\documentclass[11pt ,twoside]{amsart}  
\usepackage[left=2.5cm,right=2.5cm,top=2.5cm,bottom=2.5cm]{geometry}
\usepackage{graphicx}
\usepackage{multirow}
\usepackage{hyperref}
\usepackage{bbm}
\usepackage{amsmath,amssymb,amsfonts, amsthm,epsfig}
\usepackage{verbatim,color}
\usepackage{enumerate}
\usepackage{mathtools}
\mathtoolsset{showonlyrefs}

\numberwithin{equation}{section}

\newtheorem{thma}{Theorem}[section]
\newtheorem{lemma}[thma]{Lemma}
\newtheorem{coro}[thma]{Corollary}
\newtheorem{defi}[thma]{Definition}
\newtheorem{prop}[thma]{Propostion}

\newtheorem{claim}{Claim}
\newtheorem{remark}{Remark}
\newtheorem{condition}{Condition}

\renewcommand{\c}{c_{a,b}}

\newcommand{\g}{{\mathcal{G}^{a,b}_{\eps_1,\eps_2}}}

\renewcommand{\Re}{\mathrm{Re }}
\renewcommand{\Im}{\mathrm{Im }}

\newcommand{\eps}{\varepsilon}

\begin{document}

\title{Domino Statistics of the Two-Periodic Aztec Diamond}
\author{Sunil Chhita and Kurt Johansson}
\thanks{Both authors gratefully acknowledge the support of the Knut and Alice Wallenberg Foundation grant KAW:2010.0063. S.C. gratefully acknowledges the support of the German Research Foundation in SFB 1060 'The Mathematics of Emergent Effects'.}

\maketitle
\begin{abstract}

Random domino tilings of the Aztec diamond shape exhibit interesting features and
some of the statistical properties seen in random matrix theory. As a statistical mechanical
model it can be thought of as a dimer model or as a certain random surface. We consider the Aztec diamond
with a two-periodic weighting which exhibits all three possible phases that occur in these types of models,
often referred to as solid, liquid and gas. To analyze this model, we use entries of the inverse Kasteleyn matrix which give the
probability of any configuration of dominoes. A formula for these entries, for this particular model,
was derived by Chhita and Young (2014). In this paper, we find a major simpliﬁcation of this formula
expressing entries of the inverse Kasteleyn matrix by double contour integrals which makes it possible to investigate their asymptotics.
In a part of the Aztec diamond, where the asymptotic analysis is simpler, we use this formula to show that the entries of the inverse Kasteleyn matrix converge to the
known entries of the full-plane inverse Kasteleyn matrices for the different phases. We also study the detailed asymptotics of
the inverse Kasteleyn matrix
 at both the `liquid-solid' and `liquid-gas' boundaries, and find the extended Airy kernel in the next order asymptotics.
 Finally we provide a potential candidate
for a combinatorial description of the liquid-gas boundary.

\end{abstract}

\tableofcontents

\section{Introduction}

\subsection{Overview}

Tiling models of bounded lattice regions have been extensively researched during the last twenty years.  These tiling models are equivalent to \emph{dimer coverings} of a bipartite graph $G$, where one considers a subset of edges of $G$ so that each vertex is incident to exactly one edge while each edge present is referred to as a \emph{dimer}.

The two most commonly studied examples of these tilings are lozenge tilings, where one tiles rhombi on part of the hexagonal mesh,  and domino tilings, where one tiles dominoes (2 by 1 rectangles) on part of the square grid.  
In this paper, we study domino tilings on the Aztec diamond, where an Aztec diamond of size $n$ is all the squares of the square lattice whose centers satisfy   $ |x|+|y| \leq n$.  This model was first introduced in \cite{EKLP:92}.  

To each edge, one assigns a multiplicative weight which allows one to consider random dimer coverings:  a covering is picked with probability proportional to the product of the edge weights of the dimer covering.  This defines a discrete probability space called the \emph{dimer model}. For bipartite graphs $G$, each dimer covering encodes a three dimensional surface where the  third co-ordinate is derived from the specific dimer covering and is called the \emph{height function} \cite{Thu:90}. In this way, one gets a certain class of models of a random surface. The height function observes much of the large-scale behavior of the random dimer covers: for random dimer covers of large bounded graphs $G$, the height function tends to a deterministic limit shape \cite{CKP:01, KO:07}.  In particular, it is expected that the different phases that can occur in the model are encoded into the height function \cite{KOS:06}. Three types of limiting Gibbs measures for the dominoes/dimers are possible, often called \emph{solid}, \emph{liquid} or \emph{gas}. The difference between the three types of phases is seen in the behavior of the correlations between dominoes. They have deterministic, polynomial or exponential decay respectively in the distance between dominoes.  Note that these names for the types of phases are technical terms and the phases do not represent the physical states of matter.  

 
Uniform domino tilings of the Aztec diamond contain  solid and liquid phases and they are separated by a boundary curve, the solid-liquid boundary, which in the limit converges to an ellipse (arctic ellipse). In~\cite{Joh:05} a two-particle system, which gives a good description of the microscopic solid-liquid boundary, was used to investigate the fluctuations of this boundary. It was shown that the size of fluctuations are of 
order $n^{1/3}$ for an Aztec diamond of order $n$. Furthermore, under a suitable re-scaling, the boundary path converges to the \emph{Airy process}.  Similar situations have been observed for lozenge tilings~\cite{BF:08,OR:03,BKMM:07,Pet:13}. In particular in \cite{BF:08}, the authors introduced a general setting for a large class of models via a general particle process (with dynamics). Among other things they studied the height fluctuations in the liquid region and proved that they are given by the Gaussian free field~\cite{She:07}. These models belong to a class of processes originally formulated in~\cite{OR:03} called \emph{Schur processes} and they form a class of determinantal point processes where it is possible to 
give explicit formulas for the correlation kernel.
This class of point processes has recently been generalized to the so-called  Macdonald processes \cite{BC:11} where one leaves the determinantal framework. Here, we stay within the class of determinantal point processes, but move away from the class of Schur and Macdonald processes, by looking at a different assignment of weights to the Aztec diamond model which we call \emph{the two-periodic Aztec diamond}. This model exhibits all three phases~\cite{KO:07} and has two boundaries, one between the liquid and solid phases and one between the liquid and gas phases. Since we are able to give formulas for the inverse Kasteleyn matrix for this model that are useful for asymptotic analysis, we have the possibility of studying the 
microscopic statistical behavior of a random tiling/dimer model at a liquid-gas boundary. To our knowledge this is the first model where this is possible. In terms of the height function, the gas and solid phases correspond in the limit to facets, flat parts of the limiting height
function, whereas the liquid region corresponds to a curved surface. Before the limit, the curved surface is rough, it is expected to converge to
a Gaussian free field, the facets coming from the solid region are microscopically flat, perfect facets, whereas the facet coming from
the gas region is not completely flat but has bounded almost uncorrelated height fluctuations. Hence, the results of this paper also opens up
the possibility of studying the boundary between the rough random surface and a facet with small fluctuations.

\subsection{Definition of the model}
  
 In order to define the weights for the two-periodic Aztec diamond, we first describe (the dual of) the Aztec diamond which will be referred to as the \emph{Aztec diamond graph}.  Let 
\begin{equation}
\mathtt{W}= \{ (i,j): i \mod 2=1, j \mod 2=0, 1 \leq i \leq 2n-1, 0 \leq j \leq 2n\}
\end{equation}
and
\begin{equation}
\mathtt{B}= \{ (i,j): i \mod 2=0, j \mod 2=1, 0 \leq i \leq 2n, 1\leq j \leq 2n-1\}
\end{equation}
denote white and black vertices.  
The union, $\mathtt{B} \cup \mathtt{W}$ denotes the vertex set of the Aztec diamond graph with the edges given by $\mathtt{b} -\mathtt{w}= \pm(1,1),\pm(-1,1)$ for $\mathtt{b} \in \mathtt{B}$ and $\mathtt{w} \in \mathtt{W}$.  For an Aztec diamond of size $n=4m$ with $m \in \mathbb{N}$, define the \emph{two-periodic Aztec diamond} to be an Aztec diamond with edge  weights $a$ for all edges surrounding the faces $(i,j)$ with  $(i+j)\mod 4=2$ and edge weights $b$ for all the edges surrounding the faces $(i,j)$ with  $(i+j) \mod 4=0$; see the left figure in Fig.~\ref{fig:weights}.  In other words, the \emph{face weights}, that is the alternating product of the edge weights around each face, are given by $1,b^2/a^2,1$ and $a^2/b^2$ for the faces $(i,j-1),(i+1,j),(i,j+1)$ and $(i-1,j)$ for $(i,j) \in \mathtt{W}$ with $i+j \mod 4=3$.  

\begin{center}
\begin{figure}
\includegraphics[height=3in]{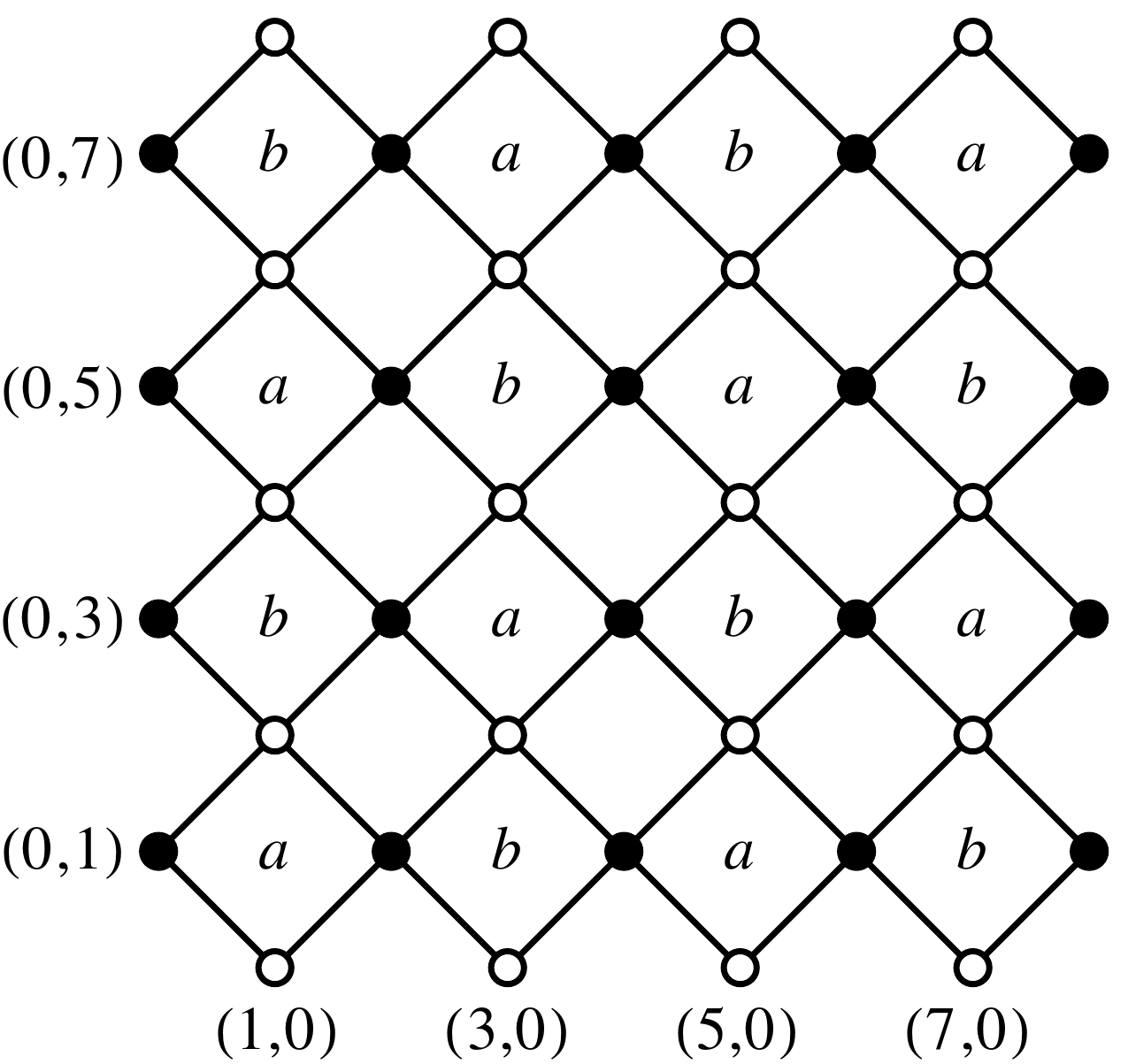}
\hspace{2mm}
\includegraphics[height=3in]{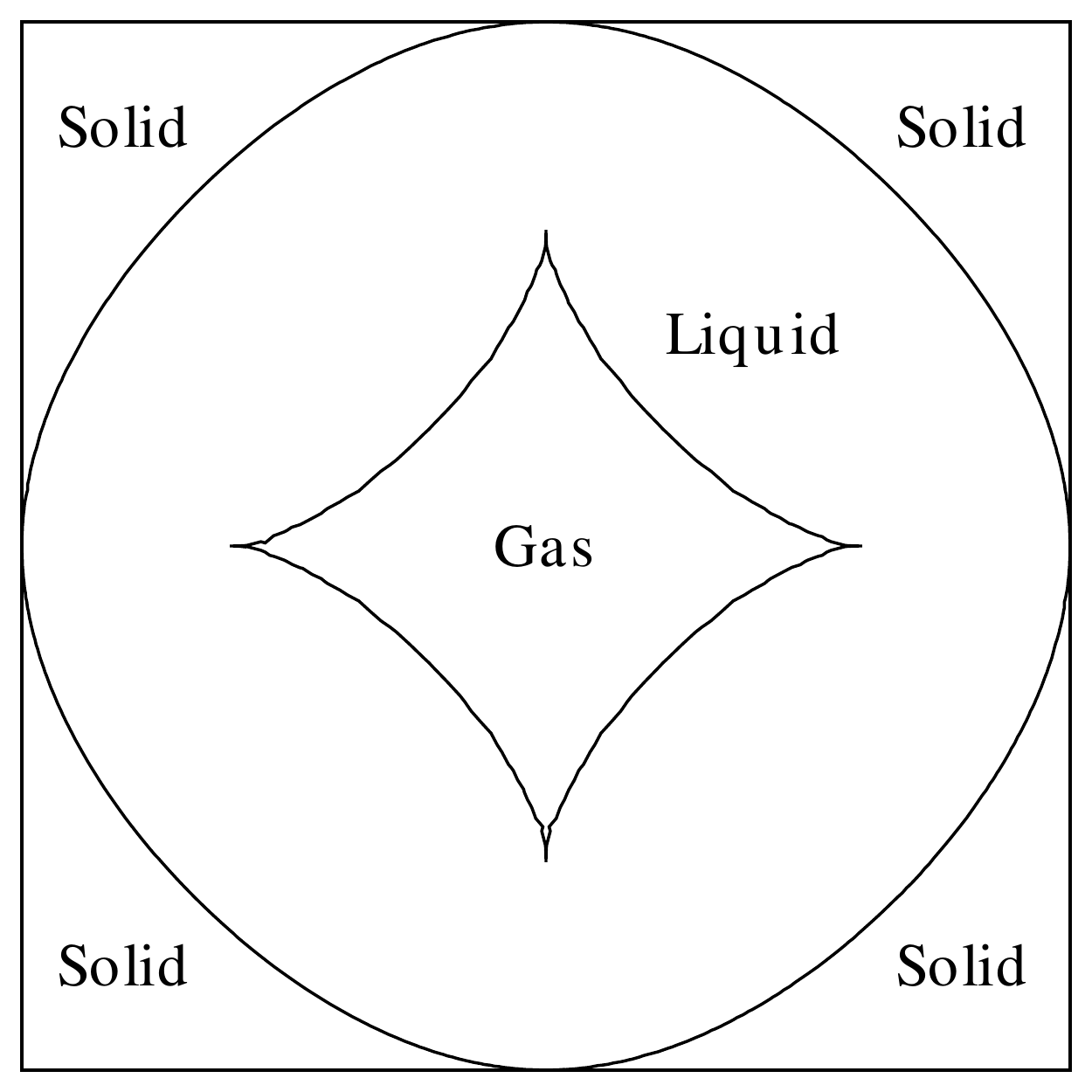}
\caption{The left figure shows the two-periodic Aztec diamond graph for $n=4$ with the edges weights  given by $a$ (or $b$) if the edge is incident to a face marked $a$ (respectively $b$). The right figure shows the limit shape when $a=0.5$ and $b=1$.}
\label{fig:weights}
\end{figure} 
\end{center}

\subsection{The Kasteleyn Approach}

To study fine asymptotic properties of tiling models, often, one tries to find the correlation functions associated with the model in the finite setting and analyze these functions as the system size gets large.
The approaches to find the correlation functions considered in \cite{Joh:05,BF:08,BV:09,BGR:10,Pet:13} and  many other tiling models are similar in many ways:  by mapping to a particle system, one computes the  correlation kernel of the  particle system by finding the inverse of the Karlin-McGregor-Lindstrom-Gessel-Viennot matrix, and using this inverse  in the Eynard Mehta theorem --- see~\cite{BR:05} for a  descriptive overview of this technique.  For Schur processes, the KMLGV matrix can be taken to be an infinite Toeplitz matrix, thus, an explicit inverse is computable using a Wiener-Hopf factorization of the symbol.  For the two-periodic Aztec diamond, the outlined approach of mapping to a particle system and using the techniques developed from random matrix theory is  mathematically complicated.  The main obstacle is finding a suitable inversion formula for a  block Toeplitz matrix since the  symbol is a two by two block matrix and Wiener-Hopf techniques are not obvious to apply.  
An alternate approach  for tiling models is to use the so-called \emph{vertex operators}~\cite{OR:03} via various commutation relations which gives methods to compute the partition function and correlation functions, as explained recently in~\cite{BCC:14, BBCCR:15}. It is not clear (at least to us), how to find a formula for the correlation functions for the two-periodic Aztec diamond using the above outlined methods.

A third approach to compute correlation kernels for tiling models, and the classical approach to dimer models,
 is via the Kasteleyn matrix. For bipartite graphs, the Kasteleyn matrix, \cite{Kas:61},
can be heuristically thought of as a signed weighted adjacency matrix whose rows and columns are indexed by black and white vertices respectively.
The sign is chosen according to a \emph{Kasteleyn orientation} of the graph.  This  means assigning a sign (or imaginary unit) to each edge weight so that the product of the edge weights around each face is negative. 
 The inverse of the Kasteleyn matrix, which we will refer to as the \emph{inverse Kasteleyn matrix} can be used to compute the correlations of dominoes using the local statistics formula for dimers on bipartite graphs found in~\cite{Ken:97}, an idea dating back to the early investigations of dimer models, e.g.~\cite{MPW:63}.  
The correlation kernel of the particles can be explicitly computed once the inverse Kasteleyn matrix is determined, compare~\cite{CJY:12}.

A derivation of the inverse Kasteleyn matrix for the two-periodic Aztec diamond considered in this paper is given in~\cite{CY:13}.
 There, the authors compute a four variable generating function whose coefficients are the entries of the  inverse Kasteleyn matrix.   Two of the variables of the generating function mark the position of the white vertex associated to the inverse Kasteleyn matrix while the other two variables mark the position of the black vertex.  Although elementary, this derivation  is computationally intensive but its main attraction is that it leads to a formula that is explicit in terms of elementary functions. We remark that  the reason behind finding a formula is due to the fact that the specific choice of edge weights, which we call two-periodic in this paper, have periodicity under four iterations of the shuffling algorithm~\cite{CY:13}. 

\subsection{Informal description of our results}

Here, we give a brief intuitive description of the model and our results.  Precise statements are found in the following section.

Random simulations of domino tilings of large two-periodic Aztec diamonds, if drawn using an appropriate choice of colors, are visually striking, since they feature all three phases.  For the model studied in this paper, the limit shape governing the separation between these phases are a family  of $8^{th}$ degree algebraic curves.  These curves can be obtained from the general machinery of Kenyon and Okounkov~\cite{KO:07},  have also recently been determined in~\cite{FSG:14},  and are also obtained from our formulas below in the standard way; e.g. see below or~\cite{OR:03} for details. The right figure in Fig.~\ref{fig:weights} shows these algebraic curves while Fig.~\ref{fig:tpn200} shows a simulation of the two-periodic Aztec diamond with two different color schemes. 
\begin{figure}
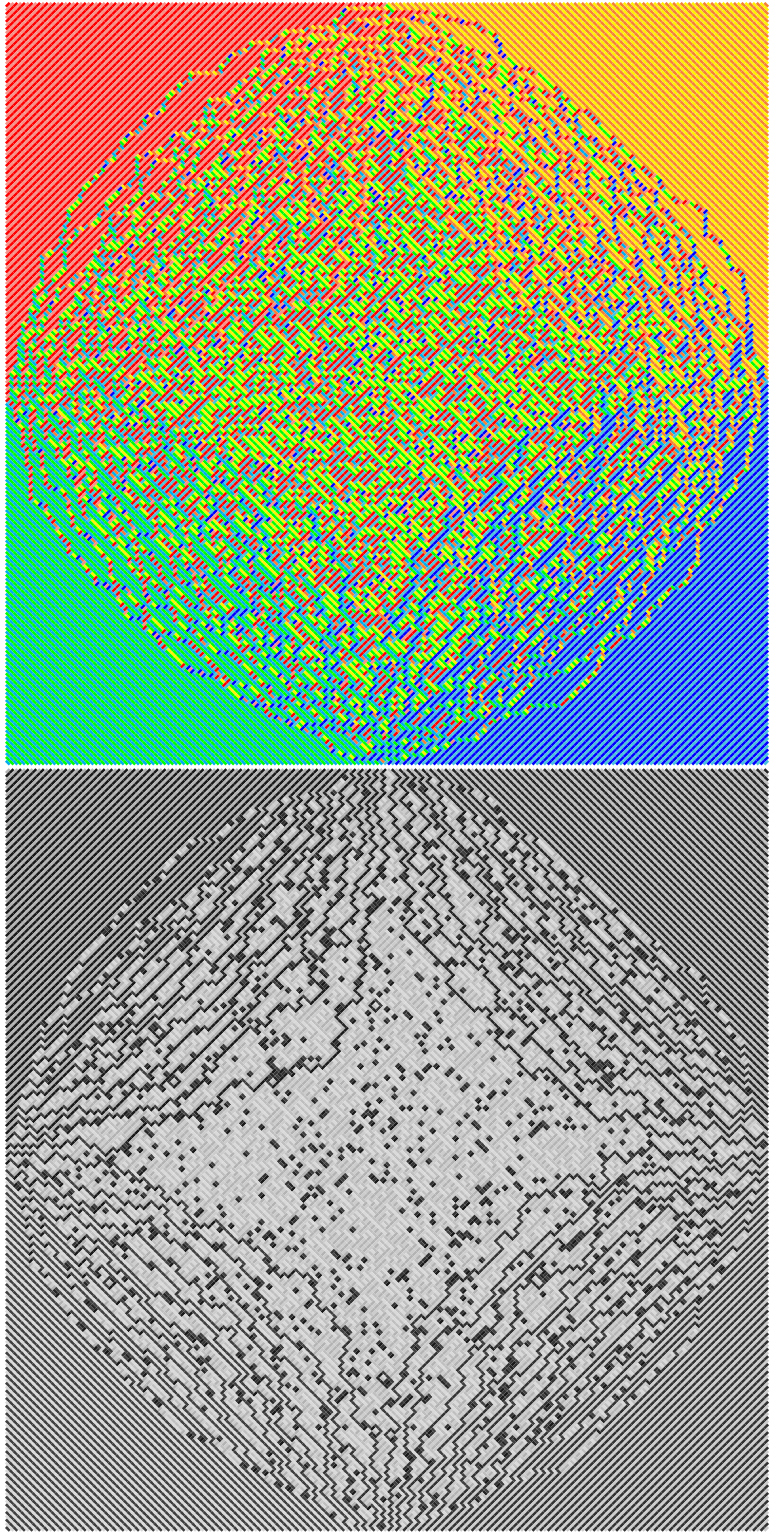

\begin{center}
\includegraphics[height=4in]{ttpn200a05.pdf}
\includegraphics[height=4in]{ttpn200a05bw.pdf}
\caption{Two different drawings of a simulation of a domino tiling of the two-periodic Aztec diamond of size 200 with $a=0.5$ and $b=1$.  The top figure contains eight different colors, highlighting the solid and liquid phases. The bottom figure contains eight different gray-scale colors to accentuate the gas phase.  }
\label{fig:tpn200}
\end{center}
\end{figure}

In this paper, we give a double contour integral formula for the inverse Kasteleyn matrix of the two-periodic Aztec diamond which follows from a major simplification of the result  in~\cite{CY:13}. This enables us to compute the leading asymptotics, as $n$ goes to infinity, of the inverse Kasteleyn matrix in each of the three regions along the main diagonal of the third quadrant of the Aztec diamond. We restrict the analysis to this line since
it is technically simpler but still allows us to investigate several of the important features of the model.
 As a consequence of this result it would be possible to deduce convergence to the limiting translation invariant Gibbs measures which have been characterized in~\cite{KOS:06}. However, we omit this argument in order not to make the paper longer than it already is.  Furthermore, we are able to find the subleading asymptotics of the inverse Kasteleyn matrix at both the liquid-solid and liquid-gas boundaries.  
At the liquid-gas boundary we find that the dominant part is given by the inverse Kasteleyn matrix for an infinite gas region plus a correction term involving the extended Airy kernel. This shows that we have strong correlation at distances  of order $n^{2/3}$ in the direction tangent to the limiting boundary curve, and  $n^{1/3}$ in the orthogonal direction, see Theorem \ref{thm:transversal1} and Corollary~\ref{coro:transversal1}. We get similar results at the liquid-solid boundary, Theorem \ref{thm:transversal2} and Corollary~\ref{coro:transversal2}, but here the dominant, ``background part'', is instead given by a coupling function for an infinite solid region.

It is possible to introduce a particle description of the tiling in a similar way as in~\cite{CJY:12} so that the liquid-solid boundary has a precise microscopic 
description in terms of last particles on appropriate lines (or, in terms of dominoes, the first domino breaking the regular brick wall pattern). In this way it should be possible, using the results of this paper, to show that
again, in an appropriate scaling limit, the liquid-solid boundary is given by an Airy process. At the liquid-gas boundary however, the situation is less clear. It is not immediate how to define a microscopic boundary although one has some indication that there is a boundary from the simulations; see the bottom figure in Fig.~\ref{fig:tpn200}.  At the liquid-gas boundary, the ``background'' is a gas phase, which is nontrivial and any configuration can occur, although its probability may be small. Hence just looking locally it is not clear how we should define a microscopic liquid-gas boundary. Still, we see from the behavior of the correlations between dominoes that the behavior is in some ways very similar at the two boundaries with the same scales and the extended Airy kernel appearing. Perhaps, at these appropriate scalings, the behavior is in some sense the same 
irrespective of whether the background is a gas phase or a solid phase. A reasonable conjecture would therefore be that with a natural and good definition of the microscopic liquid-gas boundary close to the features that we see in the pictures, we would still have convergence to an Airy process. In the last section we define a set of lattice paths which we think separate the gas and liquid phases and discuss some of their properties. These paths are global in nature and hence it is not clear how to analyze their statistical properties just using information of the inverse Kasteleyn matrix.


It is worth mentioning that models with three phases of the type studied here are well-known to both the mathematics and physics community. Note that these phases have different names in the literature. We have adopted the terminology used by Kenyon, Okounkov and Sheffield~\cite{KOS:06}.  The solid region is also known as the \emph{flat} phase, the liquid phase is often referred to as the \emph{rough} or \emph{disordered} phase and the gas phase is sometimes called the \emph{facet domain}.  
As far as we aware, Nienhuis, Hilhorst and Bl\"ote in~\cite{NHB:84} were the first to discover a model containing all three phases  through a particular solid-on-solid model.  Since then, it is expected that three phases occur for the \emph{three-periodic} lozenge tilings, the \emph{six-vertex} model with domain wall boundary conditions away from the so-called free fermion line~\cite{AR:05} (see~\cite{Res:10} for more details), and a specific parameterization in the Carroll-Peterson-Speyer grove model~\cite{PS:05,CS:04} as observed by Kenyon and Pemantle~\cite{KP:13}. 
  In all of these models, there have been no analysis (rigorous or numerical) of the behavior at the liquid-gas boundary.

\subsection{Organization}   

The rest of the paper is organized as follows:
in Section~\ref{section:Background and results} we introduce the notation and background material.  In Section~\ref{subsection:Double contour}, we state Theorem~\ref{thm:Kinverse} which gives the result for the simplified formula for entries of the inverse Kasteleyn matrix.  Theorem~\ref{thm:Kinverselimit} which gives the asymptotic  entries of the inverse Kasteleyn matrix along a particular line which passes through  the three phases and their boundaries is stated in Section~\ref{subsection:along the diagonal}. Theorem~\ref{thm:transversal1} and Theorem~\ref{thm:transversal2} which gives the asymptotic formulas for $K^{-1}_{a,1}$ at the liquid-gas and solid-liquid boundary respectively, are stated in Section~\ref{subsection:At the boundaries}.  Corollaries~\ref{coro:transversal1} and~\ref{coro:transversal2} which gives the covariance between dominoes at each of these boundaries are also stated in  Section~\ref{subsection:At the boundaries}.
We give the proofs of Theorem~\ref{thm:Kinverselimit}, Theorems~\ref{thm:transversal1} and~\ref{thm:transversal2} and Corollaries~\ref{thm:transversal1} and~\ref{thm:transversal2}  in Section~\ref{section:Asymptotics}. In Section~\ref{section:Algebra}, we give the proofs of the technical results we required for Section~\ref{section:Asymptotics}. The proof of Theorem~\ref{thm:Kinverse} is given in Section~\ref{section:Formula Simplification}.  We give our candidate for the paths which separate the liquid-gas boundary in Section~\ref{section:Combintorial}. Please note that many of our computations used computer algebra, a Mathematica file with these computations, called reviewfile.nb, is available on the arXiv.

\subsection*{Acknowledgements}
We thank Carel Faber for interesting and helpful discussions about the geometry of the octic curve. 
At various stages in this project,  we have also benefited from discussions with 
Alexei Borodin,  C\'edric Boutillier, J\'er\'emie Bouttier, Maurice Duits, Carel Faber,  Patrik Ferrari, Vadim Gorin, Martin Hairer,  Richard Kenyon, Pierre van Moerbeke,  Herbert Spohn and Craig Tracy. We would also like to thank the referee for suggestions which resulted in an improved manuscript.

\section{Background and results} \label{section:Background and results}

In this section, we provide the background material for the model and give precise statements of our results.

\subsection{Notation}
For the two-periodic  Aztec diamond, there are two types of white vertices and two types of black vertices seen from the two possibilities of edge weights around each white and each black vertex.
To distinguish between these types of vertices, we define for $i\in \{0,1\}$
\begin{equation}
 \mathtt{B}_{i} = \{ (x_1,x_2) \in \mathtt{B}: x_1+x_2 \mod 4=2i+1\}
\end{equation}
and 
\begin{equation}
  \mathtt{W}_{i} = \{ (x_1,x_2) \in \mathtt{W}: x_1+x_2 \mod 4=2i+1\}.
\end{equation}
There are four different types of dimers having weight $a$ with  $(\mathtt{W}_i,\mathtt{B}_j)$ for $i,j\in\{0,1\}$ and a further four types for dimers having weight $b$ with $(\mathtt{W}_i,\mathtt{B}_j)$ for $i,j\in\{0,1\}$. 

The Kasteleyn matrix for the two periodic Aztec diamond of size $n=4m$ with  parameters $a$ and $b$, denoted by $K_{a,b}$, is given by 
\begin{equation} \label{pf:K}
      K_{a,b}(x,y)=\left\{ \begin{array}{ll}
                     a (1-j) + b j  & \mbox{if } y=x+e_1, x \in \mathtt{B}_j \\
                     (a j +b (1-j) ) \mathrm{i} & \mbox{if } y=x+e_2, x \in \mathtt{B}_j\\
                     a j + b (1-j)  & \mbox{if } y=x-e_1, x \in \mathtt{B}_j \\
                     (a (1-j) +b j ) \mathrm{i} & \mbox{if } y=x-e_2, x \in \mathtt{B}_j\\
			0 & \mbox{if $(x,y)$ is not an edge}
                     \end{array} \right.
\end{equation}
where $\mathrm{i}^2=-1$, and we set $e_1=(1,1)$ and $e_2=(-1,1)$.    
For the rest of the paper,
we denote $\mathbb{H}$ to be the upper half-plane and set $\mathbb{H}_+$ to be the first quadrant, that is
\begin{equation}
\mathbb{H}_+ = \{ x+y\mathrm{i} : x > 0, y > 0\}.
\end{equation}
We also denote $\mathbbm{I}_{y=x}$ to be equal to 1 if $y=x$ and 0 otherwise.

\subsection{$K^{-1}$ and statistical information}

Dimers on bipartite graphs form a determinantal point process with the correlation kernel written in terms of the inverse of the Kasteleyn matrix.  
Here, we state that result for the Kasteleyn matrix given in~\eqref{pf:K}. 

Suppose that $E=\{\mathtt{e}_i\}_{i=1}^n$ are a collection of distinct edges with $\mathtt{e}_i=(\mathtt{b}_i,\mathtt{w}_i)$, where $\mathtt{b}_i$ and $\mathtt{w}_i$ denote black and white vertices. 
\begin{thma}[\cite{Ken:97}]\label{localstatisticsthm}
    The dimers form a determinantal point process on the edges of the Aztec diamond graph with correlation kernel given by
    \begin{equation}
	  L(\mathtt{e}_i,\mathtt{e}_j) = K_{a,b}(\mathtt{b}_i,\mathtt{w}_i) K_{a,b}^{-1} (\mathtt{w}_j,\mathtt{b}_i) 
    \end{equation}
 where $K_{a,b}(\mathtt{b},\mathtt{w}) = (K_{a,b})_{\mathtt{b}\mathtt{w}}$ and $K_{a,b}^{-1}(\mathtt{w},\mathtt{b})= (K^{-1}_{a,b})_{\mathtt{w}\mathtt{b}}$.
\end{thma}
The above formula is sometimes referred to as \emph{Kenyon's formula}, but the computation of correlation between dimers using the inverse Kasteleyn matrix goes back to the early days of the Kasteleyn approach.

\subsection{Gibbs measure and local statistics} \label{subsec:Gibbs}

In~\cite{KOS:06}, the authors determine the translation invariant Gibbs measure of any dimer model on a bipartite graph embedded on the plane, give techniques to compute the full plane inverse Kasteleyn matrix, and show that the possible slopes of the associated height function classify the Gibbs measures.  Conjecturally, for tiling models on bounded domains, the Gibbs measure that is seen in the limit at a certain point is determined by
the slope of the limiting height function at that point; see~\cite{KOS:06,KO:07}. This classification does not give any information on the more detailed
behavior at the boundaries between different regions, which can occur in tiling models on bounded domains.

The results from \cite{KOS:06} rely on using the smallest non-repeating unit of the graph called the \emph{fundamental domain}.  The graph considered in this paper, has a $2 \times 2$ fundamental domain, that is, the fundamental domain of the graph consists of two black and two white vertices.  More precisely, let 
\begin{equation}
\tilde{\mathtt{W}}= \{ (i,j) \in\mathbb{Z}^2: i \mod 2=1, j \mod 2=0\},
\end{equation}
\begin{equation}
\tilde{\mathtt{B}}= \{ (i,j)\in\mathbb{Z}^2: i \mod 2=0, j \mod 2=1\},
\end{equation}
and for  $i\in \{0,1\}$
\begin{equation}
\tilde{ \mathtt{B}}_{i} = \{ (x_1,x_2) \in \tilde{\mathtt{B}}: x_1+x_2 \mod 4=2i+1\}
\end{equation}
and 
\begin{equation}
  \tilde{\mathtt{W}}_{i} = \{ (x_1,x_2) \in \tilde{\mathtt{W}}: x_1+x_2 \mod 4=2i+1\}.
\end{equation}

We choose the fundamental domain of the graph embedded on the plane to have vertices $\mathtt{w}$, $\mathtt{w}+e_1$,$\mathtt{w}+e_2$ and $\mathtt{w}+e_1+e_2$ for $\mathtt{w} \in \tilde{\mathtt{W}}_0$ with the same Kasteleyn orientation and weighting as chosen for the Aztec diamond graph.

To describe the Gibbs measure, the authors in \cite{KOS:06} introduce  \emph{magnetic coordinates}.  For the fundamental domain considered here, we represent the magnetic coordinates by   $(B_1,B_2)$
 where each edge weight in the neighboring fundamental domain in the direction $e_1$ (or resp. $e_2$), is mulitplied by  $e^{B_1}$ (or $e^{B_2}$ resp.).
  Conversely,   each edge weight in the neighboring fundamental domain in the direction $-e_1$ (or $-e_2$ resp.) is multiplied  by   $e^{-B_1}$ (or $e^{-B_2}$ resp.). These 
magnetic coordinates are related 
to the average slope, that is, the Gibbs measures are characterized by the magnetic coordinates; see \cite{KOS:06}. For example, $B_1=0$ and $B_2=0$ correspond to zero average slope in both directions. 
Since our later results are stated for $b=1$, we only consider the fundamental domain for $b=1$.

\begin{center}
\begin{figure}
\includegraphics[height=2.5in]{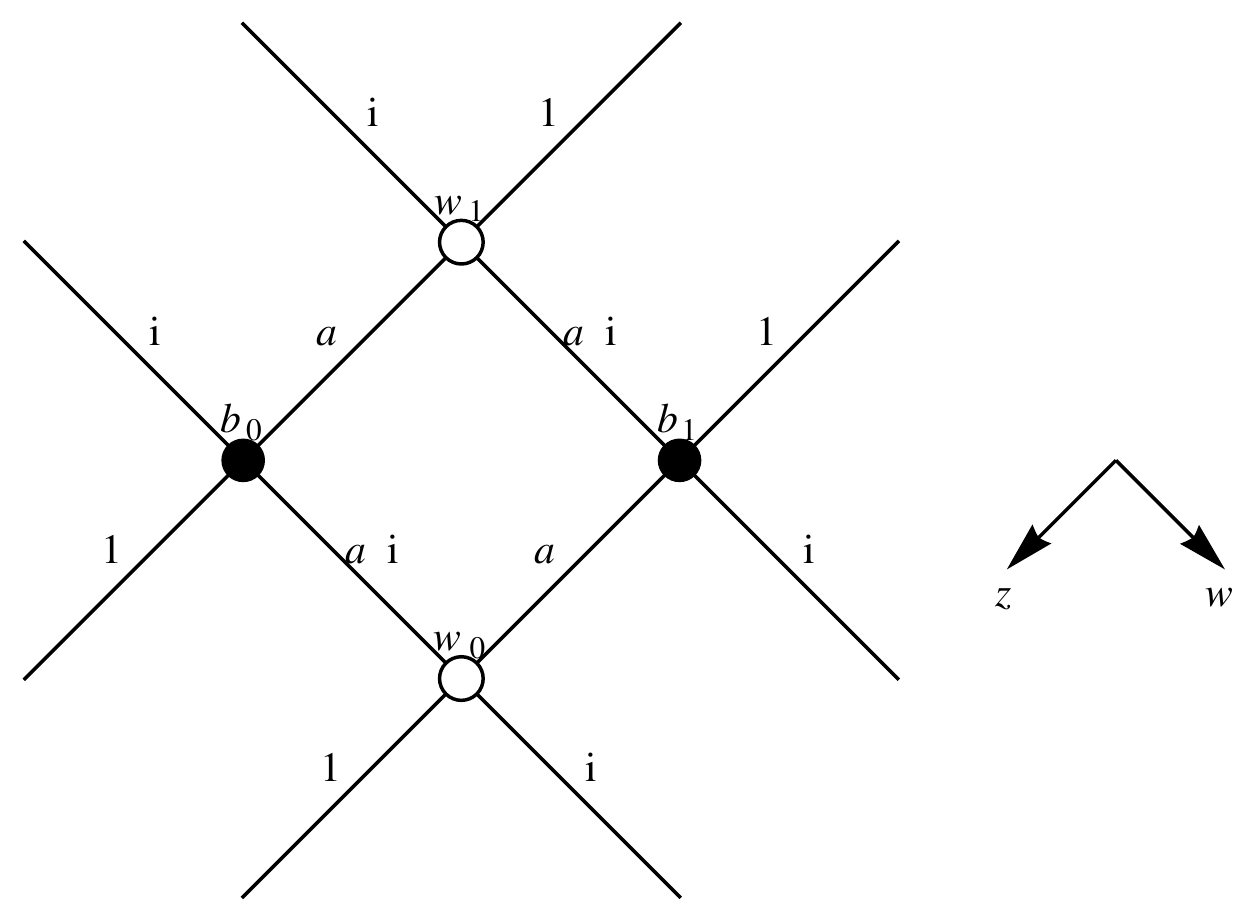}
\caption{The fundamental domain for the two-periodic Aztec diamond with the edge weights and Kasteleyn orientation. We have the labeling $w_i \in \mathtt{W}_i$ and $b_i \in \mathtt{B}_i$ for $i\in\{0,1\}$.}
\label{fig:fundamental}
\end{figure} 
\end{center}

Let $K(z,w)$ denote the Kasteleyn matrix for the above fundamental domain where $1/z$ is the multiplicative factor when crossing to a fundamental domain in the direction $e_1$ and $1/w$ is multiplicative factor when crossing to the fundamental domain in the direction $e_2$; see Fig.~\ref{fig:fundamental}.  Explicitly, we have
\begin{equation}
	K(z,w)=\left( \begin{array}{cc}
	 \mathrm{i}(a +  w^{-1})&a+z \\
	a+z^{-1} &  \mathrm{i}(a +  w) \end{array} \right).
\end{equation}
Suppose that $x\in \tilde{\mathtt{W}}_{\alpha_1}$ and $y\in \tilde{\mathtt{B}}_{\alpha_2}$ for $\alpha_{1},\alpha_2 \in \{0,1\}$ with the translation to get to the fundamental domain containing $y$ from the fundamental domain containing $x$ given by $ue_1 +ve_2$.
The whole plane inverse Kasteleyn matrix for the entries $x$ and $y$ with magnetic coordinate $(\log r_1, \log r_2)$ is denoted by $\mathbb{K}_{r_1,r_2}^{-1}(x,y)$.  Using the techniques from~\cite{KOS:06}, we have
\begin{equation} \label{res:eq:Kinvgas1}
\mathbb{K}_{r_1,r_2}^{-1}(x,y)= \frac{1}{(2\pi \mathrm{i})^2}
\int_{\Gamma_{r_1}} \frac{dz}{z} \int_{\Gamma_{r_2}} \frac{dw}{w}  [(K(z,w))^{-1}]_{\alpha_1+1,\alpha_2+1}z^uw^v,
\end{equation}
where $\Gamma_{r}$ denotes a positively oriented contour around $0$ with radius $r$.
In the above formula, $(K(z,w))^{-1}$ is the inverse of $K(z,w)$ and is given explicitly by
\begin{equation}\label{res:eq:Kinvgas2}
(K(z,w))^{-1}=\frac{1}{P(z,w)} 
\left(
\begin{array}{cc}
 \mathrm{i} (a+w) & -(a+z) \\
 -(a+ z^{-1}) & \mathrm{i} (a+ w^{-1})\\
\end{array}
\right)
\end{equation}
where
\begin{equation}
P(z,w)=-2-2 a^2-\frac{a}{w}-a w-\frac{a}{z}-a z.
\end{equation}
The function $P(z,w)$ is called the \emph{characteristic polynomial}~\cite{KOS:06}.

Using the results of~\cite{KOS:06},
it is worth noting that if the magnetic coordinates are given by $(0,0)$, then the model is in the gas region while if the magnetic coordinates are given by $(\log r_1,\log r_2)$ with $P(z,w)=0$ for some $(z,w) \in \Gamma_{r_1}\times \Gamma_{r_2}$ (from~\cite{KOS:06} the value of $(z,w)\in \Gamma_{r_1}\times \Gamma_{r_2}$ with $P(z,w)=0$ is real), then the model is in the liquid region.

\subsection{Double contour integral formula for $K_{a,1}^{-1}$}\label{subsection:Double contour}

In~\cite{CY:13}, the authors derive a four variable generating function whose coefficients are entries of the matrix $K^{-1}_{a,1}$ for $n=4m$.
  Two variables of the generating function mark the position of the white vertex while the remaining two variables mark the position of the black vertex.  

Our first main result, Theorem~\ref{thm:Kinverse}, is to derive a much simpler expression for the inverse Kasteleyn matrix.    We find that the formula of the inverse Kasteleyn matrix contains five double contour integral formulas, four of which are related by symmetry while the remaining term is given by $\mathbb{K}_{1,1}^{-1}$. 
The formula  in~\cite{CY:13} is long, and our derivation of a simpler formula primarily involves extracting coefficients of the generating function using quadruple contour integrals.  By exploiting symmetries of the model, we greatly reduce the number of quadruple contour integrals we need to reduce to double contour integrals. The reduction of the remaining quadruple integrals to double contour integrals hinges on the expression
\begin{equation}
\label{res:eq:charc}
\tilde{c}(u,v)=2(1+a^2)+a(u+u^{-1})(v+v^{-1})
\end{equation}
which is related to the characteristic polynomial under a change of variables; see Eq.~\eqref{alg:eq:char} for details.

The second main result, Corollary~\ref{corollary}, is to rewrite the four double contour integrals described above into a form that is good for saddle point analysis.  We give the prerequisite  notation for these results.
Let
\begin{equation}\label{parameter:c}
c=\frac a{1+a^2}.
\end{equation}
 This expression in the parameter $a$ will occur throughout the paper.  The results in this paper are stated and proved for $a \in (0,1)$ which means $c \in (0,1/2)$. However, similar results can be formulated for $a \in (1, \infty)$. 
Define
\begin{equation}\label{cuts:squareroot}
\sqrt{\omega^2+2c}=e^{\frac{1}{2} \log(\omega+\mathrm{i} \sqrt{2c}) +\frac{1}{2} \log(\omega-\mathrm{i} \sqrt{2c})}
\end{equation}
for $\omega \in \mathbb{C}\backslash \mathrm{i} [-\sqrt{2c},\sqrt{2c}]$
where the logarithm in the exponent has arguments in $(-\pi/2,3\pi/2)$. If we write $\sqrt{1/\omega^2+2c}$ we mean the same square root evaluated at $1/\omega$.
The following function will play an important role in many of our formulas
and computations below. Set
\begin{equation}\label{cuts:def:G}
G(\omega)=\frac{1}{\sqrt{2c}} \left(\omega -\sqrt{\omega^2+2c} \right).
\end{equation}
Recall that $\Gamma_{r}$ represents a positively oriented circle around $0$ with radius $r$. Define for $s \in\mathbb{Z}$
\begin{equation}\label{cuts:def:F}
F_s(w)=\frac{1}{2\pi \mathrm{i}} \int_{\Gamma_1} \frac{u^s}{1+a^2+a w(u+u^{-1})} \frac{du}{u}.
\end{equation}
The above function appears when we integrate two of the integrals in the quadruple contour integrals to get a double contour integral formula.
Notice that, by the change of variables $u\mapsto u^{-1}$, $F_{-s}(w)=F_{s}(w)$. 
\begin{lemma}\label{cuts:lem:F}
The function $F_s(\mathrm{i}\omega/\sqrt{2 c})$ is analytic in $\mathbb{C}\backslash ( \mathrm{i} (-\infty,1/\sqrt{2c}]
\cup \mathrm{i} [1/\sqrt{2c},\infty))$ and
\begin{equation}
F_s\left(\frac{\mathrm{i} \omega}{\sqrt{2c}} \right)= \frac{\mathrm{i}^{|s|}}{(1+a^2)\omega \sqrt{1/\omega^2+2c}}G \left(\frac{1}{\omega} \right)^{|s|} \hspace{20mm}\mbox{for } s\in\mathbb{Z}.
\end{equation}
Furthermore, for $\nu(u)=(u+1/u)/2$
\begin{equation}
	F_s (\nu(u))= 2 \frac{1}{2\pi i} \int_{\Gamma_1} \frac{ v^s}{\tilde{c}(u,v)} \frac{dv}{v}\hspace{20mm}\mbox{for } s\in\mathbb{Z}.
\end{equation}
where $\tilde{c}(u,v)$ is defined in Eq.~\eqref{res:eq:charc}.
\end{lemma}
The proof of this lemma is given in Section~\ref{section:Algebra}.  We set
\begin{equation}
\mu(u) = 1-\sqrt{1+c^2(u- u^{-1})^2 } \label{results:eq:mu} \hspace{5mm}\mbox{and}\hspace{5mm}s(u)=1-\mu(u),
\end{equation}
with the sign chosen so that $G(1/\omega)=-\mu(-u\mathrm{i})/(c(u+1/u))$ if we set $u\mathrm{i}=G(\omega)$.

We now state our theorem for the entries of $K^{-1}_{a,1}$.

\begin{thma} \label{thm:Kinverse}
For $n=4m$, $x=(x_1,x_2) \in \mathtt{W}_{\eps_1}$ and $y=(y_1,y_2) \in \mathtt{B}_{\eps_2}$ with $\eps_1,\eps_2 \in \{0,1\}$,
the entries of $K^{-1}_{a,1}$ are given by 
\begin{equation} \label{thm:eq:Kinverse}
\begin{split}
&K_{a,1}^{-1}((x_1,x_2),(y_1,y_2))= \mathbb{K}^{-1}_{1,1} ((x_1,x_2),(y_1,y_2))
-\Big( \mathcal{B}_{\eps_1,\eps_2}(a,x_1,x_2,y_1,y_2) 
\\& -\frac{\mathrm{i}}{a} (-1)^{\eps_1+\eps_2} \left( \mathcal{B}_{1-\eps_1,\eps_2}(1/a,2n-x_1,x_2,2n-y_1,y_2)   +\mathcal{B}_{\eps_1,1-\eps_2}(1/a,x_1,2n-x_2,y_1,2n-y_2) \right) \\
&  +\mathcal{B}_{1-\eps_1,1-\eps_2}(a,2n-x_1,2n-x_2,2n-y_1,2n-y_2) \Big),
\end{split}
\end{equation}
where for $a<r<1/a$, we have
\begin{equation}\label{res:eq:B}
\begin{split}
&\mathcal{B}_{\eps_1,\eps_2}(a,x_1,x_2,y_1,y_2) =
-\frac{\mathrm{i}^{\eps_1+\eps_2+1}}{(2 \pi \mathrm{i})^2} \int _{\Gamma_r} \frac{du_1}{u_1}\int_{\Gamma_r} \frac{du_2}{u_2} \frac{ F_{\frac{x_2}{2}}(\nu (u_1)) F_{\frac{y_1}{2}} (\nu(u_2))}{u_1^{\frac{x_1-1}{2}} u_2^{\frac{y_2-1}{2}}}
\\
&\times \left(\frac{1}{4 c^2} u_1^2 u_2^2 (2-\mu(-\mathrm{i}u_1))^2( 2-\mu(-\mathrm{i}u_2))^2 \right)^{m}
\sum_{\gamma_1,\gamma_2=0}^1Y^{\eps_1,\eps_2}_{\gamma_1,\gamma_2}(u_1,u_2).
\end{split}
\end{equation}
The function $Y_{\gamma_1,\gamma_2}^{\eps_1,\eps_2}$ is defined in~\eqref{asfo:eq:Y}.
\end{thma}
The proof of this result is given in Section~\ref{section:Formula Simplification}. The proof, combined with the proof of \cite[Theorem 6.2]{CY:13} gives a complete derivation of the double contour integral formulas for entries of $K^{-1}_{a,1}$.

In Eq.~\eqref{res:eq:B}, $\sum_{\gamma_1,\gamma_2=0}^1Y^{\eps_1,\eps_2}_{\gamma_1,\gamma_2}(u_1,v_2)$ is a relatively long but explicit expression which is not dependent on $n$. We postpone its definition. 
Roughly speaking, the formula  for each entry of $K_{a,1}^{-1}$  as given in~\eqref{thm:eq:Kinverse} can be thought of as the formula for the inverse Kasteleyn matrix for the whole plane gas region plus four parts which relate to the boundary of the Aztec diamond. The next corollary gives formulas for $ \mathcal{B}_{\eps_1,\eps_2}$ more suitable for asymptotic analysis. Define
\begin{equation} \label{asfo:lem:H}
{ H_{x_1,x_2}(\omega)}=\frac{\omega^{2m}  G \left(\omega \right)^{2m-\frac{x_1}{2}}}{G \left(\omega^{-1} \right)^{2m-\frac{x_2}{2}}},
\end{equation}
where $0 < x_1,x_2 < n$.  We have

\begin{coro} \label{corollary}
For $n=4m$, let $x=(x_1,x_2) \in \mathtt{W}_{\eps_1}$ and $y=(y_1,y_2) \in \mathtt{B}_{\eps_2}$ with $\eps_1,\eps_2 \in \{0,1\}$.  Suppose that  $0 < x_1,x_2,y_1,y_2 < n$, $\sqrt{2c}<p<1$ and  $Q^{\eps_1,\eps_2}_{\gamma_1,\gamma_2}$ is as given in~\eqref{asfo:eq:Q}. We have
\begin{equation}
~\label{B}
\mathcal{B}_{\eps_1,\eps_2} (a,x_1,x_2,y_1,y_2) =  \frac{ \mathrm{i}^{\frac{x_2-x_1+y_1-y_2}{2}}}{(2\pi \mathrm{i})^2}  \int_{\Gamma_{p}}\frac{d \omega_1}{\omega_1} \int_{\Gamma_{1/p}}  d \omega_2 \frac{\omega_2  }{\omega_2^2-\omega_1^2}\frac{ H_{x_1+1,x_2}(\omega_1)}{H_{y_1,y_2+1}(\omega_2)} \sum_{\gamma_1,\gamma_2=0}^1 Q_{\gamma_1,\gamma_2}^{\eps_1,\eps_2}(\omega_1 ,\omega_2),
\end{equation}
 
\begin{equation}
\begin{split}~\label{asfo:lemproof:exponential:change3}
&\frac{1}{a}\mathcal{B}_{1-\eps_1,\eps_2}(a^{-1},2n-x_1,x_2,2n-y_1,y_2)=-\frac{\mathrm{i}^{\frac{x_1-x_2-y_1-y_2}{2}}}{(2 \pi \mathrm{i})^2}\\
&\times \int _{\Gamma_p} \frac{d\omega_1}{\omega_1}\int_{\Gamma_{1/p}} {d\omega_2} \frac{\omega_2}{\omega_2^2-\omega_1^2} \frac{ H_{x_1+1,x_2}(\omega_1)}{H_{2n-y_1,y_2+1}(\omega_2)} \sum_{\gamma_1,\gamma_2=0}^1  (-1)^{\eps_2+\gamma_2}Q_{\gamma_1,\gamma_2}^{\eps_1,\eps_2}(\omega_1 ,\omega_2),
\end{split}
\end{equation}
\begin{equation}
\begin{split}~\label{asfo:lemproof:exponential:change4}
&\frac{1}{a}\mathcal{B}_{\eps_1,1-\eps_2}(a^{-1},x_1,2n-x_2,y_1,2n-y_2)=-\frac{\mathrm{i}^{\frac{y_2-y_1-x_2-x_1}{2}}}{(2 \pi \mathrm{i})^2}\\
& \times \int _{\Gamma_p} \frac{d\omega_1}{\omega_1}\int_{\Gamma_{1/p}} {d\omega_2} \frac{\omega_2}{\omega_2^2-\omega_1^2} \frac{ H_{x_1+1,2n-x_2}(\omega_1)}{H_{y_1,y_2+1}(\omega_2)} \sum_{\gamma_1,\gamma_2=0}^1  (-1)^{\eps_1+\gamma_1}Q_{\gamma_1,\gamma_2}^{\eps_1,\eps_2}(\omega_1 ,\omega_2)
\end{split}
\end{equation}
and
\begin{equation}
\begin{split}~\label{asfo:lemproof:exponential:change5}
&\mathcal{B}_{1-\eps_1,1-\eps_2}(a,2n-x_1,2n-x_2,2n-y_1,2n-y_2)=-\frac{\mathrm{i}^{\frac{y_2+y_1+x_2+x_1}{2}}}{(2 \pi \mathrm{i})^2}\\
&\times \int _{\Gamma_p} \frac{d\omega_1}{\omega_1}\int_{\Gamma_{1/p}} {d\omega_2} \frac{\omega_2}{\omega_2^2-\omega_1^2} \frac{ H_{x_1+1,2n-x_2}(\omega_1)}{H_{2n-y_1,y_2+1}(\omega_2)} \sum_{\gamma_1,\gamma_2=0}^1  (-1)^{\eps_1+\gamma_1+\eps_2+\gamma_2}Q_{\gamma_1,\gamma_2}^{\eps_1,\eps_2}(\omega_1 ,\omega_2).
\end{split}
\end{equation}

\end{coro}

The formulas in Corollary~\ref{corollary} contain the term $Q^{\eps_1,\eps_2}_{\gamma_1,\gamma_2}$ which is independent of $n$ and depends on $x,y$ only through $\varepsilon_1, \varepsilon_2$.
We remark that the formulas given in Corollary~\ref{corollary} only hold for $0 < x_1,x_2,y_1,y_2 <n$, that is, the third quadrant of the Aztec diamond. Similar formulas can be found for the remaining quadrants of the Aztec diamond but we will not state them.

\subsection{$K_{a,1}^{-1}$ along the diagonal}\label{subsection:along the diagonal}

To state the asymptotic results, we need some further notation used throughout the paper.
For $\eps_1,\eps_2 \in \{0,1\}$, denote
\begin{equation}\label{h}
h(\eps_1,\eps_2) = \eps_1(1-\eps_2) + \eps_2 (1-\eps_1).
\end{equation}
 For  $k,\ell\in\mathbb{Z}$, we define
\begin{equation}\label{Skl}
\tilde{S}(k,\ell)=\frac{\mathrm{i}^{-k-\ell}}{2(1+a^2)(2\pi\mathrm{i})^2}\int_{\Gamma_R}d\omega_1\int_{\Gamma_R}d\omega_2
\frac{G(\omega_1)^{\ell}G(\omega_2)^k}{(1-\omega_1\omega_2)\sqrt{\omega_1^2+2c}\sqrt{\omega_2^2+2c}},
\end{equation}
where $R>1$.  
Define, for $x=(x_1,x_2) \in \mathtt{W}_{\eps_1}$ and $y=(y_1,y_2) \in \mathtt{B}_{\eps_2}$,
\begin{equation}\label{Sxy}
S(x,y)=-\mathrm{i}^{1+h(\varepsilon_1,\varepsilon_2)}\left(a^{\varepsilon_2}\tilde{S}(k_1,\ell_1)+a^{1-\varepsilon_2}\tilde{S}(k_2,\ell_2)\right),
\end{equation}
where
\begin{equation}\label{kldef}
k_1=\frac{x_2-y_2-1}{2}+h(\varepsilon_1,\varepsilon_2)\,,\,\ell_1=\frac{y_1-x_1-1}{2}\,,\,
k_2=\frac{x_2-y_2+1}{2}-h(\varepsilon_1,\varepsilon_2)\,,\,\ell_2=\frac{y_1-x_1+1}{2}.
\end{equation}

The function $S(x,y)$ can be interpreted as an `inverse Kasteleyn matrix' for a whole plane solid phase which is seen from the following lemma.
\begin{lemma}\label{lem:Solidphase}
We have that
\begin{equation}\label{Sxyidentity}
S(x,y)=0 \hspace{5mm} \mbox{if $y_1-x_1> -1$ or $x_2-y_2 > -1$}
\end{equation}
Also, if $f\in\{\pm e_1,\pm e_2\}$ and $x\in \mathtt{W}_{\eps_1}$, then $S(x,x+f)\neq 0$ only if $f=e_2$, and then $S(x,x+e_2)=-\mathrm{i}a^{\eps_1-1}$. This shows that 
$S(x,y)$ can be thought of as an `inverse Kasteleyn matrix' or  a coupling function for a frozen phase.
\end{lemma}
The proof of this lemma is given in Section~\ref{Sec:Gaussianlimitsproofs}.

To consider asymptotic results, we must rescale the Aztec diamond. The following condition gives the first rescaling considered in this paper. Recall $n=4m$.
\begin{condition}\label{condition1}
Rescale and center the Aztec diamond so that $x=([4m+2m\xi_1]+\overline{x}_1,[4m+2m\xi_2]+\overline{x}_2) \in \mathtt{W}_{\eps_1}$ and $y=([4m+2m\xi_1]+\overline{y}_1,[4m+2m\xi_2]+\overline{y}_2) \in \mathtt{B}_{\eps_2}$ for $\eps_1,\eps_2\in \{0,1\}$ and $-1<\xi_1,\xi_2<1$ where $\overline{x}_1,\overline{x}_2,\overline{y}_1,\overline{y}_2$ are at most order $m^{2/3}$ if $(\xi_1,\xi_2)$ is on the liquid-gas or solid-liquid boundary and at most order $m^{1/2}$ otherwise.
\end{condition}
In Condition~\ref{condition1}, $\xi_1$ and $\xi_2$ are the \emph{asymptotic coordinates} of the rescaled  Aztec diamond. The asymptotic results are based on asymptotic analysis of the integrals in Eqs.~\eqref{B},~\eqref{asfo:lemproof:exponential:change3},~\eqref{asfo:lemproof:exponential:change4} and~\eqref{asfo:lemproof:exponential:change5}.
The asymptotic analysis of~\eqref{B} uses the method of steepest descent.   Under Condition~\ref{condition1}, the saddle point function of the integral in Eq.~\eqref{B} is given by
\begin{equation}\label{asym:eq:saddle}
g_{\xi_1,\xi_2}(\omega)= \log \omega- \xi_1 \log G \left( \omega \right) + \xi_2 \log G \left( \omega^{-1} \right).
\end{equation}
This function has a critical point when 
\begin{equation}\label{asym:eq:saddleeqn}
\omega g'_{\xi_1,\xi_2} (\omega)=1+ \frac{ \omega}{\sqrt{ \omega^2+2c}}\xi_1 + \frac{1}{\omega \sqrt{\omega^{-2}+2c}} \xi_2 =0.
\end{equation}  

The condition that $g_{\xi_1,\xi_2}$ has a double critical point leads to an
equation for $\xi_1,\xi_2$, which is the equation of the $8^{th}$ degree
curve written in Appendix~\ref{limitshape} and is also depicted in
Fig.~\ref{fig:weights}.  This curve gives the boundaries of the different
phases although we do not fully prove this here.
If we think of $c$ as
complex parameter it can be shown that
the nonsingular model of the plane curve
is a genus one curve with a marked
point, that is an elliptic curve. In addition, it has a marked
point that is of order 4 for the group structure.
Varying $c$ gives a one-parameter family of elliptic curves which
maps onto the moduli space of all elliptic curves (in fact onto
that of elliptic curves with a point of order 4)~\cite{Fab}.

To make the analysis explicit, we only consider the asymptotic behavior on the line $\xi_1=\xi_2=\xi$ in the bottom left quadrant of the Aztec diamond, that is $-1 <\xi<0$.  In this case, we have a nice explicit description of the critical points of $g_{\xi,\xi}$ in $\mathbb{H}_+\cup \mathbb{R}_{>0} \cup (\sqrt{2c},1/\sqrt{2c})\mathrm{i}$:
\begin{enumerate}
\item two roots of~\eqref{asym:eq:saddleeqn}, $\omega_c$ and $-1/\omega_c$ with $\omega_c \in \mathrm{i}(\sqrt{2c},1) $ corresponds to the gas region,
\item a double root  of~\eqref{asym:eq:saddleeqn} at $\omega_c=\mathrm{i}$ corresponds to a liquid-gas boundary,
\item  a single root  of~\eqref{asym:eq:saddleeqn}, $\omega_c \in \mathbb{H}_+$, with $|\omega_c|=1$  corresponds to the liquid region,
\item  a double root  of~\eqref{asym:eq:saddleeqn} at $\omega_c=1$ corresponds to the solid-liquid boundary and
\item two roots  of~\eqref{asym:eq:saddleeqn}, $\omega_c$ and $1/\omega_c$ with $\omega_c \in (1,\infty)$ corresponds to the solid region.
\end{enumerate}
The three remaining integrals given in Eqs.~\eqref{asfo:lemproof:exponential:change3},~\eqref{asfo:lemproof:exponential:change4} and~\eqref{asfo:lemproof:exponential:change5} will be shown to be exponentially small in the bottom left quadrant of the Aztec diamond; see Lemma~\ref{asfo:lem:exponential}.  Similar situations happen for the remaining three quadrants of the Aztec diamond, that is, one of the integrals in~\eqref{B},~\eqref{asfo:lemproof:exponential:change3},~\eqref{asfo:lemproof:exponential:change4} or~\eqref{asfo:lemproof:exponential:change5} give a contribution while the remaining three are negligible.

Our first asymptotic result gives the leading order terms  for the elements of the inverse Kasteleyn matrix.

\begin{thma} \label{thm:Kinverselimit}
Choose $x$ and $y$ as given in Condition~\ref{condition1} with $\xi_1=\xi_2=\xi$ and $-1<\xi<0$. We find
\begin{equation}
K^{-1}_{a,1}(x,y) = \left\{ \begin{array}{ll}
\mathbb{K}^{-1}_{1,1}\left((\overline{x}_1 ,\overline{x}_2) ,(\overline{y}_1 ,\overline{y}_2) \right)+O(e^{- c_1m}) & \mbox{if }  -\frac{1}{2}\sqrt{1-2c}< \xi <0, \mbox{with }c_1>0\\
\mathbb{K}^{-1}_{1,1}((\overline{x}_1 ,\overline{x}_2) ,(\overline{y}_1 ,\overline{y}_2) )+O(m^{-1/3}) & \mbox{if } \xi=-\frac{1}{2}\sqrt{1-2c} \\
\mathbb{K}^{-1}_{r_1,r_2} ((\overline{x}_1 ,\overline{x}_2) ,(\overline{y}_1 ,\overline{y}_2) )+O(m^{-1/2})& \mbox{if } -1/2\sqrt{1+2c}<\xi <-1/2 \sqrt{1-2c}\\
S(x,y) +O(m^{-1/3}) &\mbox{if }   \xi = -1/2\sqrt{1+2c}\\
S(x,y) +O(m^{-1/3}) &\mbox{if }   \xi < -1/2\sqrt{1+2c},\\
\end{array}
\right. 
\end{equation}
where $r_1=1$ and $r_2=1/|G(\omega_c)|^2$ with $\omega_c\in\mathbb{T}$ equal to the solution in $\mathbb{H}_+$ of Eq.~\eqref{asym:eq:saddleeqn} with $\xi_1=\xi_2=\xi$, and $S(x,y)$ is defined in~\eqref{Sxy}.
\end{thma}
In the above theorem, we have that
\begin{enumerate}
\item $-\frac{1}{2} \sqrt{1-2c} <\xi <0$ corresponds to the gas region,
\item $\xi=-\frac{1}{2} \sqrt{1-2c} $ corresponds to the liquid-gas boundary,
\item $-\frac{1}{2} \sqrt{1+2c} <\xi <-\frac{1}{2} \sqrt{1-2c}$ corresponds to the liquid region,
\item $\xi=-\frac{1}{2} \sqrt{1+2c} $ corresponds to the solid-liquid boundary, and
\item $ \xi<-\frac{1}{2} \sqrt{1+2c}$ corresponds to the solid region.
\end{enumerate}

The proof of the above theorem is found in Section~\ref{section:Asymptotics}. The fact that it is easy to describe the critical points in this case simplifies the analysis considerably and for this reason we restrict to this case.
We expect the behavior to be similar at all points in the interior of the four quadrants in the Aztec diamond. Close to the boundary or when $\xi_1=0$ or $\xi_2=0$ the nature of the problem changes and we will not discuss it here.

\subsection{At the boundaries} \label{subsection:At the boundaries}
 
Close to the liquid-gas and solid-liquid boundaries it is interesting to investigate further the subleading asymptotics in 
Theorem~\ref{thm:Kinverselimit}. This enables us to study the correlation of dominoes close to these boundaries. To state these results, we need a more precise scaling than the one considered in Condition~\ref{condition1}.
\begin{condition}\label{condition2}
Rescale and center the Aztec diamond so that for vertices $x=(x_1,x_2)$ and $y=(y_1,y_2)$ in the Aztec diamond graph, we have $x_1=[ 4m+4 m \xi- 2 \beta_x\lambda_2 (2m)^{2/3}  + 2 \alpha_x \lambda_1(2m)^{1/3}] +u_1$, $x_2= [4m+4 m \xi+ 2 \beta_x \lambda_2(2m)^{2/3}  + 2 \alpha_x \lambda_1(2m)^{1/3}] +u_2$, $y_1= [4m+4 m \xi- 2 \beta_y \lambda_2(2m)^{2/3}  + 2 \alpha_y \lambda_1(2m)^{1/3}] +v_1$, $y_2= [4m+4 m \xi+ 2 \beta_y \lambda_2(2m)^{2/3}  + 2 \alpha_y \lambda_1(2m)^{1/3}] +v_2$ where $\alpha_x,\alpha_y,\beta_x, \beta_y \in \mathbb{R}$ and $u_1,u_2,v_1,v_2 \in \mathbb{Z}$ are fixed.
\end{condition}
In the above condition, $\alpha_x,\alpha_y, \beta_x$ and $\beta_y$ are continuum variables in the scaling limit while $u_1,u_2,v_1$ and $v_1$ are kept as discrete variables.  The parities of $(u_1,u_2)$ and $(v_1,v_2)$ are the same as the parities of $(x_1,x_2)$ and $(y_1,y_2)$ respectively. That is for $\eps_1,\eps_2 \in \{0,1\}$, $(x_1,x_2) \in \mathtt{W}_{\eps_1}$ if and only if we have $(u_1,u_2) \in \mathtt{W}_{\eps_1}$ and 
$(v_1,v_2) \in \mathtt{B}_{\eps_2}$ if and only if we have $(y_1,y_2) \in \mathtt{B}_{\eps_2}$. 
 The scale factors $\lambda_1$ and $\lambda_2$ are dependent on the type of boundary and are determined later. We next define the extended Airy kernel, introduced in~\cite{PS:02} in its integral representation convenient for this paper.  

Let $\mathcal{C}_1$ consist of the rays $t e^{\pi \mathrm{i}/6}$ and $t e^{5\pi\mathrm{i}/6}$ for $t\geq 0$, oriented from right to left. Let $\mathcal{C}_2$ be $\mathcal{C}_1$ reflected in the $x$ axis oriented from right to left.  Denote the Airy function by $\mathrm{Ai}$.
Define
\begin{equation} \label{asfo:eq:Airymod}
\begin{split}
\tilde{\mathcal{A}}((\tau_1,\zeta_1);(\tau_2,\zeta_2))&= 
\int_0^\infty e^{-\lambda (\tau_1-\tau_2) } \mathrm{Ai} (\zeta_1 +\lambda) \mathrm{Ai} (\zeta_2+\lambda) d\lambda\\
&=\frac{e^{\frac{1}{3}(\tau_2^3-\tau_1^3)+\zeta_1 \tau_1 -\zeta_2 \tau_2}}{\mathrm{i}(2\pi \mathrm{i})^2} \int_{\mathcal{C}_1} dw \int_{\mathcal{C}_2} dz \frac{e^{\frac{\mathrm{i}}{3} (z^3-w^3) -\tau_2 z^2+\tau_1 w^2+\mathrm{i}z (\zeta_2-\tau_2^2) -\mathrm{i}w (\zeta_1-\tau_1^2)}}{z-w},
\end{split}
\end{equation}
and
\begin{equation}\label{asfo:eq:Airyphi}
\phi_{\tau_1,\tau_2} (\zeta_1 ,\zeta_2) =
\left\{
\begin{array}{ll}
 \frac{1}{\sqrt{4 \pi (\tau_2-\tau_1)}} e^{-\frac{(\zeta_1-\zeta_2)^2}{4(\tau_2-\tau_1)}-\frac{(\tau_2-\tau_1)(\zeta_1+\zeta_2)}{2}+\frac{(\tau_2-\tau_1)^3}{12}} &\mbox{if } \tau_1<\tau_2 \\
0 & \mbox{otherwise,}
\end{array}\right.
\end{equation}
which we refer to as the \emph{Gaussian part} of the extended Airy kernel.
The \emph{extended Airy kernel}, ${\mathcal{A}}((\tau_1,\zeta_1);(\tau_2,\zeta_2))$, is defined by
\begin{equation}\label{asfo:eq:Airymod2}
{\mathcal{A}}((\tau_1,\zeta_1);(\tau_2,\zeta_2))=\tilde{\mathcal{A}}((\tau_1,\zeta_1);(\tau_2,\zeta_2))-\phi_{\tau_1,\tau_2} (\zeta_1 ,\zeta_2).
\end{equation}
Note that
\begin{equation}
\mathrm{K}_{\mathrm{Ai}} ( \zeta_1,\zeta_2)=\tilde{\mathcal{A}}((\tau,\zeta_1);(\tau,\zeta_2))
\end{equation}
is the standard Airy kernel.
Introduce the following notation
\begin{equation}
\mathtt{g}_{\eps_1,\eps_2} =
\left\{
\begin{array}{ll}
\frac{\mathrm{i} \left(\sqrt{a^2+1}+a\right)}{1-a}
 &  \mbox{if } (\eps_1,\eps_2)=(0,0) \\
\frac{\sqrt{a^2+1}+a-1}{\sqrt{2a} (1-a) }
 & \mbox{if } (\eps_1,\eps_2)=(0,1) \\
-\frac{\sqrt{a^2+1}+a-1}{\sqrt{2a} (1-a) }
&\mbox{if } (\eps_1,\eps_2)=(1,0)\\
\frac{\mathrm{i}\left(\sqrt{a^2+1}-1\right)}{(1-a) a}
&\mbox{if } (\eps_1,\eps_2)=(1,1).
\end{array}
\right.
\end{equation}
We have the following theorem for the asymptotics of $K^{-1}_{a,1}$ at the liquid-gas boundary.

\begin{thma}\label{thm:transversal1}

\label{asfo:lem:transKinvgasform}
Let $x\in \mathtt{W}_{\eps_1}$, $y\in \mathtt{B}_{\eps_2}$ for $\eps_1,\eps_2\in \{0,1\}$. Choose the scaling in  Condition~\ref{condition2} with $\xi=-\frac{1}{2} \sqrt{1-2c}$,
\begin{equation}
c_0=\frac{(1-2c)^{\frac{2}{3}} }{(2c(1+2c))^{\frac{1}{3}}}, \hspace{5mm}
\lambda_1= \frac{\sqrt{1-2c}}{2c_0} \hspace{5mm} \mbox{and} \hspace{5mm}
\lambda_2= \frac{(1-2c)^{\frac{3}{2}}}{2cc_0^2}.
\end{equation}
If $\beta_x=\beta_y=\beta$, we have
\begin{equation}
\begin{split}
K_{a,1}^{-1}(x,y)&=\mathbb{K}^{-1}_{1,1}(x,y) -\mathrm{i}^{y_1-x_1+1} 
|G(\mathrm{i})|^{\frac{1}{2}(-2-x_1+x_2+y_1-y_2)}e^{\beta( \alpha_x-\alpha_y)} c_0 \mathtt{g}_{\eps_1,\eps_2}\\
&\times \mathrm{K}_{\mathrm{Ai}} ( \alpha_x+\beta^2,\alpha_y+\beta^2)(2m)^{-\frac{1}{3}}
\left(1 +O\left(m^{-1/3}\right)\right).
\end{split}
\end{equation}
Otherwise, if $\beta_x\not =\beta_y$, we have
\begin{equation}
\begin{split}
K_{a,1}^{-1}(x,y)&=-
\mathrm{i}^{y_1-x_1+1}
|G(\mathrm{i})|^{\frac{1}{2}(-2-x_1+x_2+y_1-y_2)}e^{\beta_x \alpha_x-\beta_y \alpha_y+\frac{2}{3}(\beta_x^3-\beta_y^3)} c_0 \mathtt{g}_{\eps_1,\eps_2}\\
&\times\mathcal{A}( (-\beta_x, \alpha_x+\beta_x^2);(-\beta_y,\alpha_y+\beta_y^2))(2m)^{-1/3}\left(1 +O\left(m^{-\frac{1}{3}}\right)\right).
\end{split}
\end{equation}
\end{thma}
The proof is given in Section~\ref{section:Asymptotics}.  As a corollary of this result, we are able to compute the joint probabilities of dominoes at the liquid-gas boundary. We find that

\begin{coro} \label{coro:transversal1}
Let $x, y \in \mathtt{W}_0$ with the same scaling and scaling relations as in Theorem~\ref{thm:transversal1}. 
 Then, if $(\alpha_x,\beta_x)=(\alpha_y,\beta_y)$, the covariance between the two dimers $(x,x+e_2)$ and $(y, y+e_2)$ is equal to
\begin{equation}
-a^2 \mathbb{K}^{-1}_{1,1}( (u_1,u_2),(v_1-1,v_2+1))\mathbb{K}^{-1}_{1,1}( (v_1,v_2),(u_1-1,u_2+1)) +O(m^{-1/3}).
\end{equation}
However, if $\alpha_x \not = \alpha_y$ or $\beta_x \not = \beta_y$, then the covariance between the two dimers $(x,x+e_2)$ and $(y, y+e_2)$ is equal to 
\begin{equation} \label{res:thmeq:transversal}
\frac{(a\mathtt{g}_{0,0}c_0)^2}{|G(\mathrm{i})|^{4}}
\mathcal{A}( (-\beta_x, \alpha_x+\beta_x^2);(-\beta_y,\alpha_y+\beta_y^2))
\mathcal{A}( (-\beta_y, \alpha_y+\beta_y^2);(-\beta_x,\alpha_x+\beta_x^2))(2m)^{-\frac{2}{3}}\left(1+O\left( m^{-\frac{1}{3}} \right)\right)
\end{equation} 
\end{coro}

  In the above corollary, we make a specific choice of dimers based on parity, that is we choose the dimer $(x,x+e_2)$ for $x \in \mathtt{W}_0$. Out of the eight possible dimers,  in the third quadrant of the Aztec diamond this choice of  dimer has the largest probability of being observed.   
 This is also observed from Fig.~\ref{fig:tpn200} --- the dark grey lines in the bottom left  corner of the Aztec diamond are mainly given by dimers of the form $(x,x+e_2)$ for $x \in \mathtt{W}_0$. 
An intuitive description of Theorem~\ref{thm:transversal1} and Corollary~\ref{coro:transversal1} is that at the liquid-gas boundary,  we actually see is the gas phase, but the correlation between distant dominoes is much greater than if we are really inside a gas phase, where correlations decay exponentially with distance. 
 
We now turn to the solid-liquid boundary to compare with the results at the liquid-gas boundary.  Define
\begin{equation}
\mathtt{s}_{\eps_1,\eps_2}=\left\{ \begin{array}{ll}
\frac{a-\sqrt{a^2+1}}{2( a+1)}& \mbox{if }(\eps_1,\eps_2)=(0,0)\\
\frac{\sqrt{a^2+1}-a-1}{2 \sqrt{2a} (a+1)}& \mbox{if }(\eps_1,\eps_2)=(0,1)\\
\frac{\sqrt{a^2+1}-a-1}{2 \sqrt{2a}(a+1)}& \mbox{if }(\eps_1,\eps_2)=(1,0)\\
\frac{1-\sqrt{a^2+1}}{2a(a+1)}& \mbox{if }(\eps_1,\eps_2)=(1,1).\\
\end{array} \right. 
\end{equation}

 We have the following theorem for the asymptotics of $K^{-1}_{a,1}$ at the solid-liquid boundary.

\begin{thma}\label{thm:transversal2}
Let $x\in \mathtt{W}_{\eps_1}$, $y\in \mathtt{B}_{\eps_2}$ for $\eps_1,\eps_2\in \{0,1\}$. Choose the scaling in  Condition~\ref{condition2} with  $\xi=-\frac{1}{2} \sqrt{1+2c}$, 
\begin{equation}
c_0=\frac{(1+2c)^{2/3}}{(2c(1-c))^{1/3}}, \hspace{5mm}\lambda_1 = \frac{\sqrt{1+2c}}{ 2 c_0} \hspace{5mm}\mbox{and} \hspace{5mm}\lambda_2= \frac{(1+2c)^{3/2}}{2 c c_0^2}.
\end{equation}
If $\beta_x=\beta_y=\beta$, we have
 \begin{equation}
\begin{split}
&K_{a,1}^{-1}(x,y)=
 S(x,y)+\mathrm{i}^{\frac{x_2-x_1+y_1-y_2}{2}}
G(1)^{\frac{1}{2}(-2-x_1+x_2+y_1-y_2)}c_0
\\&\times
\mathtt{s}_{\eps_1,\eps_2}e^{(\alpha_y-\alpha_x)\beta }
\mathrm{K_{Ai}}(\alpha_x+\beta^2,\alpha_y+\beta^2)(2m)^{-1/3}(1+O(m^{-1/3})). 
\end{split}
\end{equation}
Otherwise, if $\beta_x \not = \beta_y$, we have 
 \begin{equation}
\begin{split}\label{asym:lem:eq:transKinvsolidform}
&K_{a,1}^{-1}(x,y)=\mathrm{i}^{\frac{x_2-x_1+y_1-y_2}{2}}
G(1)^{\frac{1}{2}(-2-x_1+x_2+y_1-y_2)}c_0
\\
&\times 
\mathtt{s}_{\eps_1,\eps_2}
e^{\alpha_y \beta_y-\alpha_x \beta_x +\frac{2}{3} (\beta_y^3-\beta_x^3)} 
\mathcal{A}((\beta_x,\alpha_x+\beta_x^2);(\beta_y,\alpha_y+\beta_x^2))(2m)^{-1/3}(1+O(m^{-1/3})).
\\
\end{split}
\end{equation}
\end{thma}

The proof is given in Section~\ref{section:Asymptotics}.  As a comparison to Corollary~\ref{coro:transversal1}, we compute the joint probabilities of dominoes at the solid-liquid boundary.  We find that

\begin{coro}\label{coro:transversal2}
Let $x, y \in \mathtt{W}_0$ with the same scaling and scaling relations as in Theorem~\ref{thm:transversal2}. 
 Then, the covariance between the two dimers $(x,x-e_2)$ and $(y, y-e_2)$ is equal to
\begin{equation}
(c_0 \mathtt{s}_{0,0})^2
\mathcal{A}( (\beta_x, \alpha_x+\beta_x^2);(\beta_y,\alpha_y+\beta_y^2))
\mathcal{A}( (\beta_y, \alpha_y+\beta_y^2);(\beta_x,\alpha_x+\beta_x^2))(2m)^{-\frac{2}{3}}\left(1+O\left( m^{-\frac{1}{3}} \right)\right).
\end{equation} 

\end{coro}
 The proof is given in Section~\ref{section:Asymptotics}.  The motivation behind the choice of dimers given in Corollary~\ref{coro:transversal2} is to choose dimers that are not present in the solid region and whose vertices are in $(\mathtt{W}_0,\mathtt{B}_0)$.  


We can compare the Airy expressions of both $m^{-2/3}$ terms in Corollary~\ref{coro:transversal1} and Corollary~\ref{coro:transversal2} --- both exhibit long range correlations but have opposite curvature which is apparent through the parameters of the extended Airy kernel. 
Thus, the behavior at the two types of boundaries is in some respect very similar. An important difference is that at the solid-liquid boundary the dominant part are given by correlations in a solid phase, and since this phase has a trivial structure it is much easier to identify the true microscopic boundary as discussed above.

\subsection{Open problems}

We have seen that it is not obvious to define a microscopic liquid-gas boundary. In the last section, we propose a set of lattice paths which seem to separate the liquid-gas boundary starting from the bottom or top boundaries and ending at the left or right boundaries.  
We call these paths \emph{tree paths}. These paths are constructed by using a modification of Temperley's bijection~\cite{Tem:74}. These paths are motivated by the long dark lines with `clumps' of dominoes  that appear in Fig.~\ref{fig:tpn200}. Is it possible to understand the asymptotic properties of these paths? Does the path giving the liquid-gas interface converge to the Airy process after appropriate rescaling? { In a forthcoming article with V. Beffara, we partially answer both of these questions.}

The result in Theorem~\ref{thm:Kinverselimit} holds for the diagonal in the lower left corner of the Aztec diamond graph. 
We expect that a similar result should also hold for the rest of the Aztec diamond graph with the exception of the cusp and tangency points.
An extension of Theorem~\ref{thm:Kinverselimit} to the whole Aztec diamond should prove that the phase picture in Fig.~\ref{fig:weights} is true everywhere, not just the specific line investigated here.
  In this paper, we have avoided the cusps and the tangency points.  We expect that the tangency points should have similar behavior to that observed in other tiling models.  However, the behavior of the cusp is intriguing since the background is given by the gas phase.  Perhaps one needs to analyze the tree paths at the cusp, in order to describe the full picture. 

The liquid region exhibits polynomial decay of pairwise correlations between dominoes.  Locally, one expects that the height function fluctuations are given by the Gaussian free field. Heuristically, as one is gluing all these local pieces of the liquid region together, and since the liquid region is topologically an annulus, we expect that the height function fluctuations of the liquid region converges to the Gaussian free field with conformal structure diffeomorphic to the annulus. 

The derivation of the formula that we obtain for the elements of the inverse Kasteleyn matrix in this paper is very long, complicated and
not very intuitive. Since the final formula is not that complicated it would be interesting to understand or derive it in a more direct 
way.

\section{Asymptotic Computations} \label{section:Asymptotics}

In this section,  we give the bulk of the asymptotic computations for the paper and as a result, we prove Theorems~\ref{thm:Kinverselimit},~\ref{thm:transversal1} and~\ref{thm:transversal2} and  Corollaries~\ref{coro:transversal1} and~\ref{coro:transversal2}, assuming the results of Theorem~\ref{thm:Kinverse} and Corollary~\ref{corollary}.

These asymptotic computations require many prerequisite computations which we list and prove later in the section or in  Section~\ref{section:Algebra}.  
The organization of this section is as follows:
\begin{enumerate}
\item in Section~\ref{subsection:Definitions}, we introduce the definitions that are required to state all the details of the simplified version of $K^{-1}_{a,1}$ given in Theorem~\ref{thm:Kinverse} and Corollary~\ref{corollary}. We will also introduce definitions that are of importance for later computations considered in the section, including some notation, as well as stating a result which relates the three macroscopic phases.
\item In Section~\ref{Sec:GaussianLimits}, we show that for particular choices of $x$ and $y$ in $\mathbb{K}_{1,1}^{-1}(x,y)$  gives the  Gaussian part of the extended Airy kernel.  An analogous statement holds for $S(x,y)$.
\item In Section~\ref{subsection:requiredresults}, we state results that are required for the saddle point analysis. The proofs of these results are suppressed either to the end of this section or to the following section.  We also highlight the `large $n$' contribution in~\eqref{B} and give details of the saddle point function  used in all the following asymptotic computations.  
\item  We analyze~\eqref{B} when  sequentially moving the asymptotic (local) coordinate from the gas region to the solid region.  In each region or boundary, we perform a saddle point analysis. These computations are given in Sections~\ref{subsection:gas}, ~\ref{subsection:gasliquid},~\ref{subsection:liquid},~\ref{subsection:liquidsolid} and~\ref{subsection:solid} with each subsection analyzing a specific region or boundary.
\item In Section~\ref{subsection:Exponential}, we prove all the exponentially small estimates that are required in order to conclude the proofs of Theorems~\ref{thm:Kinverselimit},~\ref{thm:transversal1} and~\ref{thm:transversal2} and Corollaries~\ref{coro:transversal1} and~\ref{coro:transversal2}. These conclusions are drawn in Section~\ref{subsection:conclusion}.
\end{enumerate}

\subsection{Definitions and basic formulas} \label{subsection:Definitions}
In this subsection, we expand the definitions that were alluded to in Section~\ref{subsection:Double contour}, introduce the remaining notation used frequently throughout the rest of the paper and state a result which relates all three macroscopic phases. 

The following definition  introduces the terms $\mathtt{y}_{\gamma_1,\gamma_2}^{\eps_1,\eps_2}(a,b,u,v)$.     These terms have many remarkable symmetries  which we extensively exploit to find the simplified  formula given in Theorem~\ref{thm:Kinverse} as detailed in Section~\ref{section:Formula Simplification}.

\begin{defi} \label{def:coefficients}
Let
\begin{equation}
\begin{split}
f_{a,b}(u,v)=&
 \left(2 a^2 u v+2 b^2 u v-a b \left(-1+u^2\right) \left(-1+v^2\right)\right)\\&\times \left(2 a^2 u v+2 b^2 u v+a b \left(-1+u^2\right) \left(-1+v^2\right)\right).
\end{split}
\end{equation}
Define the following rational functions:
\begin{equation}
\begin{split}
&\mathtt{y}_{0,0}^{0,0}(a,b,u,v) =  \frac{1}{4 \left(a^2+b^2\right)^2f_{a,b}(u,v)}\left(
     2 a^7 u^2 v^2-a^5 b^2 \left(1+u^4+u^2 v^2-u^4 v^2+v^4-u^2 v^4\right)\right. \\
&\left. -a^3 b^4 \left(1+3 u^2+3 v^2+2 u^2 v^2+u^4 v^2+u^2 v^4-u^4 v^4\right) -a b^6 \left(1+v^2+u^2 +3u^2 v^2\right)
\right),\\
\end{split}
\end{equation}
\begin{equation}
\mathtt{y}_{0,1}^{0,0}(a,b,u,v)= \frac{a \left(b^2+a^2 u^2\right) \left(2 a^2 v^2+b^2 \left(1+v^2-u^2+u^2 v^2\right)\right)}{4 \left(a^2+b^2\right)f_{a,b}(u,v)}, \\
\end{equation}
\begin{equation}
\mathtt{y}_{1,0}^{0,0}(a,b,u,v)= \frac{a \left(b^2+a^2 v^2\right) \left(2 a^2 u^2+b^2 \left(1-v^2+u^2+u^2 v^2\right)\right)}{4 \left(a^2+b^2\right)f_{a,b}(u,v)} \\
\end{equation}
and
\begin{equation}
\mathtt{y}_{1,1}^{0,0}(a,b,u,v)= \frac{ a \left(2 a^2 u^2 v^2+b^2 \left(-1+v^2+u^2+u^2 v^2\right)\right) }{4 f_{a,b}(u,v)}.\\
\end{equation}
For $i,j \in \{0,1\}$,  define $\mathtt{y}^{0,1}_{i,j}(a,b,u,v)$,$\mathtt{y}^{1,0}_{i,j}(a,b,u,v)$ and $\mathtt{y}^{1,1}_{i,j}(a,b,u,v)$ by
\begin{equation}
\mathtt{y}^{0,1}_{i,j}(a,b,u,v)=\frac{\mathtt{y}^{0,0}_{i,j}(b,a,u,v^{-1})}{v^2},
\end{equation}
\begin{equation}
\mathtt{y}^{1,0}_{i,j}(a,b,u,v)=\frac{\mathtt{y}^{0,0}_{i,j}(b,a,u^{-1},v)}{u^2}
\end{equation}
and
\begin{equation}
\mathtt{y}^{1,1}_{i,j}(a,b,u,v)=\frac{\mathtt{y}^{0,0}_{i,j}(a,b,u^{-1},v^{-1})}{u^2v^2}.
\end{equation}
When $b=1$, we write $y^{\eps_1,\eps_2}_{i,j}(a,1,u,v)=y^{\eps_1,\eps_2}_{i,j}(u,v)$.
\end{defi}

The expressions in Definition~\ref{def:coefficients} appear in Theorem~\ref{thm:Kinverse} and Corollary~\ref{corollary} in the following way: we have
\begin{equation}
\label{asfo:eq:Y}
Y^{\eps_1,\eps_2}_{\gamma_1,\gamma_2}(u_1,u_2)= (1+a^2)^2   (-1)^{\eps_1 \eps_2+\gamma_1(1 +\eps_2)+\gamma_2(1+\eps_1)}
  s(-\mathrm{i}u_1)^{\gamma_1} s(-\mathrm{i} u_2)^{\gamma_2} \mathtt{y}_{\gamma_1,\gamma_2}^{\eps_1,\eps_2} (-\mathrm{i}u_1 ,- \mathrm{i} u_2)u_1^{\eps_1} u_2^{\eps_2},
\end{equation}
where $s$ is defined in~\eqref{results:eq:mu}. Also, we set
\begin{equation}\label{cuts:def:x}
\mathtt{x}^{\eps_1,\eps_2}_{\gamma_1,\gamma_2} \left(  \omega_1 , \omega_2   \right) =   \frac{ G\left( \omega_1 \right)G\left( \omega_2 \right)}{\prod_{j=1}^2 \sqrt{ \omega_j^2+2c} \sqrt{ \omega_j^{-2}+2c} }\mathtt{y}_{\gamma_1,\gamma_2}^{\eps_1,\eps_2} \left(  G\left( \omega_1 \right),G\left( \omega_2 \right)  \right) \left( 1- \omega_1^2 \omega_2^2 \right),
\end{equation}
where $G(\omega)$ is defined in~\eqref{cuts:def:G}
and 
then we define
\begin{equation}
\begin{split}\label{asfo:eq:Q}
Q_{\gamma_1,\gamma_2}^{\eps_1,\eps_2} \left( \omega_1,\omega_2 \right)& = (-1)^{\eps_1 +\eps_2 +\eps_1 \eps_2 +\gamma_1 (1+\eps_2) +\gamma_2(1+\eps_1) }
s\left( G\left( \omega_1 \right)\right)^{\gamma_1}   s \left( G\left( \omega_2^{-1} \right)\right)^{\gamma_2}\\& \times G\left( \omega_1 \right) ^{\eps_1} G\left( \omega_2^{-1}  \right)^{\eps_2} \mathtt{x}^{\eps_1,\eps_2}_{\gamma_1,\gamma_2} \left(  \omega_1 , \omega_2^{-1}\right),
\end{split}
\end{equation}
which appears in~\eqref{B}.  A small computation yields
\begin{equation} \label{asfo:eq:s}
s \left( G\left( \omega \right)\right)=\omega \sqrt{\omega^{-2}+2c} \hspace{2mm} \mbox{and} \hspace{2mm}
s \left( G\left( \omega^{-1} \right)\right)=\frac{1}{\omega} \sqrt{\omega^{2}+2c}.
\end{equation}
We remark that $Q^{\eps_1,\eps_2}_{\gamma_1,\gamma_2} \left( \omega_1, \omega_2 \right)$ is analytic on $\mathbb{C}\backslash \left( (-\infty,-1/\sqrt{2c}] \mathrm{i} \cup  [-\sqrt{2c},\sqrt{2c}] \mathrm{i} \cup [1/\sqrt{2c},\infty)\mathrm{i}\right)$.  This is seen by the fact that
\begin{equation}
\frac{ 1-\omega_1^2\omega_2^2}{f_{a,1}(G(\omega_1), G(\omega_2))}= \frac{1}{4 (1+a^2)^2 G(\omega_1)^2 G(\omega_2)^2},
\end{equation}
which follows from a computation.
For convenience, we will use the notation
\begin{equation}\label{asfo:eq:Zstar}
\begin{split}
Z_{\eps_1,\eps_2} \left( \omega_1, \omega_2 \right)&= 
\sum_{\gamma_1,\gamma_2=0}^1
Q_{\gamma_1,\gamma_2}^{\eps_1,\eps_2} \left(  \omega_1 , \omega_2   \right)
 \end{split}
\end{equation}
and so that~\eqref{B} is rewritten as 
\begin{equation}
~\label{BZ}
\mathcal{B}_{\eps_1,\eps_2} (a,x_1,x_2,y_1,y_2) =  \frac{ \mathrm{i}^{\frac{x_2-x_1+y_1-y_2}{2}}}{(2\pi \mathrm{i})^2}  \int_{\Gamma_{p}}\frac{d \omega_1}{\omega_1} \int_{\Gamma_{1/p}}  d \omega_2 \frac{\omega_2  }{\omega_2^2-\omega_1^2}\frac{ H_{x_1+1,x_2}(\omega_1)}{H_{y_1,y_2+1}(\omega_2)} Z_{\eps_1,\eps_2}(\omega_1 ,\omega_2).
\end{equation}
We set
\begin{equation}\label{asfo:eq:V}
V_{\eps_1,\eps_2}  \left( \omega_1, \omega_2 \right)= \frac{1}{2} \left(Z_{\eps_1,\eps_2} \left( \omega_1, \omega_2 \right)+(-1)^{\eps_2+1}Z_{\eps_1,\eps_2} \left( \omega_1, -\omega_2 \right)\right).
\end{equation}
By the choice of branch cut~\eqref{cuts:squareroot}, we have
\begin{equation}\label{cuts:sign1}
\sqrt{(-\omega)^2+2c}=-\sqrt{\omega^2+2c} \hspace{5mm} \mbox{and} \hspace{5mm} \sqrt{\overline{\omega}^2+2c}=\overline{\sqrt{\omega^2+2c}},
\end{equation}
for $\omega \in \mathbbm{C}\backslash \mathrm{i} [-\sqrt{2c},\sqrt{2c}]$ and we have the formulas
\begin{equation} \label{asfo:eq:Gsigns}
G\left( -\omega \right) =-G \left( \omega \right) \hspace{5mm} \mbox{and} \hspace{5mm}
G\left( \overline{\omega}\right) =\overline{G \left( \omega\right)}.
\end{equation}
As a consequence of the first formula, we are able to write out $V_{\eps_1,\eps_2}(\omega_1,\omega_2)$ using the definitions given in~\eqref{asfo:eq:Zstar} and~\eqref{asfo:eq:V}:
\begin{equation} \label{asfo:eq:V2}
\begin{split}
& V_{\eps_1,\eps_2}(\omega_1, \omega_2)=
 \frac{1}{2}\sum_{\gamma_1,\gamma_2=0}^1  (-1)^{\eps_1 +\eps_2 +\eps_1 \eps_2 +\gamma_1 (1+\eps_2) +\gamma_2(1+\eps_1) }\\
&\times s\left( G\left( \omega_1 \right)\right)^{\gamma_1}   s \left( G\left( \omega_2^{-1} \right)\right)^{\gamma_2}G\left( \omega_1 \right) ^{\eps_1} G\left( \omega_2^{-1}  \right)^{\eps_2} \left(
 \mathtt{x}^{\eps_1,\eps_2}_{\gamma_1,\gamma_2} \left(  \omega_1 , \omega_2^{-1}   \right)
- \mathtt{x}^{\eps_1,\eps_2}_{\gamma,\gamma} \left(  \omega_1 ,- \omega_2^{-1}   \right) \right),
\end{split}
\end{equation}
which follows from observing that $s(G(\omega))$ and $s(G(-\omega))$ are invariant under the map $\omega \mapsto -\omega$, directly seen from~\eqref{asfo:eq:s}.

Using the fact that
\begin{equation}
\frac{\omega_2}{\omega_2^2-\omega_1^2}= \frac{1}{2} \left( \frac{1}{\omega_2 -\omega_1} + \frac{1}{\omega_2+\omega_1} \right),
\end{equation}
 the integral in~\eqref{BZ} becomes
\begin{equation}
\begin{split}
 \mathcal{B}_{\eps_1,\eps_2}(a,x_1,x_2,y_1,y_2)&=\frac{ \mathrm{i}^{(x_2-x_1+y_1-y_2)/2}}{2(2\pi \mathrm{i})^2} \int_{\Gamma_{r}}\frac{d \omega_1}{\omega_1} \int_{\Gamma_{1/r}}  d \omega_2 \\
&\times \left( \frac{1}{\omega_2 -\omega_1} + \frac{1}{\omega_2+\omega_1} \right)
 Z_{\eps_1,\eps_2}(\omega_1,\omega_2)\frac{ H_{x_1+1,x_2}(\omega_1)}{H_{y_1,y_2+1}(\omega_2)}.
\end{split}
\end{equation}
We split the above equation into two separate double integrals with the double integral containing a term $1/(\omega_2-\omega_1)$ and the other double integral containing the term $1/(\omega_2+\omega_1)$.  For the second integral, we make the change of variables $\omega_2 \mapsto -\omega_2$. Since we have
\begin{equation}
H_{y_1,y_2+1}(-\omega) = (-1)^{\frac{y_2+1-y_1}{2}} H_{y_1,y_2+1}(\omega) = (-1)^{\eps_2+1}  H_{y_1,y_2+1}(\omega) 
\end{equation}
by using~\eqref{asfo:lem:H},~\eqref{cuts:def:G},~\eqref{cuts:sign1} and $y_2+y_1 \mod 4= 2 \eps_2+1$ for $(y_1,y_2) \in \mathtt{W}_{\eps_2}$ for $\eps_2 \in \{0,1\}$. As a result of this change of variables, we obtain, for $\sqrt{2c}<r<1$,
\begin{equation}\label{asfo:lemproof:liquidV}
  \mathcal{B}_{\eps_1,\eps_2}(a,x_1,x_2,y_1,y_2)=\frac{ \mathrm{i}^{(x_2-x_1+y_1-y_2)/2}}{(2\pi \mathrm{i})^2} \int_{\Gamma_{r}}\frac{d \omega_1}{\omega_1} \int_{\Gamma_{1/r}}  d \omega_2 \frac{ V_{\eps_1,\eps_2}(\omega_1,\omega_2)}{ \omega_2-\omega_1}\frac{ H_{x_1+1,x_2}(\omega_1)}{H_{y_1,y_2+1}(\omega_2)}
\end{equation}
by using~\eqref{asfo:eq:V}. We use the above integral as the starting point of the asymptotics of $\mathcal{B}_{\eps_1,\eps_2}(a,x_1,x_2,y_1,y_2)$.

The next lemma gives a remarkably simple formula for $V_{\eps_1,\eps_2}(\omega,\omega)$.
\begin{lemma} \label{asfo:lem:V}
We have the following formula
\begin{equation}
 V_{\eps_1,\eps_2}(\omega, \omega) =\frac{ (-1)^{1+h(\eps_1,\eps_2)} a^{\eps_2}  G\left( \omega^{-1} \right)^{h(\eps_1,\eps_2)} + a^{1-\eps_2} G\left( \omega \right)  G\left( \omega^{-1} \right)^{1-h(\eps_1,\eps_2)} }{2(1+a^2) \sqrt{\omega^2+2c} \sqrt{\omega^{-2}+2c} }.
\end{equation}
\end{lemma}
The formula will be proved in Section~\ref{subsection:manipulations}. Note that the coefficients $\mathtt{g}_{\eps_1,\eps_2}$ and $\mathtt{s}_{\eps_1,\eps_2}$ are built from the above lemma; see Lemmas~\ref{asfo:lem:gascoeff} and~\ref{asfo:lem:solidcoeff} respectively.

Let $\omega_c=e^{\mathrm{i}\theta_c}$, $\theta_c\in [0,\pi/2]$ and let $\Gamma_{\omega_c}$ be the contour defined by
\begin{equation}\label{Gammaomegac}
\Gamma_{\omega_c}\,:\,[\theta_c,\pi-\theta_c]\cup[\pi+\theta_c,2\pi-\theta_c]\ni\theta\to e^{\mathrm{i}\theta}.
\end{equation}
Note that $\Gamma_1$ is the unit circle as before. For $x=(x_1,x_2) \in \mathtt{W}_{\eps_1}$ and $y=(y_1,y_2) \in \mathtt{B}_{\eps_2}$ we
define
\begin{equation}\label{Comegac}
C_{\omega_c}(x,y)=\frac{\mathrm{i}^{(x_2-x_1+y_1-y_2)/2}}{2\pi\mathrm{i}}\int_{\Gamma_{\omega_c}}
V_{\eps_1,\eps_2}(\omega,\omega) G(\omega)^{\frac{y_1-x_1-1}{2}}G(\omega^{-1})^{\frac{x_2-y_2-1}{2}}\frac{d\omega}{\omega}.
\end{equation}
The next lemma
in a sense relates the gas, liquid and solid phases.
\begin{lemma} \label{asfo:lem:liquidcomps}
For $x=(x_1,x_2) \in \mathtt{W}_{\eps_1}$ and $y=(y_1,y_2) \in \mathtt{B}_{\eps_2}$ we have the identity
\begin{equation} \label{Gasidentity}
 \mathbb{K}_{1,1}^{-1}(x,y)=C_1(x,y)+S(x,y).
 \end{equation}
Let $\omega_c=e^{\mathrm{i}\theta_c}$, $\theta_c\in(0,\pi/2)$ and set $r_1=1$, $r_2=1/|G(\omega_c)|^2$. Then
\begin{equation}\label{Liquididentity}
\mathbb{K}_{r_1,r_2}^{-1}(x,y)=\mathbb{K}_{1,1}^{-1}(x,y)-C_{\omega_c}(x,y)
\end{equation}
and $\mathbb{K}_{r_1,r_2}^{-1}(x,y)$ is the inverse Kasteleyn matrix for a liquid phase.
\end{lemma}

The proof of the above lemma is postponed to Section~\ref{subseq:lem:liquidcomps}.

\subsection{Gaussian limits}\label{Sec:GaussianLimits}

To get the Gaussian part of the extended Airy kernel we need some Gaussian limits of $\mathbb{K}_{1,1}^{-1}(x,y)$ and $S(x,y)$.
It is perhaps a little surprising that $\mathbb{K}_{1,1}^{-1}(x,y)$ which is a gas inverse Kasteleyn matrix is responsible for a part of the
extended Airy kernel, since one would think that it has nothing to do with the Airy asymptotics.

\begin{prop}\label{Prop:GasGausslimit}
Let $x=(x_1,x_2)\in \mathtt{W}_{\eps_1}$, $y=(y_1,y_2)\in \mathtt{B}_{\eps_2}$ and assume the scaling in Condition~\ref{condition2}. Also, assume that 
$|\alpha_x-\alpha_y|+|\beta_x-\beta_y|>0$. Then, for fixed $\eps_1,\eps_2\in\{0,1\}$,
\begin{equation}\label{GasGausslimit}
\lim_{m\to\infty}(2m)^{1/3}|G(\mathrm{i})|^{\frac{x_1-x_2+y_2-y_1+2}2} \mathbb{K}^{-1}_{1,1}(x,y) =\mathbbm{I}_{\beta_x>\beta_y}\mathrm{i}^{y_1-x_1+1} \frac{ \mathtt{g}_{\eps_1,\eps_2} c_0}{\sqrt{4 \pi( \beta_x-\beta_y)}}e^{-\frac{(\alpha_x-\alpha_y)^2}{4(\beta_x-\beta_y)} }
\end{equation}

\end{prop}

We have a similar proposition for $S(x,y)$.
\begin{prop}\label{Prop:SolidGausslimit}
Let $x=(x_1,x_2)\in \mathtt{W}_{\eps_1}$, $y=(y_1,y_2)\in \mathtt{B}_{\eps_2}$ and assume the scaling in Condition~\ref{condition2}. Also, assume that 
$|\alpha_x-\alpha_y|+|\beta_x-\beta_y|>0$. Then, for fixed $\eps_1,\eps_2\in\{0,1\}$,
\begin{equation}\label{SolidGausslimit}
\lim_{m\to\infty}(2m)^{1/3}|G(1)|^{\frac{x_1-x_2+y_2-y_1+2}2}S(x,y)=
-\mathbbm{I}_{\beta_x<\beta_y}
\frac{\mathrm{i}^{\frac{x_2-x_1+y_1-y_2}{2}}}{\sqrt{4 \pi(\beta_y-\beta_x)}}\mathtt{s}_{\eps_1,\eps_2} c_0 e^{-\frac{(\alpha_x-\alpha_2)^2}{4(\beta_y-\beta_x)}}
\end{equation}

\end{prop}
The proof of both propositions will be given in Section~\ref{Sec:Gaussianlimitsproofs}.

\subsection{Results required for the saddle point analysis}\label{subsection:requiredresults}

In this section we state results and partially analyze the saddle point equation which are required for our later asymptotic arguments. The proofs of these results are given later in the paper.  These results are as follows:
\begin{enumerate}

\item informally, Lemma~\ref{asfo:lem:exponential} says that when the asymptotic coordinate is located on the diagonal in the third quadrant of the Aztec diamond, then~\eqref{asfo:lemproof:exponential:change3},~\eqref{asfo:lemproof:exponential:change4} and~\eqref{asfo:lemproof:exponential:change5} are exponentially small (in $n$).  Consequently, in order to study asymptotics of $K^{-1}_{a,1}$, from Theorem~\ref{thm:Kinverse} and Corollary~\ref{corollary}, we only need to analyze~\eqref{B} in the form \eqref{asfo:lemproof:liquidV} and $\mathbb{K}^{-1}_{1,1}(x,y)$.


\item We extract the saddle point function of~\eqref{asfo:lemproof:liquidV} and  Lemma~\ref{asym:lem:saddle} determines the number of roots of the saddle point equation. 
In Lemma~\ref{asym:lem:imaginary}, we give the behavior of the imaginary part of the saddle point function on $\partial \mathbb{H}_+$.

\end{enumerate}

We now proceed in stating these results. Recall $n=4m$.

\begin{lemma}\label{asfo:lem:exponential}
Under Condition~\ref{condition1} with $-1<\xi_1=\xi_2=\xi<0$, there are  positive constants $C_1$ and $C_2$ so that
\begin{equation}
\left|\frac{1}{a}\mathcal{B}_{1-\eps_1,\eps_2}(a^{-1},2n-x_1,x_2,2n-y_1,y_2) \right|\leq  C_1 e^{-C_2 n},
\end{equation}
\begin{equation}
\left|\frac{1}{a}\mathcal{B}_{\eps_1,1-\eps_2}(a^{-1},x_1,2n-x_2,y_1,2n-y_2) \right| \leq  C_1 e^{-C_2 n},
\end{equation}
and
\begin{equation}
\left|\mathcal{B}_{1-\eps_1,1-\eps_2}(a,2n-x_1,2n-x_2,2n-y_1,2n-y_2) \right| \leq  C_1 e^{-C_2 n}.
\end{equation}
\end{lemma}
The proof of this lemma is found in Section~\ref{subsection:Exponential}.




We begin the analysis of $\mathcal{B}_{\eps_1,\eps_2}(a,x_1,x_2,y_1,y_2)$ as stated in~\eqref{asfo:lemproof:liquidV}, by identifying  the main asymptotic term in $n$ in the integrand.  
Since $V_{\eps_1,\eps_2}(\omega_1,\omega_2)$ is independent of $m$, the $m$ dependence in Eq.~\eqref{asfo:lemproof:liquidV} is contained in the factor 
\begin{equation}
\frac{ H_{x_1+1,x_2}(\omega_1)}{H_{y_1,y_2+1}(\omega_2)}.
\end{equation}
If we set $x_1 \sim 4m + 2m \xi_1$,  $y_1 \sim 4m + 2m \xi_1$, $x_2 \sim 4m + 2m \xi_2$, $y_2 \sim 4m + 2m \xi_2$, where $\sim$ means the terms of order $m$, (i.e. ignoring lower order terms in $m$, the integer part, etc.), then the first order contribution in $m$ in the above equation, ignoring constant terms, is given by
\begin{equation}
  \exp \left(2m (g_{\xi_1,\xi_2}(\omega_1)-g_{\xi_1,\xi_2}(\omega_2)) \right),
\end{equation}
where $g_{\xi_1,\xi_2}$ is given by~\eqref{asym:eq:saddle}, restated below for convenience
\begin{equation}
g_{\xi_1,\xi_2}(\omega)= \log \omega- \xi_1 \log G \left( \omega \right) + \xi_2 \log G \left( \omega^{-1} \right).
\end{equation}
We refer to this as the \emph{saddle point function}. Although we are free to choose any branch for the above logarithm, we will keep the same convention chosen above, that is, the logarithm has arguments in  $(-\pi/2,3\pi/2)$.
 By differentiating this equation with respect to $\omega$, setting the result equal to zero and multiplying by $\omega$ we obtain~\eqref{asym:eq:saddleeqn}, that is
\begin{equation}
1+ \frac{ \omega}{\sqrt{ \omega^2+2c}}\xi_1 + \frac{1}{\omega \sqrt{\omega^{-2}+2c}} \xi_2 =0. 
\end{equation}
\begin{lemma}\label{asym:lem:saddle}
Eq.~\eqref{asym:eq:saddleeqn} has at most four roots in $\mathbb{C} \backslash (\mathrm{i}[-\infty,-1/\sqrt{2c}] \cup \mathrm{i}[-\sqrt{2c},\sqrt{2c}] \cup \mathrm{i} [1/\sqrt{2c},\infty))$ for $-1<\xi_1,\xi_2<0$.
\end{lemma}
The proof of the above lemma is postponed to Section~\ref{subsection:cutsandsaddle}.   For the rest of the paper, we set $\xi=\xi_1=\xi_2$.  Our later analysis shows that this choice in parameters gives explicit solutions to Eq.~\eqref{asym:eq:saddleeqn}, which naturally splits the analysis into five cases coinciding with the three phases and their boundaries. This choice of parameterization also ensures symmetry between the quadrants of $\mathbb{C}$.
 The next lemma gives the imaginary part of the saddle point function at the boundary of the first quadrant which is useful for finding the contours of steepest ascent and descent.
\begin{lemma} \label{asym:lem:imaginary}
The behavior of  $\mathrm{Im}\, g_{\xi,\xi} (\omega)$ for $\omega \in \partial \mathbb{H}_+$ is given by
\begin{enumerate}
\item $\mathrm{Im}\, g_{\xi,\xi} ( \mathrm{i}t)$ decreases from $(1+\xi ) \frac{\pi}{2}$ to $(1+2\xi ) \frac{\pi}{2}$ for $t$ increasing in $(0,\sqrt{2c})$,
\item $\mathrm{Im}\, g_{\xi,\xi} ( \mathrm{i}t)=(1+2\xi ) \frac{\pi}{2}$ for $t \in(\sqrt{2c},1/\sqrt{2c})$,
\item $\mathrm{Im}\, g_{\xi,\xi} ( \mathrm{i}t)$ increases from $(1+2 \xi ) \frac{\pi}{2}$ to $(1+\xi ) \frac{\pi}{2}$ for $t$ increasing in $(1/\sqrt{2c},\infty)$,
\item $\mathrm{Im}\, g_{\xi,\xi} ( t)= 0$ for $t \in (0,\infty)$,
\item  $\mathrm{Im}\, g_{\xi,\xi} (\omega)$ increases from $0$ to $(1+\xi ) \frac{\pi}{2}$  depending on the angle of approach as $\omega$ tends to 0,
\item  $\mathrm{Im}\, g_{\xi,\xi} (\omega)$ decreases from $(1+\xi ) \frac{\pi}{2}$  to $0$ depending on the angle of approach as $\omega$ tends to infinity. 

\end{enumerate}
\end{lemma}
The proof of this lemma is given in Section~\ref{subsection:cutsandsaddle}.

\subsection{Gas region} \label{subsection:gas}

We begin the asymptotic analysis of~\eqref{asfo:lemproof:liquidV} by starting inside the gas phase.   To proceed with the analysis, we  find the saddle points of~\eqref{asfo:lemproof:liquidV} and the contours of steepest ascent and steepest descent.  

The next lemma identifies the saddle points of~\eqref{asfo:lemproof:liquidV} inside the gas phase.

\begin{lemma} \label{asym:lem:rootsgas}
The saddle point of the function $g_{\xi,\xi}(\omega)$ for $\omega\in \mathbb{H}_+ \cup\partial \mathbb{H}_+$ has
two distinct saddle points in the interval $(\sqrt{2c},1/\sqrt{2c}) \mathrm{i}$ (but not equal to $\mathrm{i}$) if and only if  $-1/2 \sqrt{1-2c} < \xi <0$. These saddle points are  given by $\omega_c \in (\sqrt{2c},1) \mathrm{i}$ and $-\omega_c^{-1}$.  We also have $\mathrm{Re}g_{\xi,\xi}(\omega_c )<0$ and $\mathrm{Re}g_{\xi,\xi}({-\omega_c^{-1} })>0$.
\end{lemma}
\begin{proof}
For $\xi_1=\xi_2=\xi$, it is immediate that if $\omega_c=\mathrm{i}t$ for $t \in (\sqrt{2c},1)$  is a root of  Eq.~\eqref{asym:eq:saddleeqn}, then so are $\mathrm{i}/t$, $-\mathrm{i}/t$ and $-\mathrm{i} t$.  These are the four roots of~\eqref{asym:eq:saddleeqn} described by Lemma~\ref{asym:lem:saddle}.
By writing 
\begin{equation}
\xi=\frac{-1}{\frac{\omega_c}{\sqrt{\omega_c^2+2c}}+\frac{\omega_c^{-1}}{\sqrt{\omega_c^{-2}+2c}}}=-\frac{1}{\frac{t}{\sqrt{t^2-2c}}+\frac{t^{-1}}{\sqrt{t^{-2}-2c}}},
\end{equation}
 observe that $\xi$ is decreasing for increasing $t \in (\sqrt{2c},1)$ which gives the interval  $-1/2 \sqrt{1-2c} < \xi <0$. 

For the second condition, using the above parameterization of $\xi$ and~\eqref{asym:eq:saddleeqn} we find
\begin{equation}\label{asfo:lemproof:exponential:bound2}
g_{\xi,\xi}( \mathrm{i}t)=\log\mathrm{i}t - \frac{1}{\frac{t}{\sqrt{t^2-2c}}+\frac{t^{-1}}{\sqrt{t^{-2}-2c}}}  \log G\left(\frac{1}{\mathrm{i}t}\right)   + \frac{1}{\frac{t}{\sqrt{t^2-2c}}+\frac{t^{-1}}{\sqrt{t^{-2}-2c}}}   \log G\left(\mathrm{i}t\right) .
\end{equation}
Using the above equation, we can compute $\Re g_{\xi,\xi}(\mathrm{i} t)$ because
\begin{equation}
G\left(\mathrm{i}t\right)=\frac{\mathrm{i}}{\sqrt{2c}} \left( t-\sqrt{t^2-2c} \right) \hspace{2mm} \mbox{and}\hspace{2mm}G\left(\frac{1}{\mathrm{i}t}\right)=-\frac{\mathrm{i}}{\sqrt{2c}} \left( \frac{1}{t}-\sqrt{t^{-2}-2c} \right)
\end{equation}
for $t \in (\sqrt{2c},1/\sqrt{2c})$. By using the above two equations, it follows that $\Re g_{\xi,\xi}(\mathrm{i} t)$ is increasing for $t \in (\sqrt{2c},1)$, i.e. by differentiating. Since $\Re  g_{\xi,\xi}(\mathrm{i}) =0$ (because $|G(\mathrm{i})|=|G(1/\mathrm{i})|$) 
 the statement $\Re g_{\xi,\xi}(\mathrm{i} t)<0$ for $t \in (\sqrt{2c},1)$ is immediate. The statement $\mathrm{Re}g_{\xi,\xi}( -\mathrm{i}t^{-1})>0$ follows from an analogous computation.
\end{proof}

We next state and prove a lemma which aides in determining the local behavior of the contours of steepest ascent and descent around the saddle point.
\begin{lemma} \label{asym:lem:derivativesgas}
For $-\frac{1}{2}\sqrt{1-2c}<\xi <0$, choose $\omega_c$ such that  $g_{\xi,\xi}'( \omega_c)=0$  with $\omega_c \in (\sqrt{2 c},1)\mathrm{i}$, then $g_{\xi,\xi}'( -\omega_c^{-1})=0$,
\begin{equation}
g_{\xi,\xi}''( \omega_c)<0  \hspace{5mm} \mbox{and} \hspace{5mm} g_{\xi,\xi}''(- \omega_c^{-1})>0.
\end{equation}
\end{lemma}

\begin{proof}
From Lemma~\ref{asym:lem:rootsgas}, it is immediate that $g_{\xi,\xi}'( -\omega_c^{-1})=0$.
We differentiate $\omega g'_{\xi,\xi}(\omega)$ to obtain
\begin{equation}\label{asym:eqn:double}
\omega g'_{\xi,\xi}(\omega)+\omega^2 g''_{\xi,\xi}(\omega)= 2c \xi \left( \frac{\omega}{(\omega^2+2c)^{3/2}} - \frac{ \omega^{-1}}{(\omega^{-2}+2c)^{3/2}} \right).
\end{equation}
Since $g'_{\xi,\xi}(\omega_c)=0$ and $\omega_c=\mathrm{i}t$ we find after simplification that
\begin{equation}
\begin{split}
-t^2g''_{\xi,\xi}(\mathrm{i}t)=2c \xi \left( \frac{\mathrm{i}t}{(\mathrm{i}\sqrt{t^2+2c})^{3}} + \frac{ \mathrm{i}t^{-1}}{(-\mathrm{i}\sqrt{t^{-2}+2c})^{3}} \right),
\end{split}
\end{equation}
where $\xi<0$. A computation gives
\begin{equation}
\begin{split}
\frac{t}{(\sqrt{t^2+2c})^{3}} > \frac{ t^{-1}}{(\sqrt{t^{-2}+2c})^{3}}
\end{split}
\end{equation}
for $\sqrt{2c}<t<1$ and the reverse inequality holds for $1<t<1/\sqrt{2c}$ as required.

\end{proof}

The next lemma describes the contours of steepest descent and steepest ascent.
\begin{lemma} \label{asym:lem:contoursgas}
Assume the same conditions given in Lemma~\ref{asym:lem:derivativesgas}. Then there is a path of steepest descent from $\omega_c$ to $0$ and a path of steepest ascent from $-\omega_c^{-1}$ to infinity for $\mathrm{Re} \, g_{\xi,\xi}$.
\end{lemma}

\begin{proof}
Since $g_{\xi,\xi} (\omega)$ is analytic for $\omega \in \mathbb{H}_+ \cup \partial \mathbb{H}_+\backslash (\mathrm{i}[0,\sqrt{2c}]\cup\mathrm{i}[1/\sqrt{2c},\infty))$ the contours of steepest ascent and descent of $\mathrm{Re} \, g_{\xi,\xi} (\omega)$ are the level lines of $\mathrm{Im} \, g_{\xi,\xi} (\omega)$. For our steepest descent argument we need the descent path from $\omega_c$ ($\omega_1$-integral), and the ascent path from $-\omega_c^{-1}$ ($\omega_2$-integral). By Lemma \ref{asym:lem:imaginary} we have that $\mathrm{Im} \, g_{\xi,\xi} (\omega_c)=(1+2\xi)\frac{\pi}2$. From Lemma \ref{asym:lem:derivativesgas} we see that the descent path leaves 
$\omega_c$ perpendicularly and goes into $ \mathbb{H}_+ $. The value $(1+2\xi)\frac{\pi}2$ ( note $-\frac 12<\xi<0$ ) is taken only in $\mathrm{i}[\sqrt{2c}, 1/\sqrt{2c}]$, at $0$ and at infinity by Lemma \ref{asym:lem:imaginary} . Hence the descent path from $\omega_c$ has to go to $0$ or infinity, whereas the ascent path goes along the imaginary axis. We also have $\mathrm{Im} \, g_{\xi,\xi} (-\omega_c^{-1})=(1+2\xi)\frac{\pi}2$ and from Lemma \ref{asym:lem:imaginary} we see that the ascent path goes perpendicularly into  $ \mathbb{H}_+ $, and hence the descent path goes along the imaginary axis. Since the descent path from $\omega_c$ cannot cross the ascent path from $-\omega_c^{-1}$ in  $ \mathbb{H}_+ $, we see that the descent path from $\omega_c$ goes to $0$ and the ascent path from $-\omega_c^{-1}$ goes to infinity.

\end{proof}
Up to orientation, these contours are symmetric in both the real and imaginary axis.  The orientations are easily determined.
The contours described in Lemma~\ref{asym:lem:contoursgas} feature an infinite contour.  In order to reduce our saddle point analysis to a local saddle point argument, we need an estimate controlling the contribution of the contour outside of $|\omega|=R$ for $R$ large. A crude estimate gives
\begin{equation} \label{asfo:lemproof:crudebound}
\frac{1}{|H_{y_1,y_2+1} ({\omega})|} = \frac{1}{R^{2m(1+\xi)}} (1+O(1/R)).
\end{equation}

\begin{prop}\label{asfo:lem:Bgas}
Under the assumptions given in Theorem~\ref{thm:Kinverselimit} and with $-1/2 \sqrt{1-2c}<\xi<0$, we have
\begin{equation}
\mathcal{B}_{\eps_1,\eps_2}(a,x_1,x_2,y_1,y_2) =  O(e^{-Cm}),
\end{equation}
where $C>0$ is a constant.
\end{prop}
\begin{proof}
For the proof, we use the form of $\mathcal{B}_{\eps_1,\eps_2}(a,x_1,x_2,y_1,y_2)$ as stated in~\eqref{asfo:lemproof:liquidV} and ignore the integer parts because their effect is negligible.  To this integral, we first take out a prefactor of $H_{x_1+1,x_2}(\omega_c)/H_{y_1,y_2+1}(-\omega_c^{-1}))$ which decays exponentially fast to zero as $m$ tends to infinity.  This follows from Lemma~\ref{asym:lem:rootsgas} since the highest order term of this prefactor is given by $e^{2m(g_{\xi,\xi}(\omega_c)-g_{\xi,\xi}(-\omega_c^{-1}))}$.
 For the contour of integration with respect to $\omega_1$, deform to the contour of steepest descent given in Lemma~\ref{asym:lem:contoursgas}, and use symmetry to  extend to the remaining three quadrants.  For the contour of integration with respect to $\omega_2$, deform to the contour of steepest ascent given  in Lemma~\ref{asym:lem:contoursgas}, and use symmetry to  extend to the remaining three quadrants. The contours do not cross under these deformations.   

 Standard saddle point arguments and the crude estimate given in~\eqref{asfo:lemproof:crudebound} which controls the saddle point function on the infinite length contour means that we only need to consider the local saddle point contributions of $\mathcal{B}_{\eps_1,\eps_2}(a,x_1,x_2,y_1,y_2)$.  By symmetry, we only need to consider one pair of saddle points; we choose the pair $(\omega_c,- 1/\omega_c)$ for $\omega_c \in (\sqrt{2c},1) \mathrm{i}$. Write the local descent and ascent contours in a neighborhood of the critical points as 
\begin{equation}
\omega_1 = \omega_c+ (2m)^{-\frac{1}{2}} w \hspace{5mm}\mbox{and} \hspace{5mm}\omega_2 = -\omega_c^{-1}+ (2m)^{-\frac{1}{2}} z ,
\end{equation}
where $|w|,|z|<m^{\delta}$ with $0<\delta<1/6$. Up to leading order, these contours are parallel to the real axis.

In the neighborhood of the critical points, we apply a Taylor expansion to the exponential term contained in the integrand in~\eqref{asfo:lemproof:liquidV} 
\begin{equation}
\frac{H_{x_1+1,x_2}(\omega_1)}{H_{y_1,y_2+1}(\omega_2)}=
\frac{H_{x_1+1,x_2}(\omega_c)}{H_{y_1,y_2+1}(-\omega_c^{-1})}
e^{\frac{w^2g''_{\xi,\xi}(\omega_c)- z^2g''_{\xi,\xi}(-\omega_c^{-1})}{2}+\tilde{f}_m(\overline{x}_1,\overline{x}_2,\omega_c)w- \tilde{f}_m(\overline{y}_1,\overline{y}_2,-\omega_c^{-1})z+O(m^{-\frac{1}{2}}w^3,m^{-\frac{1}{2}}z^3)},
\end{equation}
where $\tilde{f}_m(\overline{x}_1,\overline{x}_2,\omega)=\frac{\overline{x}_1}{(2m)^\frac{1}{2}\sqrt{\omega^2+2c}} +\frac{\overline{x}_2}{(2m)^{\frac{1}{2}}\omega^2\sqrt{\omega^{-2}+2c}}$.  Using the above expansion, we have an integral for the local contribution of $ \mathcal{B}_{\eps_1,\eps_2}(a,x_1,x_2,y_1,y_2)$ with an error term in the exponent of the integrand.
We use $|e^t-1|<|t|e^{|t|}$  by setting $t$ equal to the error term, and find that $ \mathcal{B}_{\eps_1,\eps_2}(a,x_1,x_2,y_1,y_2)$ is equal to 
\begin{equation}
\begin{split}\label{asym:lemproof:B:gas}
&\frac{H_{x_1+1,x_2}(\omega_c)}{H_{y_1,y_2+1}(-\omega_c^{-1})}\frac{4 } {(2\pi )^2(2m)^{\frac{1}{2}}} \Bigg(  \int_{-m^{\delta}}^{m^\delta}dw \int_{-m^{\delta}}^{m^\delta}dz \frac{ V_{\eps_1,\eps_2} \left(\omega_c,-\omega_c^{-1}\right)}{\omega_c(( -\omega_c^{-1}-\omega_c)+(2m)^{-1/2}(w-z))}\\
& \times e^{\frac{1}{2}(g_{\xi,\xi}''(\omega_c) w^2-g_{\xi,\xi}''(-\omega_c^{-1}) z^2)} 
e^{ \tilde{f}_m(\overline{x}_1,\overline{x}_2,\omega_c)w- \tilde{f}_m(\overline{y}_1,\overline{y}_2,-\omega_c^{-1})z} +O(m^{-\frac{1}{2}})\Bigg).
\end{split}
\end{equation}
where the error term in the above equation contains the lower order terms of the local contribution and remaining global contribution. 
The singularity $1/(w-z)$ in the above integral is integrable and by using the information for $g''_{\xi,\xi}(\omega_c)$ and $g''_{\xi,\xi}(-\omega_c^{-1})$ provided by Lemma~\ref{asym:lem:derivativesgas}, the integral in the above equation is finite.  Due to the prefactor $\frac{H_{x_1+1,x_2}(\omega_c)}{H_{y_1,y_2+1}(-\omega_c^{-1})}$ and 
 since the above integral is a  Gaussian integral which has a finite limit, we conclude that the local contribution of $ \mathcal{B}_{\eps_1,\eps_2}(a,x,y)$ around its saddle point is exponentially small as required.

\end{proof}

\subsection{Liquid-gas boundary} \label{subsection:gasliquid}
We continue the analysis of the asymptotic behavior of~\eqref{asfo:lemproof:liquidV} when the asymptotic coordinate is at the liquid-gas boundary.  We find the behavior of the saddle points, the contours of steepest ascent and steepest descent.  Finally, we compute asymptotics of the integral in 
 Eq.~\eqref{asfo:lemproof:liquidV} at the liquid-gas boundary.

The next lemma gives the saddle point of the integral in Eq.~\eqref{asfo:lemproof:liquidV} when the asymptotic coordinate is at the liquid-gas boundary.

\begin{lemma} \label{asym:lem:rootsliquidgas}
The saddle point of the function $g_{\xi,\xi}(\omega)$ for $\omega\in \mathbb{H}_+ \cup\partial \mathbb{H}_+$ has  a double critical point at $\omega=\mathrm{i}$ if  and only if $\xi=-\frac{1}{2} \sqrt{1-2c}$.
\end{lemma}
\begin{proof}
The system of equations $\{g'_{\xi,\xi}(\omega_c)=0, g''_{\xi,\xi}(\omega_c)=0\}$ have the solutions $\omega_c=\pm\mathrm{i}$ for $\xi=-1/2 \sqrt{1-2c}$. 
\end{proof}
The next lemma determine the angles of approach of the contours of steepest ascent and steepest descent at the double critical point.

\begin{lemma} \label{asym:lem:derivativesliquidgas}
For $\xi=-\frac{1}{2}\sqrt{1-2c}$, we have
\begin{equation}
g'''_{\xi,\xi}(\mathrm{i} ) =-\frac{4c(c+1)}{(1-2c)^2} \mathrm{i}.
\end{equation}
\end{lemma}
\begin{proof}
We differentiate~\eqref{asym:eqn:double} and set $g'_{\xi,\xi}(\omega )=g''_{\xi,\xi}(\omega )=0$. After some computation, we find
\begin{equation}
\begin{split} \label{asym:eqn:triple}
g'''_{\xi,\xi}(\omega)= \frac{2c}{\omega^2}\left( \xi \frac{ 2c- \omega^2}{(\omega^2+2c)^{5/2}} + \xi \frac{ 2c/\omega^2 -2/\omega^4}{(1/\omega^2+2c)^{5/2}} \right).
\end{split}
\end{equation}
We can set $\omega=\mathrm{i}$ in the above equation as required.

\end{proof}

\begin{lemma} \label{asym:lem:contoursliquidgas}
For $\xi=-\frac{1}{2}\sqrt{1-2c}$,  there is a path of steepest descent leaving $\mathrm{i}$ at angle $-\pi/6$ going to $0$ for $\mathrm{Re}\,g_{\xi,\xi}$
 and a path of steepest ascent leaving $\mathrm{i}$ at angle $\pi/6$ going to infinity for $\mathrm{Re}\,g_{\xi,\xi}$.
\end{lemma}

\begin{proof}
Since $ g_{\xi,\xi} (\omega)$ is analytic for $\omega \in \mathbb{H}_+ \cup \partial \mathbb{H}_+ \backslash ( \mathrm{i} [0,\sqrt{2c}] \cup \mathrm{i} [1/\sqrt{2c},\infty))$ the contours of steepest ascent and descent of $\mathrm{Re} \, g_{\xi,\xi} (\omega)$ are the level lines of $\mathrm{Im} \, g_{\xi,\xi} (\omega)$.
From Lemma~\ref{asym:lem:derivativesliquidgas}, we see that the descent path  ($\omega_1$-integral) leaves $\mathrm{i}$ at an angle $-\pi/6$ and goes into $\mathbb{H}_+$. From Lemma~\ref{asym:lem:derivativesliquidgas}, we see that the ascent path  ($\omega_2$-integral) leaves $\mathrm{i}$ at an angle $\pi/6$ and goes into $\mathbb{H}_+$.  The endpoints of these contours follows by mirroring the argument given in the proof of Lemma~\ref{asym:lem:contoursgas} using Lemma~\ref{asym:lem:imaginary}.

\end{proof}
The top left figure in Fig.~\ref{fig:contoursall} shows a realization of these contours.
Up to orientation, these contours are symmetric in both the real and imaginary axis.  The orientations are easily determined.

\begin{figure}
\begin{center}
\includegraphics[height=6cm]{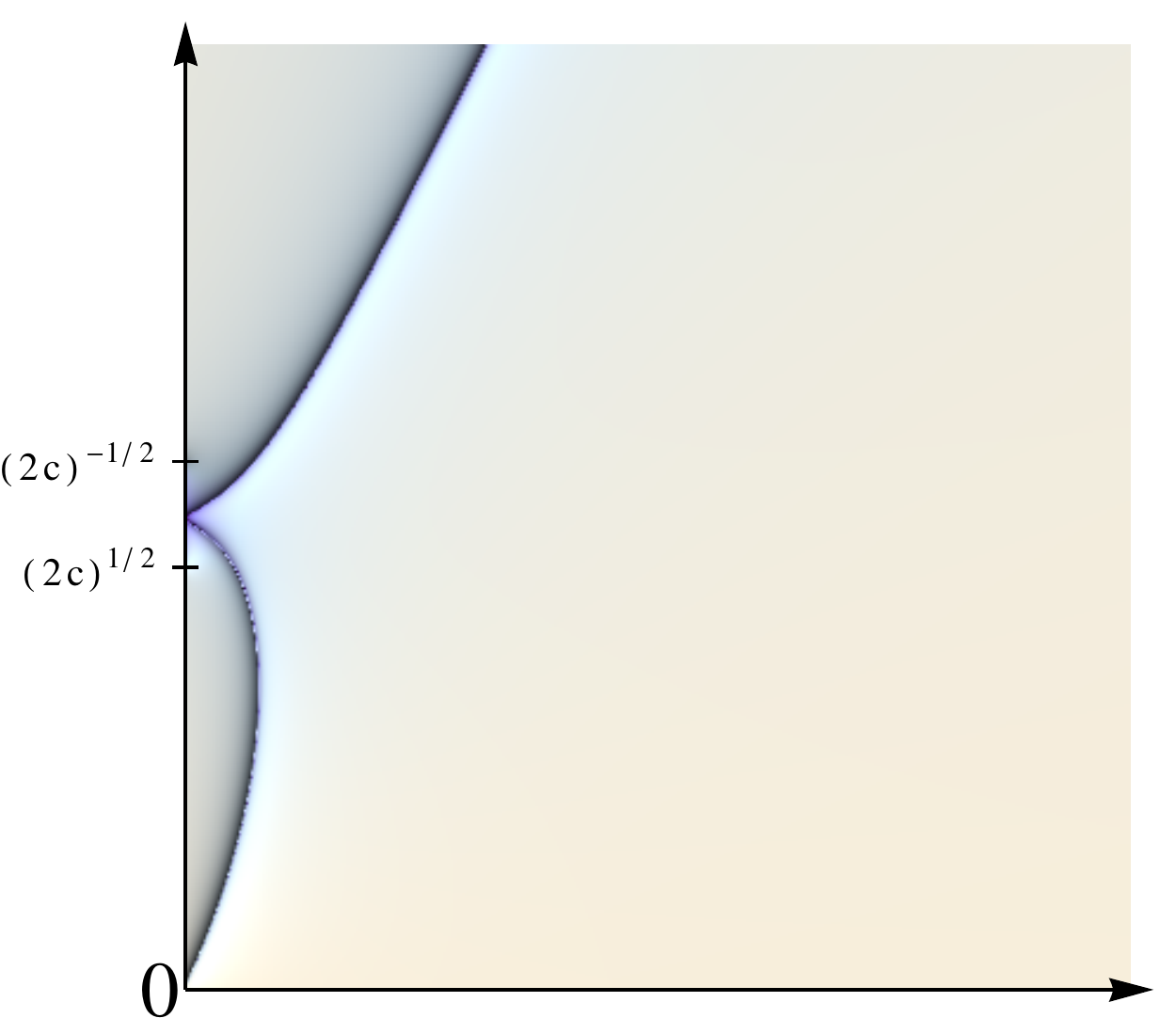}
\includegraphics[height=6cm]{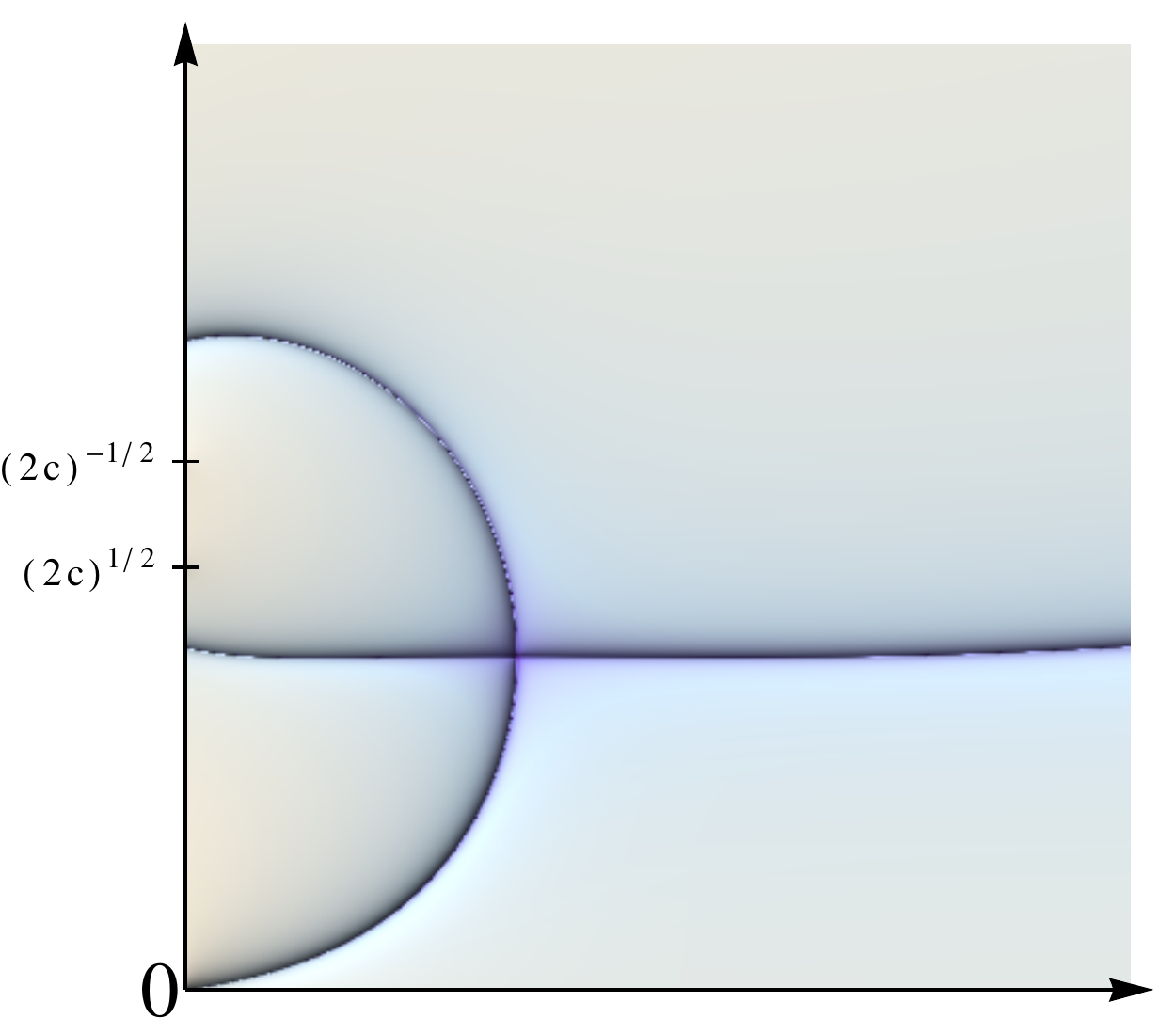}
\includegraphics[height=6cm]{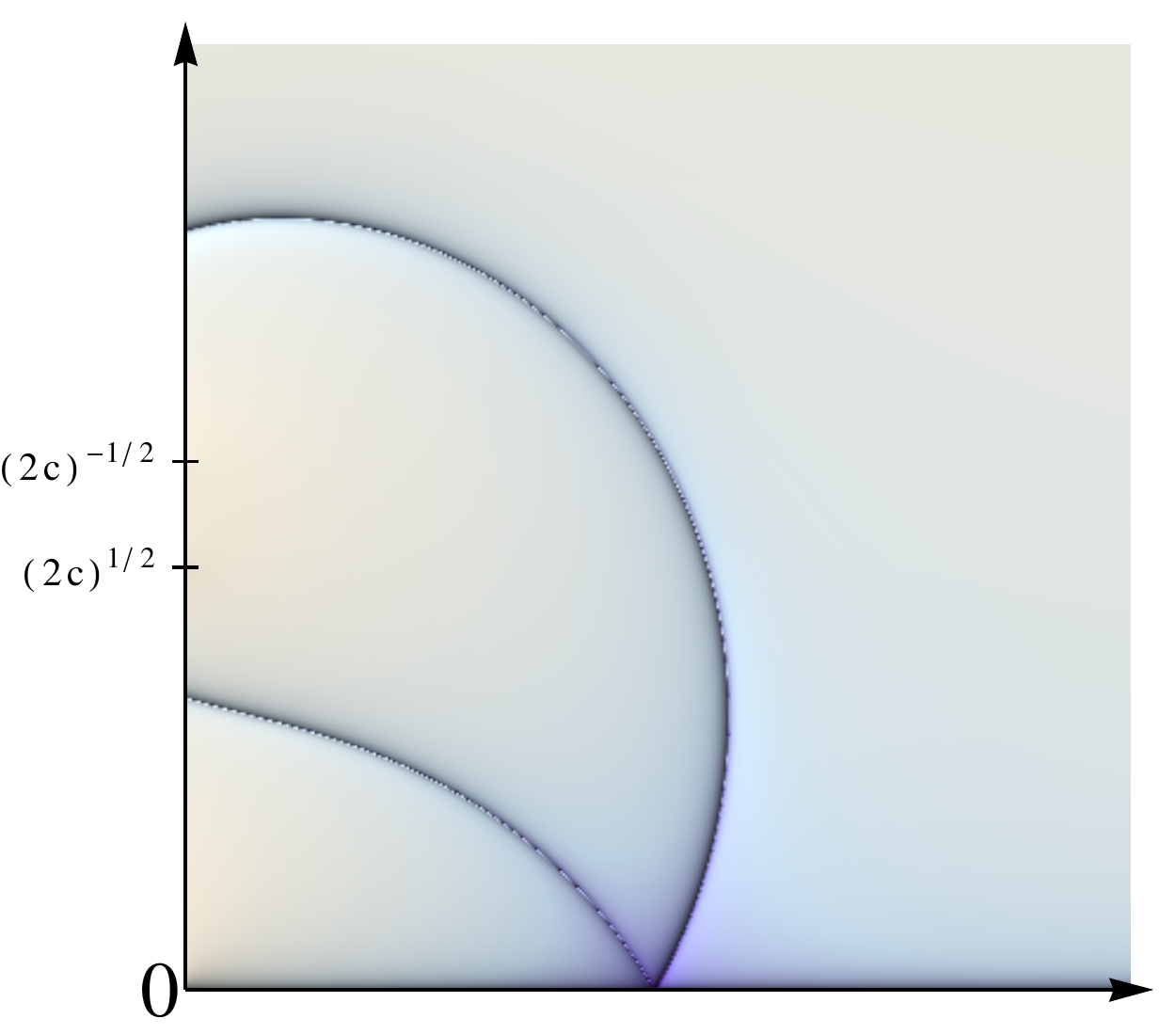}
\includegraphics[height=6cm]{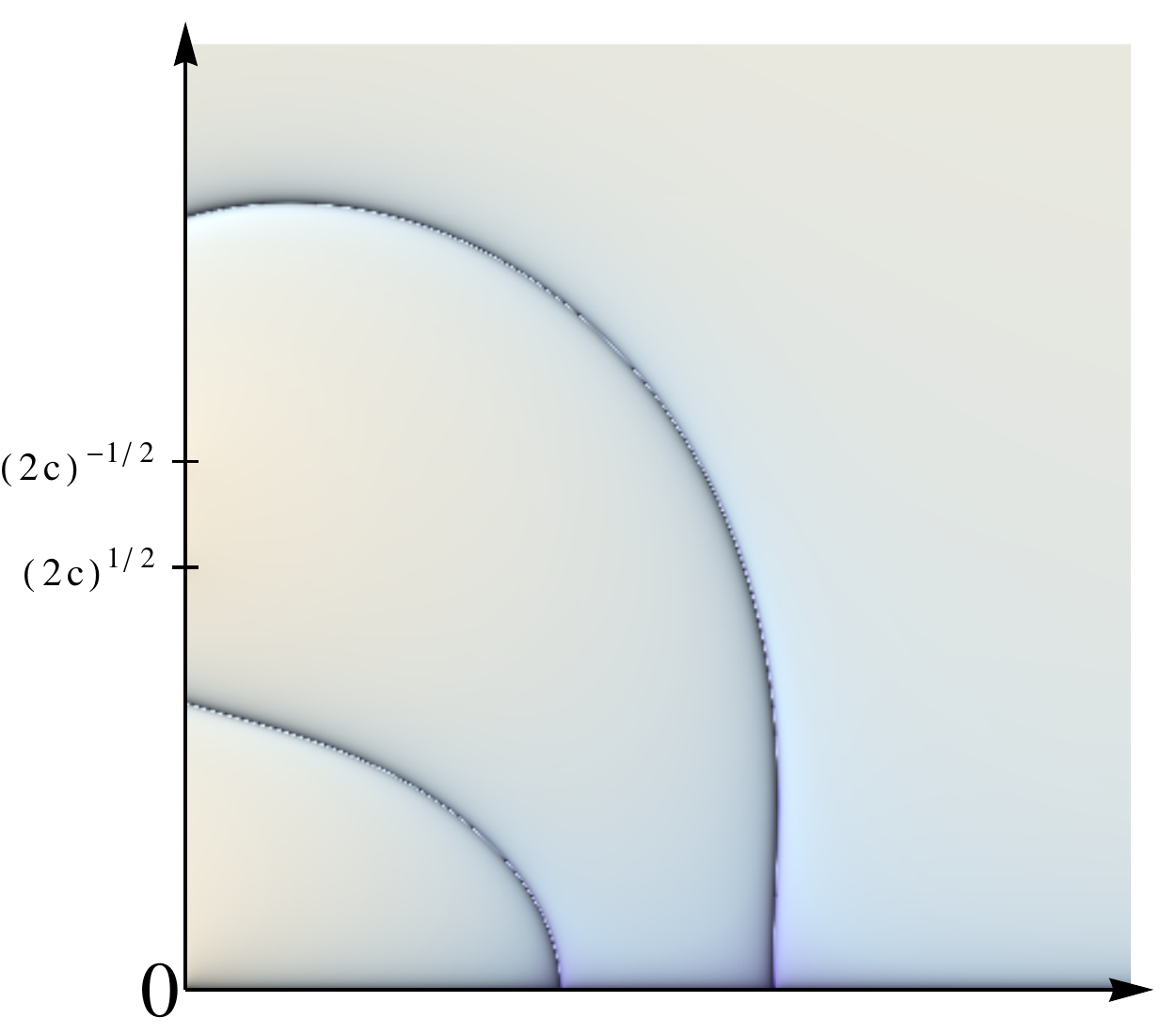}
\caption{The contours of steepest ascent and descent of $g_{\xi,\xi}(\omega)$ for  $\omega \in \mathbb{H}_+ \cup \partial  \mathbb{H}_+  $ and $a=0.5$ which means $c=0.4$.  The plots are made through a relief plot of $\log | \Im( g_{\xi,\xi}(\omega) - g_{\xi,\xi}(\omega_c))|$ where the logarithm is used to sharpen the plot. The top left figure has $\omega_c= \mathrm{i}$ corresponding to the liquid-gas boundary. The top right figure has $\omega_c=e^{\frac{\mathrm{i}\pi }{4}}$ corresponding to the liquid region. The bottom left figure has $\omega_c=1$ corresponding to the solid-liquid boundary. The bottom right figure has $\omega_c=0.8$ corresponding to the solid region.
} 
\label{fig:contoursall}
\end{center}
\end{figure}
The following lemma is required for the evaluation of the asymptotics of the integral in Eq.~\eqref{asfo:lemproof:liquidV}.
\begin{lemma}\label{asfo:lem:gascoeff}
\begin{equation}
\mathtt{g}_{\eps_1,\eps_2} = -2\mathrm{i} V_{\eps_1,\eps_2} (\mathrm{i},\mathrm{i})
\end{equation}
\end{lemma}

\begin{proof}
The proof follows setting $\omega=\mathrm{i}$ in the formula for $V_{\eps_1,\eps_2}(\omega,\omega)$ given in Lemma~\ref{asfo:lem:V} and simplifying. Note that $\sqrt{1/\mathrm{i}^2+2c}=-\sqrt{\mathrm{i}^2+2c}$ and $G(1/\mathrm{i})=-G(\mathrm{i})$ which both follow from the choice of branch cut. We have $\sqrt{\mathrm{i}^2+2c}=\mathrm{i} \sqrt{1-2c}$ and $G(\mathrm{i}) = \frac{\mathrm{i}}{\sqrt{2c}} (1-\sqrt{1-2c})$ which also follow from the choice of branch cut.
\end{proof}

We have the following proposition.
\begin{prop}\label{asfo:lem:Bliquidgas}
Choose $x$ and $y$ as given in Theorem~\ref{thm:transversal1} with $\xi=-\frac{1}{2} \sqrt{1-2c}$. Then, we have
\begin{equation}
\begin{split}\label{asym:lem:eq:transKinvgas}
\mathcal{B}_{\eps_1,\eps_2}(a,x_1,x_2,y_1,y_2)&=
\mathrm{i}^{y_1-x_1+1} 
|G(\mathrm{i})|^{\frac{1}{2}(-2-x_1+x_2+y_1-y_2)}e^{\beta_x \alpha_x-\beta_y \alpha_y+\frac{2}{3}(\beta_x^3-\beta_y^3)} c_0 \mathtt{g}_{\eps_1,\eps_2}\\
&\times\tilde{\mathcal{A}}( (-\beta_x, \alpha_x+\beta_x^2);(-\beta_y,\alpha_y+\beta_y^2))(2m)^{-1/3}\left(1 +O\left(m^{-\frac{1}{3}}\right)\right).
\end{split}
\end{equation}
where $c_0$ and the scale factors $\lambda_1$ and $\lambda_2$ are given in the statement of Theorem~\ref{thm:transversal1}.
\end{prop}

\begin{proof}
For the proof, we use the form of $\mathcal{B}_{\eps_1,\eps_2}(a,x_1,x_2,y_1,y_2)$ as given in~\eqref{asfo:lemproof:liquidV} and ignore the integer parts because  their effect is  negligible.   For the contour of integration with respect to $\omega_1$, deform to the contour of steepest descent given in Lemma~\ref{asym:lem:contoursliquidgas}, and use symmetry to  extend to the remaining three quadrants.  For the contour of integration with respect to $\omega_2$, deform to the contour of steepest ascent given  in Lemma~\ref{asym:lem:contoursliquidgas}, and use symmetry to  extend to the remaining three quadrants. The contours do not cross under these deformations.  

Standard saddle point arguments and the crude estimate given in~\eqref{asfo:lemproof:crudebound} which controls the saddle point function on the infinite length contour means that we only need to consider the local saddle point contributions of $\mathcal{B}_{\eps_1,\eps_2}(a,x_1,x_2,y_1,y_2)$.  By the symmetry in the real axis for our expressions, we only need to consider the local saddle point contribution at $\mathrm{i}$ since the local saddle point contribution  at $-\mathrm{i}$ has the same value.

Write the local descent and ascent contours in a neighborhood of the double critical point as
\begin{equation} \label{asfo:lemproof:liquidgas:local}
\omega_1=\mathrm{i}+c_0(2m)^{-1/3}w  \hspace{5mm}\mbox{and}\hspace{5mm}\omega_2=\mathrm{i}+c_0(2m)^{-1/3}z,
\end{equation}
where $|w|,|z|<m^\delta$ with $0<\delta<1/12$ and $c_0$ is given in Theorem~\ref{thm:transversal1}.  This gives contours ${\mathcal{C}}_{1,m}$ and ${\mathcal{C}}_{2,m}$  for the integral with respect to  $w$ and $z$ respectively.  

We apply an expansion of $H_{x_1+1,x_2}(\omega_1)$ and $H_{y_1,y_2+1}(\omega_2)$  for $\omega_1$ and $\omega_2$ chosen in~\eqref{asfo:lemproof:liquidgas:local} and $x$ and $y$ taken from the statement of the proposition.  
After some computation we find that
\begin{equation}
\begin{split}
\frac{H_{x_1+1,x_2}(\omega_1)}{H_{y_1,y_2+1}(\omega_2)}&=\frac{ G\left(\mathrm{i}\right)^{-\frac{x_1+1}{2}} G\left(\mathrm{i}^{-1}\right)^{-\frac{y_2+1}{2}}} { G\left(\mathrm{i}\right)^{-\frac{x_2}{2}} G\left(\mathrm{i}^{-1}\right)^{-\frac{y_1}{2}}}
  \exp\left(-\frac{2  \mathrm{i}c (1+c) w^3 c_0^3}{3 (1-2 c)^2}+\frac{2 \mathrm{i} c (1+c) z^3 c_0^3}{3 (1-2 c)^2}\right.\\ &\left. -\frac{2 \mathrm{i} w c_0 \alpha _x \lambda _1}{\sqrt{1-2 c}}+\frac{2  \mathrm{i}z c_0 \alpha _y \lambda _1}{\sqrt{1-2 c}}-\frac{2 c w^2 c_0^2 \beta _x \lambda _2}{(1-2 c)^{3/2}}+\frac{2 c z^2 c_0^2 \beta _y \lambda _2}{(1-2 c)^{3/2}}+\mathrm{Err.} \right) 
\end{split}
\end{equation}
by using the form of $H_{x_1+1,x_2}$ given in~\eqref{asfo:lem:H} and the exponential term follows by using a Taylor series approximation of $g_{\xi,\xi}(\omega_i)$ with the local change of variables~\eqref{asfo:lemproof:liquidgas:local}. The term $\mathrm{Err.}$ in the above equation is equal to $m^{-1/3}( z^4 R_m(z)+w^4 S_m(w))$ where $R_m(z)$ and $S_m(w)$ tend to constants as $m$ tends to infinity for $|z|,|w|<m^{\delta}$.
Using the definition of $c_0$, $\lambda_1$ and $\lambda_2$ as given in the statement of Theorem~\ref{thm:transversal1}, we find that  the right side of the above equation reduces to  
\begin{equation}
\begin{split}
\frac{ G\left(\mathrm{i}\right)^{-\frac{x_1+1}{2}} G\left(\mathrm{i}^{-1}\right)^{-\frac{y_2+1}{2}}} { G\left(\mathrm{i}^{-1}\right)^{-\frac{x_2}{2}} G\left(\mathrm{i}\right)^{-\frac{y_1}{2}}}
e^{  -\frac{\mathrm{i}}{3}( w^3- z^3) -(\beta_x w^2-\beta_y z^2) -\mathrm{i}(\alpha_x w - \alpha_y z)+\mathrm{Err.} }.
\end{split}
\end{equation}
We have $G\left(\mathrm{i}^{-1}\right)=G\left(-\mathrm{i}\right)=-G\left(\mathrm{i}\right)$ by~\eqref{cuts:sign1}, which means that 
\begin{equation}
\begin{split}
 \frac{H_{x_1+1,x_2}(\omega_1)}{H_{y_1,y_2+1}(\omega_2)}&=\mathrm{i}^{\frac{y_1+y_2-x_1-x_2}{2}} \left| G\left(\mathrm{i}\right)\right|^{\frac{-2+x_2-x_1+y_1-y_2}{2}}  e^{ -\frac{\mathrm{i}}{3}( w^3- z^3) -(\beta_x w^2-\beta_y z^2) -\mathrm{i}(\alpha_x w - \alpha_y z)+\mathrm{Err.} } .
\end{split}
\end{equation}
The rest of the integrand in~\eqref{asfo:lemproof:liquidV} under the change of variables in~\eqref{asfo:lemproof:liquidgas:local} is given by
\begin{equation}
\mathrm{i}^{\frac{x_2-x_1+y_1-y_2}{2}} \frac{V_{\eps_1,\eps_2} (\omega_1,\omega_2)}{\omega_1(\omega_2-\omega_1)}  = \mathrm{i}^{\frac{x_2-x_1+y_1-y_2}{2}}\frac{V_{\eps_1,\eps_2}(\mathrm{i},\mathrm{i})}{\mathrm{i} (z-w)c_0(2m)^{-\frac{1}{3}}} (1+O(m^{-\frac{1}{3}}))
\end{equation}
Using the above two equations, 
we find that the local contribution of $\mathcal{B}_{\eps_1,\eps_2}(a,x_1,x_2,y_1,y_2)$ around the double critical point $\mathrm{i}$ is given by
\begin{equation}
\begin{split}
&\mathrm{i}^{y_1-x_1} \left| G\left(\mathrm{i}\right)\right|^{\frac{-2+x_2-x_1+y_1-y_2}{2}} c_0 (2m)^{-1/3}  (1+O(m^{-\frac{1}{3}}))\\
&\times \frac{1}{(2\pi \mathrm{i})^2} \int_{{\mathcal{C}}_{1,m}} dw \int_{{\mathcal{C}}_{2,m}} dz\frac{V_{\eps_1,\eps_2}(\mathrm{i},\mathrm{i})}{\mathrm{i}} \frac{e^{-\frac{\mathrm{i}}{3}( w^3- z^3) -(\beta_x w^2-\beta_y z^2) -\mathrm{i}(\alpha_x w - \alpha_y z) + \mathrm{Err.}}}{z-w} .
\end{split}
\end{equation} 
In the local neighborhood of $\mathrm{i}$, the curves ${\mathcal{C}}_{1,m}$ and ${\mathcal{C}}_{2,m}$ converge to $\mathcal{C}_1$ and $\mathcal{C}_2$ respectively. We then use the symmetry in the real axis to find the contribution from the double critical point at $-\mathrm{i}$ which gives the same contribution as above. Using Lemma~\ref{asfo:lem:gascoeff}, we find that
\begin{equation}
\begin{split} \label{asfo:propproof:Airylg}
&\mathcal{B}_{\eps_1,\eps_2}(a,x_1,x_2,y_1,y_2)=\mathrm{i}^{-1} \mathrm{i}^{y_1-x_1+1} c_0 (2m)^{-1/3}
\left| G\left(\mathrm{i}\right)\right|^{\frac{-2+x_2-x_1+y_1-y_2}{2}}   \mathtt{g}_{\eps_1,\eps_2} \\
&\times \frac{1}{(2\pi \mathrm{i})^2} \int_{{\mathcal{C}}_1} dw \int_{{\mathcal{C}}_2} dz \frac{e^{ -\frac{\mathrm{i}}{3}( w^3- z^3) -(\beta_x w^2-\beta_y z^2) -\mathrm{i}(\alpha_x w - \alpha_y z) }}{z-w}(1 +O(m^{-\frac{1}{3}})).
\end{split}
\end{equation} 
The result follows from Eq.~\eqref{asfo:eq:Airymod}.

\end{proof}

\subsection{Liquid region} \label{subsection:liquid}

We continue the asymptotic study of~\eqref{asfo:lemproof:liquidV} by  moving the asymptotic coordinate into the liquid region.  As above, we analyze the behavior of the saddle points and determine the behavior of the contours of steepest descent and ascent before proceeding with the analysis.

The next lemma gives the behavior of the saddle points for~\eqref{BZ} when the asymptotic coordinate, $\xi$, is in $(-\sqrt{1+2c}/2,-\sqrt{1-2c}/2)$.

\begin{lemma} \label{asym:lem:rootsliquid}
The saddle point of the function $g_{\xi,\xi}(\omega)$ for $\omega\in \mathbb{H}_+ \cup\partial \mathbb{H}_+$ has
a single critical point $\omega_c$  for  $-1/2 \sqrt{1+2c} <\xi<-1/2 \sqrt{1-2c} $. This single critical point is given by $\omega_c=e^{\mathrm{i} \theta_c}$ for $\theta_c \in (0,\pi/2)$ where $\theta_c$ is defined by the value of $\xi$. Moreover, for $\xi$ chosen so that $g'_{\xi,\xi}(\omega_c)=0$, we have $\mathrm{Re} g_{\xi,\xi}(\omega_c)=0$.

\end{lemma}
\begin{proof}
From setting $\xi_1=\xi_2=\xi$ in  Eq.~\eqref{asym:eq:saddleeqn}, it is immediate that if $\omega_c$ is a zero of~\eqref{asym:eq:saddleeqn}, then so are $1/\omega_c$ and $-\omega_c$ and $1/\omega_c$ and their complex conjugates.  
If $\omega_c \in \mathbb{H}_+$  is a root of~\eqref{asym:eq:saddleeqn} with $\xi_1=\xi_2=\xi$, then it follows that $\overline{1/\omega_c}$ is also a root and in the same quadrant. 
Since there can only be at most four roots of~\eqref{asym:eq:saddleeqn} by Lemma~\ref{asym:lem:saddle}, it follows that 
$\omega_c=\overline{1/\omega_c}$ which means that $|\omega_c|=1$. The condition $-1/2 \sqrt{1+2c} <\xi<-1/2 \sqrt{1-2c} $ is determined from
\begin{equation}
\xi=\frac{-1}{\frac{\omega_c}{\sqrt{\omega_c^2+2c}}+\frac{\omega_c^{-1}}{\sqrt{\omega_c^{-2}+2c}}}
\end{equation}
for $\omega_c=e^{\mathrm{i} \theta_c}$ with $\theta_c \in (0,\pi/2)$ . The last statement in the lemma follows immediately from the fact that $|G(e^{\mathrm{i}\theta})|=|G(e^{-\mathrm{i}\theta})|$.
\end{proof} 

We now find $g''_{\xi,\xi}(\omega_c)$ where $\omega_c$ is the single critical point.

\begin{lemma} \label{asym:lem:derivativesliquid}
For $-\frac{1}{2} \sqrt{1+2c} < \xi< -\frac{1}{2} \sqrt{1-2c}$, choose $\omega_c=e^{i \theta_c}$ with $0<\theta_c<\pi/2$ such that $g_{\xi,\xi}'( \omega_c)=0$ and  set $s_1, s_2 \in \mathbb{R}$ and $\phi_1,\phi_2$ such that
\begin{equation}
e^{\mathrm{i} \theta_c} +(-1)^{k+1} \mathrm{i} \sqrt{2c} = s_{k} e^{\mathrm{i} \phi_k}
\end{equation}
for $k \in \{1,2\}$, then 
\begin{equation}
\begin{split}
g''_{\xi,\xi}(\omega_c)&= \frac{4 c |\xi|}{(s_1s_2)^{3/2}}e^{\mathrm{i}(\pi/2-2 \theta_c)} \sin \left(  \frac{3}{2}(\phi_1+\phi_2)-\theta_c \right).\\
\end{split}
\end{equation}
\end{lemma}

\begin{proof}
For $-\frac{1}{2} \sqrt{1+2c} < \xi< -\frac{1}{2} \sqrt{1-2c}$, the above definitions means that we can write
\begin{equation}
\sqrt{\omega^2_c+2c}= \sqrt{s_1 s_2} e^{\mathrm{i} (\phi_1+\phi_2)/2}
\end{equation}
and
\begin{equation}
\sqrt{\omega^{-2}_c+2c}= \sqrt{s_1 s_2} e^{-\mathrm{i} (\phi_1+\phi_2)/2}.
\end{equation}
From~\eqref{asym:eqn:double},  setting $g'_{\xi,\xi}(\omega_c)=0$ and using the fact that
\begin{equation}
\left( \frac{\omega_c}{(\omega^2_c+2c)^{3/2}} - \frac{ \omega^{-1}_c}{(\omega^{-2}_c+2c)^{3/2}} \right) = \frac{2 \mathrm{i}}{(s_1s_2)^{3/2}} \sin \left( \theta - \frac{3}{2}(\phi_1+\phi_2) \right)
\end{equation}
we obtain
\begin{equation}
\begin{split}
g''_{\xi,\xi}(\omega_c)&= \frac{4 c \xi \mathrm{i}}{(s_1s_2)^{3/2}}e^{-2\mathrm{i} \theta_c} \sin \left( \theta_c - \frac{3}{2}(\phi_1+\phi_2) \right)\\
&= \frac{4 c |\xi|}{(s_1s_2)^{3/2}}e^{\mathrm{i}(\pi/2-2 \theta_c)} \sin \left(  \frac{3}{2}(\phi_1+\phi_2)-\theta_c \right).\\
\end{split}
\end{equation}
\end{proof}

We now describe the contours of steepest ascent and descent. The top right figure in Fig.~\ref{fig:contoursall} shows a realization of these contours.
Up to orientation, these contours are symmetric in both the real and imaginary axis.  The orientations are easily determined.

\begin{lemma} \label{asym:lem:contoursliquid}
Assume the same conditions given in Lemma~\ref{asym:lem:derivativesliquid}.  There is 
\begin{itemize}
\item a contour of steepest ascent leaving $\omega_c$ at an angle $\theta_c-\pi/4$ ending at infinity,
\item a contour of steepest descent leaving $\omega_c$ at  an angle $\theta_c+\pi/4$ ending at a cut and a descent contour ending at $\mathrm{i}/\sqrt{2c}$ traveling via the cut $\mathrm{i}[1/\sqrt{2c}, \infty)$,
\item a contour of steepest ascent leaving $\omega_c$ at an angle  $\theta_c+3\pi/4$ ending at a cut and an ascent contour ending at $\mathrm{i}\sqrt{2c}$ traveling via the cut $\mathrm{i}[0,\sqrt{2c}]$ and
\item a contour of steepest descent leaving $\omega_c$ at an angle  $\theta_c-3\pi/4$ ending at zero.
\end{itemize}
for $\mathrm{Re}\, g_{\xi,\xi}$.
\end{lemma}

\begin{proof}
Since $ g_{\xi,\xi} (\omega)$ is analytic for $\omega \in \mathbb{H}_+ \cup \partial \mathbb{H}_+\backslash ( \mathrm{i} [0,\sqrt{2c}] \cup \mathrm{i}[1/\sqrt{2c}, \infty))$ the contours of steepest ascent and descent of $\mathrm{Re} \, g_{\xi,\xi} (\omega)$ are the level lines of $\mathrm{Im} \, g_{\xi,\xi} (\omega)$.
The local analysis of the contours around $\omega_c$ as given in the statement of the lemma follows from Lemma~\ref{asym:lem:derivativesliquid}.  Indeed, one can show that both $\phi_1+\phi_2$ and $\phi_1+\phi_2 - \theta_c$ are increasing for $\theta_c \in (0,\pi/2)$ which means that $\frac{3}{2}(\phi_1+\phi_2) - \theta_c$ is also increasing in $\theta_c$ and takes values in $(0,\pi)$, hence $\sin(\frac{3}{2}(\phi_1+\phi_2) - \theta_c)>0$.

 Along the curve $e^{\mathrm{i} \theta}$ for $\theta \in (0,\pi/2)$, we have $\mathrm{Re}\, g_{\xi,\xi}(e^{\mathrm{i} \theta})=0$.  We find that $\mathrm{Im}\, g_{\xi,\xi}(e^{\mathrm{i} \theta})$ is increasing for increasing $\theta$ at $\theta=0^+$ which is found by computing the derivative 
(that is, by analyzing the arguments of $G(1+\mathrm{i}\delta)$ and $G(1/(1+\mathrm{i}\delta))$ for $\delta>0$).   We also find that $\mathrm{Im}\, g(e^{\mathrm{i} \theta})$ is decreasing for increasing $\theta$ at $\theta=\frac{\pi}{2}^-$ which is found by computing the derivative 
(that is, by analyzing the arguments of $G(\mathrm{i}+\delta)$ and $G(1/(\mathrm{i}+\delta))$ for $\delta>0$).   There is exactly one saddle point along the quarter circle $e^{\mathrm{i} \theta}$ for $\theta \in (0,\pi/2)$, because $\mathrm{Re}\, g_{\xi,\xi}(e^{\mathrm{i} \theta})=0$ on this quarter circle. We deduce that the  $\mathrm{Im}\, g_{\xi,\xi}(e^{\mathrm{i} \theta})$ is maximized at $\theta=\theta_c$ and so  $\mathrm{Im}\, g_{\xi,\xi}(e^{\mathrm{i} \theta_c}) > \max (\mathrm{Im}\, g_{\xi,\xi}(1),\mathrm{Im}\, g_{\xi,\xi}(\mathrm{i}))$.    

The paths of ascent and descent   travel to  the origin, infinity or to either of the cuts, because the imaginary part of $g_{\xi,\xi}(\omega)$ in each of these regions is larger than $\max (\mathrm{Im}\, g_{\xi,\xi}(1),\mathrm{Im}\, g_{\xi,\xi}(\mathrm{i}))$; see Lemma~\ref{asym:lem:imaginary} to compare the $\mathrm{Im}\,g_{\xi,\xi}$ in each of these regions.  The contour terminating at infinity is an ascent path ($\omega_2$-integral) and since the contour cannot cross the unit circle (because it would force a descent path to cross the unit circle which contradicts $\mathrm{Re}\, g_{\xi,\xi}(e^{\mathrm{i} \theta})=0$), it follows that the contour leaving $\omega_c$ at an angle $\theta_c-\pi/4$ is a contour of steepest ascent terminating at infinity. The contour terminating at zero is a descent path ($\omega_1$-integral), it follows that the contour leaving $\omega_c$ at an angle $\theta_c-3\pi/4$ is a contour of steepest descent terminating at zero.  Since the contours cannot intersect, the remaining path of steepest ascent from $\omega_c$  must pass into the cut $[0,\sqrt{2c}]\mathrm{i}$ while the remaining path of steepest descent from $\omega_c$ must pass  into the cut $[1/\sqrt{2c},\infty)\mathrm{i}$, ensuring that $\mathrm{Im}\, g_{\xi,\xi}$ is constant along both of these contours.  We continue the path along the cut to $\mathrm{i}/\sqrt{2c}$. It remains to show that this is a descent path.

Suppose that $l$ is a point in the cut $[1/\sqrt{2c},\infty)\mathrm{i}$ such that $\mathrm{Im}\, g_{\xi,\xi}(\omega)=\mathrm{Im}\, g_{\xi,\xi}(l)$.
To show that there is a  descent path traveling from a point $l$ to $1/\sqrt{2c} \mathrm{i}$, it suffices  to show that $\mathrm{Re}\, g_{\xi,\xi}(\mathrm{i}t)$ is increasing for $t$ increasing in $[1/\sqrt{2c},\infty)$.   From~\eqref{cuts:squareroot}, we have
\begin{equation}
\lim_{\delta \to 0^+}G(\mathrm{i}t+\delta) = \frac{ t-\sqrt{t^2-2c}}{\sqrt{2c}} \mathrm{i} \hspace{2mm} \mbox{and} \hspace{2mm}\lim_{\delta \to 0^+}G\left( (\mathrm{i}t+\delta)^{-1} \right)= -\frac{t^{-1}\mathrm{i}+\sqrt{2c-t^{-2}} }{\sqrt{2c}  }
\end{equation}
and so
\begin{equation}
\lim_{\delta \to 0^+}\mathrm{Re}\, g_{\xi,\xi}(\mathrm{i}t+\delta) = \log t -\xi \log \frac{t-\sqrt{t^2-2 c} }{\sqrt{2c}} ,
\end{equation}
which means that  $\mathrm{Re}\, g_{\xi,\xi}(\mathrm{i}t)$ is increasing for $t$ increasing in $[1/\sqrt{2c},\infty)$.  An analogous argument holds for the remaining cut.

\end{proof}

\begin{prop}
\label{asfo:lem:Bliquid}
Under the assumptions given in Theorem~\ref{thm:Kinverselimit} and with $-\frac{1}{2}\sqrt{1+2c}<\xi<-\frac{1}{2}\sqrt{1-2c}$, we have
\begin{equation}
\mathcal{B}_{\eps_1,\eps_2}(a,x_1,x_2,y_1,y_2) = (\mathbb{K}^{-1}_{1,1}-\mathbb{K}^{-1}_{r_1,r_2} )((\overline{x}_1 ,\overline{x}_2) ,(\overline{y}_1 ,\overline{y}_2) ) +O(m^{-1/2}),\end{equation}
where $r_1$ and $r_2$ are given in Theorem~\ref{thm:Kinverselimit}.
\end{prop}

\begin{proof}
We use the formula~\eqref{asfo:lemproof:liquidV} for $B_{\varepsilon_1,\varepsilon_2}(x,y)$, $\sqrt{2c}<r<1$. Move $\Gamma_{r}$ and $\Gamma_{1/r}$ to the steepest descent
and descent contours respectively for the saddle points $\pm\omega_c$, $\pm\overline{\omega}_c$. Since the $\omega_2$-contour has to pass over the $\omega_1$-contour
we pick up a single integral contribution which equals $C_{\omega_c}(x,y)$. The new double contour integral with the steepest descent and ascent contours can be shown to be
$O(n^{-1/2})$ by an argument similar to that used in \cite{OR:03}. We omit the details. Thus
\begin{equation}
B_{\varepsilon_1,\varepsilon_2}(x,y)=C_{\omega_c}(x,y)+O(m^{-1/2})=\mathbb{K}_{1,1}(x,y)-\mathbb{K}_{r_1,r_2}(x,y)+O(m^{-1/2}),
\end{equation}
by Lemma~\ref{asfo:lem:liquidcomps}.
\end{proof}

\subsection{Solid-liquid boundary} \label{subsection:liquidsolid}

We continue the asymptotic study of~\eqref{asfo:lemproof:liquidV}  by moving to  the solid-liquid boundary.  Similar to  the liquid-gas boundary, we determine the behavior of the saddle point, the contours of steepest ascent and steepest descent and find an expansion of the expressions found in~\eqref{asfo:lemproof:liquidV} at the saddle point. 

The next lemma determines the behavior of the saddle point of~\eqref{asfo:lemproof:liquidV} at the solid-liquid boundary.

\begin{lemma} \label{asym:lem:rootsliquidsolid}
The saddle point of the function $g_{\xi,\xi}(\omega)$ for $\omega\in \mathbb{H}_+ \cup \partial \mathbb{H}_+ $ has  a double critical point at $w=1$ if  and only if $\xi=-1/2 \sqrt{1+2c}$.
\end{lemma}
\begin{proof}
The system of equations $\{g'_{\xi,\xi}(\omega_c)=0,g''_{\xi,\xi}(\omega_c)=0\}$  have the solutions $\omega_c=\pm1$ for $\xi=-1/2 \sqrt{1+2c}$. 
\end{proof}
Since $g'_{\xi,\xi}(1)=g''_{\xi,\xi}(1)=0$ because $1$ is a double critical point, to characterize the local nature of the contours of steepest ascent and descent, we must first ascertain $g'''_{\xi,\xi}(1)$.
\begin{lemma} \label{asym:lem:derivativesliquidsolid}
For $\xi=-\frac{1}{2}\sqrt{1+2c}$, we have
\begin{equation}
g'''_{\xi,\xi}(1) = \frac{4(1-c)}{1+2c}.
\end{equation}
\end{lemma}

\begin{proof}
From~\eqref{asym:eqn:triple}, we can set $\omega=1$ to obtain the result.
\end{proof}

We now describe the contours of steepest ascent and descent.

\begin{lemma} \label{asym:lem:contourssolidliquid}
For $\xi=-1/2\sqrt{1+2c}$, there is a contour of steepest descent leaving $1$ at angle $\pi/3$  finishing at a cut and a descent path finishing at $\mathrm{i}/\sqrt{2c} $ traveling via the cut $[1/\sqrt{2c},\infty)\mathrm{i}$, and a contour of steepest ascent leaving $1$ at angle $2\pi/3$  finishing at a cut and an ascent path ending at  $\mathrm{i}\sqrt{2c}$ traveling via the cut $[0,\sqrt{2c}]\mathrm{i}$ for $\mathrm{Re}\,g_{\xi,\xi}$.
\end{lemma}

\begin{proof}
Since $ g_{\xi,\xi} (\omega)$ is analytic for $\omega \in \mathbb{H}_+ \cup \partial \mathbb{H}_+\backslash ( \mathrm{i}[0,\sqrt{2c}] \cup \mathrm{i} [1/\sqrt{2c},\infty))$, the contours of steepest ascent and descent of $\mathrm{Re} \, g_{\xi,\xi} (\omega)$ are the level lines of $\mathrm{Im} \, g_{\xi,\xi} (\omega)$.
From Lemma~\ref{asym:lem:derivativesliquidsolid}, the contours which enter $\mathbb{H}_+$ from $1$ are a  descent path ($\omega_1$-integral), leaving $1$ at an angle $\pi/3$; and  an ascent path ($\omega_2$-integral), leaving $1$ at an angle $-\pi/3$.  Since there are also contours of steepest descent and ascent traveling along the real axis, the  contours which enter $\mathbb{H}_+$ cannot end at $0$ or infinity.  Since $\mathrm{Im} \,g_{\xi,\xi}(1)=0$ and by noting the value of $\mathrm{Im} \,g_{\xi,\xi}(\omega)$ for $\omega$ in both $\mathrm{i}[0,\sqrt{2c}]$ and $\mathrm{i}[1/\sqrt{2c},\infty)$ from Lemma~\ref{asym:lem:imaginary} for  $\xi<-1/2$, it is clear that the contours  which enter $\mathbb{H}_+$ from $1$ travel  to the cuts.  Since the steepest ascent and descent contours do not intersect, the contour of steepest ascent travels to the cut $[0,\sqrt{2c}]\mathrm{i}$ and the contour of  steepest descent  travels to the cut $[1/\sqrt{2c},\infty)\mathrm{i}$.  Following the argument given in Lemma~\ref{asym:lem:contoursliquid}, we see that we can get contours ending at the desired endpoints. 

\end{proof}
The bottom left figure in Fig.~\ref{fig:contoursall} shows a realization of these contours.
Up to orientation, these contours are symmetric in both the real and imaginary axis.  The orientations are easily determined. 

The following lemma is required for the evaluation of the asymptotics of the integral in~\eqref{asfo:lemproof:liquidV}

\begin{lemma}\label{asfo:lem:solidcoeff}
\begin{equation}
\mathtt{s}_{\eps_1,\eps_2} = 2 V_{\eps_1,\eps_2} (1,1)
\end{equation}
\end{lemma}
\begin{proof}
The proof follows setting $\omega=\mathrm{i}$ in the formula for $V_{\eps_1,\eps_2}(\omega,\omega)$ given in Lemma~\ref{asfo:lem:V} and simplifying. Here, we use the fact that  $G(1)=(1-\sqrt{1+2c})/\sqrt{2c}$.

\end{proof}

 We have the following proposition.

\begin{prop} \label{asfo:lem:Bliquidsolid}
Choose $x$ and $y$ as given in Condition~\ref{condition2} and suppose further that  $x \in \mathtt{W}_{\eps_1}$, $y\in \mathtt{B}_{\eps_2}$ for $\eps_1,\eps_2\in\{0,1\}$  and $\xi=-\frac{1}{2}\sqrt{1+2c}$. Then we have
\begin{equation}
\begin{split}\label{asym:lem:eq:transKinvsolid}
&\mathcal{B}_{\eps_1,\eps_2}(a,x_1,x_2,y_1,y_2)= \mathbb{K}_{1,1}^{-1}(x,y) -S(x,y)-G(1)^{\frac{1}{2}(-2-x_1+x_2+y_1-y_2)}
\\&\times
c_0\mathrm{i}^{\frac{x_2-x_1+y_1-y_2}{2}}\mathtt{s}_{\eps_1,\eps_2}e^{\alpha_y\beta_y-\alpha_x\beta_x +\frac{2}{3}(\beta_y-\beta_x)}
\mathcal{A}((\beta_x,\alpha_x+\beta_x^2),(\beta_y,\alpha_y+\beta_y^2))2m)^{-1/3}(1+O(m^{-1/3})), \\
\end{split}
\end{equation}
where $c_0$ and the scale factors $\lambda_1$ and $\lambda_2$ are given in the statement of Theorem~\ref{thm:transversal2}.
\end{prop}

\begin{proof}
For the proof, we start with the $\mathcal{B}_{\eps_1,\eps_2}(a,x_1,x_2,y_1,y_2)$ as given in~\eqref{asfo:lemproof:liquidV} and we ignore the integer parts since these are negligible.  Deform the contour with respect to $\omega_1$ to the contour of steepest descent given in Lemma~\ref{asym:lem:contourssolidliquid} and use symmetry to extend to the remaining three quadrants.
Deform the contour with respect to $\omega_2$ to the contour of steepest ascent given in Lemma~\ref{asym:lem:contourssolidliquid} and use symmetry to extend to the remaining three quadrants.  Let $I^{\mathtt{SL}}_{\eps_1,\eps_2}(a,x_1,x_2,y_1,y_2)$ be the double contour integral which has the same integrand as $\mathcal{B}_{\eps_1,\eps_2}(a,x_1,x_2,y_1,y_2)$ but has contours of integration described by the above deformations. 
A consequence of these contour deformations means that the contours cross completely and we pick up an additional  single integral contribution from the residue at $\omega_2$ with respect to $\omega_1$.  This single contour integral contribution is given by $C_1(x,y)$. Using  Lemma~\ref{asfo:lem:liquidcomps}, we obtain
\begin{equation}
\mathcal{B}_{\eps_1,\eps_2}(a,x_1,x_2,y_1,y_2) =\mathbb{K}^{-1}_{1,1}(x,y) -S(x,y)+I^{\mathtt{SL}}_{\eps_1,\eps_2}(a,x_1,x_2,y_1,y_2).
\end{equation}
Standard saddle point arguments means that we only need to consider the local saddle point contributions of $I^{\mathtt{SL}}_{\eps_1,\eps_2}(a,x_1,x_2,y_1,y_2)$. By the symmetry of our expressions in the imaginary axis, we only need to consider the local saddle point contribution at $1$ since the local saddle point contribution at $-1$ has the same value.

Write the local descent and ascent contours in a neighborhood of the double critical point as 
\begin{equation} \label{asfo:lemproof:solidliquid:local}
\omega_1={1}+c_0(2m)^{-1/3}w  \hspace{5mm}\mbox{and}\hspace{5mm}\omega_2={1}+c_0(2m)^{-1/3}z ,
\end{equation}
where  $|w|,|z|<m^{\delta}$ with $0<\delta<1/12$ and  $c_0$ is given in Theorem~\ref{thm:transversal2}. This gives contours $\tilde{\mathcal{C}}_{1,m}$ and $ \tilde{\mathcal{C}}_{2,m}$ for $w$ and $z$ respectively.

We apply an expansion of $H_{x_1+1,x_2}(\omega_1)$ and $H_{y_1,y_2+1}(\omega_2)$  for $\omega_1$ and $\omega_2$ chosen in~\eqref{asfo:lemproof:solidliquid:local} and $x$ and $y$ taken from the statement of the theorem.  
After some computation we find that
\begin{equation}
\begin{split}
\frac{H_{x_1+1,x_2}(\omega_1)}{H_{y_1,y_2+1}(\omega_2)}&=\frac{ G\left(1\right)^{-\frac{x_1+1}{2}} G\left(1\right)^{-\frac{y_2+1}{2}}} { G\left(1\right)^{-\frac{x_2}{2}} G\left(1\right)^{-\frac{y_1}{2}}}
  \exp\left(\frac{2 c (1-c) w^3 c_0^3}{3 (1+2 c)^2}-\frac{2 c (1-c) z^3 c_0^3}{3 (1+2 c)^2}\right.\\ &\left.- \frac{2  w c_0 \alpha _x \lambda _1}{\sqrt{1+2 c}}+\frac{2  z c_0 \alpha _y \lambda _1}{\sqrt{1+2 c}}-\frac{2 c w^2 c_0^2 \beta _x \lambda _2}{(1+2 c)^{3/2}}+\frac{2 c z^2 c_0^2 \beta _y \lambda _2}{(1+2 c)^{3/2}} +\mathrm{Err.}\right) 
\end{split}
\end{equation}
by using the form $H_{x_1,x_2}$ given in~\eqref{asfo:lem:H} and the exponential term follows by using a Taylor series approximation in $g_{\xi,\xi}$ with the local change of variables given in~\eqref{asfo:lemproof:solidliquid:local}. The term $\mathrm{Err.}$  is equal to $m^{-1/3}(z^4 R_m(z)+w^4 S_m(w))$ where $R_m(z)$ and $S_m(w)$ tend to constants as $m$ tends to infinity for $|w|,|z|<m^{\delta}$.   Using the definition of $c_0$, $\lambda_1$ and $\lambda_2$ as given in the statement of Theorem~\ref{thm:transversal2}, the right side of the above equation reduces to  
\begin{equation}
\begin{split}
\frac{ G\left({1}\right)^{-\frac{x_1+1}{2}} G\left({1}\right)^{-\frac{y_2+1}{2}}} { G\left({1}\right)^{-\frac{x_2}{2}} G\left({1}\right)^{-\frac{y_1}{2}}}
\exp \left( \frac{1}{3}( w^3- z^3) -(\beta_x w^2-\beta_y z^2) -(\alpha_x w - \alpha_y z)+ \mathrm{Err.}\right).
\end{split}
\end{equation}
We can write the above equation as 
\begin{equation}
\begin{split}
 G\left(1\right)^{\frac{-2+x_2-x_1+y_1-y_2}{2}} \exp \left( \frac{1}{3}( w^3- z^3) -(\beta_x w^2-\beta_y z^2) -(\alpha_x w - \alpha_y z)+ \mathrm{Err.}\right) .
\end{split}
\end{equation}
We also have under the change of variables~\eqref{asfo:lemproof:solidliquid:local}
\begin{equation}
\frac{V_{\eps_1,\eps_2}(\omega_1,\omega_2)}{\omega_1(\omega_2-\omega_1)}=\frac{V_{\eps_1,\eps_2}(1,1) }{c_0(z-w)(2m)^{-\frac{1}{3}}}.
\end{equation}
Using the above two local contributions, we find that the contribution of $I^{\mathtt{SL}}_{\eps_1,\eps_2}(a,x_1,x_2,y_1,y_2)$ around the double critical point 1 is given by
\begin{equation}
\begin{split}
  &\frac{\mathrm{i}^{\frac{x_2-x_1+y_1-y_2}{2}} c_0 (2m)^{-1/3} V_{\eps_1,\eps_2}(1,1)G(1)^{\frac{-2+x_2-x_1+y_1-y_2}{2}}}{(2\pi \mathrm{i})^2} \int_{\tilde{\mathcal{C}}_{1,m}} dw \int_{\tilde{\mathcal{C}}_{2,m}} dz \\ &\times  \frac{\exp \left( \frac{1}{3}( w^3- z^3) -(\beta_x w^2-\beta_y z^2) -(\alpha_x w - \alpha_y z)+\mathrm{Err.} \right)}{z-w}.
\end{split}
\end{equation}
To the above integral,
we apply the change of variables $w\mapsto w \mathrm{i}$ and $z\mapsto z \mathrm{i}$ which means that the contours $\tilde{\mathcal{C}}_{1,m}$ and $\tilde{\mathcal{C}}_{2,m}$ are mapped  ${\mathcal{C}}_{1,m}$ and ${\mathcal{C}}_{2,m}$ respectively, which are both defined in the proof of Proposition~\ref{asfo:lem:Bliquidgas}. Under these change of variables, the above equation is equal to
\begin{equation}
\begin{split}
  &-\frac{\mathrm{i}^{\frac{x_2-x_1+y_1-y_2}{2}} c_0 (2m)^{-1/3} V_{\eps_1,\eps_2}(1,1)G(1)^{\frac{-2+x_2-x_1+y_1-y_2}{2}}}{\mathrm{i}(2\pi \mathrm{i})^2} \int_{{\mathcal{C}}_{1,m}} dw \int_{{\mathcal{C}}_{2,m}} dz \\ &\times  \frac{\exp \left( -\frac{\mathrm{i}}{3}( w^3- z^3) +(\beta_x w^2-\beta_y z^2) -\mathrm{i}(\alpha_x w - \alpha_y z)+\mathrm{Err.} \right)}{z-w}.
\end{split}
\end{equation}
In the local neighborhood of $\mathrm{i}$, the curves ${\mathcal{C}}_{1,m}$ and ${\mathcal{C}}_{2,m}$ converge to $\mathcal{C}_1$ and $\mathcal{C}_2$ respectively.  We use the symmetry in the imaginary axis to find the contribution from the double critical point at $-{1}$ which gives an extra factor of $2$.  Using Lemma~\ref{asfo:lem:solidcoeff}, we find
\begin{equation}
\begin{split}
&I^{\mathtt{SL}}_{\eps_1,\eps_2}(a,x_1,x_2,y_1,y_2)= 
-\frac{\mathrm{i}^{\frac{x_2-x_1+y_1-y_2}{2}} c_0 (2m)^{-1/3} \mathtt{s}_{\eps_1,\eps_2}G(1)^{\frac{-2+x_2-x_1+y_1-y_2}{2}}}{\mathrm{i}(2\pi \mathrm{i})^2} \int_{{\mathcal{C}}_{1}} dw \int_{{\mathcal{C}}_{2}} dz \\ &\times  \frac{\exp \left( -\frac{\mathrm{i}}{3}( w^3- z^3) +(\beta_x w^2-\beta_y z^2) -\mathrm{i}(\alpha_x w - \alpha_y z)+\mathrm{Err.} \right)}{z-w}.
\end{split}
\end{equation}
The result follows from Eq.~\eqref{asfo:eq:Airymod}.

\end{proof}

\subsection{Solid region} \label{subsection:solid}

We continue the asymptotic study of~\eqref{asfo:lemproof:liquidV} by moving the asymptotic coordinate into the solid region. Similar to the previous cases, we determine the behavior of the saddle points and the contours of steepest ascent and steepest descent. 

\begin{lemma} \label{asym:lem:rootssolid}
The saddle point of the function $g_{\xi,\xi}(\omega)$ for $\omega\in \mathbb{H}_+ \cup\partial \mathbb{H}_+$ has
two distinct real positive saddle points in the interval (not equal to 1)  if and only if  $-1<\xi<-1/2 \sqrt{1+2c}$. These saddle points are  given by $\omega_c\in (1,\infty) $ and $\omega_c^{-1}$.  We also have $\mathrm{Re}g_{\xi,\xi}(\omega_c )<0$ and $\mathrm{Re}g_{\xi,\xi}(\omega_c^{-1})>0$
\end{lemma}
\begin{proof}
From setting $\xi_1=\xi_2=\xi$ in  Eq.~\eqref{asym:eq:saddleeqn}, it is immediate that if $\omega_c$ is a root of~\eqref{asym:eq:saddleeqn}, then so are $1/\omega_c$ and $-\omega_c$ and $-1/\omega_c$ which give the four possible roots of~\eqref{asym:eq:saddleeqn} as characterized by Lemma~\ref{asym:lem:saddle}.
By writing 
\begin{equation}
\xi=\frac{-1}{\frac{\omega_c}{\sqrt{\omega_c^2+2c}}+\frac{\omega_c^{-1}}{\sqrt{\omega_c^{-2}+2c}}}
\end{equation}
we can set $\omega_c= t$ for $t \in (1,\infty)$ and observe that $\xi$ is decreasing in $t$, giving the interval $-1<\xi<-1/2 \sqrt{1+2c}$.
 
For the second condition, using the above parameterization of $\xi$ and~\eqref{asym:eq:saddleeqn} we find
\begin{equation}\label{asfo:lemproof:exponential:bound2}
g_{\xi,\xi}(\omega_c)=\log\omega _c-\frac{\log G\left(\frac{1}{\omega _c}\right) \sqrt{2 c+\frac{1}{\omega _c^2}} \omega _c \sqrt{2 c+\omega _c^2}}{\sqrt{2 c+\frac{1}{\omega _c^2}} \omega _c^2+\sqrt{2 c+\omega _c^2}}+\frac{\log G\left(\omega _c\right) \sqrt{2 c+\frac{1}{\omega _c^2}} \omega _c \sqrt{2 c+\omega _c^2}}{\sqrt{2 c+\frac{1}{\omega _c^2}} \omega _c^2+\sqrt{2 c+\omega _c^2}}.
\end{equation}
 Using the above equation it can be checked after some computation that $\Re g_{\xi,\xi}(\omega_c)$ is decreasing for $\omega_c= t$ for $t \in (1,\infty)$, i.e. by differentiating.  It can be easily checked that $\Re g_{\xi,\xi}(1)=0$. This gives
$\mathrm{Re}g_{\xi,\xi}(\omega_c )<0$ and
$\mathrm{Re}g_{\xi,\xi}(\overline{-\omega_c^{-1} })>0$ follows by a similar computation.
\end{proof}

To determine the nature of the contours of steepest ascent and steepest descent, we need the following lemma:
\begin{lemma} \label{asym:lem:derivativessolid}
For $-1< \xi<-\frac{1}{2}\sqrt{1+2c}$, choose $\omega_c$ such that  $g_{\xi,\xi}'( \omega_c)=0$ with $\omega_c \in(1,\infty)$, then $g_{\xi,\xi}'( \omega_c^{-1})=0$,
\begin{equation}
g_{\xi,\xi}''( \omega_c) <0 \hspace{5mm} \mbox{and} \hspace{5mm} g_{\xi,\xi}''( \omega_c^{-1}) >0.
\end{equation}

\end{lemma}

\begin{proof}
It is immediate from Lemma~\ref{asym:lem:rootssolid} that $g_{\xi,\xi}'( \omega_c^{-1})=0$.
In Eq.~\eqref{asym:eqn:double}, we set $g^{(1)}_{\xi,\xi}(\omega)=0$ and $\omega=t$ to find
\begin{equation}
g_{\xi,\xi}''(t)=\frac{2c \xi}{t^2} \left( \frac{t}{(\sqrt{t^2+2c})^3}-\frac{t^{-1}}{(\sqrt{1/t^2+2c})^3}\right).
\end{equation}
A computation shows that  
\begin{equation}
\begin{split}
 \frac{t}{(\sqrt{t^2+2c})^3}<\frac{t^{-1}}{(\sqrt{1/t^2+2c})^3}
\end{split}
\end{equation}
for $0<t<1$ and the reverse inequality holds for $1<t<\infty$.
The formula for $ g''_{\xi,\xi}(\omega)$ has been given in~\eqref{asym:eqn:double} in the proof of Lemma~\ref{asym:lem:derivativesgas}. 

\end{proof}
\begin{lemma} \label{asym:lem:contourssolid}
Assume the same conditions given in Lemma~\ref{asym:lem:derivativessolid}. There is a contour of steepest descent from $\omega_c$ to a cut and a descent path ending at $ \mathrm{i}/\sqrt{2c}$ traveling via the cut $[1/\sqrt{2c},\infty)\mathrm{i}$, and a contour of steepest ascent from $\omega_c^{-1}$ to a cut and an ascent path ending at $ \mathrm{i}\sqrt{2c}$ traveling via the cut $[0,\sqrt{2c}]\mathrm{i}$ for $\mathrm{Re}\,g_{\xi,\xi}$.
\end{lemma}

\begin{proof}
Since $ g_{\xi,\xi} (\omega)$ is analytic for $\omega \in \mathbb{H}_+ \cup \partial \mathbb{H}_+ \backslash ( \mathrm{i} [0,\sqrt{2c}] \cup [1/\sqrt{2c} , \infty))$ the contours of steepest ascent and descent of $\mathrm{Re} \, g_{\xi,\xi} (\omega)$ are the level lines of $\mathrm{Im} \, g_{\xi,\xi} (\omega)$.
From Lemma~\ref{asym:lem:derivativessolid}, the contour  from $\omega_c$ is a descent path ($\omega_1$-integral) and enters $\mathbb{H}_+$ perpendicularly.   Similarly, from Lemma~\ref{asym:lem:derivativessolid}, the contour  from $\omega_c^{-1}$ is an ascent path ($\omega_2$-integral) and enters $\mathbb{H}_+$ perpendicularly.    Since there are contours of steepest descent and ascent traveling along the real axis (compare with the proof of Lemma~\ref{asym:lem:contoursgas}), the two contours which enter $\mathbb{H}_+$ cannot end at zero or infinity. It follows that these contours travel to their desired endpoints; see the proofs of Lemmas~\ref{asym:lem:contoursliquid} and~\ref{asym:lem:contourssolidliquid} for details.  
\end{proof}
The bottom right figure in Fig.~\ref{fig:contoursall} shows a realization of these contours.
Up to orientation, these contours are symmetric in both the real and imaginary axis.

\begin{prop} \label{asfo:lem:Bsolid}
Under the assumptions given in Theorem~\ref{thm:Kinverselimit} and with $-1<\xi<-\frac{1}{2}\sqrt{1+2c}$, we have
\begin{equation}
\mathcal{B}_{\eps_1,\eps_2}(a,x_1,x_2,y_1,y_2) = \mathbb{K}^{-1}_{1,1}(x,y )-S(x,y) + O(e^{-Cm}), 
\end{equation}
where $C>0$ is a positive constant.
\end{prop}
\begin{proof}

For the proof, we use the form of $\mathcal{B}_{\eps_1,\eps_2}(a,x_1,x_2,y_1,y_2)$ as stated in~\eqref{asfo:lemproof:liquidV} and ignore the integer parts.   Deform the contour with respect to $\omega_1$ to the contour of steepest descent given in Lemma~\ref{asym:lem:contourssolid} and use symmetry to extend to the remaining three quadrants.
Deform the contour with respect to $\omega_2$ to the contour of steepest ascent given in Lemma~\ref{asym:lem:contourssolid} and use symmetry to extend to the remaining three quadrants. Let $I^{\mathtt{S}}_{\eps_1,\eps_2}(a,x_1,x_2,y_1,y_2)$ be the double contour integral which has the same integrand as $\mathcal{B}_{\eps_1,\eps_2}(a,x_1,x_2,y_1,y_2)$ but has contours of integration described by the above deformations. A consequence of the above contour deformations means that the contours cross completely and we pick up an additional  single integral contribution from the residue at $\omega_2$ with respect to $\omega_1$.  This single contour integral contribution has been computed in the first equation of Lemma~\ref{asfo:lem:liquidcomps} and is given by $\mathbb{K}^{-1}_{1,1}(x,y) -S(x,y)$ which means that 
\begin{equation}
\mathcal{B}_{\eps_1,\eps_2}(a,x_1,x_2,y_1,y_2) =\mathbb{K}^{-1}_{1,1}(x,y) -S(x,y)+I^{\mathtt{S}}_{\eps_1,\eps_2}(a,x_1,x_2,y_1,y_2).
\end{equation}

The rest of the proof is based on showing that $I^{\mathtt{S}}_{\eps_1,\eps_2}(a,x_1,x_2,y_1,y_2)$ is exponentially small. To this integral, we first take out a prefactor of $H_{x_1+1,x_2}(\omega_c)/H_{y_1,y_2+1}(\omega_c^{-1})$ which decays exponentially fast to zero as $m$ tends to infinity.  This follows from Lemma~\ref{asym:lem:rootsgas} since the highest order term of this prefactor is given by $e^{2m(g_{\xi,\xi}(\omega_c)-g_{\xi,\xi}(\omega_c^{-1}))}$.
Standard saddle point arguments means that we only need to consider the local saddle point contributions of $ I^{\mathtt{S}}_{\eps_1,\eps_2}(a,x_1,x_2,y_1,y_2)$.  By symmetry, we only need to consider one pair of saddle points; we choose the pair $(\omega_c,1/\omega_c)$ for $\omega_c \in (1,\infty)$. Write the local descent and ascent contours in a neighborhood of the critical points as
\begin{equation}
\omega_1 = \omega_c+ (2m)^{-1/2} w \hspace{5mm}\mbox{and} \hspace{5mm}\omega_2 = \omega_c^{-1}+ (2m)^{-1/2} z 
\end{equation}
for $|w|,|z| < m^\delta$ for $0<\delta<1/6$. Up to leading order, these contours are parallel to the imaginary axis.

Using these change of variables, we apply a Taylor expansion to the exponential term contained in the integrand in~\eqref{asfo:lemproof:liquidV}
\begin{equation}
\frac{H_{x_1+1,x_2}(\omega_1)}{H_{y_1,y_2+1}(\omega_2)}=
\frac{H_{x_1+1,x_2}(\omega_c)}{H_{y_1,y_2+1}(\omega_c^{-1})}
e^{\frac{w^2g''_{\xi,\xi}(\omega_c)- z^2g''_{\xi,\xi}(\omega_c^{-1})}{2}+\tilde{f}_m(\overline{x}_1,\overline{x}_2,\omega_c)w- \tilde{f}_m(\overline{y}_1,\overline{y}_2,\omega_c^{-1})z+O(m^{-\frac{1}{2}}w^3,m^{-\frac{1}{2}}z^3)},
\end{equation}
where $\tilde{f}_m(\overline{x}_1,\overline{x}_2,\omega)=\frac{\overline{x}_1}{(2m)^\frac{1}{2}\sqrt{\omega^2+2c}} +\frac{\overline{x}_2}{(2m)^{\frac{1}{2}}\omega^2\sqrt{\omega^{-2}+2c}}$.  Using the above expansion, we have an integral for the local contribution of $ I^{\mathtt{S}}_{\eps_1,\eps_2}(a,x_1,x_2,y_1,y_2)$ with an error in the exponent of the integrand.  We use $|e^t-1|<|t|e^{|t|}$ by setting $t$  equal to this error term, and find that $ I^{\mathtt{S}}_{\eps_1,\eps_2}(a,x_1,x_2,y_1,y_2)$  is equal to\begin{equation}
\begin{split}\label{asym:lemproof:B:solid}
&\frac{H_{x_1+1,x_2}(\omega_c)}{H_{y_1,y_2+1}(\omega_c^{-1})}\frac{4 } {(2\pi )^2(2m)^{\frac{1}{2}}} \Bigg(  \int_{-m^{\delta}}^{m^\delta}dw \int_{-m^{\delta}}^{m^\delta}dz \frac{V_{\eps_1,\eps_2} \left(\omega_c,\omega_c^{-1}\right)}{\omega_c( \omega_c-\omega_c^{-1}+2m^{-1/2}(w-z))}e^{\frac{1}{2}(g_{\xi,\xi}''(\omega_c) w^2-g_{\xi,\xi}''(\omega_c^{-1}) z^2)} \\
& \times e^{ \tilde{f}_m(\overline{x}_1,\overline{x}_2,\omega_c)w- \tilde{f}_m(\overline{y}_1,\overline{y}_2,\omega_c^{-1})z} +O(m^{-\frac{1}{2}})\Bigg),
\end{split}
\end{equation}
where the error term in the above equation contains the lower order terms of the local contribution and remaining global contribution. 
The singularity $1/(w-z)$ in the above integral is integrable and by using the information for $g''_{\xi,\xi}(\omega_c)$ and $g''_{\xi,\xi}(\omega_c^{-1})$ provided by Lemma~\ref{asym:lem:derivativessolid}, the integral in the above equation is finite.  Due to the prefactor $\frac{H_{x_1+1,x_2}(\omega_c)}{H_{y_1,y_2+1}(\omega_c^{-1})}$ and 
 since the above integral is a  Gaussian integral which has a finite limit, we conclude that $ I^{\mathtt{S}}_{\eps_1,\eps_2}(a,x_1,x_2,y_1,y_2)$ is exponentially small as required.

\end{proof}

\subsection{Proof of Lemma~\ref{asfo:lem:exponential}} \label{subsection:Exponential}



\begin{proof}[
Proof of Lemma~\ref{asfo:lem:exponential}]

We will only give the argument for the first estimate in the lemma since the remaining estimates follow by slight modifications as explained below.  For the integral with respect to $\omega_2$ in Eq.~\eqref{asfo:lemproof:exponential:change3}, deform the contour so that it surrounds the cut $\mathrm{i}(-\infty, -1/\sqrt{2c}]$ and the cut $\mathrm{i}[1/\sqrt{2c},\infty)$ and travels through infinity.  Call this contour $\mathcal{D}$.  For $-1<\xi<0$, find the critical point $\omega_c \in \mathbb{H}_+\cup (1,\infty) \cup (\sqrt{2c},1)\mathrm{i}$  (see  Lemmas~\ref{asym:lem:rootsgas},~\ref{asym:lem:rootsliquidgas},~\ref{asym:lem:rootsliquid},~\ref{asym:lem:rootsliquidsolid} or~\ref{asym:lem:rootssolid}). For the integral with respect to $\omega_1$ remove a prefactor of $H_{x_1+1,x_2}(\omega_c)$.  For the contour with respect to $\omega_1$, deform the contour as given for the appropriate value of $\xi$ (see  Lemmas~\ref{asym:lem:contoursgas},~\ref{asym:lem:contoursliquidgas},~\ref{asym:lem:contoursliquid},~\ref{asym:lem:contourssolidliquid} or~\ref{asym:lem:contourssolid}) in $\mathbb{H}_+$.  If the contour is supposed to travel via the cut to the branch point $\mathrm{i}/\sqrt{2c}$, then deform the contour so that it now travels at  distance $\delta_n$ away from the cut (that is, the contour remains in $\mathbb{H}_+$) until it terminates at $\mathrm{i}/\sqrt{2c}$, where $\delta_n$ tends slowly to 0 as $n$ tends to infinity. Extend this contour to the remaining quadrants by symmetry.  For the rest of the computation, we only consider the liquid region; the other regions follow from similar estimates.  With this choice of contour for $\omega_1$, it follows that the $\Re g_{\xi,\xi}(\omega_1)$ is decreasing when $\omega_1$ travels close to the cut until it terminates at $\mathrm{i}/\sqrt{2c}$. Indeed, when the contour with respect to $\omega_1$ is close to the cut, we take a series expansion by selecting the real part, we find that
\begin{equation}
\Re g_{\xi,\xi}(\mathrm{i}t+\delta_n) = \Re g_{\xi,\xi}(\mathrm{i}t) +\delta_n \frac{-1}{\sqrt{2c t^2-1}}+O(\delta_n^2).
\end{equation}
Both expressions on the right side of the above equation are increasing for $t \in [1/\sqrt{2c},\infty)$; see the proof of Lemma~\ref{asym:lem:contoursliquid} for details verifying this statement for the first term.

For convenience, we take $\overline{x}_1,\overline{x}_2,\overline{y}_1$ and $\overline{y_2}$ to be equal to zero --- these parameters have no effect on the exponential bound.   As determined in Lemma~\ref{asym:lem:rootsliquid}, for $\xi$ in the liquid region there is a single critical point $\omega_c \in \mathbb{H}_+$.  Standard saddle point techniques, since $\min (\omega_1-\omega_2)>\delta_n$ for the above choice of contours and provided we bound above by $1/\delta_n$, we only need to estimate the local contribution of the integral with respect to $\omega_1$.  
To do so, we approximate the leading order terms by the change of variables $\omega_1=\omega_c + n^{-1/2} w \kappa$, where $\kappa$ is the direction of the contour of steepest descent at $\omega_c$.  Using a Taylor expansion in the exponent of the integrand, we find that the local contribution of the integral at the single critical point is given by
\begin{equation}
\begin{split}~\label{asfo:lemproof:exponential:local}
&\left|a^{-1} \mathcal{B}_{1-\eps_1,\eps_2}(a^{-1},2n-x_1,x_2,2n-y_1,y_2) \right|= \delta_n^{-1}
\Bigg| \frac{  1 }{(2 \pi \mathrm{i})^2 n^{1/2}} \int_{\mathcal{D}} {d\omega_2} \frac{H_{x_1+1,x_2}(\omega_c)}{H_{2n-y_1,y_2+1}(\omega_2)}
\frac{\omega_2}{\omega_2^2-\omega_c^2}\\
&\times  
 \left(\int _{-n^{\delta}}^{n^{\delta}} \frac{dw}{\omega_c} e^{g^{(2)}_{\xi,\xi}(\omega_c) \frac{ \kappa^2w^2}{2}}
 \sum_{\gamma_1,\gamma_2=0}^1  (-1)^{\eps_2+\gamma_2}Q_{\gamma_1,\gamma_2}^{\eps_1,\eps_2}(\omega_c ,\omega_2)+O(n^{-1/2+2\delta}) \right) \Bigg|,
\end{split}
\end{equation}
where the error term is found by using $|e^t-1|\leq |t|e^{|t|}$ where $t$ is the error from the Taylor expansion.   By extending the contours of integration with respect to $w$ to infinity, the  integral with respect to $w$ is a Gaussian integral and has a finite limit.  Note that $|H_{x_1+1,x_2}(\omega_c)|=1$, since we are in the liquid region.
We now estimate the contribution from the integral with respect to $\omega_2$.  We want to find the minimum of $|H_{2n-y_1,y_2+1}(\omega_2)|$  for $\omega_2 \in \mathcal{D}$.   By symmetry, it suffices to consider the integral on one side of one of the cuts --- we restrict the contour $\mathcal{D}$ to the right side of the cut $\mathrm{i}[1/\sqrt{2c},\infty)$ and compute the corresponding square roots.
For $\omega_2= \epsilon + \mathrm{i} t$ and for $t>1/\sqrt{2c}$, we have
\begin{equation}
\sqrt{\omega^2_2+2c} \to \left\{ \begin{array}{ll}
\mathrm{i}\sqrt{t^2-2c} & \mbox{if } \epsilon \to 0^+\\
\mathrm{i}\sqrt{t^2-2c} & \mbox{if } \epsilon \to 0^-
\end{array} \right. \hspace{3mm} \mbox{and} \hspace{3mm}
\sqrt{\omega^{-2}_2+2c} \to \left\{ \begin{array}{ll}
\sqrt{2c-t^{-2}} & \mbox{if } \epsilon \to 0^+\\
-\sqrt{2c-t^{-2}} & \mbox{if } \epsilon \to 0^-.
\end{array} \right.
\end{equation}
From the above limits, it is clear that $|G(( \mathrm{i}t)^{-1})|=1$. 
Since
\begin{equation}
\frac{1}{H_{2n-y_1,y_2+1}(\omega_2)}\sim e^{-2m( \log \mathrm{i}t + \xi \log G(\mathrm{i}t) + \xi \log G((\mathrm{i}t)^{-1}) )}
\end{equation}
we need to show that the minimum of 
\begin{equation}
\log| \mathrm{i}t| +\xi\log |G(\mathrm{i}t)|+\xi\log |G((\mathrm{i}t)^{-1})|=\log t+\xi\log\frac{t-\sqrt{t^2-2c}}{\sqrt{2c}}
\end{equation}
is positive for $t \in (1/\sqrt{2c} ,\infty)$.   It is easy to check that $({t-\sqrt{t^2-2c}})/\sqrt{2c}$ is decreasing for increasing  $t \in (1/\sqrt{2c} ,\infty)$ and is less than 1. Since $\xi$ is negative, it follows that the right side of the above equation is bounded below by  $\log 1/\sqrt{2c} >0$.  This means that
\begin{equation}
\left|\frac{1}{H_{2n-y_1,y_2+1}(\omega_2)} \right| <C_1 e^{-C_2 n}
\end{equation}
for $\omega_2 \in \mathcal{D}$ and for positive constants $C_1$ and $C_2$
which concludes the proof of the first estimate in Lemma~\ref{asfo:lem:exponential}.

For the second estimate given in the lemma, we use~\eqref{asfo:lemproof:exponential:change4} and follow an adaptation of the above argument where we use saddle point analysis to bound the integral with respect to $\omega_2$ and deform the contour of integration with respect to $\omega_1$ to the inner cut. 
 For the third estimate given in the lemma, we use~\eqref{asfo:lemproof:exponential:change5} and deform the contour of integration with respect to  $\omega_1$ to the inner cut and the contour of integration with respect to $\omega_2$ to $\mathcal{D}$ and use the arguments above to get the exponential decay bound.

\end{proof}

\subsection{Conclusion of the proofs Theorems~\ref{thm:Kinverselimit},~\ref{thm:transversal1} and~\ref{thm:transversal2}}\label{subsection:conclusion}

We first complete the proof of Theorem~\ref{thm:Kinverselimit}.  

\begin{proof}[Proof of Theorem~\ref{thm:Kinverselimit}]
For $-1<\xi <0$, Lemma~\ref{asfo:lem:exponential} implies that the only asymptotic contribution for the formula $K_{a,1}^{-1}(x,y)$ in
 Theorem~\ref{thm:Kinverse}  is from $\mathcal{B}_{\eps_1,\eps_2}(a,x_1,x_2,y_1,y_2)$, described in~\eqref{B}, and $\mathbb{K}_{1,1}(x,y)$. Thus, we only need to consider
\begin{equation}
K_{a,1}^{-1}(x,y)=\mathbb{K}^{-1}_{1,1}(x,y)-\mathcal{B}_{\eps_1,\eps_2}(a,x_1,x_2,y_1,y_2) +O(e^{-Cn}),
\end{equation}
where $C>0$ is some constant.
 Asymptotic expansions of the term $\mathcal{B}_{\eps_1,\eps_2}(a,x_1,x_2,y_1,y_2)$ have been given for the required degree of accuracy in the gas, liquid and solid regions in Propositions~\ref{asfo:lem:Bgas},~\ref{asfo:lem:Bliquid} and~\ref{asfo:lem:Bsolid}. We directly substitute these results into the above equation.   More precise results for $\mathcal{B}_{\eps_1,\eps_2}(a,x_1,x_2,y_1,y_2)$ at the liquid-gas and solid-liquid boundaries are given in Propositions~\ref{asfo:lem:Bliquidgas} and~\ref{asfo:lem:Bliquidsolid}.  It is enough to set $\alpha_x=\alpha_y=\beta_x=\beta_y=0$ in each of these propositions and substitute these expressions into the above equation.

\end{proof}

We now give the proof of Theorem~\ref{thm:transversal1}.

\begin{proof}[Proof of Theorem~\ref{thm:transversal1}]
We first extract the term $|G(\mathrm{i})|^{1/2(-2-x_1+x_2+y_1-y_2)}$ from the formula for $K^{-1}_{a,1}$ given in Theorem~\ref{thm:Kinverse}. We obtain
\begin{equation}
\begin{split}
K_{a,1}^{-1}(x,y)&=|G(\mathrm{i})|^{\frac{-2-x_1+x_2+y_1-y_2}{2}} \bigg(|G(\mathrm{i})|^{-\frac{-2-x_1+x_2+y_1-y_2}{2}} \mathbb{K}^{-1}_{1,1}(x,y)\\&-|G(\mathrm{i})|^{-\frac{-2-x_1+x_2+y_1-y_2}{2}}\big( \mathcal{B}_{\eps_1,\eps_2}(a,x_1,x_2,y_1,y_2) +O(e^{-Cm})\big)\bigg)
\end{split}
\end{equation}
where the last error term follows from the result in Lemma~\ref{asfo:lem:exponential}. The term $|G(\mathrm{i})|^{-1/2(-2-x_1+x_2+y_1-y_2)}$ grows at most at rate $e^{C m^{2/3}}$ for some $C>0$ provided that $\beta_x-\beta_y>0$, is order $1$ when $\beta_x=\beta_y$ and decays at rate $e^{-Cm^{2/3}}$ for some $C>0$ for $\beta_x-\beta_y<0$. Thus, we obtain
\begin{equation}
\begin{split}
K_{a,1}^{-1}(x,y)&=|G(\mathrm{i})|^{\frac{-2-x_1+x_2+y_1-y_2}{2}} \bigg(|G(\mathrm{i})|^{-\frac{-2-x_1+x_2+y_1-y_2}{2}} \mathbb{K}^{-1}_{1,1}(x,y)\\&-|G(\mathrm{i})|^{-\frac{-2-x_1+x_2+y_1-y_2}{2}} \mathcal{B}_{\eps_1,\eps_2}(a,x_1,x_2,y_1,y_2) +O(e^{-Cm})\bigg)
\end{split}
\end{equation}
From the above equation, the case $\beta_x=\beta_y$ follows immediately from Proposition~\ref{asfo:lem:Bliquidgas}.
 
Using Proposition~\ref{Prop:GasGausslimit},  Proposition~\ref{asfo:lem:Bliquidgas} and the above equation, we have
\begin{equation}
\begin{split}
K_{a,1}^{-1}(x,y)&=(2m)^{-1/3}(1+O(m^{-1/3}))\times \mathrm{i}^{y_1-x_1+1}|G(\mathrm{i})|^{\frac{1}{2}(-2-x_1+x_2+y_1-y_2)} c_0 \mathtt{g}_{\eps_1,\eps_2}  
\\&\times\left( \frac{\mathbbm{I}_{\beta_x >\beta_y}}{ \sqrt{4\pi(\beta_x-\beta_y)}}e^{-\frac{(\alpha_x-\alpha_y)^2}{4(\beta_x-\beta_y)}} - e^{\beta_x \alpha_x -\beta_y \alpha_y +\frac{2}{3} (\beta_x^3-\beta_y^3)} \tilde{\mathcal{A}}((-\beta_x,\alpha_x+\beta_x^2);(-\beta_y,\alpha_y+\beta_y^2)) \right).
\end{split}
\end{equation}
Using a substitution, we find that 
\begin{equation}
\frac{\mathbbm{I}_{\beta_x >\beta_y}}{ \sqrt{4\pi(\beta_x-\beta_y)}}e^{-\frac{(\alpha_x-\alpha_y)^2}{4(\beta_x-\beta_y)}}e^{-\beta_x \alpha_x +\beta_y \alpha_y -\frac{2}{3} (\beta_x^3-\beta_y^3)}= \phi_{-\beta_x,-\beta_y}(\alpha_x+\beta_x^2,\alpha_y+\beta_y^2)
\end{equation}
which means that 
\begin{equation}
\begin{split}
K_{a,1}^{-1}(x,y)&=\mathrm{i}^{y_1-x_1+1}|G(\mathrm{i})|^{\frac{1}{2}(-2-x_1+x_2+y_1-y_2)} c_0 \mathtt{g}_{\eps_1,\eps_2}e^{\beta_x \alpha_x -\beta_y \alpha_y +\frac{2}{3} (\beta_x^3-\beta_y^3)}   
\\&\times\left(  \phi_{-\beta_x,-\beta_y}(\alpha_x+\beta_x^2,\alpha_y+\beta_y^2)
-\tilde{\mathcal{A}}((-\beta_x,\alpha_x+\beta_x^2);(-\beta_y,\alpha_y+\beta_y^2)) \right)(2m)^{-1/3}(1+O(m^{-1/3})).
\end{split}
\end{equation}
and the result follows immediately from~\eqref{asfo:eq:Airymod2}.
\end{proof}

We now give the proof of Theorem~\ref{thm:transversal2}.

\begin{proof}[Proof of Theorem~\ref{thm:transversal2}]
We first extract the term $|G({1})|^{1/2(-2-x_1+x_2+y_1-y_2)}$ from the formula for $K^{-1}_{a,1}$ given in Theorem~\ref{thm:Kinverse}. We obtain
\begin{equation}
\begin{split}
K_{a,1}^{-1}(x,y)&=|G({1})|^{\frac{-2-x_1+x_2+y_1-y_2}{2}} \bigg(|G({1})|^{-\frac{-2-x_1+x_2+y_1-y_2}{2}} \mathbb{K}^{-1}_{1,1}(x,y)\\&-|G({1})|^{-\frac{-2-x_1+x_2+y_1-y_2}{2}}\big( \mathcal{B}_{\eps_1,\eps_2}(a,x_1,x_2,y_1,y_2) +O(e^{-Cm})\big)\bigg)
\end{split}
\end{equation}
where the last error term follows from the result in Lemma~\ref{asfo:lem:exponential}. The term $|G({1})|^{-1/2(-2-x_1+x_2+y_1-y_2)}$ grows at most at rate $e^{C m^{2/3}}$ for some $C>0$ provided that $\beta_x-\beta_y>0$, is order $1$ when $\beta_x=\beta_y$ and decays at rate $e^{-Cm^{2/3}}$ for some $C>0$ for $\beta_x-\beta_y<0$. Thus, we obtain
\begin{equation}
\begin{split}
K_{a,1}^{-1}(x,y)&=|G({1})|^{\frac{-2-x_1+x_2+y_1-y_2}{2}} \bigg(|G({1})|^{-\frac{-2-x_1+x_2+y_1-y_2}{2}} \mathbb{K}^{-1}_{1,1}(x,y)\\&-|G({1})|^{-\frac{-2-x_1+x_2+y_1-y_2}{2}} \mathcal{B}_{\eps_1,\eps_2}(a,x_1,x_2,y_1,y_2) +O(e^{-Cm})\bigg)
\end{split}
\end{equation}
From the above equation, the case $\beta_x=\beta_y$ follows immediately from Proposition~\ref{asfo:lem:Bliquidsolid}.

Using Proposition~\ref{Prop:SolidGausslimit}, Proposition~\ref{asfo:lem:Bliquidsolid} and the above equation, we have
\begin{equation}
\begin{split}
K_{a,1}^{-1}(x,y)&=(2m)^{-1/3}(1+O(m^{-1/3}))\times \mathrm{i}^{\frac{x_2-x_1+y_1-y_2}{2}}|G(1)|^{\frac{1}{2}(-2-x_1+x_2+y_1-y_2)} c_0 \mathtt{s}_{\eps_1,\eps_2}  
\\&\times\left(  e^{-\beta_x \alpha_x +\beta_y \alpha_y -\frac{2}{3} (\beta_x^3-\beta_y^3)} \tilde{\mathcal{A}}((\beta_x,\alpha_x+\beta_x^2);(\beta_y,\alpha_y+\beta_y^2))-\frac{\mathbbm{I}_{\beta_y >\beta_x}}{ \sqrt{4\pi(\beta_y-\beta_x)}}e^{-\frac{(\alpha_x-\alpha_y)^2}{4(\beta_y-\beta_x)}}  \right).
\end{split}
\end{equation}
Using a substitution, we find that 
\begin{equation}
\frac{\mathbbm{I}_{\beta_y >\beta_x}}{ \sqrt{4\pi(\beta_y-\beta_x)}}e^{-\frac{(\alpha_y-\alpha_x)^2}{4(\beta_y-\beta_x)}}e^{\beta_x \alpha_x -\beta_y \alpha_y +\frac{2}{3} (\beta_x^3-\beta_y^3)}= \phi_{\beta_x,\beta_y}(\alpha_x+\beta_x^2,\alpha_y+\beta_y^2)
\end{equation}
which means that 
\begin{equation}
\begin{split}
K_{a,1}^{-1}(x,y)&=
 \mathrm{i}^{\frac{x_2-x_1+y_1-y_2}{2}}|G(1)|^{\frac{1}{2}(-2-x_1+x_2+y_1-y_2)} c_0 \mathtt{s}_{\eps_1,\eps_2}  e^{-\beta_x \alpha_x +\beta_y \alpha_y -\frac{2}{3} (\beta_x^3-\beta_y^3)} \\
&\times( \tilde{\mathcal{A}}((\beta_x,\alpha_x+\beta_x^2);(\beta_y,\alpha_y+\beta_y^2))-\phi_{\beta_x,\beta_y}(\alpha_x+\beta_x^2,\alpha_y+\beta_y^2))(2m)^{-1/3}(1+O(m^{-1/3}))
\end{split}
\end{equation}
and the result follows immediately from~\eqref{asfo:eq:Airymod2}.
\end{proof}

We now give the proofs of Corollary~\ref{coro:transversal1} and~\ref{coro:transversal2}.

\begin{proof}[Proof of Corollary~\ref{coro:transversal1}]
Suppose that $x$ and $y$ are both in $\mathtt{W}_0$ as given in the statement of the corollary. From  Theorem~\ref{localstatisticsthm}, the covariance between the edges $(x,x+e_2)$ and $(y,y+e_2)$ is given by
\begin{equation} \label{asfo:thmproof:localstats}
\begin{split}
&\mathbb{P}\left[ \mbox{Edges at } (x,x+e_2) \mbox{ and } (y,y+e_2)\right]- \mathbb{P}\left[ \mbox{Edge at } (x,x+e_2)\right] \mathbb{P}\left[ \mbox{Edge at }(y,y+e_2)\right]\\&= K_{a,1}(x,x+e_2)K_{a,1}(y,y+e_2)
K_{a,1}^{-1}(x,y+e_2)K_{a,1}^{-1}(x,y+e_2)=-a^2K_{a,1}^{-1}(x,y+e_2)K_{a,1}^{-1}(x,y+e_2)
\end{split}
\end{equation}
since $K_{a,1}(x,x+e_2)=K_{a,1}(y,y+e_2)=a\mathrm{i}$. The corollary immediate follows from inserting the expansions of $K^{-1}_{a,1}$ in Theorem~\ref{thm:transversal1} into the above formula.

\end{proof}

\begin{proof}[Proof of Corollary~\ref{coro:transversal2}]
Suppose that $x$ and $y$ are both in $\mathtt{W}_0$ as given in the statement of the corollary. From  Theorem~\ref{localstatisticsthm}, the covariance between the edges $(x,x+e_2)$ and $(y,y+e_2)$ is given by
\begin{equation} \label{asfo:thmproof:localstats1}
\begin{split}
&\mathbb{P}\left[ \mbox{Edges at } (x,x-e_2) \mbox{ and } (y,y-e_2)\right]- \mathbb{P}\left[ \mbox{Edge at } (x,x-e_2)\right] \mathbb{P}\left[ \mbox{Edge at }(y,y-e_2)\right]\\&= K_{a,1}(x,x-e_2)K_{a,1}(y,y-e_2)
K_{a,1}^{-1}(x,y-e_2)K_{a,1}^{-1}(x,y-e_2)=-K_{a,1}^{-1}(x,y-e_2)K_{a,1}^{-1}(x,y-e_2)
\end{split}
\end{equation}
since $K_{a,1}(x,x-e_2)=K_{a,1}(y,y-e_2)=\mathrm{i}$. The corollary immediate follows from inserting the expansions of $K^{-1}_{a,1}$ in Theorem~\ref{thm:transversal2} into the above formula.

\end{proof}

\section{Postponed proofs}
\label{section:Algebra}

In this section we will give the postponed proofs of some lemmas and propositions in Section~\ref{section:Asymptotics}.

\subsection{Proof of Lemma~\ref{cuts:lem:F}}\label{subsection:cutsandsaddle}
We will now give the proof of Lemma~\ref{cuts:lem:F}.

\begin{proof}[Proof of Lemma~\ref{cuts:lem:F}]
Without loss of generality, we will assume that $s \geq 0$ because $F_s(w)=F_{-s}(w)$ seen by the change of variables $u \mapsto 1/u$ in Eq.~\eqref{cuts:def:F}.  
From~\eqref{cuts:def:F}, we make the change $u \mapsto u \mathrm{i}$ and set $w=\mathrm{i} \omega/\sqrt{2c}$ and after simplification, we find that
\begin{equation}
F_s\left(\frac{\mathrm{i} \omega}{\sqrt{2c}} \right)=\frac{\mathrm{i}^s}{(1+a^2)} \frac{1}{2 \pi \mathrm{i}} \int_{\Gamma_1} \frac{u^s}{1-\omega \sqrt{\frac{c}{2} }(u-u^{-1})} \frac{du}{u}.
\end{equation}
We need to check that the choice of square root ensures that   $F_s(\mathrm{i}\omega/\sqrt{2 c})$ is analytic.
Observe that $F_s(\mathrm{i}\omega/\sqrt{2 c})$ is analytic in $\omega$ unless $ 1-\omega \sqrt{\frac{c}{2} }(u-u^{-1})$ has a zero for some $|u|=1$. 
Since $\mathrm{i}(u-u^{-1})$ is real for $|u|=1$, we conclude that non-analyticity requires $1/(\sqrt{2c} \mathrm{i} \omega) \in [-1,1]$ which means that $\omega \in \Omega$ where we define $\Omega= \mathrm{i} (-\infty,-1/\sqrt{2c}] \cup \mathrm{i}[1/\sqrt{2c}, \infty)$. 

The roots of  $ 1-\omega \sqrt{\frac{c}{2} }(u-u^{-1})=0$  are given by 
\begin{equation}\label{u2omega}
u_{\pm} (\omega)=\frac{1}{\omega \sqrt{2c}} \pm\sqrt{ \frac{1}{2c\omega^2}+1}.
\end{equation}
The square root must be defined so that we have analyticity in $ \mathbb{C}\setminus\Omega$. Note that
\begin{equation}
\frac{1}{G(\omega)}=-\frac{1}{\sqrt{2c}} (\omega +\sqrt{\omega^2+2c} ).
\end{equation}
Thus, $u_+(\omega)=-1/G(1/\omega)$ and $u_-(\omega)=G(1/\omega)$ are the solutions we want, since this implies
$u_+(\omega) u_-(\omega)=1$ and $u_+(\omega)+u_-(\omega)=\sqrt{2/c}\omega^{-1}$.

Notice that if $|u_-(\omega)|=1$, that is $u_-(\omega)=e^{\mathrm{i}\phi}$, then 
\begin{equation}
\sqrt{\frac{2}{c}} \frac{1}{\omega}=u_+(\omega)+u_-(\omega)=-\frac{1}{u_-(\omega)}+u_-(\omega)=-e^{-\mathrm{i}\phi}+e^{\mathrm{i}\phi}=
2\sin\phi\mathrm{i},
\end{equation}
which happens only in the cut. It follows that $|u_-(\omega)|\not =1$ if $\omega \not \in \Omega$.
If $\omega = t>0$, then by the choice of square root~\eqref{cuts:squareroot}, we have
\begin{equation}
u_-(t)=G(t^{-1})=\frac{1}{\sqrt{2c}} \left( \frac{1}{t} -\sqrt{\frac{1}{t^2}+2c} \right)=-\frac{\sqrt{2c}}{t^{-1} +\sqrt{t^{-2}+2c}},
\end{equation}
which means that $|u_-(t)|<1$.  By continuity, we have $|u_-(\omega)|<1$ for all $\omega \in \mathbb{C}\setminus\Omega$. As a consequence, by applying the residue theorem, we have
\begin{equation}
\begin{split}
F_s\left(\frac{\mathrm{i}\omega }{\sqrt{2c}} \right)& = \frac{- \mathrm{i}^s}{(1+a^2) \omega}\sqrt{\frac{2}{c}}
\frac{u_-(\omega)^s}{u_-(\omega)-u_+(\omega)}\\
 &= \frac{ \mathrm{i}^s}{(1+a^2)\omega \sqrt{ \omega^{-2}+2c}} G(\omega^{-1})^{s},
\end{split}
\end{equation}
which follows because $u_-(\omega)-u_+(\omega)= \sqrt{2/c} \sqrt{\omega^{-2}+2c}$.
\end{proof}

Changing $\omega$ to $1/\omega$ in~\eqref{u2omega} and letting
$u=u_{-}(1/\omega)$, we see from the proof of Lemma \ref{cuts:lem:F} that
\begin{equation}\label{omegauchange}
\omega=\omega(u)=\sqrt{\frac{c}{2}}\left(u-\frac 1{u}\right)
\end{equation}
gives a bijection from $\{u\,;\,|u|<1\}$ to $\mathbb{C}\setminus \mathrm{i}[-\sqrt{2c},\sqrt{2c}]$ with inverse $u=G(\omega)$. A computation shows that
\begin{equation}\label{dudomega}
\frac{du}{u}=-\frac{d\omega}{\sqrt{\omega^2+2c}}.
\end{equation}
This change of variables between $u$- and $\omega$-variables is important in many computations.

\subsection{Proof of Lemma \ref{asfo:lem:V}}\label{subsection:manipulations}

In this subsection, we give the proof of Lemma~\ref{asfo:lem:V}. 
It seems feasible that one could determine an expression for $V_{\eps_1,\eps_2}(\omega,\omega)$ by directly `plugging in' but we found that the resulting expression hard to analyze due to its length and complexity.  Instead, we exploit some of the structure of the formulas.
The steps listed below involve finding relations in terms of $G$ which, after computer algebra, gives more compact formulas of the expressions which make $V_{\eps_1,\eps_2}(\omega,\omega)$. These relations are given below.  Although this approach uses computer algebra, it seems to uncover some of the structure behind the model.

\begin{lemma}\label{asfo:lem:spectral}
We have the following identity:
\begin{equation}\label{asfo:eq:spectral}
G\left( \omega \right)^2 G\left(\omega^{-1} \right)^2=-1+G\left( \omega \right)^2+ G\left(\omega^{-1} \right)^2 +\frac{2}{c}G\left( \omega \right) G\left(\omega^{-1} \right).
\end{equation}
\end{lemma}
 It is worth noting that the spectral curve equation~\cite{KOS:06,KO:06} is extremely important relation for these computations. The above lemma is equivalent to the spectral curve equation $P(z,w)=0$. This is seen by setting $u_1=G(\omega)$ and $u_2=G(1/\omega)$, using the change of variables $u_2=w u_1$ where $u_1$ is fixed followed by the change of variables $z=w u_1^2$.

\begin{proof}[Proof of Lemma~\ref{asfo:lem:spectral}]
The lemma follows by expanding out the definition of $G$ on both sides of the equation and comparing the result of each expansion.
\end{proof}

We now list four relations which will be used to determine a good formula for $V_{\eps_1,\eps_2}(\omega,\omega)$.
\begin{lemma}\label{alg:lem:relations}
We have the following relations
\begin{equation}
(1+a^2) \omega  G\left( \omega \right) \left( a -  G\left( \omega \right) G\left( \omega^{-1} \right) \right) = \frac{ \left(  G\left( \omega^{-1} \right)+ a  G\left( \omega \right)
\right) \left( a^2+  G\left( \omega \right)^2 \right) } {\sqrt{ \omega^{-2} +2c} },
\end{equation}
\begin{equation}
-(1+a^2) \omega  G\left( \omega \right) \left( a  G\left( \omega^{-1} \right)+ G\left( \omega \right) \right) = \frac{ \left( -1+a  G\left( \omega^{-1} \right)  G\left( \omega \right)
\right) \left( a^2+  G\left( \omega \right)^2 \right) } {\sqrt{ \omega^{-2} +2c} },
\end{equation}
\begin{equation}
(1+a^2) \omega  G\left( \omega \right) \left(  a  G\left( \omega \right)+ G\left( \omega^{-1} \right) \right) = \frac{ \left(a-  G\left( \omega^{-1} \right)  G\left( \omega \right)
\right) \left( 1+a^2  G\left( \omega \right)^2 \right) } {\sqrt{ \omega^{-2} +2c} }
\end{equation}
and
\begin{equation}
-(1+a^2) \omega  G\left( \omega \right) \left(- 1 + a G\left( \omega \right) G\left( \omega^{-1} \right) \right) = \frac{ \left( a G\left( \omega^{-1} \right)+   G\left( \omega \right)
\right) \left(1+ a^2  G\left( \omega \right)^2 \right) } {\sqrt{ \omega^{-2} +2c} }.
\end{equation}
\end{lemma}
\begin{proof}
The proof of these relations follows by simply expanding out both sides of each equation and checking that the left side is equal to the right side.

\end{proof}
\begin{proof}[Proof of Lemma \ref{asfo:lem:V}]

The proof is based on the decomposition of $V_{\eps_1,\eps}(\omega_1,\omega_2)$ defined in~\eqref{asfo:eq:V} when setting  $\omega_1=\omega_2$, using the spectral curve given in Lemma~\ref{asfo:lem:spectral}, and using the relations given in Lemma~\ref{alg:lem:relations}.  Due to the length of the computations, we only give explicit computations for $V_{\eps_1,\eps_2}(\omega,\omega)$ with $(\eps_1,\eps_2)=(0,0)$.  The remaining  values of $(\eps_1,\eps_2)$ follow from similar computations.  We used computer algebra for these computations.

We first write out the equations $\mathtt{x}_{\gamma_1,\gamma_2}^{\eps_1,\eps_2} (\omega_1,\omega_2)$ for $\eps_1,\eps_2,\gamma_1,\gamma_2 \in \{0,1\}$.  Recall the definition of $f_{a,b}(u,v)$ given in Definition~\ref{def:coefficients}.  A simplification yields
\begin{equation}\label{asfo:lemproof:V:denomiator1}
\frac{(1-\omega_1^2 \omega_2^2) }{f_{a,1}\left(  G\left( \omega_1 \right), G\left( \omega_{2}\right) \right)}=\frac{1}{4(1+a^2)^2 G\left( \omega_1 \right)^2 G\left( \omega_{2}\right)^2 }
\end{equation}
and
\begin{equation}\label{asfo:lemproof:V:denomiator2}
\frac{(1-\omega_1^2 \omega_2^2) }{f_{1,a}\left(  G\left( \omega_1 \right), G\left( \omega_{2}\right) \right)}=\frac{1}{4(1+a^2)^2 G\left( \omega_1 \right)^2 G\left( \omega_{2}\right)^2 },
\end{equation}
which are both seen by using the change of variables  $u_1=G(\omega_1)$  and $v_1=G(\omega_2)$, see~\eqref{omegauchange}, and computing the left sides of the above equations.
Using the definition of $\mathtt{x}_{\gamma_1,\gamma_2}^{\eps_1,\eps_2}$ given in~\eqref{cuts:def:x}, the Eqs.~\eqref{asfo:lemproof:V:denomiator1} and~\eqref{asfo:lemproof:V:denomiator2} and the definition of $\mathtt{y}_{\gamma_1,\gamma_2}^{\eps_1,\eps_2}$ given in Definition~\ref{def:coefficients}, it is possible to write out $\mathtt{x}_{\gamma_1,\gamma_2}^{\eps_1,\eps_2} (\omega_1,\omega_2^{-1})$ for $\eps_1,\eps_2,\gamma_1,\gamma_2 \in \{0,1\}$.   For $\gamma_1,\gamma_2\in\{0,1\}$, we will only write out the terms $\mathtt{x}_{\gamma_1,\gamma_2}^{0,0} (\omega_1,\omega_2^{-1})$.  The values of $\mathtt{x}_{\gamma_1,\gamma_2}^{\eps_1,\eps_2} (\omega_1,\omega_2)$ for $(\eps_1,\eps_2)\not=(0,0)$ can be  determined explicitly from $\mathtt{x}_{\gamma_1,\gamma_2}^{0,0} (\omega_1,\omega_2^{-1})$ by using the relations given in Definition~\ref{def:coefficients} for $\gamma_1,\gamma_2\in\{0,1\}$ but we do not list these terms here.  We find that 
\begin{equation}
\begin{split}
 \mathtt{x}_{0,0}^{0,0} (\omega_1,\omega_2^{-1})&=
\frac{1}{16 \left(1+a^2\right)^4 G\left(\omega _1\right) G\left(\frac{1}{\omega _2}\right)\prod_{j=1}^2 \sqrt{\omega_j^2+2c}\sqrt{\omega_j^{-2}+2c}  }
\\
&\times \Bigg(
-a \left(1+a^2+a^4+\left(1+3 a^2\right) G\left(\omega _1\right){}^2+a^4 G\left(\omega _1\right){}^4\right) \\
&+a \left(-1-3 a^2+\left(-3-2 a^2-a^4+2 a^6\right) G\left(\omega _1\right){}^2+a^2 \left(-1+a^2\right) G\left(\omega _1\right){}^4\right) G\left(\frac{1}{\omega _2}\right){}^2 \\
&+a^3 \left(-1+G\left(\omega _1\right){}^2\right) \left(a^2+G\left(\omega _1\right){}^2\right) G\left(\frac{1}{\omega _2}\right){}^4 \Bigg),
\end{split}
\end{equation}
\begin{equation}
 \mathtt{x}_{0,1}^{0,0} (\omega_1,\omega_2^{-1})= \frac{a \left(1+a^2 G\left(\omega _1\right){}2\right) \left(1+\left(1+2 a^2\right) G\left(\frac{1}{\omega _2}\right){}^2+G\left(\omega _1\right){}^2 \left(-1+G\left(\frac{1}{\omega _2}\right){}^2\right)\right)}{16 \left(1+a^2\right)^3 G\left(\omega _1\right)G\left(\frac{1}{\omega _2}\right)\prod_{j=1}^2 \sqrt{\omega_j^2+2c}\sqrt{w_j^{-2}+2c}},
\end{equation}
\begin{equation}
 \mathtt{x}_{1,0}^{0,0} (\omega_1,\omega_2^{-1})=\frac{a \left(1+a^2 G\left(\frac{1}{\omega _2}\right){}^2\right) \left(1+2 a^2 G\left(\omega _1\right){}^2-G\left(\frac{1}{\omega _2}\right){}^2+G\left(\omega _1\right){}^2 \left(1+G\left(\frac{1}{\omega _2}\right){}^2\right)\right)}{16 \left(1+a^2\right)^3 G\left(\omega _1\right){} G\left(\frac{1}{\omega _2}\right){}\prod_{j=1}^2 \sqrt{\omega_j^2+2c}\sqrt{w_j^{-2}+2c}}
\end{equation}
and
\begin{equation}
\mathtt{x}_{1,1}^{0,0} (\omega_1,\omega_2^{-1})=\frac{a \left(-1+G\left(\frac{1}{\omega _2}\right){}^2+G\left(\omega _1\right){}^2 \left(1+\left(1+2 a^2\right) G\left(\frac{1}{\omega _2}\right){}^2\right)\right)}{16 \left(1+a^2\right)^2 G\left(\omega _1\right){} G\left(\frac{1}{\omega _2}\right){}\prod_{j=1}^2 \sqrt{\omega_j^2+2c}\sqrt{w_j^{-2}+2c}}. 
\end{equation}
From the expressions for $\mathtt{x}_{\gamma_1,\gamma_2}^{\eps_1,\eps_2}(\omega_1,\omega_2^{-1})$, we are able to compute $\mathtt{x}_{\gamma_1,\gamma_2}^{\eps_1,\eps_2}(\omega_1,-\omega_2^{-1})$ using~\eqref{asfo:eq:Gsigns} and we find that
\begin{equation}
\begin{split}
\mathtt{x}_{\gamma_1,\gamma_2}^{\eps_1,\eps_2}(\omega_1,-\omega_2^{-1})=-\mathtt{x}_{\gamma_1,\gamma_2}^{\eps_1,\eps_2}(\omega_1,\omega_2^{-1}).
\end{split}
\end{equation}
We proceed by evaluating the right side of~\eqref{asfo:eq:V2} for $\omega_1=\omega_2$ and $(\eps_1,\eps_2)=(0,0)$ term by term. We apply the relation given in Lemma~\ref{asfo:lem:spectral} to reduce each of the numerators of the following expressions into simpler terms, that is, we apply a polynomial reduction under the spectral curve. We also use the formula for $s(\cdot)$ given in~\eqref{asfo:eq:s}. We find that
\begin{equation}
\begin{split} \label{alg:lemproof:sumV00}
&\frac{1}{2}(\mathtt{x}_{0,0}^{0,0}(\omega,\omega^{-1})-\mathtt{x}_{0,0}^{0,0}(\omega,-\omega^{-1}))= \frac{\left(a G\left(\frac{1}{\omega }\right)+G(\omega )\right) \left(G\left(\frac{1}{\omega }\right)+a G(\omega )\right) \left(-1+a G\left(\frac{1}{\omega }\right) G(\omega )\right)}{8 \left(1+a^2\right)^2 G\left(\frac{1}{\omega }\right) G(\omega ) (\omega^2+2c)(\omega^{-2}+2c) },\\
&-\frac{1}{2}s(G(\omega^{-1}))(\mathtt{x}_{0,1}^{0,0}(\omega,\omega^{-1})-\mathtt{x}_{0,1}^{0,0}(\omega,-\omega^{-1}))=-\omega\frac{\left(a G\left(\frac{1}{\omega }\right)+G(\omega )\right) \left(1+a^2 G(\omega )^2\right)}{8 \left(1+a^2\right)^2  G(\omega )\sqrt{\omega^2+2c} (\omega^{-2}+2c) },\\
&-\frac{1}{2}s(G(\omega))(\mathtt{x}_{1,0}^{0,0}(\omega,\omega^{-1})-\mathtt{x}_{1,0}^{0,0}(\omega,-\omega^{-1}))= -\frac{\left(1+a^2 G\left(\frac{1}{\omega }\right)^2\right) \left(G\left(\frac{1}{\omega }\right)+a G(\omega )\right)}{8\omega \left(1+a^2\right)^2 G\left(\frac{1}{\omega }\right)\sqrt{\omega^{-2}+2c} (\omega^{2}+2c)} \mbox{ and}\\
&\frac{1}{2}s(G(\omega))s(G(\omega^{-1}))(\mathtt{x}_{1,1}^{0,0}(\omega,\omega^{-1})-\mathtt{x}_{1,1}^{0,0}(\omega,-\omega^{-1}))=\frac{-1+a G\left(\frac{1}{\omega }\right) G(\omega )}{8 \left(1+a^2\right)\sqrt{\omega^2+2c} \sqrt{\omega^{-2}+2c} }.
\end{split}
\end{equation}
To the above four equations (and where appropriate), we apply the appropriate relations given in Lemma~\ref{alg:lem:relations}.  More precisely, for the last equation in~\eqref{alg:lemproof:sumV00}, we apply no relations; for the penultimate equation in~\eqref{alg:lemproof:sumV00}, we apply the fourth equation in Lemma~\ref{alg:lem:relations} with $\omega\mapsto \omega^{-1}$; for the second equation in~\eqref{alg:lemproof:sumV00}, we apply the last equation of Lemma~\ref{alg:lem:relations} as written; for the first equation in~\eqref{alg:lemproof:sumV00}, we apply the first relation in Lemma~\ref{alg:lem:relations} twice, once as written and for the second application with $\omega \mapsto \omega^{-1}$, which gives 
\begin{equation}
\begin{split}
\frac{1}{2}(\mathtt{x}_{0,0}^{0,0}(\omega,\omega^{-1})-\mathtt{x}_{0,0}^{0,0}(\omega,-\omega^{-1}))&= \frac{\left(a-G\left(\frac{1}{\omega }\right) G(\omega )\right)^2 \left(-1+a G\left(\frac{1}{\omega }\right) G(\omega )\right)}{8 \left(a^2+G\left(\frac{1}{\omega }\right)^2\right) \left(a^2+G(\omega )^2\right)\sqrt{\omega^2+2c}\sqrt{\omega^{-2}+2c}}\\
&=\frac{-1+a G\left(\frac{1}{\omega }\right) G(\omega )}{8 \left(1+a^2\right)\sqrt{\omega^2+2c} \sqrt{\omega^{-2}+2c} },
\end{split}
\end{equation}
where the last line follows from applying a polynomial reduction under the spectral curve given in Lemma~\ref{asfo:lem:spectral} in both the numerator and denominator separately. Under the above simplications, we sum up the block of four equations in~\eqref{alg:lemproof:sumV00}, to arrive at
\begin{equation}
V_{0,0}(\omega, \omega)=\frac{-1+a G\left(\frac{1}{\omega }\right) G(\omega )}{2 \left(1+a^2\right)\sqrt{\omega^2+2c} \sqrt{\omega^{-2}+2c} },
\end{equation}
which is the expression we wanted.
We apply a similar procedure for the remaining values of $\eps_1$ and $\eps_2$.

\end{proof}

\subsection{Proof of Lemma~\ref{asfo:lem:liquidcomps}}\label{subseq:lem:liquidcomps}
Before we can prove Lemma \ref{asfo:lem:liquidcomps} we need some further notation and results. 
Let $\gamma_R$, $R<1$, be the ellipse in the $\omega$-plane given by
\begin{equation}
\omega=\sqrt{\frac{c}{2}}\left(Re^{\mathrm{i}\theta}-\frac 1Re^{-\mathrm{i}\theta}\right),
\end{equation}
$0\le \theta\le 2\pi$. Note that the map~\eqref{omegauchange} maps the circle $\Gamma_R$ in the $u$-plane to the ellipse $-\gamma_R$ in the $\omega$-plane, where $-\gamma_R$ denotes the contour $\gamma_R$ with negative orientation.

For two contours $\gamma,\gamma'$ and $k,\ell\in\mathbb{Z}$ we define
\begin{equation}\label{Dtilde1}
\tilde{D}_{\gamma,\gamma'}(k,\ell)=\frac{\mathrm{i}^{-k-\ell}}{2(1+a^2)(2\pi\mathrm{i})^2}
\int_{\gamma}d\omega_1\int_{\gamma'}d\omega_2\frac{G(\omega_1)^{\ell}G(\omega_2)^k}{(1-\omega_1\omega_2)\sqrt{\omega_1^2+2c}\sqrt{\omega_2^2+2c}}
\end{equation}
provided the double contour integral is well-defined. Also, let
\begin{equation}\label{Dgamma}
D_{\gamma,\gamma'}(x,y)=-\mathrm{i}^{1+h(\varepsilon_1,\varepsilon_2)}\left(a^{\varepsilon_2}\tilde{D}_{\gamma,\gamma'}(k_1,\ell_1)
+a^{1-\varepsilon_2}\tilde{D}_{\gamma,\gamma'}(k_2,\ell_2)\right),
\end{equation}
where $k_i,\ell_i$ are given by~\eqref{kldef}. Furthermore, for $k,\ell\in\mathbb{Z}$ we define
\begin{equation}\label{Ctilde}
\tilde{C}_{\omega_c}(k,\ell)=\frac{\mathrm{i}^{-k-\ell}}{2(1+a^2)2\pi\mathrm{i}}
\int_{\Gamma_{\omega_c}}\frac{d\omega}{\omega}\frac{G(\omega)^{\ell}G(\omega^{-1})^k}{\sqrt{\omega^2+2c}\sqrt{1/\omega^2+2c}},
\end{equation}
where $\Gamma_{\omega_c}$ is given by~\eqref{Gammaomegac}.

\begin{lemma}\label{lem:GasLiquidintegral}
Let $R_1=\sqrt{r_1/r_2}$ and $R_2=1/\sqrt{r_1r_2}$. Then
\begin{equation}\label{GasLiquiduintegral}
\mathbb{K}^{-1}_{r_1,r_2}(x,y)=
-\frac{\mathrm{i}^{1+h(\varepsilon_1,\varepsilon_2)}}{(2\pi \mathrm{i})^2}\int_{\Gamma_{R_1}}\frac{du_1}{u_1}\int_{\Gamma_{R_2}}\frac{du_2}{u_2}
\frac{a^{\varepsilon_2}u_2^{h(\varepsilon_1,\varepsilon_2)-1}+a^{1-\varepsilon_2}u_1u_2^{-h(\varepsilon_1,\varepsilon_2)}}
{\tilde{c}(u_1,u_2)u_1^{(x_1-y_1+1)/2}u_2^{(y_2-x_2-1)/2}}.
\end{equation}
\end{lemma}

\begin{proof}
If we make the change of variables $u_2\to1/u_2$ and use the identity
\begin{equation}
a^{\varepsilon_2}u_2^{1-h(\varepsilon_1,\varepsilon_2)}+a^{1-\varepsilon_2}u_1u_2^{h(\varepsilon_1,\varepsilon_2)}=
a^{1-\varepsilon_1}u_1^{1-h(\varepsilon_1,\varepsilon_2)}+a^{\varepsilon_1}u_1^{h(\varepsilon_1,\varepsilon_2)} u_2
\end{equation}
for $\varepsilon_1,\varepsilon_2\in\{0,1\}$, we see the equation~\eqref{GasLiquiduintegral} is equivalent to
\begin{equation}\label{GasLiquiduintegral2}
\mathbb{K}^{-1}_{r_1,r_2}(x,y)=
-\frac{\mathrm{i}^{1+h(\varepsilon_1,\varepsilon_2)}}{(2\pi \mathrm{i})^2}\int_{\Gamma_{R_1}}\frac{du_1}{u_1}\int_{\Gamma_{1/R_2}}\frac{du_2}{u_2}
\frac{a^{1-\varepsilon_1}u_1^{1-h(\varepsilon_1,\varepsilon_2)}+a^{\varepsilon_1}u_1^{h(\varepsilon_1,\varepsilon_2)}u_2}
{\tilde{c}(u_1,u_2)u_1^{(x_1-y_1+1)/2}u_2^{(x_2-y_2+1)/2}}.
\end{equation}

We will only verify~\eqref{GasLiquiduintegral} for $(\eps_1,\eps_2)=(0,0)$ and $(0,1)$.  The remaining two cases are analogous. 
We first take the change of variables $u_2=w u_1$ in~\eqref{GasLiquiduintegral2}, where $u_1$ is fixed followed by the change of variables $z=w u_1^2$.  For the second integral, the range of integration is doubled but this is compensated by the Jacobian.  Notice that
\begin{equation}
\label{alg:eq:char}
\tilde{c}(u_1,u_2)=\tilde{c}\left(\sqrt{\frac{z}{w}},\sqrt{w z}\right) =- P(z,w).
\end{equation}
The right side of equation~\eqref{GasLiquiduintegral2} reads
\begin{equation}\label{asym:lemproof:gaschange1}
\begin{split}
&\frac{ \mathrm{i}^{1+h(\eps_1,\eps_2)}}{(2\pi \mathrm{i})^2} \int_{\Gamma_1} \frac{ dz}{z}\int_{\Gamma_1} \frac{ dw}{w} \frac{ a^{1-\eps_1} \left( \sqrt{\frac{z}{w}} \right)^{1-h(\eps_1,\eps_2)} +a^{\eps_1} \sqrt{w z}\left( \sqrt{\frac{z}{w}} \right)^{h(\eps_1,\eps_2)}}{P(z,w) z^{\frac{ x_1-y_1+x_2-y_2+2}{4}} w^{\frac{x_2-y_2-x_1+y_1}{4}}}.\\
\end{split}
\end{equation}
It remains to check that for each combination of $\eps_1$ and $\eps_2$, the integrand in the above equation is equal to the integrand in~\eqref{res:eq:Kinvgas1}.

If $x \in \mathtt{W}_0$ and $y\in \mathtt{B}_0$ then the translation of the fundamental domains, defined in Section~\ref{subsec:Gibbs}, between the white vertex and black vertex  is 
\begin{equation}
\left( \frac{x_1-(y_1+1)}{4} + \frac{x_2+1-y_2}{4},  \frac{x_2+1-y_2}{4}- \frac{x_1-(y_1+1)}{4}\right),
\end{equation}
where the first coordinate is the translation with respect to the $z$ variables and the second coordinate is the translation with respect to the $w$ variables for the fundamental domain. This  means the factor appearing in the integrand of~\eqref{res:eq:Kinvgas1} from this translation is given by
\begin{equation}
z^{\frac{x_1-y_1+x_2-y_2}{4}} w^{\frac{x_2-y_2-x_1+y_1+2}{4}} .
\end{equation}
Therefore the integrand of~\eqref{res:eq:Kinvgas1} reads
\begin{equation}
\frac{  \mathrm{i} (a  + w ) }{P(z,w)z^{\frac{x_1-y_1+x_2-y_2}{4}} w^{\frac{x_2-y_2-x_1+y_1+2}{4}}}.
\end{equation}
On the other hand, the integrand in~\eqref{asym:lemproof:gaschange1} for $\eps_1=0$ and $\eps_2=0$ is given by 
\begin{equation}
 \frac{\mathrm{i}\left(a \sqrt{\frac{z}{w}} + \sqrt{{zw}} \right)}{P(z,w)    z^{\frac{ x_1-y_1+x_2-y_2+2}{4}} w^{\frac{x_2-y_2-x_1+y_1}{4}}}=\frac{  \mathrm{i} (a  + w ) }{P(z,w)z^{\frac{x_1-y_1+x_2-y_2}{4}} w^{\frac{x_2-y_2-x_1+y_1+2}{4}}}
\end{equation}
and so we have verified Eq.~\eqref{GasLiquiduintegral2} for $\eps_1=0$ and $\eps_2=0$.

If $x \in \mathtt{W}_0$ and $y\in \mathtt{B}_1$, the translation of the fundamental domains between the white vertex and black vertex is
\begin{equation}
\left( \frac{x_1-(y_1-1)}{4} + \frac{x_2+1-y_2)}{4},  \frac{x_2+1-y_2)}{4}-\frac{x_1-(y_1-1)}{4} \right)
\end{equation}
using the same coordinate notation as above
which means the factor appearing in the integrand of~\eqref{res:eq:Kinvgas1} from this translation is given by
\begin{equation}
z^{\frac{x_1-y_1+x_2-y_2+2}{4}} w^{\frac{x_2-y_2-x_1+y_1}{4}}.
\end{equation}
Therefore the integrand of~\eqref{res:eq:Kinvgas1} reads
\begin{equation}
\frac{-  (a+ z)  }{P(z,w)z^{\frac{x_1-y_1+x_2-y_2+2}{4}} w^{\frac{x_2-y_2-x_1+y_1}{4}}}.
\end{equation}
On the other hand, the integrand in~\eqref{asym:lemproof:gaschange1} for $\eps_1=0$ and $\eps_2=1$ is given by 
\begin{equation}
 \frac{-\left(a  + \sqrt{\frac{z}{w}}\sqrt{zw} \right)}{P(z,w)    z^{\frac{ x_1-y_1+x_2-y_2+2}{4}} w^{\frac{x_2-y_2-x_1+y_1}{4}}}= \frac{-\left(a  + z \right)}{P(z,w)    z^{\frac{ x_1-y_1+x_2-y_2+2}{4}} w^{\frac{x_2-y_2-x_1+y_1}{4}}} 
\end{equation}
and so we have verified Eq.~\eqref{GasLiquiduintegral2} for $\eps_1=0$ and $\eps_2=1$. The remaining two cases follow from similar computations.

\end{proof}

The next lemma gives expressions for the whole plane liquid and gas inverse Kasteleyn matrices in the $\omega$-variables. Note that
$S(x,y)=D_{\Gamma_R,\Gamma_R}(x,y)$ with $R>1$.

\begin{lemma}\label{GasLiquidDformula}
Assume that $R_1=\sqrt{r_1/r_2}$, $R_2=1/\sqrt{r_1r_2}$ and $R_1,R_2<1$. Then,
\begin{equation}\label{LiquidD}
\mathbb{K}^{-1}_{r_1,r_2}(x,y)=D_{\gamma_{R_1},\gamma_{R_2}}(x,y).
\end{equation}
We also have that
\begin{equation}\label{GasD}
\mathbb{K}^{-1}_{1,1}(x,y)=D_{\Gamma_{R},\Gamma_{R}}(x,y),
\end{equation}
for $\sqrt{2c}<R<1$. Furthermore,
\begin{equation}\label{Comegakl}
C_{\omega_c}(x,y)=-\mathrm{i}^{1+h(\varepsilon_1,\varepsilon_2)}\left(a^{\varepsilon_2}\tilde{C}_{\omega_c}(k_1,\ell_1)+a^{1-\varepsilon_2}\tilde{C}_{\omega_c}(k_2,\ell_2)\right),
\end{equation}
where $k_i,\ell_i$ are given by~\eqref{kldef}.
\end{lemma}

\begin{proof}
Using ~\eqref{kldef} in~\eqref{GasLiquiduintegral} we see that
\begin{equation}\label{GasLiquiduintegral3}
\mathbb{K}^{-1}_{r_1,r_2}(x,y)=-\mathrm{i}^{1+h(\varepsilon_1,\varepsilon_2)}\left[\frac{a^{\varepsilon_2}}{(2\pi\mathrm{i})^2}\int_{\Gamma_{R_1}}\frac{du_1}{u_1}
\int_{\Gamma_{R_2}}\frac{du_2}{u_2}\frac{u_1^{\ell_1}u_2^{k_1}}{\tilde{c}(u_1,u_2)}+
\frac{a^{1-\varepsilon_2}}{(2\pi\mathrm{i})^2}\int_{\Gamma_{R_1}}\frac{du_1}{u_1}
\int_{\Gamma_{R_2}}\frac{du_2}{u_2}\frac{u_1^{\ell_2}u_2^{k_2}}{\tilde{c}(u_1,u_2)}\right].
\end{equation}
Note that
\begin{equation}
\tilde{c}(u_1/\mathrm{i},u_2/\mathrm{i})=2(1+a^2)(1-\frac c2(u_1-1/u_1)(u_2-1/u_2)).
\end{equation}
We now first make the change of variables $u_1\to u_1\mathrm{i}$, $u_2\to u_2\mathrm{i}$ in~\eqref{GasLiquiduintegral3}, and then the change of variables
$u_1=G(\omega_1)$, $u_2=G(\omega_2)$ as in~\eqref{omegauchange}, which is possible since $R_1,R_2<1$. This gives us
\begin{equation}
\mathbb{K}^{-1}_{r_1,r_2}(x,y)=
-\mathrm{i}^{1+h(\varepsilon_1,\varepsilon_2)}\left(a^{\varepsilon_2}\tilde{D}_{\gamma_{R_1},\gamma_{R_2}}(k_1,\ell_1)+a^{1-\varepsilon_2}\tilde{D}_{\gamma_{R_1},\gamma_{R_2}}(k_2,\ell_2)\right)=D_{\gamma_{R_1},\gamma_{R_2}}(x,y),
\end{equation}
by~\eqref{Dtilde1} and~\eqref{Dgamma}.

Taking $r_1=r_2=1$ in~\eqref{GasLiquiduintegral3} we see that
\begin{equation}\label{Gasuintegral}
\mathbb{K}^{-1}_{1,1}(x,y)=-\mathrm{i}^{1+h(\varepsilon_1,\varepsilon_2)}\left[\frac{a^{\varepsilon_2}}{(2\pi\mathrm{i})^2}\int_{\Gamma_{1}}\frac{du_1}{u_1}
\int_{\Gamma_{1}}\frac{du_2}{u_2}\frac{u_1^{\ell_1}u_2^{k_1}}{\tilde{c}(u_1,u_2)}+
\frac{a^{1-\varepsilon_2}}{(2\pi\mathrm{i})^2}\int_{\Gamma_{1}}\frac{du_1}{u_1}
\int_{\Gamma_{1}}\frac{du_2}{u_2}\frac{u_1^{\ell_2}u_2^{k_2}}{\tilde{c}(u_1,u_2)}\right].
\end{equation}
In~\eqref{Gasuintegral} we can deform $\Gamma_1$ to $\Gamma_r$ with $r<1$ close to 1. The same change of variables as above then gives 
$\mathbb{K}^{-1}_{1,1}(x,y)=D_{\gamma_r,\gamma_r}(x,y)$. We can now deform the ellipse $\gamma_r$ to the circle $\Gamma_R$ with $R<1$
without passing points such that $\omega_1\omega_1=1$ in~\eqref{Dtilde1}. This proves~\eqref{GasD}.

It follows from Lemma \ref{asfo:lem:V}, \eqref{Comegac} and~\eqref{kldef} that
\begin{equation}\label{Comegaformula}
C_{\omega_c}(x,y)=\frac{\mathrm{i}^{(x_2-x_1+y_1-y_2)/2}}{2\pi\mathrm{i}}\int_{\Gamma_{\omega_c}}\frac{(-1)^{1+h(\varepsilon_1,\varepsilon_2)}a^{\varepsilon_2}G(\omega)^{\ell_1}G(\omega^{-1})^{k_1}+a^{1-\varepsilon_2}G(\omega)^{\ell_2}G(\omega^{-1})^{k_2}}{\omega\sqrt{\omega^2+2c}\sqrt{1/\omega^2+2c}}d\omega.
\end{equation}
Using~\eqref{kldef} and the fact that
\begin{equation}
\mathrm{i}^{x_2-x_1+y_1-y_2}=(-1)^{\varepsilon_1+\varepsilon_2+1},
\end{equation}
we see that
\begin{align}
\mathrm{i}^{(x_2-x_1+y_1-y_2)/2-1-h(\varepsilon_1,\varepsilon_2)}&=-(-1)^{-k_1-\ell_1}(-1)^{\varepsilon_1+\varepsilon_2+1}\\
\mathrm{i}^{(x_2-x_1+y_1-y_2)/2-1-h(\varepsilon_1,\varepsilon_2)}&=-(-1)^{-k_2-\ell_2}.
\end{align}
Together with~\eqref{Ctilde} and~\eqref{Comegaformula} this proves~\eqref{Comegakl}.
\end{proof}

The image of the curve $-\gamma_R$ under the map $\omega\to1/\omega$ is a curve $\tilde{\gamma}_R$, where $\tilde{\gamma}_R$ has positive orientation.
In order to prove Lemma \ref{asfo:lem:liquidcomps} we need the following result.

\begin{lemma}\label{Curveintersection}
Let $R_1\in(|G(1)|,|G(\mathrm{i})|)$. Then $\gamma_{R_1}$ and $\tilde{\gamma}_{R_1}$ intersect at exactly four points $\pm\omega_c$, $\pm\overline{\omega}_c$,
where $\omega_c=e^{\mathrm{i}\theta_c}$, $\theta_c\in(0,\pi/2)$ and $R_1=|G(\omega_c)|$. If $R_1=|G(\mathrm{i})|$, then $\gamma_{R_1}$ and $\tilde{\gamma}_{R_1}$
intersect at $\pm\mathrm{i}$ and $\tilde{\gamma}_{R_1}$ is outside $\gamma_{R_1}$.  If $R_1=|G(1)|$, then $\gamma_{R_1}$ and $\tilde{\gamma}_{R_1}$
intersect at $\pm 1$ and $\tilde{\gamma}_{R_1}$ is inside $\gamma_{R_1}$. 
\end{lemma}
\begin{center}
\begin{figure}
\includegraphics[height=3in]{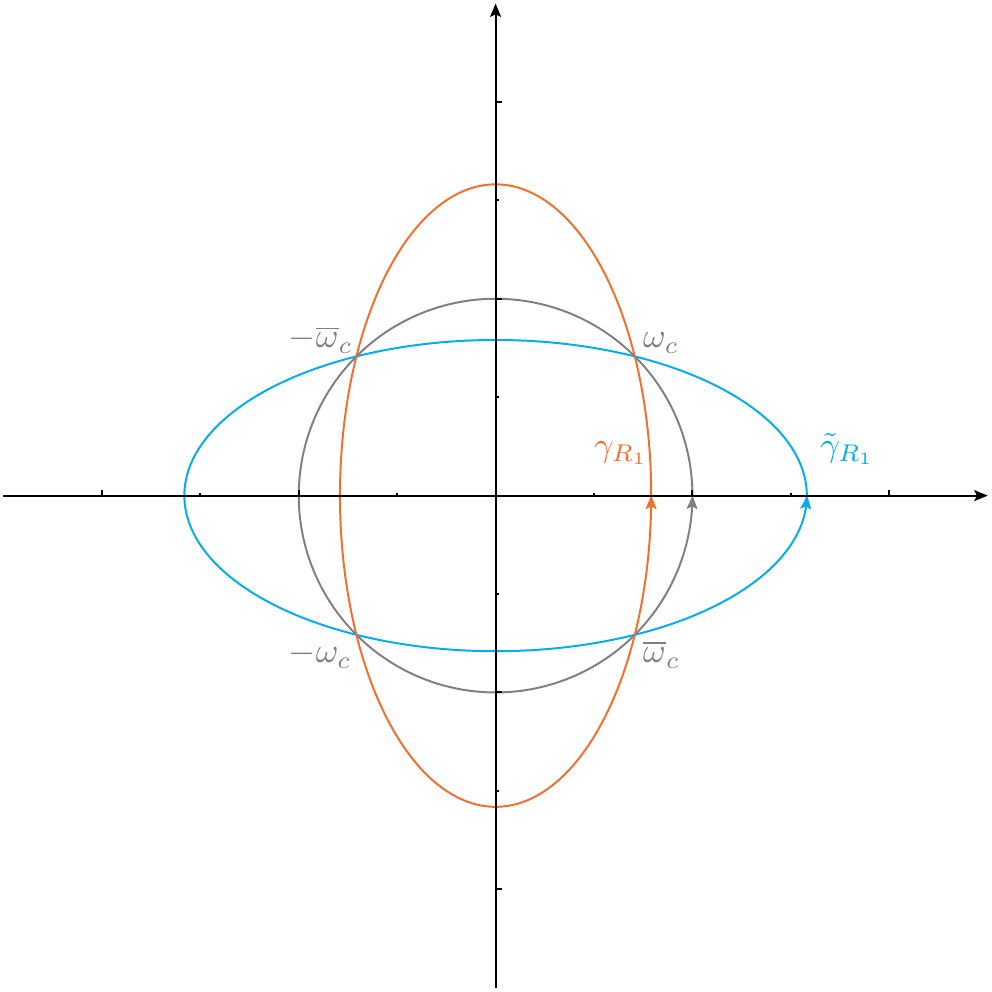}
\caption{A diagram of the contours $\gamma_{R_1}$ and $\tilde{\gamma}_{R_1}$ in Lemma~\ref{Curveintersection}. The gray circle is the unit circle.}
\label{fig:ellipses}
\end{figure} 
\end{center}
Figure~\ref{fig:ellipses} shows a realization of Lemma~\ref{Curveintersection}.

\begin{proof}

Let 
\begin{equation}\label{GammaRpar}
\omega_1=\sqrt{\frac c2}\left(R_1e^{\mathrm{i}\theta_1}-\frac 1{R_1}e^{-\mathrm{i}\theta_1}\right)\,,\,
\omega_2=\sqrt{\frac c2}\left(R_1e^{\mathrm{i}\theta_2}-\frac 1{R_1}e^{-\mathrm{i}\theta_2}\right),
\end{equation}
be two parametrizations of $\gamma_{R_1}$. The curves $\gamma_{R_1}$ and $\tilde{\gamma}_{R_1}$ intersect if $\omega_1\omega_2=1$, i.e.
\begin{equation}
\frac c2\left(-\left(e^{\mathrm{i}(\theta_1-\theta_2)}+e^{-\mathrm{i}(\theta_1-\theta_2)}\right)+R_1^2e^{\mathrm{i}(\theta_1+\theta_2)}
+\frac 1{R_1^2}e^{-\mathrm{i}(\theta_1+\theta_2)}\right)=1.
\end{equation}
If we take the imaginary part in this equation we get
\begin{equation}
\left(R_1^2-\frac 1{R_1^2}\right)\sin(\theta_1+\theta_2)=0,
\end{equation}
i.e. $\theta_1=\theta_2+n\pi$. If $\theta_2=-\theta_1+\pi$ we get $\omega_2=-\overline{\omega}_1$, which gives $\omega_1\omega_2=-|\omega_1|^2\neq 1$, so
we must have $\theta_2=-\theta_1$ from which we see that $\omega_2=\overline{\omega}_1$, i.e. $|\omega_1|=|\omega_2|=1$ at a point of intersection.
From~\eqref{GammaRpar} we see that $\gamma_{R_1}$ is an ellipse with half-axes $\sqrt{\frac c2}(1/R_1-R_1)$ in the $x$-direction and 
$\sqrt{\frac c2}(1/R_1+R_1)$ in the $y$-direction. We see that if $R_1\in[|G(1)|,|G(\mathrm{i})|]$, then $\sqrt{\frac c2}(1/R_1-R_1)\le 1$ with equality if and only if
$R_1=|G(1)|=(\sqrt{1+2c}-1)/\sqrt{2c}$, and $\sqrt{\frac c2}(1/R_1+R_1)\ge 1$ with equality if and only if
$R_1=|G(\mathrm{i})|=(1-\sqrt{1-2c})/\sqrt{2c}$. This gives the other statements since $\gamma_{R_1}$ will intersect the unit circle at exactly four points if
$R_1\in(|G(1)|,|G(\mathrm{i})|)$, at $\pm 1$ if $R_1=|G(1)|$ and at $\pm\mathrm{i}$ if $R_1=|G(\mathrm{i})|$.
\end{proof}

We are now ready for the proof of Lemma~\ref{asfo:lem:liquidcomps}.

\begin{proof}[Proof of Lemma~\ref{asfo:lem:liquidcomps}]
Let us first prove~\eqref{Gasidentity}. By~\eqref{GasD},
\begin{equation}
\mathbb{K}^{-1}_{1,1}(x,y)=D_{\Gamma_{R_1},\Gamma_{R_1}}(x,y),
\end{equation}
for $\sqrt{2c}<R_1<1$. Now,
\begin{equation}\label{Dtilde2}
\tilde{D}_{\Gamma_{R_1},\Gamma_{R_1}}(k,\ell)=\frac{\mathrm{i}^{-k-\ell}}{2(1+a^2)(2\pi\mathrm{i})^2}
\int_{\Gamma_{R_1}}d\omega_1\int_{\Gamma_{R_1}}\frac{d\omega_2}{\omega_2}\frac{G(\omega_1)^{\ell}G(\omega_2)^k}{(1/\omega_2-\omega_1)\sqrt{\omega_1^2+2c}\sqrt{\omega_2^2+2c}}.
\end{equation}
We can move the $\omega_1$-contour $\Gamma_{R_1}$ to $\Gamma_{R}$ where $R>1/R_1$, and we then pass the pole at $\omega_1=1/\omega_2$ which will give
a single contour integral contribution. In the remaining double contour integral we can deform the $\omega_2$-contour $\Gamma_{R_1}$ to $\Gamma_{R}$ without
passing any more poles. This gives
\begin{equation}
\tilde{D}_{\Gamma_{R_1},\Gamma_{R_1}}(k,\ell)=\tilde{C}_1(k,\ell)+\tilde{S}(k,\ell)
\end{equation}
and hence, by~\eqref{Sklidentity}, \eqref{Dgamma} and~\eqref{Comegakl}, we see that
\begin{equation}
\mathbb{K}^{-1}_{1,1}(x,y)=C_1(x,y)+S(x,y),
\end{equation}
which proves~\eqref{Gasidentity}.

Let us now prove~\eqref{Liquididentity}. If we take $r_1=1$ and $r_2=1/|G(\omega_c)|^2$ in~\eqref{LiquidD} we get
\begin{equation}
\mathbb{K}^{-1}_{r_1,r_2}(x,y)=D_{\gamma_{R_1},\gamma_{R_1}}(x,y),
\end{equation}
since $R_1=R_2=|G(\omega_c)|<1$. Now,
\begin{align}\label{Dtilde3}
\tilde{D}_{\gamma_{R_1},\gamma_{R_1}}(k,\ell)&=\frac{\mathrm{i}^{-k-\ell}}{2(1+a^2)(2\pi\mathrm{i})^2}
\int_{\gamma_{R_1}}d\omega_1\int_{\gamma_{R_1}}\frac{d\omega_2}{\omega_2}\frac{G(\omega_1)^{\ell}G(\omega_2)^k}{(1/\omega_2-\omega_1)\sqrt{\omega_1^2+2c}\sqrt{\omega_2^2+2c}}\\
&=\frac{\mathrm{i}^{-k-\ell}}{2(1+a^2)(2\pi\mathrm{i})^2}
\int_{\gamma_{R_1}}d\omega_1\int_{\tilde{\gamma}_{R_1}}\frac{d\omega_2}{\omega_2}\frac{G(\omega_1)^{\ell}G(1/\omega_2)^k}{(\omega_2-\omega_1)\sqrt{\omega_1^2+2c}\sqrt{1/\omega_2^2+2c}}.
\end{align}
We can now move the contour $\tilde{\gamma}_{R_1}$ inside $\gamma_{R_1}$. By Lemma \ref{Curveintersection} this gives pole contributions if $\omega_1$ is on the parts
of $\gamma_{R_1}$ between $\overline{\omega}_c$ and $\omega_c$, or between $-\overline{\omega}_c$ and $-\omega_c$. In the remaining double contour integral we can
deform the $\omega_1$-contour to $\Gamma_{R_1}$ and the $\omega_2$-contour to $\Gamma_{1/R}$, where $R_1R>1$. Thus,
\begin{align}
\tilde{D}_{\gamma_{R_1},\gamma_{R_1}}(k,\ell)&=\tilde{C}_1(k,\ell)-\tilde{C}_{\omega_c}(k,\ell)\\&+
\frac{\mathrm{i}^{-k-\ell}}{2(1+a^2)(2\pi\mathrm{i})^2}
\int_{\Gamma_{R_1}}d\omega_1\int_{\Gamma_{1/R}}\frac{d\omega_2}{\omega_2}\frac{G(\omega_1)^{\ell}G(1/\omega_2)^k}{(\omega_2-\omega_1)\sqrt{\omega_1^2+2c}\sqrt{1/\omega_2^2+2c}}.
\end{align}
In the last double contour integral we change $\omega\to1/\omega_2$ and then deform $\Gamma_{R_1}$ to $\Gamma_{R}$.
This gives
\begin{equation}
\tilde{D}_{\gamma_{R_1},\gamma_{R_1}}(k,\ell)=\tilde{C}_1(k,\ell)-\tilde{C}_{\omega_c}(k,\ell)+\tilde{S}(k,\ell).
\end{equation}
Using~\eqref{Sxy}, \eqref{Dgamma} and~\eqref{Comegakl} we see that
\begin{equation}
\mathbb{K}^{-1}_{r_1,r_2}(x,y)=C_1(x,y)-C_{\omega_c}(x,y)+S(x,y)=\mathbb{K}^{-1}_{1,1}(x,y)-C_{\omega_c}(x,y),
\end{equation}
by~\eqref{Gasidentity}, which is what we wanted to prove.

We see that the intersection points between $\gamma_{R_1}$ and $\tilde{\gamma}_{R_1}$ are mapped to points $\pm u_c$, $\pm\overline{u}_c$, where
$u_c=G(\omega_c)$, and these in turn, under $z=u_1u_2$, $w=u_2/u_1$, are mapped to points $(z_0,w_0)\in\Gamma_{r_1}\times\Gamma_{r_2}$
and $(\overline{z}_0,\overline{w}_0)\in\Gamma_{r_1}\times\Gamma_{r_2}$, where $P(z_0,w_0)=P(\overline{z}_0,\overline{w}_0)=0$. See the proof of 
Lemma \ref{lem:GasLiquidintegral}.
This shows that we are in the liquid phase. Compare the remark at the end of section~\ref{subsec:Gibbs}.
\end{proof}

\subsection{Proof of Propositions~\ref{Prop:GasGausslimit} and \ref{Prop:SolidGausslimit} and Lemma~\ref{lem:Solidphase}}\label{Sec:Gaussianlimitsproofs}

Before we can prove Proposition~\ref{Prop:GasGausslimit} about Gaussian asymptotics we need some notation and some further results. For $k,\ell\in\mathbb{Z}$, set
\begin{equation}\label{Ekl}
E_{k,\ell}=\frac 1{(2\pi\mathrm{i})^2}\int_{\Gamma_1}\frac{du_1}{u_1}\int_{\Gamma_1}\frac{du_2}{u_2}\frac{u^{\ell}u_2^{k}}{\tilde{c}(u_1,u_2)}.
\end{equation}
We see immediately from~\eqref{Ekl} that
\begin{equation}\label{Eklsym}
E_{k,\ell}=E_{\ell,k}=E_{-k,\ell}=E_{k,-\ell}=E_{-k,-\ell}.
\end{equation}
From the proof of Lemma~\ref{GasLiquidDformula} and~\eqref{Dtilde1} we see that
\begin{equation}\label{GasE}
\mathbb{K}^{-1}_{1,1}(x,y)=-\mathrm{i}^{1+h(\eps_1,\eps_2)}\left(a^{\eps_2}E_{k_1,\ell_1}+a^{1-\eps_2}E_{k_2,\ell_2}\right),
\end{equation}
where $k_i,\ell_i$ are given by~\eqref{kldef}. Also, for $r<1$ close to $1$ the change of variables $u_i=G(\omega_i)$ gives,
\begin{align}
E_{k,\ell}&=\tilde{D}_{\gamma_r,\gamma_r}(k,\ell)=\tilde{D}_{\Gamma_r,\Gamma_r}(k,\ell)\notag\\
&=\frac{\mathrm{i}^{-k-\ell}}{2(1+a^2)(2\pi\mathrm{i})^2}\int_{\Gamma_r}\frac{d\omega_1}{\omega_1}\int_{\Gamma_r}d\omega_2
\frac{G(\omega_1)^{\ell}G(\omega_2)^{k}}{(1/\omega_1-\omega_2)\sqrt{\omega_1^2+2c}\sqrt{\omega_2^2+2c}}.
\notag
\end{align}
From this we see that if $k\ge 0$, then
\begin{equation}\label{Eomega}
E_{k,\ell}=\frac{\mathrm{i}^{-k-\ell}}{2(1+a^2)2\pi\mathrm{i}}\int_{\Gamma_1}\frac{d\omega}{\omega}
\frac{G(\omega)^{\ell}G(1/\omega)^{k}}{\sqrt{\omega^2+2c}\sqrt{1/\omega^2+2c}}.
\end{equation}
This follows since the only pole in the $\omega_2$-integral outside $\Gamma_r$ is at $1/\omega_1$ if $k\ge 0$. Note that 
$G(\omega)$ decays like $1/\omega$ at infinity, and that the condition $k\ge 0$ is essential in~\eqref{Eomega}.

The next lemma implements the relations~\eqref{Eklsym} in the $\omega$-formula for $E_{k,\ell}$.
\begin{lemma}\label{lem:Eformula}
Let $A_m$, $B_m$, $m\ge 1$, be given and set $b_m=\max(|A_m|,|B_m|)$, and
\begin{equation}
a_m=\begin{cases} 
A_m   &\quad \text{if } b_m=|B_m| \\
B_m   &\quad \text{if } b_m=|A_m|. \\
\end{cases}
\end{equation}
Then,
\begin{equation}\label{Eformula}
E_{B_m+A_m,B_m-A_m}=\frac{(-1)^{b_m}}{2(1+a^2)2\pi\mathrm{i}}\int_{\Gamma_1}\frac{d\omega}{\omega}
\frac{G(\omega)^{b_m-a_m}G(1/\omega)^{b_m+a_m}}{\sqrt{\omega^2+2c}\sqrt{1/\omega^2+2c}}.
\end{equation}
\end{lemma}

\begin{proof} We have to check the different cases. Assume that $b_m=|A_m|$, $a_m=B_m$ and $A_m\ge 0$, so that $A_m\ge |B_m|$. Then, $b_m+a_m=A_m+B_m\ge 0$
and $b_m-a_m=A_m-B_m$. Using~\eqref{Eklsym} we see that
\begin{equation}
E_{B_m+A_m,B_m-A_m}=E_{b_m+a_m,a_m-b_m}=E_{b_m+a_m,b_m-a_m},
\end{equation}
which gives~\eqref{Eformula} by~\eqref{Eomega}. If $b_m=|A_m|=-A_m$ and $a_m=B_m$, so that $-A_m\ge |B_m|$, then $b_m+a_m=-A_m+B_m\ge 0$ and
$b_m-a_m=-A_m-B_m$. Hence, by~\eqref{Eklsym},
\begin{equation}
E_{B_m+A_m,B_m-A_m}=E_{-A_m+B_m,-A_m+B_m}=E_{b_m+a_m,-(b_m-a_m)}=E_{b_m+a_m,b_m-a_m},
\end{equation}
and again we get~\eqref{Eformula} by~\eqref{Eomega}. The case when $b_m=|B_m|$ is treated analogously.
\end{proof}

The next lemma gives an asymptotic formula for $E_{B_m+A_m,B_m-A_m}$. Together with~\eqref{GasE} this will enable us to prove
Proposition~\ref{Prop:GasGausslimit}.

\begin{lemma}\label{lem:Easymp}
Let $A_m,B_m,a_m,b_m$ be as in Lemma~\ref{lem:Eformula}. 
\begin{enumerate}
\item Assume that $b_m \to \infty$ as $m \to \infty$ and $|a_m| \leq b_m^{{7}/{12}}$ for large $m$.  Then, there exists a constant $d_1 >0$ so that
\begin{equation}\label{Easymp}
E_{B_m+A_m,B_m-A_m}=\frac{(-1)^{a_m+b_m}|G(\mathrm{i})|^{2b_m}\left(e^{-\frac{\sqrt{1-2c}}{2c}\frac{a_m^2}{b_m}}\left(1+O\left(b_m^{-{1}/{4}}\right)\right) +O\left(e^{-d_1 b_m^{{1}/{6}}}\right)\right)}{2(1+a^2)(1-2c)^{1/4}\sqrt{2\pi cb_m}}
\end{equation}
as $m\to\infty$.
\item Assume that $b_m >0$, $m\geq 1$.  There exists constants $C,d_1,d_2 >0$ so that
\begin{equation} \label{Easympeq2}
|E_{B_m+A_m,B_m-A_m} | \leq \frac{C}{\sqrt{b_m}} |G(\mathrm{i})|^{2 b_m} \left( e^{-d_1 \frac{a_m^2}{b_m}}+e^{-d_2 b_m} \right)
\end{equation}
for all $m\geq 1$.
\end{enumerate}

\end{lemma}
\begin{proof}
Write
\begin{equation}\label{g1g2}
g_1(\omega)=\log G(\omega)+\log G(1/\omega) \hspace{5mm} \mbox{and} \hspace{5mm}
g_2(\omega)=\log G(\omega)-\log G(1/\omega).
\end{equation}
Then, by~\eqref{Eformula}
\begin{equation}\label{Esaddle}
E_{B_m+A_m,B_m-A_m}=\frac{(-1)^{b_m}}{2(1+a^2)2\pi\mathrm{i}}\int_{\Gamma_1}\frac{d\omega}{\omega}
\frac{e^{b_mg_1(\omega)-a_mg_2(\omega)}}{\sqrt{\omega^2+2c}\sqrt{1/\omega^2+2c}}.
\end{equation}
We will prove~\eqref{Easymp} by a saddle-point analysis of the right side of~\eqref{Esaddle}.  Since $|a_m| \leq b_m$, we treat $g_1(\omega)$ as the dominant function and seek zeros of $g_1'(\omega)$.
 A computation gives
\begin{equation}
\omega g_1'(\omega)=-\frac{\omega}{\sqrt{\omega^2+2c}}+\frac{1/\omega}{\sqrt{1/\omega^2+2c}}
\end{equation}
from which we see that $g_1'(\omega)=0$ if and only if $\omega=\pm 1, \pm \mathrm{i}$. Further computation gives
\begin{equation}
g_1''(\pm\mathrm{i})=-\frac{4c}{(1-2c)^{3/2}}\,,\,g_1''(\pm 1)=-\frac{4c}{(1+2c)^{3/2}}.
\end{equation}
Similarly, $g_2'(\pm\mathrm{i})=\pm2\mathrm{i}/\sqrt{1-2c}$, and we also have $g''_2(\mathrm{i})=\mathrm{i} g'_2(\mathrm{i})$. By~\eqref{asfo:eq:Gsigns}, $G(e^{\mathrm{i}\theta})G(e^{-\mathrm{i}\theta})=|G(e^{\mathrm{i}\theta})|^2$ and we see that the unit circle from $\pm\mathrm{i}$
to $\pm 1$ gives a steepest descent path for $\pm\mathrm{i}$.  It is not hard to see that if $\Gamma_1^+$ is the upper half of the unit circle from $1$ to $-1$, then
\begin{equation}
\begin{split}
E_{B_m+A_m,B_m-A_m}&=\frac{(-1)^{b_m}}{(1+a^2) 2 \pi \mathrm{i}} \int_{\Gamma_1^+} \frac{d \omega}{\omega} \frac{ e^{b_m g_1(\omega)-a_m g_2(\omega)}}{\sqrt{\omega^2+2c} \sqrt{1/\omega^2+2c}} \\
&=\frac{(-1)^{a_m+b_m}|G(\mathrm{i})|^{2 b_m}}{(1+a^2) 2 \pi \mathrm{i}} \int_{\Gamma_1^+} \frac{d \omega}{\omega} \frac{ e^{b_m( g_1(\omega)-g_1(\mathrm{i})) -a_m( g_2(\omega)-g_2(\mathrm{i}))}}{\sqrt{\omega^2+2c} \sqrt{1/\omega^2+2c}}. \\
\end{split}
\end{equation}
We parameterize $\Gamma_1^+$ by
\begin{equation} \label{omegat}
\omega(t)= \mathrm{i}e^{\mathrm{i}t/\sqrt{b_m}}\quad,\quad |t|\leq \frac{\pi}{2}\sqrt{b_m}.
\end{equation}
Take $\epsilon$ small.  A Taylor expansion gives 
\begin{equation} \label{g1exp}
b_m (g_1(\omega(t)) -g_1(\mathrm{i})) =-\frac{2c}{(1-2c)^{\frac{3}{2}}} t^2 + \frac{t^3}{\sqrt{b_m}} \Phi_1(t),
\end{equation}
where $|\Phi_1(t)| \leq C$ for $|t| \leq 2 \epsilon \sqrt{b_m}$, $t \in \mathbb{C}$.  Using $g_2''(\mathrm{i}) = \mathrm{i} g_2'(\mathrm{i})$, we find 
\begin{equation}\label{g2exp}
a_m (g_2(\omega(t)) -g_2(\mathrm{i})) = - \frac{2 \mathrm{i}}{\sqrt{1-2c}} \frac{a_m}{\sqrt{b_m}} t+ \frac{a_m}{b_m^{3/2}}t^3 \Phi_2(t),
\end{equation}
where $|\Phi_2(t)| \leq C$ for $|t| \leq 2 \epsilon \sqrt{b_m}$, $t \in \mathbb{C}$.
Let $\Gamma_{1,\epsilon}^+$ be the part of $\Gamma_1^+$ given by~\eqref{omegat} for $\epsilon \sqrt{b_m} \leq |t| \leq \frac{\pi}{2} \sqrt{b_m}$.  It follows from~\eqref{g1exp},~\eqref{g2exp} and the fact that we have a steepest descent contour that
\begin{equation} \label{estimate1}
\left|\frac{1}{ 2 \pi \mathrm{i}} \int_{\Gamma_{1,\epsilon}^+} \frac{d \omega}{\omega} \frac{ e^{b_m( g_1(\omega)-g_1(\mathrm{i})) -a_m( g_2(\omega)-g_2(\mathrm{i}))}}{\sqrt{\omega^2+2c} \sqrt{1/\omega^2+2c}}\right|  \leq  C e^{-c_2 \epsilon^2 b_m}
\end{equation}
for some $c_2>0$. Consider the remaining part of the integral
\begin{equation}
\begin{split} \label{mainintegral}
&\frac{1}{ 2 \pi \mathrm{i}} \int_{\Gamma_1^+ \backslash \Gamma_{1,\epsilon}^+} \frac{d \omega}{\omega} \frac{ e^{b_m( g_1(\omega)-g_1(\mathrm{i})) -a_m( g_2(\omega)-g_2(\mathrm{i}))}}{\sqrt{\omega^2+2c} \sqrt{1/\omega^2+2c}} \\
&= \frac{1}{ 2 \pi \sqrt{b_m}} \int_{-\epsilon \sqrt{b_m}}^{\epsilon \sqrt{b_m}} \frac{ e^{-\frac{2c}{(1-2c)^{3/2}} t^2-\frac{2 \mathrm{i}}{\sqrt{1-2c}} \frac{a_m}{\sqrt{b_m}} t+ \frac{t^3}{\sqrt{b_m}} \Phi_3(t)}}{ \sqrt{\omega(t)^2+2c}\sqrt{1/\omega(t)^2+2c}} dt,
\end{split}
\end{equation}
where $\Phi_3(t) =\Phi_1(t)+a_m \Phi_2(t)/b_m$.  Since $|a_m| \leq b_m$ we see that $|\Phi_3(t)| \leq C$ for $|t| \leq 2 \epsilon \sqrt{b_m}$, $t \in \mathbb{C}$.  First consider the case $|a_m| \leq b_m^{7/12}$.  Take $\epsilon = \epsilon_0 b_m^{-5/12}$ so that $|t| \leq \epsilon_0 b_m^{1/12}$.   On the right side of~\eqref{mainintegral}, we make the change of variables 
\begin{equation}
t= s - \frac{ \mathrm{i} (1-2c)a_m}{2c \sqrt{b_m}}.
\end{equation}
Note that $\left| \frac{ \mathrm{i} (1-2c)a_m}{2c \sqrt{b_m}} \right| \leq \epsilon_0 b_m^{1/12}$ if we take $\epsilon_0$ large enough.  We shift the contour back to the real line ensuring that the error remains smaller than the one on the right side of~\eqref{estimate1}. Consequently, we get the integral
\begin{equation}
\frac{ e^{-\frac{\sqrt{1-2c}}{2c} \frac{a_m^2}{b_m}}}{2 \pi \sqrt{b_m}} \int_{-\epsilon_0 b_m^{1/12}}^{\epsilon_0 b_m^{1/12}} ds\frac{ e^{-\frac{2c}{(1-2c)^{3/2}} s^2 +\frac{1}{\sqrt{b_m}} (s- \mathrm{i} \kappa_0 \frac{a_m}{\sqrt{b_m}} )^3 \Phi_3 ( s- \mathrm{i} \kappa_0 \frac{a_m}{\sqrt{b_m}} )}}{ \sqrt{\omega(s- \mathrm{i} \kappa_0 \frac{a_m}{\sqrt{b_m}})^2 +2c}\sqrt{1/\omega(s- \mathrm{i} \kappa_0 \frac{a_m}{\sqrt{b_m}})^2 +2c}}
\end{equation}
where $\kappa_0 = (1-2c)/(2c)$. Since 
\begin{equation}
\frac{1}{\sqrt{b_m}} \left| (s- \mathrm{i} \kappa_0 \frac{a_m}{\sqrt{b_m}})^3 \Phi_3(s- \mathrm{i} \kappa_0 \frac{a_m}{\sqrt{b_m}})^3 \right|\leq  C b_m^{-1/4}
\end{equation}
some further computations and estimates give~\eqref{Easymp}.

We also want to prove the estimate~\eqref{Easympeq2}. Consider again~\eqref{mainintegral} and make the change of variables
\begin{equation}
t = s-\mathrm{i} \kappa \frac{a_m}{\sqrt{b_m}} 
\end{equation}
where $\kappa$, $0<\kappa\leq \epsilon$, will be chosen.  Note that $|\kappa a_m/\sqrt{b_m}| \leq \epsilon \sqrt{b_m}$.  After moving the integration contour back to the real line, which gives an error of the same size as the right side of~\eqref{estimate1}, we get the estimate
\begin{equation}
\begin{split} \label{estimate2}
&\left| \frac{1}{2\pi \sqrt{b_m}} \int_{-\epsilon \sqrt{b_m}}^{\epsilon \sqrt{b_m}} \frac{ e^{-\frac{2c}{(1-2c)^{3/2}} t^2-\frac{2 \mathrm{i}}{\sqrt{1-2c}} \frac{a_m}{\sqrt{b_m}} t+ \frac{t^3}{\sqrt{b_m}} \Phi_3(t)}}{ \sqrt{\omega(t)^2+2c}\sqrt{1/\omega(t)^2+2c}} dt \right|\\
& \leq  \frac{ C e^{-\frac{2c}{(1-2c)^{3/2}} \kappa^2 \frac{a_m^2}{b_m}}}{\sqrt{b_m} }
\int_{-\epsilon \sqrt{b_m}}^{\epsilon \sqrt{b_m}} e^{-\frac{2c}{(1-2c)^{3/2}}s^2+\frac{C_0}{\sqrt{b_m}} |s- \mathrm{i} \kappa \frac{a_m}{\sqrt{b_m}}|^3} ds +Ce^{-c_3 \epsilon^2 b_m} 
\end{split}
\end{equation}
where $C_0,c_3>0$. Now, for $|s| \leq \epsilon \sqrt{b_m}$ 
\begin{equation}
\begin{split}
\frac{1}{\sqrt{b_m}} |s- \mathrm{i} \kappa \frac{a_m}{\sqrt{b_m}}|^3 &\leq \frac{|s|^3}{\sqrt{b_m}} +3 \kappa \frac{a_m}{b_m}s^2+3 \kappa^2 \frac{a_m^2}{b_m^{3/2}} |s| +\kappa^3\frac{a_m^3}{b_m^2}  \\ 
&\leq (\eps+3 \kappa)s^2+\frac{a_m^2}{b_m} \kappa^2 (3 \epsilon+\kappa)
\end{split}
\end{equation}
since $|a_m| \leq b_m$.  Choose $\epsilon$ and $\kappa$ so that
\begin{equation}
3 C_0 (\epsilon+\kappa) \leq \frac{c}{(1-2c)^{3/2}}.
\end{equation}
Then, the right side of~\eqref{estimate2} is less than or equal to
\begin{equation}
 \frac{ C e^{-\frac{2c}{(1-2c)^{3/2}} \kappa^2 \frac{a_m^2}{b_m}}}{\sqrt{b_m} } \int_{-\infty}^\infty e^{-\frac{c}{(1-2c)^{3/2}}s^2} ds+Ce^{-c_3 \epsilon^2 b_m} 
\end{equation}
from which~\eqref{Easympeq2} follows.
\end{proof}

We are now ready to prove Proposition~\ref{Prop:GasGausslimit}.

\begin{proof}[Proof of Proposition~\ref{Prop:GasGausslimit}]
We will use the formula~\eqref{GasE} together with Lemma~\ref{lem:Easymp}. From $B_m+A_m=k_i$ and $B_m-A_m=\ell_i$ we see that in the lemma we will
have
\begin{equation}
B_m=\frac{k_i+\ell_i}2\quad,\quad A_m=\frac{k_i-\ell_i}2.
\end{equation}
Thus, for $i=1,2$,
\begin{align}\label{AmBm}
2B_m&=\frac{x_2-x_1+y_1-y_2}2+(-1)^{i}(1-h(\eps_1,\eps_2))\\&=
2(\beta_x-\beta_y)\lambda_2(2m)^{2/3}+\frac{u_2-v_2+v_1-u_1}2+(-1)^{i}(1-h(\eps_1,\eps_2)),
\notag\\
2A_m&=2(\alpha_x-\alpha_y)\lambda_1(2m)^{1/3}+\frac{u_2-v_2+u_1-v_1}2-(-1)^{i}h(\eps_1,\eps_2).\notag
\end{align}
We can assume that $m$ is large. Assume first that $\beta_x\neq\beta_y$. Then, using the definition in Lemma~\ref{lem:Eformula} we see that
$a_m=A_m$ and $b_m=|B_m|$.  The $G(\mathrm{i})$-factor in the left side of~\eqref{GasGausslimit} together with the $G(\mathrm{i})$-factor in~\eqref{Easympeq2}
gives the factor
\begin{equation}\label{Gi}
|G(\mathrm{i})|^{2(|B_m|-B_m)}
\end{equation}
by~\eqref{AmBm} if we ignore $\pm(1-h(\eps_1,\eps_2))$. If $\beta_x<\beta_y$, then $B_m<0$ and we see that the expression in~\eqref{Gi} goes to zero fast as $m\to\infty$.
Note that $|G(\mathrm{i})|=(1-\sqrt{1-2c})/\sqrt{2c}<1$.
Assume that $\beta_x>\beta_y$, so that $b_m=B_m$. Then, we have $a_m+b_m=k_i$. We obtain using~\eqref{Easymp}
\begin{equation}
\begin{split}
(2m)^{1/3} |G(\mathrm{i})|^{\frac{x_1-x_2+y_2-y_1+2}{2}} E_{k_i,\ell_i}&= \frac{(-1)^{k_i} |G(\mathrm{i})|^{1+(-1)^{i}(1-h(\eps_1,\eps_2))} 
e^{-\frac{\lambda_1^2\sqrt{1-2c}}{2c\lambda_2}\frac{(\alpha_x-\alpha_y)^2}{\beta_x-\beta_y}}}
{2(1+a^2)(1-2c)^{1/4}\sqrt{2\pi c\lambda_2}\sqrt{\beta_x-\beta_y}}(1+O(m^{-1/6})) \\
&=\frac{(-1)^{k_i} |G(\mathrm{i})|^{1+(-1)^{i}(1-h(\eps_1,\eps_2))} c_0
e^{-\frac{(\alpha_x-\alpha_y)^2}{4(\beta_x-\beta_y)}}}
{(1+a^2)(1-2c)\sqrt{4\pi }\sqrt{\beta_x-\beta_y}}(1+O(m^{-1/6})) \\
\end{split}
\end{equation}
where we have substituted in the scale factors given in Theorem~\ref{thm:transversal1} to simplify the above equation. Using the above equation and~\eqref{GasE}, we find that
\begin{equation}
\begin{split}\label{Easymp2}
&(2m)^{1/3} |G(\mathrm{i})|^{\frac{x_1-x_2+y_2-y_1+2}{2}} \mathbb{K}_{1,1}^{-1}(x,y) \\
&= (2m)^{1/3} |G(\mathrm{i})|^{\frac{x_1-x_2+y_2-y_1+2}{2}} (-a^{\eps_2} \mathrm{i}^{1+h(\eps_1,\eps_2)} E_{k_1,\ell_1} -a^{1-\eps_2} \mathrm{i}^{1+h(\eps_1,\eps_2)} E_{k_2,\ell_2} ) \\
&=\frac{ 
(1+O(m^{-1/6}))c_0e^{-\frac{(\alpha_x-\alpha_y)^2}{4(\beta_x-\beta_y)}}}
{(1+a^2)(1-2c)\sqrt{4\pi }\sqrt{\beta_x-\beta_y}}  \\ &\times
\left(-a^{\eps_2} \mathrm{i}^{2k_1+1+h(\eps_1,\eps_2)}|G(\mathrm{i})|^{h(\eps_1,\eps_2)} -a^{1-\eps_2} \mathrm{i}^{2k_2+1+h(\eps_1,\eps_2)}|G(\mathrm{i})|^{2-h(\eps_1,\eps_2)} \right).
\end{split}
\end{equation}
Observe that
\begin{equation}
\mathrm{i}^{2k_1+1}=\mathrm{i}^{x_2-y_2+2h(\eps_1,\eps_2)}=
\mathrm{i}^{x_2-y_2+2\eps_1+2\eps_2+4\eps_1 \eps_2}=
-\mathrm{i}^{2x_2-2y_2+y_1-x_1}=\mathrm{i}^{y_1-x_1}
\end{equation}
because $\mathrm{i}^{x_2-x_1}=\mathrm{i}^{2\eps_1-1}$, $\mathrm{i}^{y_1-y_2}=\mathrm{i}^{2\eps_2-1}$ and $x_2-y_2 \mod 2=1$.  We also have $\mathrm{i}^{2k_2+1}=-\mathrm{i}^{2k_1+1}$.   These sign simplifications can be used in~\eqref{Easymp2} to obtain
\begin{equation}
\begin{split} \label{Easymp3simplify}
&(2m)^{1/3} |G(\mathrm{i})|^{\frac{x_1-x_2+y_2-y_1+2}{2}} \mathbb{K}_{1,1}^{-1}(x,y) 
=\frac{ 
(1+O(m^{-1/6}))c_0e^{-\frac{(\alpha_x-\alpha_y)^2}{4(\beta_x-\beta_y)}}}
{(1+a^2)(1-2c)\sqrt{4\pi }\sqrt{\beta_x-\beta_y}}  \\ &\times
\mathrm{i}^{y_1-x_1}
\left(-a^{\eps_2} \mathrm{i}^{h(\eps_1,\eps_2)}|G(\mathrm{i})|^{h(\eps_1,\eps_2)} +a^{1-\eps_2} \mathrm{i}^{h(\eps_1,\eps_2)}|G(\mathrm{i})|^{2-h(\eps_1,\eps_2)} \right).
\end{split}
\end{equation}
Notice that from $G(\mathrm{i})=\mathrm{i}(1-\sqrt{1-2c})/\sqrt{2c}$, $G(\mathrm{i}^{-1})=-G(\mathrm{i})$, Lemma~\ref{asfo:lem:V} and Lemma~\ref{asfo:lem:gascoeff}, we have
\begin{equation}
\begin{split}
&\left(-a^{\eps_2} \mathrm{i}^{h(\eps_1,\eps_2)}|G(\mathrm{i})|^{h(\eps_1,\eps_2)} +a^{1-\eps_2} \mathrm{i}^{h(\eps_1,\eps_2)}|G(\mathrm{i})|^{2-h(\eps_1,\eps_2)} \right)  \\ 
&=\left(a^{\eps_2} (-1)^{1+h(\eps_1,\eps_2)}G(\mathrm{i}^{-1})^{h(\eps_1,\eps_2)} +a^{1-\eps_2} G(\mathrm{i}) G(\mathrm{i}^{-1})^{1-h(\eps_1,\eps_2)}\right)\\
&=2(1+a^2) (1-2c)V_{\eps_1,\eps_2}(\mathrm{i},\mathrm{i})  = \mathrm{i} (1+a^2)(1-2c)
\mathtt{g}_{\eps_1,\eps_2} \end{split}
\end{equation}
and using the above simplification in~\eqref{Easymp3simplify} gives
\begin{equation}
(2m)^{1/3} |G(\mathrm{i})|^{\frac{x_1-x_2+y_2-y_1+2}{2}} \mathbb{K}_{1,1}^{-1}(x,y) 
=\frac{ 
(1+O(m^{-1/6}))\mathrm{i}^{y_1-x_1+1} c_0 \mathtt{g}_{\eps_1,\eps_2} }
{\sqrt{4\pi (\beta_x-\beta_y)}} e^{-\frac{(\alpha_x-\alpha_y)^2}{4(\beta_x-\beta_y)}}.
\end{equation}

If $\beta_x=\beta_y$ and $\alpha_x\neq\alpha_y$, then $b_m=|A_m|$, $a_m=B_m$ and, by~\eqref{AmBm}, $(x_2-x_1+y_2-y_1)/2$ is bounded. Since $|A_m|\to\infty$,
it follows from~\eqref{Easympeq2} that $E_{k_i,\ell_i}$ goes to zero like $|G(\mathrm{i})|^{2|A_m|}$. The proposition is proved.
\end{proof}

We will now give the proof of Lemma~\ref{lem:Solidphase} which concerns properties of $S(x,y)$ defined by~\eqref{Sxy}.  Recall that $e_1=(1,1)$ and $e_2=(-1,1)$.

\begin{proof}[Proof of Lemma~\ref{lem:Solidphase}]
Consider $\tilde{S}(k,\ell)$ defined by~\eqref{Skl}. Analyticity outside $\Gamma_R$ in $\omega_1$ or $\omega_2$ including at infinity shows that
\begin{equation}\label{Sklidentity}
\tilde{S}(k,\ell)=0 \hspace{5mm} \mbox{if $k\geq 0$ or $l \geq 0$}.
\end{equation}
Here we use that $G(\omega)\sim-\sqrt{c/2}\,\omega^{-1}$ as $\omega\to\infty$. From~\eqref{Sklidentity} and~\eqref{kldef} we see that~\eqref{Sxyidentity}
follows. A consequence of~\eqref{Sxyidentity} is that $S(x,x+f)=0$ only if $f\neq e_2$. By~\eqref{Sxy} and~\eqref{Sklidentity}, $S(x,x+e_2)=-\mathrm{i}a^{\eps_1}
\tilde{S}(-1,-1)$, since $x+e_2\in B_{\eps_1}$, i.e. $\eps_1=\eps_2$. Thus, with $R>1$,
\begin{equation}\label{Se2}
S(x,x+e_2)=\frac{a^{\eps_1}\mathrm{i}}{2(1+a^2)(2\pi\mathrm{i})^2}\int_{\Gamma_R}d\omega_1\int_{\Gamma_R}d\omega_2
\frac{G(\omega_1)^{-1}G(\omega_2)^{-1}}{(1-\omega_1\omega_2)\sqrt{\omega_1^2+2c}\sqrt{\omega_2^2+2c}}.
\end{equation}
Using $G(\omega)\sim-\sqrt{c/2}\omega^{-1}$ and $\sqrt{\omega^2+2c}\sim\omega$ as $\omega\to\infty$, we see that as $R\to\infty$ the right side of~\eqref{Se2}
converges to
\begin{equation}
-\frac{a^{\eps_1}\mathrm{i}}{2(1+a^2)}\frac 2c=-\mathrm{i}a^{\eps_1-1}.
\end{equation}

To see that $S(x,y)$ works as a coupling function for a solid phase, we let $w_1,\dots,w_k$ be white vertices, $w_i=(x_1^i,x_2^i)$. Let $f_i=(f_1^i,f_2^i)\in\{\pm e_1,\pm e_2\}$,
and consider $\det(S(w_j,w_i+f_i))$, compare Theorem~\ref{localstatisticsthm}. Without changing the value of the determinant we can assume that
$w_1,\dots,w_k$ is lexicographically ordered, i.e.  $x_1^1\le\dots\le x_1^k$ and $x_2^i<x_2^{i+1}$ if $x_1^i=x_1^{i+1}$. Now,
\begin{equation}
w_i+f_i-w_j=(x_1^i-a_i^j+f_1^i,x_2^i-x_2^j+f_2^i).
\end{equation}
If $i>j$ and $x_1^i>x_1^j$, then~\eqref{Sxyidentity} shows that $S(w_j,w_i+f_i)=0$.
If $i>j$ and $x_1^i=x_1^j$, then $x_2^i>x_2^j$, and again $S(w_j,w_i+f_i)=0$. Hence, the determinant $\det(S(w_j,w_i+f_i))$ is upper-triangular and the diagonal
elements are $\neq 0$ if and only if $f_i=e_2$.
\end{proof}

Next, we turn to the proof of Proposition~\ref{Prop:SolidGausslimit}.

\begin{proof}[Proof of Proposition~\ref{Prop:SolidGausslimit}]
By~\eqref{Gasidentity} and~\eqref{Sxyidentity}
\begin{equation}\label{Gasidentity2}
S(x,y)=\mathbbm{I}_{y_1-x_1\le -1}\mathbbm{I}_{x_2-y_2\le -1} (\mathbb{K}_{1,1}^{-1}(x,y)-C_1(x,y)).
\end{equation}
Now,
\begin{equation}\label{Solidasymp1}
(2m)^{1/3}|G(1)|^{\frac{x_1-x_2+y_2-y_1}2}\mathbb{K}_{1,1}^{-1}(x,y)=\left|\frac{G(1)}{G(\mathrm{i})}\right|^{\frac{x_1-x_2+y_2-y_1}2}
|G(\mathrm{i})|^{\frac{x_1-x_2+y_2-y_1}2}(2m)^{1/3}\mathbb{K}_{1,1}^{-1}(x,y).
\end{equation}
It follows from $\sqrt{1\pm 2c}<1\pm c$, that $|G(1)/G(\mathrm{i})|<1$. By~\eqref{kldef} and Condition~\ref{condition2},
\begin{align}\label{kldef2}
k_i&=\frac{x_2-y_2+(-1)^i}2-(-1)^ih(\eps_1,\eps_2)\\&=
(\beta_x-\beta_y)\lambda_2(2m)^{2/3}+(\alpha_x-\alpha_y)\lambda_1(2m)^{1/3}+\frac{u_2-v_2+(-1)^i}2-(-1)^ih(\eps_1,\eps_2),
\notag\\
\ell_i&=(\beta_x-\beta_y)\lambda_2(2m)^{2/3}-(\alpha_x-\alpha_y)\lambda_1(2m)^{1/3}+\frac{v_1-u_1+(-1)^i}2.\notag
\end{align}
We may assume that $m$ is large. Note that if $\beta_x>\beta_y$, then $S(x,y)=0$ by~\eqref{Sxyidentity}. If $\beta_x=\beta_y$ and $\alpha_x\neq\alpha_y$,
then~\eqref{Sxyidentity} again gives $S(x,y)=0$. Thus, $\beta_y>\beta_x$ is the only possibility if we want a non-zero limit.

We see from~\eqref{kldef2} that $(x_1-x_2+y_2-y_1)/2\sim 2(\beta_y-\beta_x)\lambda_2(2m)^{2/3}\to\infty$, as $m\to\infty$, and thus by~\eqref{Solidasymp1},
\begin{equation}\label{Kgasasymp}
\lim_{m\to\infty}(2m)^{1/3}|G(1)|^{\frac{x_1-x_2+y_2-y_1}2}\mathbb{K}_{1,1}^{-1}(x,y)=0.
\end{equation}
It remains to compute the asymptotics of $C_1(x,y)$. This is very similar to the proof of Lemma~\ref{lem:Easymp}. Consider~\eqref{Comegakl} with $\omega_c=1$. Write
\begin{equation}
U_i=\frac{u_2-v_2+(-1)^i}2-(-1)^ih(\eps_1,\eps_2)\,,\,V_i=\frac{v_1-u_1+(-1)^i}2,
\end{equation}
which are bounded quantities. Also, write
\begin{equation}
\beta_m=(\beta_y-\beta_x)\lambda_2(2m)^{2/3}\,,\,
\alpha_m=(\alpha_x-\alpha_y)\lambda_1(2m)^{1/3}.
\end{equation}
Then, using~\eqref{g1g2},
\begin{align}
\tilde{C}_1(k_i,\ell_i)&=\frac{\mathrm{i}^{-k_i-\ell_k}}{2(1+a^2)2\pi\mathrm{i}}\int_{\Gamma_1}\frac{d\omega}{\omega}
\frac{G(\omega)^{-\beta_m+\alpha_m+U_i}G(1/\omega)^{-\beta_m-\alpha_m+V_i}}
{\sqrt{\omega^2+2c}\sqrt{1/\omega^2+2c}}\\
&=\frac{\mathrm{i}^{-k_i-\ell_k}}{2(1+a^2)2\pi\mathrm{i}}\int_{\Gamma_1}\frac{d\omega}{\omega}
\frac{e^{-\beta_mg_1(\omega)+\alpha_mg_2(\omega)}G(\omega)^{U_i}G(1/\omega)^{V_i}}
{\sqrt{\omega^2+2c}\sqrt{1/\omega^2+2c}}.
\end{align}
If we compare with the proof of Lemma~\ref{lem:Easymp} the dominating saddle points are now $\pm 1$ instead, and the unit circle is again a steepest descent curve.
Proceeding in an analogous fashion we obtain
\begin{equation}
\begin{split}\label{Ctildeasymp}
C_1(x,y)&=-\mathrm{i}^{1+h(\eps_1,\eps_2)}(a^{\eps_2} C_1(k_1,l_1) +a^{1-\eps_2} C_1(k_2,l_2)) \\
&=-2\frac{\mathrm{i}^{-\frac{x_2-x_1+y_1-y_2}{2}+1+h(\eps_1,\eps_2)}G(1)^{\frac{x_2-x_1+y_1+y_2-2}{2}}}{2(1+a^2)(1+2c)^{1/4} \sqrt{2\pi c \lambda_2}} 
\frac{1}{\sqrt{\beta_y-\beta_x}}e^{-\frac{\lambda_1^2\sqrt{1+2c}}{2c\lambda_2}\frac{(\alpha_x-\alpha_2)^2}{\beta_y-\beta_x}}\\
&\times
 2(\mathrm{i}^{1-h(\eps_1,\eps_2)}a^{\eps_2}G(1)^{h(\eps_1,\eps_2)} +\mathrm{i}^{-1+h(\eps_1,\eps_2)}a^{1-\eps_2} G(1)^{2-h(\eps_1,\eps_2)})(2m)^{-1/3}(1+O(m^{-1/6})) \\
 \end{split}
\end{equation}
Using the fact that $\mathrm{i}^{-\frac{x_2-x_1+y_1-y_2}{2}+2h(\eps_1,\eps_2)}=-\mathrm{i}^{-\frac{x_2-x_1+y_1-y_2}{2}+2\eps_1+2\eps_2-2}=-\mathrm{i}^{\frac{x_2-x_1+y_1-y_2}{2}}$ since $\mathrm{i}^{x_2-x_1} =\mathrm{i}^{2\eps_1-1}$ and $\mathrm{i}^{y_1-y_2} =\mathrm{i}^{ 2\eps_2-1}$ and using Lemma~\ref{asfo:lem:V}, we find that the right side of~\eqref{Ctildeasymp} reduces to 
\begin{equation}\label{Ctildeasymp2}
\begin{split}
C_1(x,y)&=\frac{\mathrm{i}^{\frac{x_2-x_1+y_1-y_2}{2}}G(1)^{\frac{x_2-x_1+y_1+y_2-2}{2}}}{\sqrt{2 \pi c\lambda_2(\beta_y-\beta_x)} (1+2c)^{-3/4}}2V_{\eps_1,\eps_1}(1,1)e^{-\frac{\lambda_1^2\sqrt{1+2c}}{2c\lambda_2}\frac{(\alpha_x-\alpha_2)^2}{\beta_y-\beta_x}}(2m)^{-1/3}(1+O(m^{-1/6}))\\
&=\frac{\mathrm{i}^{\frac{x_2-x_1+y_1-y_2}{2}}G(1)^{\frac{x_2-x_1+y_1+y_2-2}{2}}}{\sqrt{4 \pi(\beta_y-\beta_x)}} 2V_{\eps_1,\eps_2}(1,1) c_0 e^{-\frac{(\alpha_x-\alpha_2)^2}{4(\beta_y-\beta_x)}}(2m)^{-1/3}(1+O(m^{-1/6}))\\
\end{split}
\end{equation}
where we have used the scale factors $\lambda_1$ and $\lambda_2$ and constant $c_0$ from the statement of Theorem~\ref{thm:transversal2}.
Using Eq.~\eqref{Kgasasymp}, Lemma~\ref{asfo:lem:solidcoeff} and the asymptotics from Eqs.~\eqref{Ctildeasymp2} substituted into~\eqref{Comegakl}, we obtain~\eqref{SolidGausslimit}.
\end{proof}

\subsection{Proofs of Lemma \ref{asym:lem:saddle} and Lemma \ref{asym:lem:imaginary}}

We now give the proof of Lemma~\ref{asym:lem:saddle}.
\begin{proof}[Proof of Lemma~\ref{asym:lem:saddle}]

We can represent $1/\sqrt{\omega^2+2c}$ as
\begin{equation}
\frac{1}{\sqrt{\omega^2+2c}} = \frac{1}{\pi} \int_{-\sqrt{2c}}^{\sqrt{2c}}\frac{1}{\omega - \mathrm{i}s} \frac{ds}{\sqrt{2c-s^2}}=\int\frac{1}{\omega-\mathrm{i}s} d \mathtt{m}(s),
\end{equation}
which gives the correct square root.
We consider a discrete approximation to the probability measure  $\mathtt{m}$.  Suppose that $-\sqrt{2c} < \gamma_1 <\dots < \gamma_m <0<\gamma_{m+1} < \dots <\gamma_{2m} <\sqrt{2c}$ and that for $n=4m$
\begin{equation}
\frac{2}{n} \sum_{j=1}^{n/2} \delta_{\gamma_j}   \to \mathtt{m} 
\end{equation}
in the sense of weak convergence. 
Then, we rewrite the saddle point equation~\eqref{asym:eq:saddleeqn} using the discrete approximation of the measure $\mathtt{m}$ and find 
\begin{equation}
\begin{split}
1+\frac{2}{n} \sum_{j=1}^{n/2}\left( \frac{\xi_1 \omega}{\omega - \mathrm{i}\gamma_j} + \frac{\xi_2 }{1-\mathrm{i}\gamma_j \omega} \right)=0.
\end{split}
\end{equation}
For notational convenience, we set $\tilde{\gamma}_j=-1/\gamma_j$ and set $\omega=\mathrm{i}z$. The above equation becomes
\begin{equation}
\begin{split}\label{alg:lemproof:roots}
1 +\frac{2}{n} \sum_{j=1}^{n/2}\left( \frac{\xi_1 z}{z- \gamma_j}-\frac{\xi_2\tilde{\gamma}_j }{z-\tilde{\gamma_j} } \right)=0.
\end{split}
\end{equation}
By a rearrangement, we have
\begin{equation}
\begin{split} 
 &\sum_{j=1}^{n/2}\left[ \prod_{k=1}^{n/2} (z- \gamma_k)(z-\tilde{\gamma_k})\right]+\xi_1 z(z-\tilde{\gamma_j})\left[ \prod_{k\not=j}^{n/2} (z- \gamma_k)(z-\tilde{\gamma_k})\right]\\&-   \xi_2 \tilde{\gamma}_j(z- \gamma_j) \left[ \prod_{k\not=j}^{n/2} (z- \gamma_k)(z-\tilde{\gamma_k})\right]=0.
\end{split}
\end{equation}
Since $|\xi_1|<1$, the above equation has $n$ roots . We count how many of these roots are contained within the cuts. By our choice of $\gamma_j$, we have
 $-\infty< \tilde{\gamma}_{m+1}<  \dots < \tilde{\gamma}_{2m} < -1/\sqrt{2c}  < -\sqrt{2c} < \gamma_1 <\dots < \gamma_m <0<\gamma_{m+1} < \dots <\gamma_{2m} <\sqrt{2c}<1/\sqrt{2c} < \tilde{\gamma}_1 < \dots<\tilde{\gamma}_m <\infty $.  It is clear that in each of the  intervals $(\gamma_{k},\gamma_{k+1})$, $(\gamma_{k+m},\gamma_{m+k+1})$, $(\tilde{\gamma}_{k},\tilde{\gamma}_{k+1})$ and $(\tilde{\gamma}_{m+k},\tilde{\gamma}_{m+k+1})$ for $1 \leq k \leq m-1$ there is a root due to a sign change on the left side of~\eqref{alg:lemproof:roots} when traversing each interval from left to right.  As a result there are at least $n-4$ roots in the branch cuts.  Therefore, there are at most $4$ roots away from the branch cuts. We take limits as $n$ tends to infinity to conclude that there are at most $4$ roots of~\eqref{asym:eq:saddleeqn} in $\mathbb{C}\backslash \mathrm{i} ( (-\infty,-1/\sqrt{2c}]\cup[-\sqrt{2c},\sqrt{2c}]\cup [1/\sqrt{2c},\infty))$.

\end{proof}

We now give the proof of Lemma~\ref{asym:lem:imaginary}.
\begin{proof}[Proof of Lemma~\ref{asym:lem:imaginary}]
We present the first, second, fourth and fifth cases since the omitted cases are analogous. We have
\begin{equation}
\mathrm{Im}\, g_{\xi,\xi} (\omega) = \arg \omega -\xi \arg G(\omega) +\xi \arg G(\omega^{-1}),
\end{equation}
where $\arg$ takes values in $(-\frac{\pi}{2},\frac{3 \pi }{2})$.

For $0<t<\sqrt{2c}$, we have
\begin{equation}
 \lim_{\delta \to 0^+} G\left(\delta +\mathrm{i}t \right)= \frac{1}{\sqrt{2c}} \left(\mathrm{i}t -\sqrt{2c-t^2}\right) \hspace{2mm} \mbox{and} \hspace{2mm} \lim_{\delta \to 0^+} G\left((\delta +\mathrm{i}t)^{-1} \right)= \frac{\mathrm{i}}{\sqrt{2c}} \left(t^{-1}-\sqrt{t^{-2}-2c}\right).
\end{equation}
It follows that on the right side of the cut $[0,\sqrt{2c}]\mathrm{i}$, we have
\begin{equation}
\arg G\left(\mathrm{i}t \right)=\mathrm{arccot}\left(-\frac{\sqrt{2c-t^2}}{t} \right) \hspace{2mm} \mbox{and} \hspace{2mm} \arg G\left(-\mathrm{i}t^{-1} \right)=\frac{3 \pi}{2},
\end{equation}
where $\mathrm{arccot}$ takes values in $(0,\pi)$. From the above equation, on the right side of the cut $[0,\sqrt{2c}]\mathrm{i}$, we have
\begin{equation}
\mathrm{Im}\, g_{\xi,\xi} (\mathrm{i}t) = \frac{\pi}{2} -\xi\,\mathrm{arccot}\left(-\frac{\sqrt{2c-t^2}}{t} \right) +\xi \frac{3 \pi}{2}.
\end{equation}
It follows from the above equation that on the right side of the cut $[0,\sqrt{2c}]\mathrm{i}$, $\mathrm{Im}\, g_{\xi,\xi} (\mathrm{i}t)$ is decreasing for increasing $t \in(0,\sqrt{2c})$ with  
\begin{equation}
 \lim_{t \to 0^+}\mathrm{Im}\, g_{\xi,\xi} (\mathrm{i}t) = \frac{\pi}{2}(1+\xi) \hspace{2mm} \mbox{and} \hspace{2mm} \lim_{t \to \sqrt{2c}^-}\mathrm{Im}\, g_{\xi,\xi} (\mathrm{i}t) = \frac{\pi}{2}(1+2\xi),
\end{equation}
which verifies the first case.

For $\sqrt{2c}<t<1/\sqrt{2c}$, we have
\begin{equation}
\begin{split}
\arg G(\mathrm{i}t ) =\arg \left(\frac{\mathrm{i} \left(t-\sqrt{t^2-2c} \right)}{\sqrt{2c}}\right)    = \frac{\pi}{2} \hspace{2mm} \mbox{and} \hspace{2mm}\arg G(-\mathrm{i}t^{-1} ) =\arg \left(\frac{\mathrm{i} \left(t^{-1}-\sqrt{t^{-2}-2c} \right)}{\sqrt{2c}} \right)   = \frac{3\pi}{2},
\end{split}
\end{equation}
which verifies the second case.

For $t \in (0,\infty)$, notice that from the equation
\begin{equation}
G(\omega) = -\frac{\sqrt{2c}}{\omega+\sqrt{\omega^2+2c}} \hspace{5mm} \mbox{for } \omega \not =\pm \mathrm{i}\sqrt{2c},
\end{equation}
we have $\lim_{\delta \to 0^+} \arg G(t+\mathrm{i}\delta )=\pi$ and $\lim_{\delta \to 0^+} \arg G((t+\mathrm{i}\delta)^{-1} )=\pi$ which verifies the fourth case.

For the fifth case, set $\omega=\delta e^{\mathrm{i}\phi}$.  We have
\begin{equation}
\arg G(\omega) = \arg \left(-\frac{\sqrt{2c}}{\omega+\sqrt{\omega^2+2c}}\right) = \pi \hspace{5mm} \mbox{as } \delta \to 0^+
\end{equation}
and
\begin{equation}
\arg G(\omega^{-1}) = \arg \left(-\frac{\sqrt{2c}}{\omega^{-1}+\sqrt{\omega^{-2}+2c}}\right) = \pi+\phi \hspace{5mm} \mbox{as }\delta \to 0^+,
\end{equation}
which verifies the fifth case.
\end{proof}

\section{Formula Simplification} \label{section:Formula Simplification}
In this section, we find the reductions of the four variable generating function for $K^{-1}_{a,1}$  computed in~\cite{CY:13} which will eventually lead to Theorem~\ref{thm:Kinverse}.  We also give the proof of Corollary~\ref{corollary}.
This section is highly computational but it results in a dramatic simplification of the previous formula for $K^{-1}_{a,1}$.   Many of these computations were executed using computer algebra.  The section is organized as follows:
\begin{enumerate} 
\item We partially bring forward the previous formula for the four variable generating function for $K^{-1}_{a,1}$ given in~\cite{CY:13}. This is stated in Section~\ref{subsection:Previous}. We refer to~\cite{CY:13} for the omitted parts of the formula.
\item In Section~\ref{subsection:Symmetrization}, we find a symmetrized version of $K^{-1}_{a,1}$.  This results in a splitting of the generating function formula for $K^{-1}_{a,1}$ into essentially four different parts (up to symmetry) in Section~\ref{subsection:split}
\item We simplify the so-called boundary generating function in Section~\ref{subsection:Boundary} into much more manageable expressions. 
\item In Section~\ref{subsection:Double}, we extract coefficients of the separate parts which we were able to identify above. After many computations, these give double contour integral formulas. 
\item From the above computations, we are able to conclude  the proof of Theorem~\ref{thm:Kinverse} in Section~\ref{subsection:Completeproof}.
\item Finally, we prove Corollary~\ref{corollary} in Section~\ref{subsection:corollary}.
\end{enumerate}

\subsection{Previous formula} \label{subsection:Previous}
In order to describe the four variable generating function obtained in~\cite{CY:13} whose coefficients are entries of the matrix $K^{-1}_{a,1}$ we need the following terms:
\begin{equation}\begin{split}
        c^a_{\partial}(w_1,w_2)&=2(1+a^2) +a (w_1^2+w_1^{-2}) (w_2^2+w_2^{-2}),\\
     s^a_{i,0}(w_1,w_2)&=-a \left(w_1^{-2} w_2^{-2} + w_1^2 w_2^{-2} \right) -a \mathtt{i} w_1^2 +a \mathtt{i} w_1^{-2} -2 a^{2i}, \\
     s^a_{i,2n}(w_1,w_2)&= w_2^{2n} \left(-a \left( w_1^2 w_2^2 +w_1^{-2} w_2^2 \right) +a \mathtt{i} w_1^{2} - a \mathtt{i} w_1^{-2} -2a^{2(1-i)} \right)
     \end{split} 
\end{equation}
for $i \in \{0,1\}$. Note that in~\cite{CY:13}, $ c^a_{\partial}$ was denoted by $c_{\partial}$. We need the $a$ dependence for our later computations.
We set $f_n(x)=(1-x^{n})/(1-x)$ -- the sum of a geometric series and we also let for $n=4m$ and $\mathtt{w}=(w_1,w_2)$ and $\mathtt{b}=(b_1,b_2)$,
\begin{equation}
\begin{split}
 &d(\mathtt{w},\mathtt{b})=  f_{2m}(w_2^4b_2^4)f_{2m}(w_1^4b_1^4) b_2 w_1 (a+w_1^2 b_1^2+w_2^2b_2^2(1+a w_1^2b_1^2)) (w_2^2+b_1^2-\mathtt{i}(1+w_2^2b_1^2)) \\
 &+\left(\left( 1-\mathtt{i} w_2^2 \right) w_1^{-1}  b_2(1+a w_2^2 b_2^2)+\left( w_2^2-\mathtt{i} \right)  w_1^{2n+1}  b_1^{2n}b_2 (a+ w_2^2 b_2^2)\right) f_{2m}(w_2^4 b_2^4)\\
 & -\frac{d_{\mathrm{sides}}(\mathtt{w},\mathtt{b})}{c^a_{\partial}(b_1,b_2)} f_{2m}(w_1^4 b_1^4),
 \end{split}
 \end{equation}
where
\begin{equation}
 \begin{split}
d_{\mathrm{sides}} (\mathtt{w},\mathtt{b})&=  s^a_{0,0}(w_1,w_2)  \left(a(  b_1^2 - \mathtt{i})  w_1 b_2 + w_1  b_2^{-1}(1-\mathtt{i} b_1^2) \right)\\
&+s^a_{1,0}(w_1,w_2)  \left(  ( b_1^2 - \mathtt{i} ) w_1^3 b_1^2 b_2 + a w_1^3 b_1^2 b_2^{-1} (1-\mathtt{i} b_1^2 ) \right) \\ 
&+ s^a_{0,2n}(w_1,w_2) \left( (1-\mathtt{i} b_1^2) w_1b_2^{2n-1} +a w_1  b_2^{2n+1} (-\mathtt{i}+b_1^2)\right)\\
&+s^a_{1,2n}(w_1,w_2) \left( a(1 -\mathtt{i} b_1^2) w_1^3 b_1^2 b_2^{2n-1} +  w_1^3 b_1^2 b_2^{2n+1} (-\mathtt{i}+b_1^2) \right).
 \end{split}
\end{equation}

The generating function of $K_{a,b}^{-1}$ for $K_{a,b}$ defined in~\eqref{pf:K}, is given by 
\begin{equation}
      G(a,b,w_1,w_2,b_1,b_2) := \sum_{x \in \mathtt{W}} \sum_{y \in \mathtt{B}} K_{a,b}^{-1} (x,y) w_1^{x_1} w_2^{x_2} b_1^{y_1} b_2^{y_2}
\end{equation}
for $x=(x_1,x_2)$ and $y=(y_1,y_2)$.

\begin{thma}[\cite{CY:13}] \label{pf:mainthm}
For an Aztec diamond of size $n=4m$ with $m\in\mathbb{N}$, the generating function of $K_{a,1}^{-1}$ with variables $\mathtt{w}=(w_1,w_2)$ and $\mathtt{b}=(b_1,b_2)$ is given by
\begin{equation}
 \begin{split}\label{pf:mainthmeqn1}
     &G(a,1,w_1,w_2,b_1,b_2) =\frac{d(\mathtt{w},\mathtt{b})}{c^a_{\partial}(w_1,w_2)} \\& +       \left( \sum_{i,j\in \{0,1\}} \sum_{k,l \in \{0,2n\}} \sum_{(x_1,k) \in \mathtt{W}_i} \sum_{(l,y_2) \in \mathtt{B}_j}\frac{ s^a_{i,k}(w_1,w_2) s^a_{j,l}(b_2,b_1) K_{a,1}^{-1} ((x_1,k),(l,y_2))}{c^a_{\partial}(w_1,w_2)c^a_{\partial}(b_1,b_2)} w_1^{x_1} b_2^{y_2}\right), \\
    \end{split}
\end{equation}
where
\begin{equation}
\begin{split} \label{pf:mainthmeqn2}
      K_{a,1}^{-1}((x_1,0),(0,y_2)) &= L_{a,1}\left(\frac{x_1-1}{2},\frac{y_2-1}{2} \right), \\
      K_{a,1}^{-1}((x_1,0),(2n,y_2)) &=\mathtt{i}^{2n-1-x_1+y_2} L_{1,a}\left(n-\frac{x_1+1}{2},\frac{y_2-1}{2} \right), \\
      K_{a,1}^{-1}((x_1,2n),(0,y_2)) &= \mathtt{i}^{2n-1+x_1-y_2} L_{1,a}\left(\frac{x_1-1}{2},n-\frac{y_2+1}{2} \right)
 \hspace{2mm} \mathrm{and} \\
      K_{a,1}^{-1}((x_1,2n),(2n,y_2)) &= L_{a,1}\left(n-\frac{x_1+1}{2},n-\frac{y_2+1}{2} \right).
\end{split}
\end{equation}
and $L_{a,b}(i,j)$ is given below in Lemma~\ref{previousbgf}.

\end{thma}

In~\cite{CY:13}, the authors derive an expression to compute $L_{a,b}(i,j)$ and find that $L_{a,b}(i,j)$ are the signed entries of  the so-called \emph{boundary generating function}; see~\cite{CY:13} for a full description.
The next lemma partially brings forward the formula for $L_{a,b}(i,j)$ as stated in \cite[Lemma 6.1]{CY:13}.

\begin{lemma}[\cite{CY:13}] \label{previousbgf}
For an Aztec diamond of size $n=4m$ with $m\in\mathbb{N}$, let $L_{a,b}(i,j)$ denote $K_{a,b}^{-1}((2i+1,0),(0,2j+1))$.  We have 
 \begin{equation}
 \begin{split}
  L_{a,b}(i,j)&=  \frac{- \mathtt{i}^{i+j+1}}{(2\pi \mathtt{i})^2} \int_{|z|=1} \int_{|w|=1} \\
  &\times \sum_{r=0}^{m-1} \sum_{k,l\in \{0,1\} } \frac{ g^{2k+l+1}_{2[i]_2+[j]_2+1}(a,b,w,z)  \alpha_k^r(a,b,w) \alpha_l^r(a,b,z)}{ w^{\lfloor i/2 \rfloor +1} z^{\lfloor j/2 \rfloor+1}} dw \, dz ,
 \end{split}
 \end{equation}
where for $1 \leq p,q\leq 4$ we have $g_p^q(a,b,w,z)=\mathbf{N}_{p,q} (a,b,w,z)$ which is given in~\cite[Appendix A]{CY:13} and
  \begin{equation} \label{fortress tilings: alpha}
      \alpha_k^r(a,b,w) =  \frac{(\beta_0(a,b,w))^{2r} +(-1)^k ( \beta_1(a,b,w))^{2r}}{4 \sqrt{ ab}(a^2+b^2)(\sqrt{w} \sqrt{(b^4 + a^4) w + a^2 b^2 (1 + w^2)})^k} 
  \end{equation}
  for $k \in \{0,1\}$ with 
  \begin{equation} \label{fortress tilings: beta}
  \beta_l(a,b,w) = \frac{  \left( (a^2 + b^2) \sqrt{w} -(-1)^l  \sqrt{(b^4 + a^4) w + a^2 b^2 (1 + w^2)}\right)}{\sqrt{2ab(a^2+b^2)}}.
   \end{equation}
\end{lemma}

\subsection{Symmetrization} \label{subsection:Symmetrization}

In this subsection, we show that $K_{a,1}^{-1}$ given in Theorem~\ref{pf:mainthm} satisfies a symmetrization property.   These symmetrization formulas are consequences of reflections in the lines $(n,0)$ and $(0,n)$. 

Define 
\begin{equation}
\mathcal{T}_1( x,y)= (2n-x,y)\hspace{5mm}\mbox{and}\hspace{5mm}\mathcal{T}_2(x,y)=(x,2n-y),
\end{equation}
which are the reflections in the lines $(n,0)$ and $(0,n)$ respectively.  We have that 
\begin{equation}
\mathcal{T}_i:\mathtt{B}_j \to \mathtt{B}_{(i-1)(1-j)+(2-i)j} \hspace{5mm}\mbox{and}\hspace{5mm}\mathcal{T}_i:\mathtt{W}_j \to \mathtt{W}_{(i-1)j+(2-i)(1-j)}
\end{equation}
for $i\in \{1,2\}$ and $j\in \{0,1\}$ which is seen directly from the definitions. 

\begin{lemma} \label{sym:lemma:K}
For $r \in \{1,2\}$, we have for $\mathtt{b}=(b_1,b_2)\in \mathtt{B}$ and $\mathtt{w}=(w_1,w_2)\in \mathtt{W}$
\begin{equation}
K_{a,b}(\mathcal{T}_r \mathtt{b} ,\mathcal{T}_r \mathtt{w} ) = (-\mathrm{i}) (-1)^{\frac{w_1-b_1+w_2-b_2}{2}} K_{b,a}(\mathtt{b},\mathtt{w}).
\end{equation}
\end{lemma}

\begin{proof}
We first perform the calculation for $\mathcal{T}_1$.  We have that if $\mathcal{T}_1\mathtt{w}-\mathcal{T}_1\mathtt{b}=\pm e_j$, then $\mathtt{w} -\mathtt{b} = \pm e_{3-j}$ for $j \in \{1,2\}$ which can be seen by evaluating each case separately.  Using~\eqref{pf:K}, we write 
\begin{equation} 
\begin{split}
      K_{a,b}(\mathcal{T}_1\mathtt{b},\mathcal{T}_1\mathtt{w})&=\left\{ \begin{array}{ll}
                     a (1-j) + b j  & \mbox{if } \mathcal{T}_1\mathtt{w}=\mathcal{T}_1\mathtt{b}+e_1, \mathtt{b} \in \mathtt{B}_j \\
                     (a j +b (1-j) ) \mathrm{i} & \mbox{if } \mathcal{T}_1\mathtt{w}=\mathcal{T}_1\mathtt{b}+e_2, \mathtt{b} \in \mathtt{B}_j\\
                     a j + b (1-j)  & \mbox{if } \mathcal{T}_1\mathtt{w}=\mathcal{T}_1\mathtt{b}-e_1, \mathtt{b} \in \mathtt{B}_j \\
                     (a (1-j) +b j ) \mathrm{i} & \mbox{if } \mathcal{T}_1\mathtt{w}=\mathcal{T}_1\mathtt{b}-e_2, \mathtt{b} \in \mathtt{B}_j
        \end{array} \right.\\
&=-(-1)^{(w_1-b_1+w_2-b_2)/2} \mathrm{i} K_{b,a}(\mathtt{b},\mathtt{w}).
\end{split}
\end{equation}

The formula for $K_{a,b}(\mathcal{T}_2\mathtt{b},\mathcal{T}_2\mathtt{w})$ is verified in a similar way as given above.


\end{proof}

Using the previous lemma, we are able to write formulas for $K^{-1}_{a,b}$ under the reflections $\mathcal{T}_1$, $\mathcal{T}_2$ and $\mathcal{T}_2 \mathcal{T}_1$.
\begin{lemma} \label{sym:lemma:Kinv}
For $r \in \{1,2\}$
\begin{equation} \label{sym:lemma:Kinveqn1}
	K_{a,b}^{-1} (\mathtt{w},\mathtt{b}) = (-1)^{\frac{w_1+w_2-b_1-b_2}{2} }(-\mathrm{i}) K_{b,a}^{-1} (\mathcal{T}_r \mathtt{w}, \mathcal{T}_r \mathtt{b}) 
\end{equation}
and 
\begin{equation} \label{sym:lemma:Kinveqn2}
K_{a,b}^{-1} (\mathtt{w},\mathtt{b}) =-(-1)^{w_2-b_2} K_{a,b}^{-1} (\mathcal{T}_2 \mathcal{T}_1 \mathtt{w}, \mathcal{T}_2 \mathcal{T}_1 \mathtt{b} ).
\end{equation}
\end{lemma}

\begin{proof}
By definition we have that for $\mathtt{b}, \mathtt{b}' \in \mathtt{B}$
\begin{equation}
	\sum_{\mathtt{w} \in \mathtt{W}} K_{a,b}(\mathtt{b}',\mathtt{w}) K_{a,b}^{-1} (\mathtt{w}, \mathtt{b}) = \delta_{\mathtt{b}', \mathtt{b}}.
\end{equation}
We make the change $\mathtt{w} \mapsto \mathcal{T}_r \mathtt{w}$ which gives
\begin{equation}
	\sum_{\mathcal{T}_r\mathtt{w} \in \mathtt{W}} K_{a,b}(\mathtt{b}',\mathcal{T}_r\mathtt{w}) K_{a,b}^{-1} (\mathcal{T}_r\mathtt{w}, \mathtt{b}) = \delta_{\mathtt{b}', \mathtt{b}}
\end{equation}
but the sum is over the same set, so we write
\begin{equation}
	\sum_{\mathtt{w} \in \mathtt{W}} K_{a,b}(\mathtt{b}',\mathcal{T}_r\mathtt{w}) K_{a,b}^{-1} (\mathcal{T}_r\mathtt{w}, \mathtt{b}) = \delta_{\mathtt{b}', \mathtt{b}}.
\end{equation}
We make the change $\mathtt{b} \mapsto \mathcal{T}_r \mathtt{b}$ and $\mathtt{b}' \mapsto \mathcal{T}_r \mathtt{b}'$ and interchange $a$ and $b$ which gives
\begin{equation}
		\sum_{\mathtt{w} \in \mathtt{W}} K_{b,a}(\mathcal{T}_r\mathtt{b}',\mathcal{T}_r\mathtt{w}) K_{b,a}^{-1} (\mathcal{T}_r\mathtt{w}, \mathcal{T}_r\mathtt{b}) = \delta_{\mathcal{T}_r\mathtt{b}', \mathcal{T}_r\mathtt{b}}=\delta_{\mathtt{b}',\mathtt{b}}.
\end{equation}
We now use Lemma~\ref{sym:lemma:K} which gives
\begin{equation}
		\sum_{\mathtt{w} \in \mathtt{W}} K_{a,b}(\mathtt{b}',\mathtt{w})(-\mathrm{i})(-1)^{\frac{w_1-b_1'+w_2-b_2'}{2}} K_{b,a}^{-1} (\mathcal{T}_r\mathtt{w}, \mathcal{T}_r\mathtt{b})=\delta_{\mathtt{b}',\mathtt{b}}.
\end{equation}
We can multiply both sides of the above equation by $(-1)^{(b_1'+b_2')/2}$.  By using the fact  that $\delta_{\mathtt{b}',\mathtt{b}} (-1)^{(b_1'+b_2')/2} = \delta_{\mathtt{b}',\mathtt{b}} (-1)^{(b_1+b_2)/2}$ the above equation becomes
\begin{equation}
		\sum_{\mathtt{w} \in \mathtt{W}} K_{a,b}(\mathtt{b}',\mathtt{w})(-\mathrm{i})(-1)^{\frac{w_1-b_1+w_2-b_2}{2}} K_{b,a}^{-1} (\mathcal{T}_r\mathtt{w}, \mathcal{T}_r\mathtt{b}) =\delta_{\mathtt{b}',\mathtt{b}}
\end{equation} 
for all $\mathtt{b}, \mathtt{b}' \in \mathtt{B}$.  Since the inverse of a matrix is unique, we obtain~\eqref{sym:lemma:Kinveqn1}.
  
To obtain~\eqref{sym:lemma:Kinveqn2}, by using~\eqref{sym:lemma:Kinveqn1} repeatedly, we find that
\begin{equation}
\begin{split}
K_{a,b}^{-1}(\mathtt{w},\mathtt{b}) &=-\mathrm{i}(-1)^{\frac{w_1+w_2-b_1-b_2}{2}} K_{b,a}^{-1}(\mathcal{T}_1\mathtt{w},\mathcal{T}_1\mathtt{b}) \\
&= - (-1)^{\frac{w_1+w_2-b_1-b_2}{2}}(-1) ^{\frac{(\mathcal{T}_1\mathtt{w})_1+(\mathcal{T}_1\mathtt{w})_2-(\mathcal{T}_1\mathtt{b})_1-(\mathcal{T}_1\mathtt{b})_2}{2}}K_{b,a}^{-1}(\mathcal{T}_2\mathcal{T}_1\mathtt{w},\mathcal{T}_2\mathcal{T}_1\mathtt{b})\\
&= - (-1)^{\frac{w_1+w_2-b_1-b_2}{2}}(-1) ^{\frac{2n-w_1+w_2-(2n-b_1)-b_2}{2}}K_{b,a}^{-1}(\mathcal{T}_2\mathcal{T}_1\mathtt{w},\mathcal{T}_2\mathcal{T}_1\mathtt{b}),
\end{split}
\end{equation}
which gives~\eqref{sym:lemma:Kinveqn2}.

\end{proof}

A consequence of Lemma~\ref{sym:lemma:Kinv} is that there is a \emph{symmetrization formula} for $K^{-1}$
\begin{equation}
\begin{split}
4 K^{-1}_{a,b}(\mathtt{w}, \mathtt{b})& = K^{-1}_{a,b} (\mathtt{w},\mathtt{b}) -(-1)^{w_2-b_2} K_{a,b}^{-1} (\mathcal{T}_2 \mathcal{T}_1 \mathtt{w}, \mathcal{T}_2 \mathcal{T}_1 \mathtt{b})\\
&-\mathrm{i} (-1)^{\frac{w_1 +w_2-b_1 -b_2}{2}} \left( K^{-1}_{a,b} (\mathcal{T}_1\mathtt{w},\mathcal{T}_1 \mathtt{b}) +K^{-1}_{a,b} (\mathcal{T}_2 \mathtt{w},\mathcal{T}_2\mathtt{b}) \right),
\end{split}
\end{equation}
although this formula is not used in our analysis.

\subsection{Formula Splitting}\label{subsection:split}

We split the generating function into smaller parts and use the symmetrization formulas found in Section~\ref{subsection:Symmetrization} to deduce a simpler formula for $G(a,1,w_1,w_2,b_1,b_2)$.  This leads to a simpler form of the generating function which is more amenable to analysis.

Let
\begin{equation} \label{split:def:d1}
d_1(a,w_1,w_2,b_1,b_2)= \frac{b_2 w_1 \left(a+b_1^2 w_1^2+b_2^2 \left(1+a b_1^2 w_1^2\right) w_2^2\right) \left(b_1^2+w_2^2-\mathrm{i} \left(1+b_1^2 w_2^2\right)\right)}{\left(1-b_1^4 w_1^4\right) \left(1-b_2^4 w_2^4\right) c^a_{\partial}(w_1,w_2) },
\end{equation}
\begin{equation}\label{split:def:d2}
d_2(a,w_1,w_2,b_1,b_2)=\frac{b_2 \left(1-\mathrm{i} w_2^2\right) \left(1+a b_2^2 w_2^2\right)}{w_1 \left(1-b_2^4 w_2^4\right) c^a_{\partial}(w_1,w_2)},
\end{equation}
\begin{equation}\label{split:def:d3}
\begin{split}
&d_3(a,w_1,w_2,b_1,b_2)=\\
&\frac{s_{0,0}^a\left(w_1,w_2\right) \left(\frac{\left(1-\mathrm{i} b_1^2\right) w_1}{b_2}+a \left(-\mathrm{i}+b_1^2\right) b_2 w_1\right)+s_{1,0}^a\left(w_1,w_2\right) \left(\frac{a b_1^2 \left(1-\mathrm{i} b_1^2\right) w_1^3}{b_2}+b_1^2 \left(-\mathrm{i}+b_1^2\right) b_2 w_1^3\right)}{(1-b_1^4 w_1^4)c^a_{\partial}(w_1,w_2) c^a_{\partial}(b_1,b_2)}
\end{split}
\end{equation}
and 
\begin{equation}
H(a,w_1,w_2,b_1,b_2) = \sum_{i,j\in \{0,1\}} \sum_{ \substack{(x_1,0) \in \mathtt{W}_i\\ (0,y_2)  \in \mathtt{B}_j}} \frac{ s_{i,0}^a(w_1,w_2) s_{j,0}^a(b_2,b_1)}{c^a_{\partial}(w_1,w_2) c^a_{\partial}(b_1,b_2)} K^{-1}_{a,1}((x_1,0),(0,y_2)) w_1^{x_1} b_2^{y_2}.
\end{equation}
Define
\begin{equation} \label{sym:def:G}
	G^0(a,w_1,w_2,b_1,b_2)= (d_1+d_2-d_3 +H)(a,w_1,w_2,b_1,b_2).
\end{equation}
Define the operation $\mathcal{S}$ which acts on four variable functions with parameters of the form $g(a,w_1,w_2,b_1,b_2)$ by
\begin{equation}
\begin{split}
&\mathcal{S} g(a,w_1,w_2,b_1,b_2)=g(a,w_1,w_2,b_1,b_2) +\frac{ \mathrm{i}}{a} (w_1b_1)^{2n} g(a^{-1},-\mathrm{i} w_1^{-1},\mathrm{i} w_2,-\mathrm{i}b_1^{-1}, \mathrm{i}b_2 )\\
& +\frac{ \mathrm{i}}{a} (w_2 b_2)^{2n} g(a^{-1}, \mathrm{i}  w_1 , -\mathrm{i}w_2^{-1}, \mathrm{i}b_1,  -\mathrm{i} b_2^{-1} )-(w_1w_2b_1b_2)^{2n} g(a,w_1^{-1} ,-w_2^{-1} ,b_1^{-1},-b_2^{-1}).
\end{split}
\end{equation}

We have that
\begin{lemma} \label{sym:lemma:G01}
\begin{equation} \label{sym:lemma:G01eqn}
G(a,1,w_1,w_2,b_1,b_2) =\mathcal{S} G^0(a,w_1,w_2,b_1,b_2).
\end{equation}
\end{lemma}

\begin{proof}
We compute the right side of~\eqref{sym:lemma:G01eqn} for each of the four terms that define $G^0(a,w_1,w_2,b_1,b_2)$. 
Before doing so, notice that
\begin{equation} \label{sym:lemproof:G0eqn1}
	a^2 s_{i,0}^{1/a} (- \mathrm{i} w_1^{-1},\mathrm{i}w_2)= s_{1-i,0}^a (w_1,w_2),
\end{equation}
\begin{equation}\label{sym:lemproof:G0eqn2}
	b_1^{2n} a^2 s_{i,0}^{1/a}(\mathrm{i}b_2,-\mathrm{i} b_1^{-1}) =s_{i,2n}^a (b_2,b_1)
\end{equation}
and
\begin{equation}\label{sym:lemproof:G0eqn3}
	w_2^{2n} s_{i,0}^a ( w_1^{-1} ,-w_2^{-1})=s_{1-i,2n}^a (w_1,w_2),
\end{equation}	
which follow from direct computations for $i \in \{0,1\}$. We also note that
\begin{equation}\label{sym:lemproof:cchange}
	c_{\partial}^{1/a} (w_1,w_2) = 1/a^2 c^a_{\partial} (w_1,w_2).
\end{equation}
We now compute each term of $G^0(a,w_1,w_2,b_1,b_2)$ when substituted into the right side of~\eqref{sym:lemma:G01eqn}.  Since the computations are quite long, we expand out each term separately and where appropriate, we have used computer algebra to help with these formulas.

For $d_1(a,w_1,w_2,b_1,b_2)$, we have using~\eqref{sym:lemproof:cchange}
\begin{equation}
\begin{split} \label{sym:lemma:G01:d11}
&\frac{ \mathrm{i}}{a} (w_1b_1)^{2n} d_1(a^{-1},-\mathrm{i} w_1^{-1},\mathrm{i} w_2,-\mathrm{i}b_1^{-1}, \mathrm{i}b_2 )\\
&=\frac{\mathrm{i}b_2 b_1^{2 n} \left(b_1^2 \left(w_2^2+\mathrm{i}\right)+\mathrm{i}w_2^2+1\right) w_1^{2 n+1} \left(a b_1^2 b_2^2 w_2^2 w_1^2+a+b_1^2 w_1^2+b_2^2 w_2^2\right)}{\left(b_1^4 w_1^4-1\right) \left(b_2^4 w_2^4-1\right)c^a_{\partial}(w_1,w_2)},
\end{split}
\end{equation}

\begin{equation}
\begin{split}\label{sym:lemma:G01:d12}
&\frac{ \mathrm{i}}{a} (w_2 b_2)^{2n} d_1(a^{-1}, \mathrm{i}  w_1 , -\mathrm{i}w_2^{-1}, \mathrm{i}b_1,  -\mathrm{i} b_2^{-1} )\\
&= \frac{\mathrm{i}w_1 b_2^{2 n+1} \left(b_1^2 \left(w_2^2+\mathrm{i}\right)+\mathrm{i}w_2^2+1\right) w_2^{2 n} \left(a b_1^2 b_2^2 w_2^2 w_1^2+a+b_1^2 w_1^2+b_2^2 w_2^2\right)}{\left(b_1^4 w_1^4-1\right) \left(b_2^4 w_2^4-1\right)c^a_{\partial}(w_1,w_2)}
\end{split}
\end{equation}
and
\begin{equation}
\begin{split}\label{sym:lemma:G01:d13}
&-(w_1w_2b_1b_2)^{2n} d_1(a,w_1^{-1} ,-w_2^{-1} ,b_1^{-1},-b_2^{-1})\\
&=\frac{b_1^{2 n} b_2^{2 n+1} \left(b_1^2 \left(1-\mathrm{i}w_2^2\right)+w_2^2-\mathrm{i}\right) w_1^{2 n+1} w_2^{2 n} \left(a b_1^2 b_2^2 w_2^2 w_1^2+a+b_1^2 w_1^2+b_2^2 w_2^2\right)}{\left(b_1^4 w_1^4-1\right) \left(b_2^4 w_2^4-1\right)c^a_{\partial}(w_1,w_2)}.
\end{split}
\end{equation}
Summing up~\eqref{sym:lemma:G01:d11}, ~\eqref{sym:lemma:G01:d12}, ~\eqref{sym:lemma:G01:d13} and $d_1(a,w_1,w_2,b_1,b_2)$ we find $\mathcal{S} d_1(a,w_1,w_2,b_1,b_2)$ is exactly equal to the first term in $d(\mathtt{w},\mathtt{b})/c^a_{\partial}(w_1,w_2)$.    

For $d_2(a,w_1,w_2,b_1,b_2)$, we have using~\eqref{sym:lemproof:cchange}
\begin{equation}
\begin{split} \label{sym:lemma:G01:d21}
&\frac{ \mathrm{i}}{a} (w_1b_1)^{2n} d_2(a^{-1},-\mathrm{i} w_1^{-1},\mathrm{i} w_2,-\mathrm{i}b_1^{-1}, \mathrm{i}b_2 )
=-\frac{b_2 \left(w_2^2-\mathrm{i}\right) b_1^{2 n} w_1^{2 n+1} \left(a+b_2^2 w_2^2\right)}{(b_2^4 w_2^4-1)c^a_{\partial}(w_1,w_2)},
\end{split}
\end{equation}

\begin{equation}
\begin{split}\label{sym:lemma:G01:d22}
&\frac{ \mathrm{i}}{a} (w_2 b_2)^{2n} d_2(a^{-1}, \mathrm{i}  w_1 , -\mathrm{i}w_2^{-1}, \mathrm{i}b_1,  -\mathrm{i} b_2^{-1} )
=\frac{\left(1-\mathrm{i}w_2^2\right) b_2^{2 n+1} w_2^{2 n} \left(a b_2^2 w_2^2+1\right)}{w_1 \left(b_2^4 w_2^4-1\right)c^a_{\partial}(w_1,w_2)} 
\end{split}
\end{equation}
and
\begin{equation}
\begin{split}\label{sym:lemma:G01:d23}
&-(w_1w_2b_1b_2)^{2n} d_2(a,w_1^{-1} ,-w_2^{-1} ,b_1^{-1},-b_2^{-1})
=\frac{\left(w_2^2-\mathrm{i}\right) b_1^{2 n} b_2^{2 n+1} w_1^{2 n+1} w_2^{2 n} \left(a+b_2^2 w_2^2\right)}{(b_2^4 w_2^4-1)c^a_{\partial}(w_1,w_2)}.
\end{split}
\end{equation}
Summing up~\eqref{sym:lemma:G01:d21}, ~\eqref{sym:lemma:G01:d22}, ~\eqref{sym:lemma:G01:d23} and $d_2(a,w_1,w_2,b_1,b_2)$, we find $\mathcal{S}d_2(a,w_1,w_2,b_1,b_2)$ is  exactly equal to  the second term in $d(\mathtt{w},\mathtt{b})/c^a_{\partial}(w_1,w_2)$.    

For $d_3(a,w_1,w_2,b_1,b_2)$, by using~\eqref{sym:lemproof:G0eqn1} and~\eqref{sym:lemproof:cchange}
\begin{equation}
\begin{split} \label{sym:lemma:G01:d31}
&\frac{\mathrm{i}}{a} w_1^{2n}b_1^{2n} d_3(a^{-1} ,-\mathrm{i}  w_1^{-1} , \mathrm{i}w_2,- \mathrm{i}b_1^{-1},  \mathrm{i} b_2)=-w_1^{2n} b_1^{2n}d_3(a,w_1,w_2,b_1,b_2),
\end{split}
\end{equation}
which is seen by expanding out $d_3$.  By using~\eqref{sym:lemproof:G0eqn2} and~\eqref{sym:lemproof:cchange} we also  have
\begin{equation}
\begin{split}\label{sym:lemma:G01:d32}
&\frac{ \mathrm{i}}{a} (w_2 b_2)^{2n} d_3(a^{-1}, \mathrm{i}  w_1 , -\mathrm{i}w_2^{-1}, \mathrm{i}b_1,  -\mathrm{i} b_2^{-1} )\\
&=b_1^{2n}\frac{a \mathrm{i} \left( -s_{0,2n}^a(w_1,w_2)\left(w_1b_2(1+\mathrm{i}b_1^2)+\frac{w_1}{ab_2}(b_1^2+\mathrm{i} ) \right)-s_{1,2n}(w_1,w_2)\left(\frac{w_1^3b_1^2 b_2}{a} (1+b_1^2 \mathrm{i})+b_1^2(\mathrm{i}+b_1^2)\frac{w_1^3}{b_2} \right) \right)}{(1-w_1^4b_1^4)c^a_{\partial}(w_1,w_2) c^a_{\partial}(b_1,b_2)}\\
&=b_1^{2n}\frac{s_{0,2n}^a(w_1,w_2)\left(a w_1b_2(b_1^2-\mathrm{i})+\frac{w_1}{b_2}(1- b_1^2\mathrm{i} ) \right)+s_{1,2n}(w_1,w_2)\left(w_1^3b_1^2 b_2 (b_1^2 -\mathrm{i})+ab_1^2(1-\mathrm{i}b_1^2)\frac{w_1^3}{b_2} \right) }{(1-w_1^4b_1^4)c^a_{\partial}(w_1,w_2) c^a_{\partial}(b_1,b_2)}.
\end{split}
\end{equation}
Using~\eqref{sym:lemproof:G0eqn3} and~\eqref{sym:lemproof:cchange}, we have
\begin{equation}
\begin{split}\label{sym:lemma:G01:d33}
&- (w_1 b_1 w_2 b_2)^{2n} d_3(a,  w_1^{-1} , -w_2^{-1}, b_1^{-1},  - b_2^{-1} )\\
&= w_1^{2n+4} b_1^{2n+4} b_2^{2n} \frac{  s_{1,2n}^a(w_1,w_2) \left((1-\mathrm{i}b_1^{-2})\frac{b_2}{w_1} + \frac{a (b_1^{-2}-\mathrm{i})}{b_2w_1}\right)+s_{0,2n}^a(w_1,w_2) \left( a b_1^{-2}(1-\mathrm{i}b_1^{-2}) \frac{ b_2}{w_1^3}+\frac{b_1^{-2}(b_1^{-2}-\mathrm{i})}{ b_2w_1^{3}} \right)}{(1-w_1^4 b_1^4)c^a_{\partial}(w_1,w_2) c^a_{\partial}(b_1,b_2)}  \\
&=w_1^{2n}b_1^{2n}b_2^{2n}\frac{  s_{1,2n}^a(w_1,w_2) \left((b_1^2-\mathrm{i}){b_1^2b_2}{w_1^3} + \frac{a (1-\mathrm{i}b_1^{2})w_1^3b_1^2}{b_2}\right)+s_{0,{2n}}^a(w_1,w_2) \left( a (b_1^2-\mathrm{i}) {w_1 b_2}+\frac{(1-\mathrm{i}b_1^{2})w_1}{ b_2} \right)}{(1-w_1^4 b_1^4)c^a_{\partial}(w_1,w_2) c^a_{\partial}(b_1,b_2)}.  \\%
%
\end{split}
\end{equation}
Summing up~\eqref{sym:lemma:G01:d31}, ~\eqref{sym:lemma:G01:d32}, ~\eqref{sym:lemma:G01:d33} and $d_3(a,w_1,w_2,b_1,b_2)$ gives exactly 
$$
\mathcal{S}d_3(a,w_1,w_2,b_1,b_2)=\frac{d_{\mathrm{sides}}(\mathtt{w},\mathtt{b})}{c^a_{\partial}(w_1,w_2) c^a_{\partial} (b_1,b_2)} f_{2m}(w_1b_1).
$$

For the term $H$, we must first rewrite  $K^{-1}_{1,a}(x,y)$ as  $K^{-1}_{1/a,1}(x,y)$ by applying gauge transformations.  Indeed, we multiply all the vertices encoded by $K^{-1}_{1/a}(x,y)$ by $a$.  Since $K^{-1}_{1/a,1}(x,y)$ is encoded by a signed count of dimer coverings of the Aztec diamond graph with  $x$ and $y$ removed divided by the total number of dimer coverings of the Aztec diamond graph, we have the relation
\begin{equation}
K^{-1}_{1/a,1} = \frac{a^{-(n(n+1)-1)}}{a^{n(n+1)}}K^{-1}_{1,a} ,
\end{equation}
which leads to the equation
\begin{equation}
	K^{-1}_{1,a} (x,y) =\frac{1}{a} K^{-1}_{1/a,1}(x,y) 
\end{equation}
for $x\in \mathtt{W}$ and $y \in \mathtt{B}$.  Using the above gauge transformation,~\eqref{sym:lemproof:G0eqn1} and~\eqref{sym:lemproof:G0eqn2}, we have
\begin{equation}
\begin{split}
&\frac{\mathrm{i}}{a} (w_1b_1)^{2n} H \left( a^{-1},\mathrm{i} w_1^{-1} , \mathrm{i} w_2, \mathrm{i} b_1^{-1} , \mathrm{i} b_2 \right)\\
&=\mathrm{i}w_1^{2n}\sum_{i,j \in \{0,1\}} \sum_{\substack{ (x_1,0)\in \mathtt{W}_i \\ (0,y_2) \in \mathtt{B}_j}} \frac{1}{a} K^{-1}_{1/a,1} ((x_1,0),(0,y_2)) \left(- \mathrm{i}w_1^{-1} \right)^{x_1} \left( \mathrm{i}b_2 \right)^{y_2} \frac{s_{1-i,0}^a(w_1,w_2) s_{j,2n}^a(b_2,b_1)}{c^a_{\partial}(w_1,w_2 )c^a_{\partial}(b_1,b_2) } \\
&=\mathrm{i}w_1^{2n}\sum_{i,j \in \{0,1\}} \sum_{\substack{ (x_1,0)\in \mathtt{W}_i \\ (0,y_2) \in \mathtt{B}_j}}  K^{-1}_{1,a} ((x_1,0),(0,y_2)) \left(- \mathrm{i}w_1^{-1} \right)^{x_1} \left( \mathrm{i}b_2 \right)^{y_2} \frac{s_{1-i,0}^a(w_1,w_2) s_{j,2n}^a(b_2,b_1)}{c^a_{\partial}(w_1,w_2) c^a_{\partial}(b_1,b_2) } \\
&=w_1^{2n}\sum_{i,j \in \{0,1\}} \sum_{\substack{ (x_1,0)\in \mathtt{W}_i \\ (0,y_2) \in \mathtt{B}_j}}  K^{-1}_{a,1} (\mathcal{T}_1(x_1,0),\mathcal{T}_1(0,y_2))\left(w_1^{-1} \right)^{x_1}  b_2 ^{y_2} \frac{s_{1-i,0}^a(w_1,w_2) s_{j,2n}^a(b_2,b_1)}{c^a_{\partial}(w_1,w_2) c^a_{\partial}(b_1,b_2) }, \\
\end{split}
\end{equation}
where the last line follows from using Lemma~\ref{sym:lemma:Kinv}. We can change the sum over $(0,y_2)$ to $(2n,y_2)$ which gives
\begin{equation} \label{sym:lemproof:G0H1}
\begin{split}
&\frac{\mathrm{i}}{a} (w_1b_1)^{2n} H \left( a^{-1},\mathrm{i} w_1^{-1} , \mathrm{i} w_2, \mathrm{i} b_1^{-1} , \mathrm{i} b_2 \right)\\
&=w_1^{2n}\sum_{i,j \in \{0,1\}} \sum_{\substack{ (x_1,0)\in \mathtt{W}_i \\ (2n,y_2) \in \mathtt{B}_j}}  K^{-1}_{a,1} ((2n-x_1,0),(2n,y_2))\left(w_1^{-1} \right)^{x_1} b_2 ^{y_2} \frac{s_{1-i,0}^a(w_1,w_2) s_{j,2n}^a(b_2,b_1)}{c^a_{\partial}(w_1,w_2) c^a_{\partial}(b_1,b_2) } \\
&=\sum_{i,j \in \{0,1\}} \sum_{\substack{ (x_1,0)\in \mathtt{W}_i \\ (2n,y_2) \in \mathtt{B}_j}}  K^{-1}_{a,1} ((x_1,0),(2n,y_2))w_1 ^{2n-x_1} b_2 ^{y_2} \frac{s_{i,0}^a(w_1,w_2) s_{j,2n}^a(b_2,b_1)}{c^a_{\partial}(w_1,w_2) c^a_{\partial}(b_1,b_2) } \\
\end{split}
\end{equation}
by reversing the sum respect to $(x_1,0)$.  The above equation is exactly equal to one of the terms in the $G(a,1,w_1,w_2,b_1,b_2)$ given in Theorem~\ref{pf:mainthm}.

We can perform a very similar calculation for $\mathrm{i}/a (w_2 b_2)^{2n} H(a^{-1}, \mathrm{i} w_1, -\mathrm{i}w_2^{-1}, \mathrm{i} b_1 , -\mathrm{i} b_2^{-1})$ as given above. We find that
\begin{equation} \label{sym:lemproof:G0H2}
\begin{split}
&\frac{\mathrm{i}}{a} (w_2 b_2)^{2n} H\left(a^{-1}, \mathrm{i} w_1, -\mathrm{i}w_2^{-1}, \mathrm{i} b_1 , -\mathrm{i} b_2^{-1}\right)\\
&=\sum_{i,j \in \{0,1\}} \sum_{\substack{ (x_1,2n)\in \mathtt{W}_i \\ (0,y_2) \in \mathtt{B}_j}}  K^{-1}_{a,1} ((x_1,2n),(0,y_2))w_1 ^{x_1}  b_2 ^{2n-y_2} \frac{s_{i,2n}^a(w_1,w_2) s_{j,0}^a(b_2,b_1)}{c^a_{\partial}(w_1,w_2) c^a_{\partial}(b_1,b_2) }. \\
\end{split}
\end{equation}
Finally, we consider $-(w_1 w_2 b_1 b_2)^{2n}H(a,w_1^{-1},-w_2^{-1} ,b_1^{-1},-b_2^{-1})$.  Using~\eqref{sym:lemproof:G0eqn3}, we find that 
\begin{equation}
\begin{split}
&- (w_1 w_2 b_1 b_2)^{2n} H\left(a,w_1^{-1} ,-w_2^{-1},b_1^{-1},-b_2^{-1} \right) \\
&=-\sum_{i,j \in \{0,1\}} \sum_{\substack{(x_1,0)\in\mathtt{W}_i \\ (0,y_2) \in \mathtt{B}_j}} K^{-1}_{a,1} \left( (x_1,0),(0,y_2)\right) w_1^{-x_1} b_2^{-y_2} (-1)^{y_2} \frac{s_{1-i,2n}^a(w_1,w_2) s_{1-j,2n}^a (b_2,b_1)}{c^a_{\partial}(w_1,w_2) c^a_{\partial}(b_1,b_2)} (w_1 b_2)^{2n}\\
&=\sum_{i,j \in \{0,1\}} \sum_{\substack{(x_1,0)\in\mathtt{W}_i \\ (0,y_2) \in \mathtt{B}_j}} K^{-1}_{a,1} \left(\mathcal{T}_1\mathcal{T}_2 (x_1,0),\mathcal{T}_1 \mathcal{T}_2(0,y_2)\right) w_1^{2n-x_1} b_2^{2n-y_2}  \frac{s_{1-i,2n}^a(w_1,w_2) s_{1-j,2n}^a (b_2,b_1)}{c^a_{\partial}(w_1,w_2) c^a_{\partial}(b_1,b_2)} 
\end{split}
\end{equation}
by using Lemma~\ref{sym:lemma:Kinv}.  We now reverse the sums and so we obtain
\begin{equation}
\begin{split}\label{sym:lemproof:G0H3}
&- (w_1 w_2 b_1 b_2)^{2n} H\left(a,w_1^{-1} ,-w_2^{-1},b_1^{-1},-b_2^{-1} \right) \\
&=\sum_{i,j \in \{0,1\}} \sum_{\substack{(x_1,2n)\in\mathtt{W}_i \\ (2n,y_2) \in \mathtt{B}_j}} K^{-1}_{a,1} \left( (x_1,2n),(2n,y_2)\right) w_1^{x_1} b_2^{y_2}  \frac{s_{i,2n}^a(w_1,w_2) s_{1,2n}^a (b_2,b_1)}{c^a_{\partial}(w_1,w_2) c^a_{\partial}(b_1,b_2)} .
\end{split}
\end{equation}
We now use~\eqref{sym:lemproof:G0H1},~\eqref{sym:lemproof:G0H2} and~\eqref{sym:lemproof:G0H3} to deduce that
\begin{equation}
\begin{split}
&\mathcal{S}H(a,w_1,w_2,b_1,b_2)\\
&=\sum_{i,j\in \{0,1\}} \sum_{k,l \in \{0,2n\}} \sum_{(x_1,k) \in \mathtt{W}_i} \sum_{(l,y_2) \in \mathtt{B}_j}\frac{ s^a_{i,k}(w_1,w_2) s^a_{j,l}(b_2,b_1) K_{a,1}^{-1} ((x_1,k),(l,y_2))}{c^a_{\partial}(w_1,w_2)c^a_{\partial}(b_1,b_2)} w_1^{x_1} b_2^{y_2}.
\end{split}
\end{equation}

\end{proof}

\begin{remark}
The above lemma suggests that $G(a,b,w_1,w_2,b_1,b_2)$ has a symmetrization property,  that is, one can show
\begin{equation} 
\begin{split}
&4G(a,b,w_1,w_2,b_1,b_2) =G(a,b,w_1,w_2,b_1,b_2) + \mathrm{i} (w_1b_1)^{2n} G(b,a,-\mathrm{i} w_1^{-1},\mathrm{i} w_2,-\mathrm{i}b_1^{-1}, \mathrm{i}b_2 )\\
& + \mathrm{i} (w_2 b_2)^{2n} G(b,a, \mathrm{i}  w_1 , -\mathrm{i}w_2^{-1}, \mathrm{i}b_1,  -\mathrm{i} b_2^{-1} )-(w_1w_2b_1b_2)^{2n} G(a,b,w_1^{-1} ,-w_2^{-1} ,b_1^{-1},-b_2^{-1}).
\end{split}
\end{equation}
The above formula follows from noting the symmetries of $K^{-1}_{a,b}$. Since we do not use this formula in the rest of the paper, we will not provide the details of this assertion.
\end{remark}


\subsection{Boundary Generating Function Simplification} \label{subsection:Boundary}

In this subsection, we consider the boundary generating function of the two-periodic Aztec diamond of size $n=4m$ parameterized by $a$ and $b$.  For this subsection only, we present our computations and results for general $b$ . Since symmetrization transformations of $G^0$ give $G$ from Section~\ref{subsection:Symmetrization}, we only need to consider the boundary generating function at one corner.  We choose the corner with white vertices $(x_1,0)$ and black vertices $(0,y_2)$ for $x_1,y_2 \in2 \mathbb{Z}+1$ and $1\leq x_1,y_1 \leq 2n-1$. 
 Most of the computations have used (`brute force') computer algebra, that is, we have large expressions which eventually simplify to expressions which are very reasonable.   Below, we summarize these computer algebra computations.  Rather than subject the reader to these long intricate computations, we will only verify that our obtained expressions are correct.  The motivation behind these simplifications was the appearance of a formula related to $\tilde{c}(u,v)$, defined in~\eqref{res:eq:charc}. 
 
We start by recalling the boundary generating function given in~\cite{CY:13}, using the same definitions given in that paper, unless mentioned specifically otherwise.  We set $\mathbf{G}(a,b,x,y,z)$ to be the boundary generating function which is defined as
\begin{equation}
\mathbf{G}(a,b,x,y,z) = \sum_{n=0}^{\infty} \sum_{j=0}^{n-1} \sum_{i=0}^{n-1} |K^{-1}_{a,b} ((2i+1,0),(0,2j+1))| x^i y^j z^n,
\end{equation}
where $K^{-1}_{a,b} ((2i+1,0),(0,2j+1))$ depends on $n$ by definition.  
Let $\mathbf{\Lambda}$ be a diagonal $4 \times 4$  matrix with entries along the diagonal given by $\{ \lambda_i \}_{i=1}^4$ with
\begin{equation}
	\lambda_{2i+j+1}= \beta_i(x)^2 \beta_j(y)^2,
\end{equation}
where we have dropped the dependency of $a$ and $b$ in the definition of $\beta_i$ given in~\eqref{fortress tilings: beta}, that is $\beta_i(w)=\beta_i(a,b,w)$ for $i\in \{0,1\}$.  Let $\mathbf{\Gamma}$ denote the $4 \times 4$ matrix with the $(i,j)$ entry given by
\begin{equation}
\begin{split}
&\left( \frac{(b^2-a^2)x-(-1)^{\lfloor \frac{j-1}{2} \rfloor} \sqrt{x} \sqrt{ (a^4 +b^4)x+a^2 b^2(1+x^2)}}{ab(1+x)} \right)^{1-\lfloor\frac{i-1}{2}\rfloor} \\
& \times \left( \frac{(b^2-a^2)y-(-1)^{\left[ \frac{j-1}{2} \right]_2} \sqrt{y} \sqrt{ (a^4 +b^4)y+a^2 b^2(1+y^2)}}{ab(1+y)} \right)^{1-\left[\frac{i-1}{2}\right]_2} ,
\end{split}
\end{equation}
where $[i]_2=i \mod 2$. 
From~\cite[Lemma 6.4]{CY:13}, the coefficient of $z^n$ in $\mathbf{G}(a,b,x,y,z)$ is given by 
\begin{equation}
\label{BGF:form1}
\sum_{k=0}^{m-1} \mathbf{ \Gamma }\mathbf{\Lambda}^k\mathbf{ \Gamma}^{-1} . (\mathbf{B_1} + \mathbf{B_2}),
\end{equation}
where we have
\begin{equation}
\mathbf{B_1}=\left( 
\begin{array}{c}
 \frac{a^2+2b^2}{2 a \left(a^2+b^2\right)} \\
 \frac{b}{2  \left(a^2+b^2\right)} \\
 \frac{b}{2  \left(a^2+b^2\right)} \\
 \frac{a }{2  \left(a^2+b^2\right)}
\end{array}
\right) \hspace{5mm} \mbox{and}\hspace{5mm}
 \mathbf{B_2}=\left(
\begin{array}{c}
 \frac{4 b^4 x y+2 a^4 (1+x) (1+y)+a^2 b^2 (1+3 x+3 y+5 x y)}{4 a \left(a^2+b^2\right)^2} \\
 \frac{4 a^4 (1+x)+2 b^4 x (1+y)+a^2 b^2 (3+y+x (7+y))}{4 b \left(a^2+b^2\right)^2} \\
 \frac{2 b^4 (1+x) y+4 a^4 (1+y)+a^2 b^2 (3+x+7 y+x y)}{4 b \left(a^2+b^2\right)^2} \\
5 \frac{a \left(8 a^4+b^4 (1+x) (1+y)+2 a^2 b^2 (4+x+y)\right)}{4 b^2 \left(a^2+b^2\right)^2}
\end{array}
\right).
\end{equation}
Note that $\mathbf{B_1}$ and $\mathbf{B_2}$ were not written out explicitly in~\cite{CY:13}.  Since they are crucial to the following argument, we have listed them above.  We remark that~\cite[Lemma 6.1]{CY:13} is obtained from~\cite[Lemma 6.4]{CY:13} by a transformation and extraction of coefficients.  We found that  the form of $\mathbf{G}(a,b,x,y,z)$ as given in~\cite[Lemma 6.4]{CY:13} (and hence~\eqref{BGF:form1}) is much more convenient to use than~\cite[Lemma 6.1]{CY:13}.

In the above definitions, we set $x=u^2$ and $y=v^2$ and we have
\begin{equation} 
	\mu_{a,b} (u) = 1- \sqrt{1+\c^2 (u-u^{-1})^2}
\end{equation}
and $s_{a,b}(u)=1-\mu_{a,b}(u)$
where $\c=ab/(b^2+a^2).$ In the specific case when $b=1$, we recover $c_{a,1}=c$ defined in~\eqref{parameter:c}, $s_{a,1}(u)=s(u)$ and $\mu_{a,b}(u)=\mu(u)$ which are both defined in~\eqref{results:eq:mu}.  
For this subsection, we use the notation that if $A$ is a vector, then $A|_{i}$ represents the $i^{th}$ row of $A$ and recall~\eqref{h} which says that 
$h(\eps,\delta)=\eps(1-\delta)+(1-\eps)\delta$ for $\eps,\delta \in\{0,1\}$.

We write $\lambda_i=\lambda_i(a,b,u,v)$ as
\begin{equation}
\label{lambda}	\lambda_{2i+j+1} = \beta_i(u^2)^2 \beta_j(v^2)^2,
\end{equation}
where 
\begin{equation}
\beta_0(u^2)^2=\frac{\c}{2} u^2 \mu_{a,b} (u)^2
\hspace{3mm} \mbox{and} \hspace{3mm}
\beta_1(u^2)^2 = \frac{\c}{2}u^2(2- \mu_{a,b} (u))^2.
\end{equation}
We rewrite $\mathbf{\Gamma}=\left\{\mathbf{ \Gamma}_{i,j} \right\}_{i,j=1}^4$ in terms of $\mu_{a,b}$.  
By noting that $\mu_{a,b}(u)^2=2\mu_{a,b}(u)+c_{a,b}^2(u-1/u)^2$,
the $(2\eps_1+\eps_2+1,2\gamma_1 +\gamma_2+1)^{th}$ entry of $\mathbf{\Gamma}$, with $\eps_1,\eps_2,\gamma_1,\gamma_2 \in\{0,1\}$ reads
\begin{equation}
\begin{split}
&(-1)^{h\left(\varepsilon _1,\gamma _2\right)+h\left(\varepsilon _2,\gamma _1\right)+\varepsilon _1-\gamma _1 \varepsilon _1+\varepsilon _2-\gamma _2 \varepsilon _2} u^{2-2 \varepsilon _1} \left(a b \left(1+u^2\right)\right)^{-1+\varepsilon _1} v^{2-2 \varepsilon _2} \left(a b \left(1+v^2\right)\right)^{-1+\varepsilon _2}\\
& \times \left(-2 a^{2-2 \gamma _1} b^{2 \gamma _1}+\left(a^2+b^2\right) \mu_{a,b} (u)\right){}^{1-\varepsilon _1} \left(-2 a^{2-2 \gamma _2} b^{2 \gamma _2}+\left(a^2+b^2\right) \mu_{a,b} (v)\right){}^{1-\varepsilon _2}.
\end{split}
\end{equation}
The inverse of $\mathbf{\Gamma}$ can also be expressed in terms of $\mu_{a,b}$.  The
$(2\eps_1+\eps_2+1,2\gamma_1 +\gamma_2+1)^{th}$ entry of $\mathbf{\Gamma}^{-1}$ for $\eps_1,\eps_2,\gamma_1,\gamma_2 \in\{0,1\}$ is given by
\begin{equation}
\begin{split}
&\frac{(-1)^{h\left(\varepsilon _1,\gamma _2\right)+h\left(\varepsilon _2,\gamma _1\right)+\gamma _1+\gamma _2-\gamma _1 \varepsilon _1-\gamma _2 \varepsilon _2}}{4 \left(a^2+b^2\right)^2 (-1+\mu_{a,b} (u)) (-1+\mu_{a,b} (v))} u^{2 \left(-1+\gamma _1\right)} \left(a b \left(1+u^2\right)\right)^{1-\gamma _1} v^{2 \left(-1+\gamma _2\right)} \left(a b \left(1+v^2\right)\right)^{1-\gamma _2}\\
& \left(-2 a^{2 \varepsilon _1} b^{2-2 \varepsilon _1}+\left(a^2+b^2\right) \mu_{a,b} (u)\right){}^{\gamma _1} \left(-2 a^{2 \varepsilon _2} b^{2-2 \varepsilon _2}+\left(a^2+b^2\right) \mu_{a,b} (v)\right){}^{\gamma _2}.
\end{split}
\end{equation}
We have introduced all the definitions that we need for our computations.  We will perform computer algebra computations on these expressions but first,
we manipulate  the boundary generating function symbolically which gives a convenient form.
To do so, we expand~\eqref{BGF:form1} in terms of $\lambda_i$ which gives
\begin{equation} \label{BGF:beforesum}
\sum_{k=0}^{m-1} \mathbf{ \Gamma }\mathbf{\Lambda}^k\mathbf{ \Gamma}^{-1} . (\mathbf{B_1} + \mathbf{B_2})
=
\sum_{ k=0}^{m-1} \sum_{i=1}^4 \mathbf{X}_i \lambda_i^k,
\end{equation}
where $\mathbf{X}_j=\mathbf{X}_j(a,b,u,v)$ are four column vectors which are the coefficients of $a_i$ in the following expression 
\begin{equation}
 \mathbf{ \Gamma }\mathrm{diag}(a_1,a_2,a_3,a_4)\mathbf{ \Gamma}^{-1} . (\mathbf{B}_1 + \mathbf{B}_2),
\end{equation}
where $\mathrm{diag}(a_1,a_2,a_3,a_4)$ is a $4 \times 4$ diagonal matrix with entries $a_1,\dots,a_4$ and for $\mathbf{B}_1$ and $\mathbf{B}_2$, we set $x=u^2$ and $y=v^2$.  
 Note that each entry of $\mathbf{X}_i$ is not a polynomial.  We write~\eqref{BGF:beforesum} in terms of the geometric series, that is, we obtain
\begin{equation}
\begin{split} \label{bgf:splitbgf}
\sum_{ k=0}^{m-1} \sum_{i=1}^4 \mathbf{X}_i \lambda_i^k &= \sum_{i=1}^4 \mathbf{X}_i \frac{1-\lambda_i^m}{1-\lambda_i} 
= \sum_{i=1}^4 \mathbf{X}_i \frac{1}{1-\lambda_i} -\sum_{i=1}^4 \mathbf{X}_i  \frac{\lambda_i^m}{1-\lambda_i}   	\\
&= \frac{ \sum_{i=1}^4  \mathbf{X}_i  \prod_{j \not = i}1-\lambda_j }{ \prod_{j =1}^4 1-\lambda_{j}} -	 \frac{ \sum_{i=1}^4  \mathbf{X}_i \lambda_i^{m} \prod_{j \not = i}1-\lambda_j }{ \prod_{j =1}^4 1-\lambda_{j}}.
\end{split}
\end{equation}
It is possible to show, after some computations, that  $\sum_{i=1}^4 \mathbf{X}_i \prod_{i \not = j} (1-\lambda_j)$ and $\prod_{j=1}^4 (1-\lambda_j)$ share a polynomial factor of $g_0(u,v)g_1(u,v)$ where 
\begin{equation}
\begin{split}
g_i(u,v)&=16+d^2 \left(16 d^6+8 \left(7+u^2\right)+8 d^4 \left(7+u^2\right)+d^2 \left(81+13 u^2+3 u^4-u^6\right)\right)\\
&+2(-1)^i \left(d+d^3\right) u \left(4 \left(-3+u^2\right)+4 d^4 \left(-3+u^2\right)+d^2 \left(-23+6 u^2+u^4\right)\right) v\\
&-d^2 \left(-1+u^2\right) \left(4 \left(2+u^2\right)+4 d^4 \left(2+u^2\right)+d^2 \left(13+14 u^2-3 u^4\right)\right) v^2\\
&+4(-1)^i \left(d+d^3\right) u \left(2+2 d^4+d^2 \left(3+2 u^2-u^4\right)\right) v^3 \\
&+d^2 \left(-1+u^2\right) \left(4 u^2+d^2 \left(-3+2 \left(7+2 d^2\right) u^2-3 u^4\right)\right) v^4
\\
&+2(-1)^i d^3 \left(1+d^2\right) u \left(-1+u^2\right)^2 v^5+d^4 \left(-1+u^2\right)^3 v^6,
\end{split}
\end{equation}
where $d^i=a^i b^{8-i}$.   Let
\begin{equation}
\g(m,u,v) = \frac{
\sum_{i=1}^4 \left. \mathbf{X}_i \right|_{2 \eps_1+\eps_2+1} \prod_{j \not = i} (1-\lambda_j) }{\prod_{j=1}^4 (1-\lambda_j)} \lambda_i^m,
\end{equation}
where $\left. \mathbf{X}_i \right|_{2\eps_1+\eps_2+1}$ denotes $(2\eps_1+\eps_2+1)^{th}$ row of $\mathbf{X}_i$ as mentioned above.  From~\eqref{bgf:splitbgf} we have 
\begin{equation} \label{bgf:splitbgf2first}
\sum_{ k=0}^{m-1} \sum_{i=1}^4 \left. \mathbf{X}_i  \right|_{2\eps_1 +\eps_2+1}\lambda_i^k= \g(0,u,v)-\g(m,u,v).
\end{equation}
For $\eps_1,\eps_2 \in \{0,1\}$, define
\begin{equation} \label{bgf:Teded}
T^{a,b}_{\eps_1,\eps_2}(u,v)=\sum_{p,q=0}^{2m-1} K_{a,b}^{-1}\left((4q+2\eps_1+1,0),(0,4p+2 \eps_2+1)\right) (\mathrm{i} u)^{2q} (\mathrm{i} v)^{2p}.
\end{equation}
Once we have obtained a suitable expression for $\g(m,a,b)$, we are able to use it in the boundary generating function because of the following lemma 
\begin{lemma}\label{bgf:lemma:corners}
\begin{equation}
\begin{split} \label{bgf:lem:corners}
T^{a,b}_{\eps_1,\eps_2} (u,v)&=- \mathrm{i}^{\eps_1+\eps_2+1}  \left( \mathcal{G}_{\eps_1,\eps_2}^{a,b}(0,u,v)-\mathcal{G}_{\eps_1,\eps_2}^{a,b}(m,u,v) \right). \\
\end{split}
\end{equation}
\end{lemma}

\begin{proof}
We have that 
\begin{equation}\label{bgf:lempf:Te0e0}
T^{a,b}_{\eps_1,\eps_2}(u,v)= \sum_{p,q=0}^{2m-1}L_{a,b}(2q+\eps_1,2p+\eps_2 ) (\mathrm{i} u)^{2q} (\mathrm{i} v)^{2p}
\end{equation}
by using~\eqref{bgf:Teded} and \cite[Lemma 6.1]{CY:13}.  The sign of $L_{a,b}(i,j)$ is given by $-\mathrm{i}^{i+j+1}$.  Therefore, the sign in the summand in~\eqref{bgf:lempf:Te0e0} is given by
\begin{equation}
	-\mathrm{i}^{2p+2q} \mathrm{i}^{2p+\eps_1+2q+\eps_2+1}=-\mathrm{i}^{\eps_1+\eps_2+1}.
\end{equation}
Therefore, we write 
\begin{equation}
T^{a,b}_{\eps_1,\eps_2}(u,v)=-\mathrm{i}^{\eps_1+\eps_2+1}
 \sum_{p,q=0}^{2m-1}\left|L_{a,b}(2q+\eps_1,2p+\eps_2 ) \right|  u^{2q}  v^{2p}.
\end{equation}
The sum in the right hand-side of the above equation is exactly equal to the row $2\eps_1+\eps_2+1$ of the boundary generating function with $x=u^2$ and $y=v^2$.  By using~\eqref{bgf:splitbgf2first}, we have that the boundary generating function is equal to $\g(0,u,v)-\g(m,u,v)$.   
\end{proof}

The above symbolic computations requires extensive computer algebra to find an explicit  expression for $\g (m,u,v)$ in terms of the rational functions defined in Definition~\ref{def:coefficients}. 
We performed these computations and we
state the resulting formulas in the lemma below.
For the proof of this lemma, we will not give these computations due to sheer complexity of 
the intermediate terms.  Instead, it suffices to verify that our assertions are correct and we provide the necessary statements which gives a much simpler verification (also by using computer algebra).
\begin{lemma}\label{bgf:lem:Gcompstatement}
We have that
\begin{equation}
\begin{split} \label{bgf:lem:Gcomp}
\g (m,u,v) &=
\frac{(-1)^{\eps_1 \eps_2} \mathtt{y}_{0,0}^{\eps_1,\eps_2} (a,b,u,v) }{s_{a,b}(u)s_{a,b}(v)}\chi_1
+\frac{ (-1)^{\eps_1(1-\eps_2)} \mathtt{y}_{0,1}^{\eps_1,\eps_2} (a,b,u,v)}{s_{a,b}(u)} \chi_2 \\
&+ \frac{ (-1)^{\eps_2 (1-\eps_1)} \mathtt{y}_{1,0}^{\eps_1,\eps_2}(a,b,u,v)}{s_{a,b}(v)} \chi_3
+ (-1)^{\eps_1+\eps_2 -\eps_1 \eps_2} \mathtt{y}_{1,1}^{\eps_1,\eps_2} (a,b,u,v) \chi_4 ,
\end{split}
\end{equation}
where $\chi_i=\chi_i(u,v)$ are given by
\begin{equation}
\chi_{2\delta_1+\delta_2+1}(u,v) = \sum_{\alpha_1,\alpha_2=0}^1 (-1)^{\alpha_1+\alpha_2+\alpha_1 \delta_1+\alpha_2\delta_2} (\lambda_{2\alpha_1+\alpha_2+1}(u,v))^m
\end{equation}
for $\delta_1,\delta_2\in \{0,1\}$, and for $\eps_1,\eps_2,\gamma_1,\gamma_2 \in \{0,1\}$, $\mathtt{y}^{\eps_1, \eps_2}_{\gamma_1,\gamma_2} (a,b,u,v)$ are given in Definition~\ref{def:coefficients}.

\end{lemma}
\begin{proof}

We claim the following four relations
\begin{equation}\label{bgf:lemproof:rel1}
\begin{split}
&\left. \mathbf{X}_1 \right|_{2 \eps_1 +\eps_2+1} s_{a,b}(u)s_{a,b}(v)= \left( (-1)^{\eps_1 \eps_2} \mathtt{y}_{0,0}^{\eps_1,\eps_2} + s_{a,b}(v) (-1)^{\eps_1(1-\eps_2)} \mathtt{y}_{0,1}^{\eps_1,\eps_2} \right.\\
&\left.+ s_{a,b}(u) (-1)^{\eps_2(1-\eps_1)} \mathtt{y}_{1,0}^{\eps_1,\eps_2} +s_{a,b}(u) s_{a,b}(v) (-1)^{\eps_1+\eps_2-\eps_1 \eps_2} \mathtt{y}_{1,1}^{\eps_1,\eps_2} \right) (1-\lambda_1),
\end{split}
\end{equation}
\begin{equation}\label{bgf:lemproof:rel2}
\begin{split}
&\left. \mathbf{X}_2 \right|_{2 \eps_1 +\eps_2+1} s_{a,b}(u)s_{a,b}(v)=\left( -(-1)^{\eps_1 \eps_2} \mathtt{y}_{0,0}^{\eps_1,\eps_2} + s_{a,b}(v) (-1)^{\eps_1(1-\eps_2)} \mathtt{y}_{0,1}^{\eps_1,\eps_2} \right. \\
&\left. - s_{a,b}(u) (-1)^{\eps_2(1-\eps_1)} \mathtt{y}_{1,0}^{\eps_1,\eps_2} +s_{a,b}(u) s_{a,b}(v) (-1)^{\eps_1+\eps_2-\eps_1 \eps_2} \mathtt{y}_{1,1}^{\eps_1,\eps_2} \right) (1-\lambda_2) ,
\end{split}
\end{equation}
\begin{equation}\label{bgf:lemproof:rel3}
\begin{split}
&\left. \mathbf{X}_3 \right|_{2 \eps_1 +\eps_2+1} s_{a,b}(u)s_{a,b}(v)=\left(- (-1)^{\eps_1 \eps_2} \mathtt{y}_{0,0}^{\eps_1,\eps_2} - s_{a,b}(v) (-1)^{\eps_1(1-\eps_2)} \mathtt{y}_{0,1}^{\eps_1,\eps_2} \right. \\
&\left. + s_{a,b}(u) (-1)^{\eps_2(1-\eps_1)} \mathtt{y}_{1,0}^{\eps_1,\eps_2} +s_{a,b}(u) s_{a,b}(v) (-1)^{\eps_1+\eps_2-\eps_1 \eps_2} \mathtt{y}_{1,1}^{\eps_1,\eps_2} \right) (1-\lambda_3) 
\end{split}
\end{equation}
and
\begin{equation}\label{bgf:lemproof:rel4}
\begin{split}
&\left. \mathbf{X}_4 \right|_{2 \eps_1 +\eps_2+1} s_{a,b}(u)s_{a,b}(v)=\left( (-1)^{\eps_1 \eps_2} \mathtt{y}_{0,0}^{\eps_1,\eps_2} - s_{a,b}(v) (-1)^{\eps_1(1-\eps_2)} \mathtt{y}_{0,1}^{\eps_1,\eps_2} \right. \\
&\left. - s_{a,b}(u) (-1)^{\eps_2(1-\eps_1)} \mathtt{y}_{1,0}^{\eps_1,\eps_2} +s_{a,b}(u) s_{a,b}(v) (-1)^{\eps_1+\eps_2-\eps_1 \eps_2} \mathtt{y}_{1,1}^{\eps_1,\eps_2} \right)(1-\lambda_4) ,
\end{split}
\end{equation}
where $\mathtt{y}_{i,j}^{\eps_1,\eps_2}=\mathtt{y}_{i,j}^{\eps_1,\eps_2}(a,b,u,v)$.  To prove these relations, we first must obtain a suitable formula $\mathbf{X}_k$.  Setting $\mathbf{B}=\mathbf{B_1}+\mathbf{B_2}$, we have by expanding out the matrix multiplication
\begin{equation}
\left. \mathbf{ \Gamma }\mathbf{\Lambda}^r\mathbf{ \Gamma}^{-1} . (\mathbf{B_1} + \mathbf{B_2})\right|_{i}=\sum_{k=1}^4 \left.\mathbf{B}\right|_{k}\sum_{j=1}^4 \lambda_j^r \mathbf{\Gamma}_{i,j} \mathbf{\Gamma}^{-1}_{j,k}
\end{equation}
for $i\in\{1,\dots,4\}$.
From the above equation, we extract the coefficient of $\lambda_j^r$ and so by definition of $\mathbf{X}_j$ given in equation~\eqref{BGF:beforesum}, we obtain
\begin{equation}
\left. \mathbf{X}_j \right|_i = \mathbf{\Gamma}_{i,j} \sum_{k=1}^4 \left. \mathbf{B}\right|_{k} \mathbf{\Gamma}_{j,k}^{-1}.
\end{equation}
Therefore, we are able to evaluate the left hand-side of equations~\eqref{bgf:lemproof:rel1},~\eqref{bgf:lemproof:rel2},~\eqref{bgf:lemproof:rel3} and~\eqref{bgf:lemproof:rel4}.  

To evaluate the right hand-side of equations~\eqref{bgf:lemproof:rel1},~\eqref{bgf:lemproof:rel2},~\eqref{bgf:lemproof:rel3} and~\eqref{bgf:lemproof:rel4}, we first simplify $\lambda_i$ by using the following simplifications:
\begin{equation}
\lambda_1= \frac{\left(\c^2 \left(-1+u^2\right)^2+2 u^2 \mu_{a,b} (u)\right) \left(\c^2 \left(-1+v^2\right)^2+2 v^2 \mu_{a,b} (v)\right)}{4 \c^2},
\end{equation}
\begin{equation}
\lambda_2= \frac{\left(\c^2 \left(-1+u^2\right)^2+2 u^2 \mu_{a,b} (u)\right) \left(4 v^2+\c^2 \left(-1+v^2\right)^2-2 v^2 \mu_{a,b} (v)\right)}{4 \c^2},
\end{equation}
\begin{equation}
\lambda_3= \frac{\left(4 u^2+\c^2 \left(-1+u^2\right)^2-2 u^2 \mu_{a,b} (u)\right) \left(\c^2 \left(-1+v^2\right)^2+2 v^2 \mu_{a,b} (v)\right)}{4 \c^2}
\end{equation}
and
\begin{equation}
\lambda_4= \frac{\left(4 u^2+\c^2 \left(-1+u^2\right)^2-2 u^2 \mu_{a,b} (u)\right) \left(4 v^2+\c^2 \left(-1+v^2\right)^2-2 v^2 \mu_{a,b} (v)\right)}{4 \c^2}.
\end{equation}
We substitute in the above expressions for $\lambda_i$ into the right hand-sides of equations~\eqref{bgf:lemproof:rel1},~\eqref{bgf:lemproof:rel2},~\eqref{bgf:lemproof:rel3} and~\eqref{bgf:lemproof:rel4}, set $s_{a,b}(u)=1-\mu_{a,b}(u)$ and $s_{a,b}(v)=1-\mu_{a,b}(v)$, and use the simplifications
$\mu_{a,b}(u)^2=2 \mu_{a,b} (u)  +\c^2(u-1/u)^2$ and $\mu_{a,b}(v)^2=2 \mu_{a,b} (v)  +\c^2(v-1/v)^2$ so that the degree of $\mu_{a,b}(u)$ and $\mu_{a,b}(v)$ is at most $1$.
 Using the above approach (and computer algebra), we have computable expressions of both sides of equations~\eqref{bgf:lemproof:rel1},~\eqref{bgf:lemproof:rel2},~\eqref{bgf:lemproof:rel3} and~\eqref{bgf:lemproof:rel4}. From these expressions, by using computer algebra we verify~\eqref{bgf:lemproof:rel1},~\eqref{bgf:lemproof:rel2},~\eqref{bgf:lemproof:rel3} and~\eqref{bgf:lemproof:rel4}. The lemma follows from equations~\eqref{bgf:lemproof:rel1},~\eqref{bgf:lemproof:rel2},~\eqref{bgf:lemproof:rel3} and~\eqref{bgf:lemproof:rel4}.  This can be seen by dividing both sides of~\eqref{bgf:lemproof:rel1} by $(1-\lambda_1)$,~\eqref{bgf:lemproof:rel2} by $(1-\lambda_2)$,~\eqref{bgf:lemproof:rel3} by $(1-\lambda_3)$ and~\eqref{bgf:lemproof:rel4} by $(1-\lambda_4)$ and then substituting into the right side of the following equation:
\begin{equation}
\g (m,u,v)=\sum_{i=1}^4 \frac{\left. \mathbf{X}_i \right|_{2 \eps_1 +\eps_2+1}}{1-\lambda_i}\lambda_i^m.
\end{equation} 


\end{proof}
From Lemma~\ref{bgf:lemma:corners}, $T^{a,1}_{\eps_1,\eps_2}(u,v)$ consists of two components which we call $S^1_{\eps_1,\eps_2}(u,v)$
and $S^2_{\eps_1,\eps_2}(u,v)$, where 
$S^1_{\eps_1,\eps_2}(u,v)$ 
is the contribution of $T^{a,1}_{\eps_1,\eps_2}(u,v)$ from $\mathcal{G}^{a,1}_{\eps_1,\eps_2}(0,u,v)$ and $S^2_{\eps_1,\eps_2}(u,v)$ is the contribution of $T^{a,1}_{\eps_1,\eps_2}(u,v)$ from
 $\mathcal{G}^{a,1}_{\eps_1,\eps_2}(m,u,v)$.  More explicitly,  we have
\begin{equation} \label{bgf:eq:decompofT}
T^{a,1}_{\eps_1,\eps_2}(u,v)=S_{\eps_1,\eps_2}^1(u,v)-S_{\eps_1,\eps_2}^2(u,v).
\end{equation}
Using Lemma~\ref{bgf:lem:Gcompstatement}, we find that 
\begin{equation}
\begin{split}
S_{\eps_1,\eps_2}^1(u,v)&=-\mathrm{i}^{\eps_1 + \eps_2 +1} \mathcal{G}^{a,1}_{\eps_1,\eps_2} (0,u,v) \\
&=-4(-1)^{\eps_1+\eps_2-\eps_1 \eps_2} \mathrm{i}^{\eps_1+\eps_2+1}{ \mathtt{y}^{\eps_1, \eps_2}_{1,1} (u,v) } 
\label{bgf:lempf:S}
\end{split}
\end{equation}
and
\begin{equation}\label{bgf:lempf:S2}
\begin{split}
S^2_{\eps_1,\eps_2}(u,v)&=-\mathrm{i}^{\eps_1 + \eps_2 +1}\mathcal{G}^{a,1}_{\eps_1,\eps_2} \left(m,u, v \right) \\
&= -\mathrm{i}^{\eps_1+\eps_2 +1} \left( \frac{ (-1)^{\eps_1 \eps_2} \mathtt{y}_{0,0}^{\eps_1,\eps_2} (u,v)}{s(u)s(v)} \chi_1 +  \frac{ (-1)^{\eps_1(1- \eps_2)} \mathtt{y}_{0,1}^{\eps_1,\eps_2} (u,v)}{s(u)} \chi_2\right. \\\, & \left.
 +\frac{ (-1)^{(1-\eps_1) \eps_2} \mathtt{y}_{1,0}^{\eps_1,\eps_2} (u,v)}{s(v)} \chi_3 +   (-1)^{ \eps_1+\eps_2-\eps_1 \eps_2} \mathtt{y}_{1,1}^{\eps_1,\eps_2} (u,v) \chi_4 \right),
\end{split}
\end{equation}
where for $1 \leq i \leq 4$, $\chi_i$ is given in Lemma~\ref{bgf:lem:Gcompstatement}  
and for $\eps_1,\eps_2,\gamma_1,\gamma_2 \in \{0,1\}$, $\mathtt{y}^{\eps_1, \eps_2}_{\gamma_1,\gamma_2} (u,v)=\mathtt{y}^{\eps_1, \eps_2}_{\gamma_1,\gamma_2} (a,1,u,v)$ as given in Definition~\ref{def:coefficients}.

\subsection{Double Contour integrals} \label{subsection:Double}

In this subsection, we consider for $\epsilon>0$ small
\begin{equation}
K^{-1}_{a,1}((x_1,x_2),(y_1,y_2))=\frac{1}{(2\pi \mathrm{i})^4} \int_{\Gamma_{1-\epsilon}} \frac{dw_1}{w_1}  \int_{\Gamma_{1-\epsilon}} \frac{dw_2}{w_2} \int_{\Gamma_{1-\epsilon}} \frac{db_1}{b_1}\int_{\Gamma_{1-\epsilon}} \frac{d b_2}{b_2} \frac{G(a,1,w_1,w_2,b_1,b_2)}{w_1^{x_1} w_2^{x_2} b_1^{y_1}b_2^{y_2}}
\end{equation}
and show that we can reduce to linear combinations of double contour integrals.
The splitting of the generating function given 
in Lemma~\ref{sym:lemma:G01} greatly reduces the complexity of the computation. However, the computation is extremely involved and so we provide a synopsis of our computations:

\begin{enumerate}
\item We first introduce the notation that we use.
\item We extract the coefficient of $w_1^{x_1}w_2^{x_2}b_1^{y_1}b_2^{y_2}$ of $\mathcal{S}(d_1+d_2)(a,w_1,w_2,b_1,b_2)$ using quadruple integrals and show that there is a double contour integral formula. By comparing this expression with Eq.~\eqref{GasLiquiduintegral2}, we find that this double contour integral is given by $\mathbb{K}^{-1}_{1,1}(x,y)$.
\item We are able to rearrange the sums in $H(a,w_1,w_2,b_1,b_2)$ and split the expression into two parts based on the simplification given in Section~\ref{subsection:Boundary}.  We are able to find double contour integral formulas for the coefficients of $w_1^{x_1}w_2^{x_2}b_1^{y_1}b_2^{y_2}$ in these two expressions.  For $(x_1,x_2) \in\mathtt{W}_{\eps_1}$ and $(y_1,y_2) \in\mathtt{B}_{\eps_2}$ with $\eps_1,\eps_2 \in \{0,1\}$, we call the first of these  double contour integral formulas $\psi_{\eps_1,\eps_2}(a,x_1,x_2,y_1,y_2)$  while the second double contour integral is equal to $\mathcal{B}_{\eps_1,\eps_2}(a,x_1,x_2,y_1,y_2)$ given in~\eqref{res:eq:B}.   That is, the first expression comes from~\eqref{bgf:lempf:S} while the second expression comes from~\eqref{bgf:lempf:S2} after the sum rearrangement.  The second expression is exactly equal to~\eqref{res:eq:B}.
\item We next extract the coefficient of $w_1^{x_1}w_2^{x_2}b_1^{y_1}b_2^{y_2}$ in $d_3(a,w_1,w_2,b_1,b_2)$ and find that this coefficient is given by $\psi_{\eps_1,\eps_2}(a,x_1,x_2,y_1,y_2)$. 

\end{enumerate}

The \emph{vertex type} of  $\mathtt{w}^{r}\mathtt{b}^{s}=w_1^{r_1} w_2^{r_2} b_1^{s_1} b_2^{s_2}$ is given by $\mathtt{W}_{([r_1+r_2]_ 4-1)/2} \times \mathtt{B}_{([s_1 + s_2]_4-1)/2}$.  A function $f$ does not change the vertex type if $f(\mathtt{w}^r \mathtt{b}^s)$ also has the vertex type $\mathtt{W}_{([r_1+r_2]_ 4-1)/2} \times \mathtt{B}_{([s_1 + s_2]_4-1)/2}$. For example, the function defined by multiplying $1/c^a_{\partial}(w_1,w_2)$ to $\mathtt{w}^r \mathtt{b}^s$ does not change the vertex type.

Let $\tilde{s}^a_\eps(u_1,u_2)= s_{\eps,0}^a(\sqrt{u_1},\sqrt{u_2})$ and we can write 
\begin{equation} \label{ovl:eqn:ssplit}
\tilde{s}_{\eps}^a(u_1,u_2)= B_{\eps,0}^a (u_1,u_2)+B_{\eps,1}^a (u_1,u_2),
\end{equation}
where
\begin{equation} \label{ovl:def:B0}
B_{\eps,0}^a (u_1,u_2) =-a (u_1+u_1^{-1}) u_2^{-1} -2 a^{2\eps}
\end{equation}
and
\begin{equation} \label{ovl:def:B1}
B_{\eps,1}^a(u_1,u_2) = -a \mathrm{i} (u_1 -u_1^{-1}).
\end{equation}
We combine~\eqref{ovl:def:B0} and~\eqref{ovl:def:B1} to obtain
\begin{equation}
B_{\eps,\delta}^a (u_1,u_2)=-\left( a(u_1+u^{-1}_1 ) \right)^{1-\delta} u_2^{\delta-1} -(2 a ^{2 \eps})^{1-\delta} -\left(a \mathrm{i} (u_1-u_1^{-1}) \right)^{\delta} +1+\delta 
\end{equation}
for $\eps,\delta \in \{0,1\}$. 
From comparing~\eqref{res:eq:charc} and $c_\delta^a$, we have
	$\tilde{c}(u_1,u_2) = c^a_{\partial} (\sqrt{u_1},\sqrt{u_2})$.

\subsubsection{Double Contour integrals from the contributions from $d_1$ and $d_2$}

We combine the terms $d_1$ and $d_2$ defined in Eqs.~\eqref{split:def:d1} and~\eqref{split:def:d2} and extract out the coefficients $w_1^{x_1}$, $w_2^{x_2}$, $b_1^{y_1}$ and $b_2^{y_2}$ for $x \in \mathtt{W}$ and $y \in \mathtt{B}$ but dependent on the type of white and black vertices. 
\begin{lemma}\label{dc:lem:gaspart}
For $x=(x_1,x_2) \in \mathtt{W}_{\eps_1}$ and $y=(y_1,y_2) \in \mathtt{B}_{\eps_2}$ and $\eps_1,\eps_2 \in \{0,1\}$ we have
\begin{equation}
\begin{split}
&\frac{1}{(2\pi \mathrm{i})^4} \int_{\Gamma_{1-\epsilon}} \frac{dw_1}{w_1}  \int_{\Gamma_{1-\epsilon}} \frac{dw_2}{w_2} \int_{\Gamma_{1-\epsilon}} \frac{db_1}{b_1}\int_{\Gamma_{1-\epsilon}} \frac{d b_2}{b_2} \frac{\mathcal{S}(d_1+d_2)(a,w_1,w_2,b_1,b_2)}{w_1^{x_1} w_2^{x_2} b_1^{y_1}b_2^{y_2}} \\
&=-\frac{ \mathtt{i}^{1+h(\eps_1,\eps_2)}}{(2\pi \mathrm{i})^2} 
\int_{\Gamma_{1}} \frac{du_1}{u_1}  \int_{\Gamma_{1}} \frac{du_2}{u_2}
\frac{ a^{1-\eps_1} u_1^{1-h(\eps_1,\eps_2)} +a^{\eps_1} u_2 u_1^{h(\eps_1,\eps_2)}}{\tilde{c}(u_1,u_2) u_1^{\frac{x_1-y_1+1}{2}} u_2^{\frac{x_2-y_2+1}{2}}} = \mathbb{K}^{-1}_{1,1}(x,y).
\end{split}
\end{equation}
\end{lemma}

When extracting out the vertex types in the proof below, our convention is that the first row corresponds to the pair vertices in $\mathtt{W}_0 \times \mathtt{B}_0$, the second row corresponds to  the pair of vertices in $\mathtt{W}_0 \times \mathtt{B}_1$, the third row corresponds to the pair of vertices $\mathtt{W}_1 \times \mathtt{B}_0$ and the fourth row corresponds to the pair of vertices in $\mathtt{W}_1 \times \mathtt{B}_1$.  
\
\begin{proof}
From  $d_1(a, w_1,w_2,b_1,b_2)$ we extract out the types of vertices and organize the result as a column vector as described above.  Notice that the series expansions of $1/((1-w_1^4 b_1^4) (1-w_2^4 b_2^4)c^a_{\partial}(w_1,w_2))$ do not change the vertex types.  Therefore, we only need to sort the type of vertices in the numerator of $d_1(a,w_1,w_2,b_1,b_2)$.  We   find that the vertex types of $d_1(a, w_1,w_2,b_1,b_2)$ are given by 
\begin{equation} \label{ovl:lemproof:G0d1}
\frac{1}{(1-w_1^4 b_1^4) (1-w_2^4 b_2^4)c^a_{\partial}(w_1,w_2)}\left(
\begin{array}{c}
 -\mathrm{i} b_2 w_1 \left(a+b_1^4 w_1^4+b_1^2 b_2^2 w_1^4+a b_1^2 b_2^2 w_1^4\right) \\
 b_2 w_1 \left(a b_1^2+b_1^2 w_1^4+b_2^2 w_1^4+a b_1^4 b_2^2 w_1^4\right) \\
 b_2 w_1^3 \left(a+b_1^4+b_1^2 b_2^2+a b_1^2 b_2^2 w_1^4\right) \\
 -\mathrm{i} b_2 w_1^3 \left(b_1^2+a b_1^2+b_2^2+a b_1^4 b_2^2 w_1^4\right)
\end{array}
\right).
\end{equation}
We also extract out the vertex types for $d_2(a,w_1,b_1,b_2)$.  We obtain
\begin{equation} \label{ovl:lemproof:G0d2}
\frac{1}{ (1-w_2^4 b_2^4)c^a_{\partial}(w_1,w_2)}\left(
\begin{array}{c}
 -i b_2 w_1 \\
 a b_2^3 w_1 \\
 \frac{b_2}{w_1} \\
 -i a b_2^3 w_1^3
\end{array}
\right).
\end{equation}
We add~\eqref{ovl:lemproof:G0d1} and~\eqref{ovl:lemproof:G0d2} to obtain an expression for $(d_1+d_2)(a,w_1,w_2,b_1,b_2)$. Entry-wise, we have that $(d_1+d_2)(a,w_1,w_2,b_1,b_2)$ is given by
\begin{equation}~\label{dc:lemproof:d1+d2}
\frac{1}{(1-w_1^4 b_1^4)(1-b_2^4 w_2^4)c^a_{\partial}(w_1,w_2)}\left(
\begin{array}{c}
 -\frac{\mathrm{i} b_2 \left(a w_1^2+w_2^2\right) \left(1+b_1^2 b_2^2 w_1^2 w_2^2\right)}{w_1} \\
 \frac{b_2 \left(b_1^2 w_1^2+b_2^2 w_2^2\right) \left(a+w_1^2 w_2^2\right)}{w_1} \\
 \frac{b_2 \left(1+a w_1^2 w_2^2\right) \left(1+b_1^2 b_2^2 w_1^2 w_2^2\right)}{w_1} \\
 -\frac{\mathrm{i} b_2 \left(w_1^2+a w_2^2\right) \left(b_1^2 w_1^2+b_2^2 w_2^2\right)}{w_1} \\
\end{array}
\right).
\end{equation}
Notice that 
\begin{equation}
\begin{split} \label{dc:lemproof:f1expand}
&\frac{1}{(2\pi \mathrm{i})^4} \int_{\Gamma_{1-\epsilon}} \frac{dw_1}{w_1}  \int_{\Gamma_{1-\epsilon}} \frac{dw_2}{w_2} \int_{\Gamma_{1-\epsilon}} \frac{db_1}{b_1}\int_{\Gamma_{1-\epsilon}} \frac{d b_2}{b_2} \frac{\mathcal{S}(d_1+d_2)(a,w_1,w_2,b_1,b_2)}{w_1^{x_1} w_2^{x_2} b_1^{y_1}b_2^{y_2}} \\
&=\frac{1}{(2\pi \mathrm{i})^4} \int_{\Gamma_{1-\epsilon}} \frac{dw_1}{w_1}  \int_{\Gamma_{1-\epsilon}} \frac{dw_2}{w_2} \int_{\Gamma_{1-\epsilon}} \frac{db_1}{b_1}\int_{\Gamma_{1-\epsilon}} \frac{d b_2}{b_2} \frac{(d_1+d_2)(a,w_1,w_2,b_1,b_2)}{w_1^{x_1} w_2^{x_2} b_1^{y_1}b_2^{y_2}} \\
\end{split}
\end{equation}
because the coefficient of $b_1^{y_1}$ is zero in the term $\frac{ \mathrm{i}}{a} (w_1b_1)^{2n} (d_1+d_2)(a^{-1},-\mathrm{i} w_1^{-1},\mathrm{i} w_2,-\mathrm{i}b_1^{-1}, \mathrm{i}b_2 )$ for $0 \leq y_1 \leq 2n$, the coefficient of $w_2^{x_2}$ is zero in the term
$\frac{ \mathrm{i}}{a} (w_2 b_2)^{2n} (d_1+d_2)(a^{-1}, \mathrm{i}  w_1 , -\mathrm{i}w_2^{-1}, \mathrm{i}b_1,  -\mathrm{i} b_2^{-1} )$ for $0\leq x_2 \leq 2n$ and the coefficient of $b_1^{y_1}$ is zero in the term
$(w_1w_2b_1b_2)^{2n} (d_1+d_2)(a,w_1^{-1} ,-w_2^{-1} ,b_1^{-1},-b_2^{-1})$ for $0\leq y_1 \leq 2n$. Each of these assertions is seen by expanding out each expression and $1/(1-w_1^4 b_1^4)$ or  $1/(1-w_2^4 b_2^4)$ in terms of its geometric series.

We consider the change of variables $(w_1,w_2)=(u_1^{1/2},u_2^{1/2})$ and  $(b_1,b_2)=(v_1^{1/2},v_2^{1/2})$ in~\eqref{dc:lemproof:f1expand} and notice that~\eqref{dc:lemproof:d1+d2} becomes
\begin{equation}\label{dc:lemproof:d1+d22} 
\frac{1}{(1-u_1^2 v_1^2)(1-u_2^2 v_2^2) \tilde{c}(u_1,u_2)} 
\left( \begin{array}{c} 
-\frac{i \sqrt{v_2} \left(a u_1+u_2\right) \left(u_1 u_2 v_1 v_2+1\right)}{\sqrt{u_1}} \\
\frac{\sqrt{v_2} \left(a+u_1 u_2\right) \left(u_1 v_1+u_2 v_2\right)}{\sqrt{u_1}}  \\
\frac{\sqrt{v_2} \left(a u_1 u_2+1\right) \left(u_1 u_2 v_1 v_2+1\right)}{\sqrt{u_1}}\\
-\frac{i \sqrt{v_2} \left(a u_2+u_1\right) \left(u_1 v_1+u_2 v_2\right)}{\sqrt{u_1}}
\end{array} \right).
\end{equation}
We claim the following equation 
\begin{equation}\label{dc:lemproof:d1change}
\begin{split}
&\frac{1}{(2\pi \mathrm{i})^4} \int_{\Gamma_{1-\epsilon}} \frac{dw_1}{w_1}  \int_{\Gamma_{1-\epsilon}} \frac{dw_2}{w_2} \int_{\Gamma_{1-\epsilon}} \frac{db_1}{b_1}\int_{\Gamma_{1-\epsilon}} \frac{d b_2}{b_2} \frac{\mathcal{S}(d_1+d_2)(a,w_1,w_2,b_1,b_2)}{w_1^{x_1} w_2^{x_2} b_1^{y_1}b_2^{y_2}} \\
&=\frac{-\mathrm{i}^{1+h(\eps_1,\eps_2)}}{(2\pi \mathrm{i})^4} \int_{\Gamma_{1-\epsilon}} \frac{du_1}{u_1}  \int_{\Gamma_{1-\epsilon}} \frac{du_2}{u_2} \int_{\Gamma_{1-\epsilon}} \frac{dv_1}{v_1}\int_{\Gamma_{1-\epsilon}} \frac{d v_2}{v_2}\\
&\hspace{20mm}\times \frac{ \left( a^{1-\eps_1} u_1^{1-h(\eps_1,\eps_2)} +a^{\eps_1} u_2 u_1^{h(\eps_1,\eps_2)}\right)\left(u_1 v_1(u_2 v_2)^{1-\eps_2}+(u_2 v_2)^{\eps_2} \right) }{u_1^{(x_1+1)/2} u_2^{x_2/2} v_1^{y_1/2}v_2^{(y_2-1)/2}}
\end{split}
\end{equation}
for $(x_1,x_2)\in\mathtt{W}_{\eps_1}$ and $(y_1,y_2) \in\mathtt{B}_{\eps_2}$. The above equation follows  due to the change of variables  
$(w_1,w_2)=(u_1^{1/2},u_2^{1/2})$ and  $(b_1,b_2)=(v_1^{1/2},v_2^{1/2})$ doubles the contour of integration for each integral  but this is compensated by the extra factor of $1/2$ from the change of variables for each integral and using~\eqref{dc:lemproof:d1+d22} entry-wise. 
We are able to compute the integrals with respect to $v_1$ and $v_2$ because
\begin{equation}
\frac{1}{(2 \pi \mathrm{i})^2} \int_{\Gamma_{1-\epsilon}} \frac{dv_1}{v_1} \int_{\Gamma_{1-\epsilon}} \frac{dv_2}{v_2} \frac{(u_1 u_2 v_1 v_2+1)}{v_1^{y_1/2} v_2^{(y_2-1)/2}(1-u_1^2 v_1^2)(1-u_2^2 v_2^2)}= u_1^{y_1/2} u_2^{(y_2-1)/2}
\end{equation}
and
\begin{equation}
\frac{1}{(2 \pi \mathrm{i})^2} \int_{\Gamma_{1-\epsilon}} \frac{dv_1}{v_1} \int_{\Gamma_{1-\epsilon}} \frac{dv_2}{v_2} \frac{(u_1v_1+u_2 v_2)}{v_1^{y_1/2} v_2^{(y_2-1)/2}(1-u_1^2 v_1^2)(1-u_2^2 v_2^2)}= u_1^{y_1/2} u_2^{(y_2-1)/2}
\end{equation}
using the residue theorem. These two equations means that~\eqref{dc:lemproof:d1change} is equal to  
\begin{equation}
-\frac{ \mathtt{i}^{1+h(\eps_1,\eps_2)}}{(2\pi \mathrm{i})^2} 
\int_{\Gamma_{1-\epsilon}} \frac{du_1}{u_1}  \int_{\Gamma_{1-\epsilon}} \frac{du_2}{u_2}
\frac{ a^{1-\eps_1} u_1^{1-h(\eps_1,\eps_2)} +a^{\eps_1} u_2 u_1^{h(\eps_1,\eps_2)}}{\tilde{c}(u_1,u_2) u_1^{\frac{x_1-y_1+1}{2}} u_2^{\frac{x_2-y_2+1}{2}}}.
\end{equation}
Since the integrands are analytic for $|u_1|=1$ and $|u_2|=1$, we deform the contours of integration $|u_1|=1-\epsilon$ and $|u_2|=1-\epsilon$ to 
$|u_1|=1$ and $|u_2|=1$ and compare with Eq.~\ref{GasLiquiduintegral2} which gives the result.
\end{proof}

\subsubsection{Double Contour integrals from the contributions from $H$}

Let
\begin{equation}
\begin{split}
&\psi_{\eps_1,\eps_2}(a,x_1,x_2,y_1,y_2) = \frac{-\mathrm{i}^{\eps_1+\eps_2+1} (1+a^2)^2 }{(2 \pi \mathrm{i})^2} \int _{\Gamma_1} \frac{du_1}{u_1} \int_{\Gamma_1} \frac{dv_2}{v_2} \frac{ F_{\frac{x_2}{2}}(\nu (u_1)) F_{\frac{y_1}{2}} (\nu(v_2))}{u_1^{\frac{x_1-1}{2}} v_2^{\frac{y_2-1}{2}}}\\
&\times  \left( \sum_{\gamma_1,\gamma_2=0}^1 (-1)^{\eps_1 \eps_2+\gamma_1(1 +\eps_2)+\gamma_2(1+\eps_1)}  s(-\mathrm{i}u_1)^{\gamma_1} s(-\mathrm{i} v_2)^{\gamma_2} \mathtt{y}_{\gamma_1,\gamma_2}^{\eps_1,\eps_2} (-\mathrm{i}u_1 ,- \mathrm{i} v_2)u_1^{\eps_1} v_2^{\eps_2}  \right).
\end{split}
\end{equation}
Here, we prove the following lemma
\begin{lemma}\label{sym:lem:f3andf4}

For $\eps_1,\eps_2 \in \{0,1\}$ with $x \in \mathtt{W}_{\eps_1}$ and $y \in \mathtt{B}_{\eps_2}$,  we have
\begin{equation} \label{dc:lem:Htof3}
\begin{split}
&\frac{1}{(2\pi \mathrm{i})^4} \int_{\Gamma_{1-\epsilon}} \frac{dw_1}{w_1}  \int_{\Gamma_{1-\epsilon}} \frac{dw_2}{w_2} \int_{\Gamma_{1-\epsilon}} \frac{db_1}{b_1}\int_{\Gamma_{1-\epsilon}} \frac{d b_2}{b_2}
 \frac{H(a,w_1,w_2,b_1,b_2)}{w_1^{x_1} w_2^{x_2} b_1^{y_1}b_2^{y_2}} \\
&= \psi_{\eps_1,\eps_2}(a,x_1,x_2,y_1,y_2)- \mathcal{B}_{\eps_1,\eps_2}(a,x_1,x_2,y_1,y_2),
\end{split}
\end{equation}
where $\mathcal{B}{\eps_1,\eps_2}(a,x_1,x_2,y_1,y_2)$ is given in~\eqref{res:eq:B}.

\end{lemma}

We shall split the proof of this lemma into several smaller computations.  Some of these computations are dependent on the relations between the coefficients $\mathtt{y}_{\gamma_1,\gamma_2}^{\eps_1,\eps_2}$ for $\gamma_1,\gamma_2,\eps_1,\eps_2\in \{0,1\}$.  We split up the entries of $H(a,w_1,w_2,b_1,b_2)$ into its  vertex types. 

\begin{lemma}\label{sym:lem:split}

Under the change of variables $(w_1^2,w_2^2)=(u_1,u_2)$ and $(b_1^2,b_2^2)=(v_1,v_2)$, the entries of $H(a,w_1,w_2,b_1,b_2)$ are given by
\begin{equation}\label{sym:H:eq:split}
\begin{split}
\sum_{i=1}^{2}\frac{(-1)^{i+1}}{ \tilde{c} (u_1,u_2) \tilde{c} (v_1,v_2)} 
\left(
\begin{array}{c}
\sum_{\gamma_1,\gamma_2=0}^1B_{\gamma_1,\gamma_1}^a(u_1,u_2) B_{\gamma_2,\gamma_2}^a(v_2,v_1) S^i_{\gamma_1,\gamma_2}(-u_1 \mathrm{i}, -v_2 \mathrm{i}) u_1^{\gamma_1+1/2} v_2^{\gamma_2+1/2} \\
\sum_{\gamma_1,\gamma_2=0}^1B_{\gamma_1,\gamma_1}^a(u_1,u_2) B_{\gamma_2,1-\gamma_2}^a (v_2,v_1) S^i_{ \gamma_1,\gamma_2}(-u_1 \mathrm{i}, -v_2 \mathrm{i}) u_1^{\gamma_1+1/2} v_2^{\gamma_2+1/2} \\
\sum_{\gamma_1,\gamma_2=0}^1B_{\gamma_1,1-\gamma_1}^a(u_1,u_2) B_{\gamma_2,\gamma_2}^a (v_2,v_1)S^i_{ \gamma_1,\gamma_2}(-u_1 \mathrm{i}, -v_2 \mathrm{i}) u_1^{\gamma_1+1/2} v_2^{\gamma_2+1/2} \\
\sum_{\gamma_1,\gamma_2=0}^1B_{\gamma_1,1-\gamma_1}^a(u_1,u_2) B_{\gamma_2,1-\gamma_2}^a (v_2,v_1) S^i_{ \gamma_1,\gamma_2}(-u_1 \mathrm{i}, -v_2 \mathrm{i}) u_1^{\gamma_1+1/2} v_2^{\gamma_2+1/2} 
\end{array}
\right),
\end{split}
\end{equation}
where the $(2 \eps_1 +2\eps_2+1)^{{th}}$ row are the terms in $\mathtt{W}_{\eps_1}\times\mathtt{B}_{\eps_2}$.
\end{lemma}
The equation~\eqref{sym:H:eq:split} has a much more compact form, namely, we write the $(2\eps_1+\eps_2+1)^{th}$ row of~\eqref{sym:H:eq:split} as
\begin{equation}
\sum_{i=1}^{2}\frac{(-1)^{i+1}}{ \tilde{c} (u_1,u_2) \tilde{c} (v_1,v_2)}\sum_{\gamma_1,\gamma_2=0}^1B_{\gamma_1,h(\gamma_1,\eps_1)}^a(u_1,u_2) B_{\gamma_2,h(\gamma_2,\eps_2)}^a (v_2,v_1) S^i_{ \gamma_1,\gamma_2}(-u_1 \mathrm{i}, -v_2 \mathrm{i}) u_1^{\gamma_1+1/2} v_2^{\gamma_2+1/2}. \\
\end{equation}

\begin{proof}
We apply the change of variables to $w_i\mapsto \sqrt{u_i}$ and $b_i \mapsto \sqrt{v_i}$ for $i \in \{1,2\}$ to $H$. We find that
\begin{equation}
H(a,\sqrt{u_1},\sqrt{u_2},\sqrt{v_1},\sqrt{v_2}) = \sum_{\gamma_1,\gamma_2 =0}^1 \sum_{\substack{ (x_1,0) \in \mathtt{W}_{\gamma_1} \\ (0,y_2) \in \mathtt{B}_{\gamma_2}}} \frac{\tilde{s}_{\gamma_1}^a(u_1,u_2) \tilde{s}_{\gamma_2}^a (v_2,v_1) }{\tilde{c} (u_1,u_2)\tilde{c}(v_1,v_2)} K^{-1}_{a,1}((x_1,0),(0,y_2))u_1^{x_1/2} v_2^{y_2/2}.
\end{equation}
We rewrite the sum over $(x_1,0) \in \mathtt{W}_i$ and $(0,y_2)\in\mathtt{B}_j$.  We obtain 
\begin{equation}
\begin{split}
H(a,\sqrt{u_1},\sqrt{u_2},\sqrt{v_1},\sqrt{v_2})& = \sum_{\gamma_1,\gamma_2 =0}^1 \frac{\tilde{s}^a_{\gamma_1}(u_1,u_2) \tilde{s}^a_{\gamma_2} (v_2,v_1) }{\tilde{c} (u_1,u_2)\tilde{c}(v_1,v_2)}\\
&\times \sum_{p,q=0}^{2m+1} K^{-1}( (4p+2 \gamma_1+1,0) ,(0,4q +2 \gamma_2+1)) u_1^{2p+\gamma_1+1/2} v_2^{2q + \gamma_2+1/2}.
\end{split}
\end{equation}
We now rewrite the second sum in the above equation using the definition of $T^{a,1}_{\gamma_1, \gamma_2}$ given in~\eqref{bgf:Teded}.  We obtain
\begin{equation}
H(a,\sqrt{u_1},\sqrt{u_2},\sqrt{v_1},\sqrt{v_2}) = \sum_{\gamma_1,\gamma_2 =0}^1 \frac{\tilde{s}_{\gamma_1}^a(u_1,u_2) \tilde{s}_{\gamma_2}^a (v_2,v_1) }{\tilde{c} (u_1,u_2)\tilde{c}(v_1,v_2)} T^{a,1}_{\gamma_1,\gamma_2} (-\mathrm{i} u_1 , -\mathrm{i}v_2) u_1^{\gamma_1+1/2} v_2^{\gamma_2+1/2}.
\end{equation}
Using Lemma~\ref{bgf:lemma:corners} and the decomposition of $T^{a,1}_{\gamma_1,\gamma_2}$ into $S^1_{\gamma_1,\gamma_2}$ and $S^2_{\gamma_1,\gamma_2}$ given in~\eqref{bgf:eq:decompofT}, we obtain
\begin{equation}
H(a,\sqrt{u_1},\sqrt{u_2},\sqrt{v_1},\sqrt{v_2}) =\sum_{i=0}^1 \sum_{\gamma_1,\gamma_2 =0}^1 \frac{\tilde{s}_{\gamma_1}^a(u_1,u_2) \tilde{s}^a_{\gamma_2} (v_2,v_1) }{\tilde{c} (u_1,u_2)\tilde{c}(v_1,v_2)}(-1)^{i} S^{i+1}_{\gamma_1,\gamma_2} (-\mathrm{i} u_1 , -\mathrm{i}v_2) u_1^{\gamma_1+1/2} v_2^{\gamma_2+1/2}.
\end{equation}

We now use the following claim which is proved later.
\begin{claim}
\label{ovl:claim:exp}
For $\mathtt{x}_{2i+j+1} \in \mathtt{W}_i \times \mathtt{B}_j$, we expand
\begin{equation}
 \sum_{\gamma_1,\gamma_2=0}^1  \frac{\tilde{s}^a_{\gamma_1} (u_1,u_2) \tilde{s}^a_{\gamma_2}(v_2,v_1)}{\tilde{c}(u_1,u_2) \tilde{c}(v_1,v_2)} \mathtt{x}_{2\gamma_1+\gamma_2+1} u_1^{\gamma_1} v_2^{\gamma_2} \sqrt{ u_1 v_2}
\end{equation}
into four terms in $\mathtt{W}_{\eps_1} \times \mathtt{B}_{\eps_2}$ for $\eps_1, \eps_2 \in \{0,1\}$ with the $\mathtt{W}_{\eps_1} \times \mathtt{B}_{\eps_2}$ entry given by
\begin{equation}
\begin{split} \label{sym:lem:eq:gensplit}
&\sum_{\gamma_1,\gamma_2=0}^1 \frac{B^a_{\gamma_1,h(\eps_1,\gamma_1)}(u_1,u_2) B^a_{\gamma_2,h(\eps_2,\gamma_2)}(v_2,v_1) }
{\tilde{c}(u_1,u_2)\tilde{c}(v_2,v_1)} \mathtt{x}_{2 \gamma_1+\gamma_2+1}u_1^{\gamma_1} v_2^{\gamma_2} \sqrt{u_1 v_2}.
\end{split}
\end{equation}
\end{claim}
By setting $\mathtt{x}_{2\gamma_1+\gamma_2+1}$ to be $S^{i+1}_{\gamma_1,\gamma_2}$ in the above claim, we obtain the lemma. 

\end{proof}

\begin{proof}[Proof of Claim~\ref{ovl:claim:exp}]
We have that $\tilde{c}(u_1,u_2)$ does not change the vertex type, that is, for $(u_1,u_2) \in  \mathtt{W}_0$, we still have $\tilde{c}(u_1,u_2) \in \mathtt{W}_0$.  We also have that $\sqrt{u_1 v_2}$ does not change the vertex type either.  
We expand 
\begin{equation}
\tilde{s}_{\gamma_1}^a (u_1,u_2) \tilde{s}_{\gamma_2}^a(v_2,v_1) \mathtt{x}_{2\gamma_1+\gamma_2+1} u_1^{\gamma_1} v_2^{\gamma_2} \sqrt{u_1} \sqrt{v_2}
\end{equation}
in terms of entries in $\mathtt{W}_{\eps_1} \times \mathtt{B}_{\eps_2}$ in the following way: to find the entry of the form $\mathtt{W}_{0} \times \mathtt{B}_0$, we split both of the  $\tilde{s}^a_{\gamma}$ terms (that is $\tilde{s}^a_{\gamma}(u_1,u_2)$ and $\tilde{s}^a_{\gamma}(v_2,v_1)$) using~\eqref{ovl:eqn:ssplit} into two parts with one part changing the vertex type and the other part not changing the vertex type.  
For $\mathtt{x}_1$, we do not change the vertex type of both vertices and hence 
\begin{equation}
B_{0, 0}^a(u_1,u_2 ) B_{0,0}^a(v_2,v_1)\sqrt{ u_1}\sqrt{ v_2} \mathtt{x}_1 \in \mathtt{W}_0 \times \mathtt{B}_0.
\end{equation}
 For $\mathtt{x}_2$, we only need to change the black vertex type because $\mathtt{x}_2 \in \mathtt{W}_0 \times \mathtt{B}_1$ and hence
\begin{equation}
B_{0, 0}^a(u_1,u_2 ) B_{0,1}^a(v_2,v_1)  \sqrt{u_1}\sqrt{ v_2} v_2\mathtt{x}_2 \in \mathtt{W}_0 \times \mathtt{B}_0.
\end{equation}
 For $\mathtt{x}_3$, we only need to change the white vertex type because $\mathtt{x}_3 \in \mathtt{W}_1 \times \mathtt{B}_0$ and hence
\begin{equation}
B_{0, 1}^a(u_1,u_2 ) B_{0,0}^a(v_2,v_1) \sqrt{u_1} \sqrt{v_2} u_1\mathtt{x}_3 \in \mathtt{W}_0 \times \mathtt{B}_0.
\end{equation}
Finally, for $\mathtt{x}_4$, we need to change both vertices because $\mathtt{x}_4 \in \mathtt{W}_1 \times \mathtt{B}_1$ and hence
\begin{equation}
B_{0, 1}^a(u_1,u_2 ) B_{0,1}^a(v_2,v_1) \sqrt{u_1} \sqrt{v_2} u_1 v_2\mathtt{x}_4 \in \mathtt{W}_0 \times \mathtt{B}_0.
\end{equation}
We sum the above four equations to verify ~\eqref{sym:lem:eq:gensplit} for $\eps_1=\eps_2=0$.  The other three pairs of vertex types can be evaluated in a similar fashion and noting
\begin{equation}
B_{0,1}^a (u_1,u_2)=B_{1,1}^a (u_1,u_2).
\end{equation}
\end{proof}

We now state and prove two technical lemmas which describe relations between the coefficients $\mathtt{y}_{\gamma_1,\gamma_2}^{\eps_1,\eps_2}(-\mathrm{i}u_1,-\mathrm{i}u_2)$  for $\gamma_1,\gamma_2,\eps_1,\eps_2\in \{0,1\}$.  The proof of these lemmas require polynomial computer algebra.
These lemmas are key to finding a good expression for the double contour integral formula for extracting coefficients of $H(a,w_1,w_2,b_1,b_2)$.  
Before doing so,
let 
\begin{equation}
\begin{split}
&\mathtt F^{\delta_1,\delta_2}(u_1,u_2,\eps_1,\eps_2,\gamma_1,\gamma_2,\kappa_1,\kappa_2) = (-\mathtt{i}^{\gamma_1+\gamma_2+1}) (-1)^{\delta_2\gamma_1+\delta_1\gamma_2-\gamma_1 \gamma_2} \mathtt{y}_{\delta_1,\delta_2}^{\gamma_1,\gamma_2} (-\mathrm{i} u_1, -\mathrm{i}u_2) u_1^{\gamma_1} u_2^{\gamma_2} \\
& \times \prod_{j=1}^2\left( \left[ 1+h(\gamma_j,\eps_j)-(2a^{2\gamma_j})^{1-h(\eps_j,\gamma_j)} -(a \mathrm{i}(u_j-u_j^{-1}))^{h(\gamma_j,\eps_j)} -(-(1+a^2))^{1-h(\eps_j,\gamma_j)}\right](1-\kappa_j) \right.\\
&\left.- \kappa_j(1-h(\gamma_j,\eps_j))(1+a^2)^{1-h(\gamma_j,\eps_j)}\right).
\end{split}
\end{equation}
for $\delta_1,\delta_2,\eps_1,\eps_2,\gamma_1,\gamma_2,\kappa_1,\kappa_2 \in \{0,1\}$.

\begin{lemma}\label{sym:H:lem:rel1}
For $\eps_1,\eps_2 \in \{0,1\}$, we have
\begin{equation}\label{sym:H:eq:rel1}
\begin{split}
&\sum_{\gamma_1,\gamma_2=0}^1 (-\mathtt{i}^{\gamma_1+\gamma_2+1}) (-1)^{\gamma_1+\gamma_2-\gamma_1 \gamma_2} \mathtt{y}_{1,1}^{\gamma_1,\gamma_2} (-\mathrm{i} u_1, -\mathrm{i}u_2) u_1^{\gamma_1} u_2^{\gamma_2} \\
& \times \prod_{j=1}^2 \left[ 1+h(\gamma_j,\eps_j)-(2a^{2\gamma_j})^{1-h(\eps_j,\gamma_j)} -(a \mathrm{i}(u_j-u_j^{-1}))^{h(\gamma_j,\eps_j)} -(-(1+a^2)\mu(-\mathrm{i}u_j))^{1-h(\eps_j,\gamma_j)}\right]\\
&=\sum_{\kappa_1,\kappa_2=0}^1 (-\mathtt{i}^{\eps_1+\eps_2+1}) (-1)^{\kappa_1(1+\eps_2)+\kappa_2(1+\eps_1)+\eps_1 \eps_2}(1+a^2)^2u_1^{\eps_1} u_2^{\eps_2} 
\mathtt{y}_{\kappa_1,\kappa_2}^{\eps_1,\eps_2} (-\mathrm{i} u_1, -\mathrm{i}u_2) s(-\mathrm{i} u_1)^{\kappa_1} s(-\mathrm{i} u_2)^{\kappa_2}.
\end{split}
\end{equation}

\end{lemma}

The left side of Eq.~\eqref{sym:H:eq:rel1} has a more compact form given by
\begin{equation}\label{sym:H:eq:rel1alt}
\begin{split}
&\frac{1}{4}\sum_{\gamma_1,\gamma_2=0}^1
B_{\gamma_1,h(\gamma_1,\eps_1)}^a\left(u_1,\left(-\frac{\mu(-\mathrm{i}u_1)}{c(u_1+1/u_1)} \right)^{-1}\right) B_{\gamma_2,h(\gamma_2,\eps_2)}^a \left(v_2,\left(-\frac{\mu(-\mathrm{i}u_2)}{c(u_2+1/u_2)} \right)^{-1}\right)\\
&\hspace{10mm} \times  S^1_{ \gamma_1,\gamma_2}(-u_1 \mathrm{i}, -u_2 \mathrm{i}) u_1^{\gamma_1} u_2^{\gamma_2}, 
\end{split}
\end{equation}
which is seen by using the formula for $S^1_{\eps_1,\eps_2}$ given in Eq.~\eqref{bgf:lempf:S} and using Eqs.~\eqref{ovl:def:B0} and~\eqref{ovl:def:B1}.
\begin{proof}
First, notice that $\mathtt F^{1,1}(u_1,u_2,\eps_1,\eps_2,\gamma_1,\gamma_2,\kappa_1,\kappa_2)$ is the coefficient of $s(-\mathrm{i}u_1)^{\kappa_1} s(-\mathrm{i}u_2)^{\kappa_2}$ in each summand with respect to $\gamma_1$ and $\gamma_2$ of the left side of~\eqref{sym:H:eq:rel1} with respect, that is, the left side of~\eqref{sym:H:eq:rel1} is equal to
\begin{equation}\label{sym:H:lemproof:rel1}
\sum_{\kappa_1,\kappa_2=0}^1\sum_{\gamma_1,\gamma_2=0}^1F^{1,1}(u_1,u_2,\eps_1,\eps_2,\gamma_1,\gamma_2,\kappa_1,\kappa_2)s(-\mathrm{i}u_1)^{\kappa_1} s(-\mathrm{i} u_2)^{\kappa_2}.
\end{equation} 
  The following four relations can be checked (using computer algebra) by verifying that the left side is equal the right side:
\begin{equation}
\sum_{\gamma_1,\gamma_2=0}^1\mathtt F^{1,1}(u_1,u_2,\eps_1,\eps_2,\gamma_1,\gamma_2,0,0)=-\mathrm{i}^{\eps_1+\eps_2+1}(-1)^{\eps_1\eps_2} (1+a^2)^2u_1^{\eps_1} u_2^{\eps_2} \mathtt{y}_{0,0}^{\eps_1,\eps_2} (-\mathrm{i} u_1,-\mathrm{i} u_2) ,
\end{equation}
\begin{equation}
\sum_{\gamma_1,\gamma_2=0}^1\mathtt F^{1,1}(u_1,u_2,\eps_1,\eps_2,\gamma_1,\gamma_2,0,1)=-\mathrm{i}^{\eps_1+\eps_2+1} (-1)^{1+\eps_1+\eps_1\eps_2}(1+a^2)^2u_1^{\eps_1} u_2^{\eps_2} \mathtt{y}_{0,1}^{\eps_1,\eps_2} (-\mathrm{i} u_1,-\mathrm{i} u_2),
\end{equation}
\begin{equation}
\sum_{\gamma_1,\gamma_2=0}^1\mathtt F^{1,1}(u_1,u_2,\eps_1,\eps_2,\gamma_1,\gamma_2,1,0)=-\mathrm{i}^{\eps_1+\eps_2+1} (-1)^{1+\eps_2+\eps_1\eps_2}(1+a^2)^2u_1^{\eps_1} u_2^{\eps_2} \mathtt{y}_{1,0}^{\eps_1,\eps_2} (-\mathrm{i} u_1,-\mathrm{i} u_2) 
\end{equation}
and
\begin{equation}
\sum_{\gamma_1,\gamma_2=0}^1\mathtt F^{1,1}(u_1,u_2,\eps_1,\eps_2,\gamma_1,\gamma_2,1,1)=-\mathrm{i}^{\eps_1+\eps_2+1} (-1)^{\eps_1+\eps_2+\eps_1\eps_2}(1+a^2)^2u_1^{\eps_1} u_2^{\eps_2} \mathtt{y}_{1,1}^{\eps_1,\eps_2} (-\mathrm{i} u_1,-\mathrm{i} u_2).
\end{equation}
Using the above four equations we have that Eq.~\eqref{sym:H:lemproof:rel1} is equal to the right side of~\eqref{sym:H:eq:rel1} as required.
\end{proof}

\begin{lemma} \label{sym:H:lem:rel2}
For $\eps_1,\eps_2 \in \{0,1\}$, we have
\begin{equation} \label{sym:H:eq:rel2}
\begin{split} 
&\sum_{\delta_1,\delta_2=0}^1\sum_{\gamma_1,\gamma_2=0}^1 (-\mathtt{i}^{\gamma_1+\gamma_2+1}) (-1)^{\delta_2\gamma_1+\delta_1\gamma_2-\gamma_1 \gamma_2} \frac{\mathtt{y}_{\delta_1,\delta_2}^{\gamma_1,\gamma_2} (-\mathrm{i} u_1, -\mathrm{i}u_2) u_1^{\gamma_1} u_2^{\gamma_2}}{ s(-\mathrm{i} u_1)^{1-\delta_1} s(-\mathrm{i} u_2)^{1-\delta_2}} \chi_{2\delta_1+\delta_2+1}(-\mathrm{i}u_1,-\mathrm{i}u_2)\\
& \times \prod_{j=1}^2 \left[ 1+h(\gamma_j,\eps_j)-(2a^{2\gamma_j})^{1-h(\eps_j,\gamma_j)} -(a \mathrm{i}(u_j-u_j^{-1}))^{h(\gamma_j,\eps_j)} -(-(1+a^2)\mu(-\mathrm{i}u_j))^{1-h(\eps_j,\gamma_j)}\right]\\
&=4\sum_{\kappa_1,\kappa_2=0}^1 (-\mathtt{i}^{\eps_1+\eps_2+1}) (-1)^{\kappa_1(1+\eps_2)+\kappa_2(1+\eps_1)+\eps_1 \eps_2}(1+a^2)^2u_1^{\eps_1} u_2^{\eps_2} 
\mathtt{y}_{\kappa_1,\kappa_2}^{\eps_1,\eps_2} (-\mathrm{i} u_1, -\mathrm{i}u_2) \\
&\times s(-\mathrm{i} u_1)^{\kappa_1} s(-\mathrm{i} u_2)^{\kappa_2}\lambda_4^m(-\mathrm{i}u_1,-\mathrm{i}u_2).
\end{split}
\end{equation}

\end{lemma}
The left hand-side of the equation~\eqref{sym:H:eq:rel2} has a more compact form given by
\begin{equation}\label{sym:H:eq:rel2alt}
\begin{split}
&\sum_{\gamma_1,\gamma_2=0}^1B_{\gamma_1,h(\gamma_1,\eps_1)}^a\left(u_1,\left(-\frac{\mu(-\mathrm{i}u_1)}{c(u_1+1/u_1)} \right)^{-1}\right) B_{\gamma_2,h(\gamma_2,\eps_2)}^a \left(v_2,\left(-\frac{\mu(-\mathrm{i}u_2)}{c(u_2+1/u_2)} \right)^{-1}\right) \\
& \hspace{10mm}\times S^2_{ \gamma_1,\gamma_2}(-u_1 \mathrm{i}, -u_2 \mathrm{i})
 u_1^{\gamma_1} u_2^{\gamma_2}, 
\end{split}
\end{equation}
which is seen by using the formula for $S_{\eps_1,\eps_2}^2$ given in Eq.~\eqref{bgf:lempf:S2} and using Eqs.~\eqref{ovl:def:B0} and~\eqref{ovl:def:B1}.

\begin{proof}
The proof is based on expansions of relations primarily using computer algebra.
In the relations below, we set $\chi_{2 \delta_1+\delta_2+1}= \chi_{2 \delta_1+\delta_2+1}(-\mathrm{i}u_1,-\mathrm{i}u_2)$ and $\lambda_{2\delta_1+\delta_2+1}^m=\lambda_{2\delta_1+\delta_2+1}(-\mathrm{i}u_1,-\mathrm{i}u_2)^m$. 
Notice that the left side of~\eqref{sym:H:eq:rel2} is equal to   
\begin{equation} \label{sym:lemproof:rel2}
\sum_{\kappa_1,\kappa_2=0}^1\sum_{\delta_1,\delta_2=0}^1 \sum_{\gamma_1,\gamma_2=0}^1 \mathtt F^{\delta_1,\delta_2}(u_1,u_2,\eps_1,\eps_2,\gamma_1,\gamma_2,\kappa_1,\kappa_2) \frac{ s(-\mathrm{i} u_1)^{\kappa_1} s(-\mathrm{i} u_2)^{\kappa_2}}{ s(-\mathrm{i} u_1)^{1-\delta_1} s(-\mathrm{i} u_2)^{1-\delta_2} }  \chi_{2\delta_1+\delta_2+1}.\\
\end{equation}
To compute the above expression, we first evaluate
\begin{equation} \label{sym:lemproof:rel2a}
 \sum_{\gamma_1,\gamma_2=0}^1 \mathtt F^{\delta_1,\delta_2}(u_1,u_2,\eps_1,\eps_2,\gamma_1,\gamma_2,\kappa_1,\kappa_2) \\
\end{equation}
for each possible pair of values of $(\delta_1,\delta_2)$ where $\delta_1,\delta_2 \in \{0,1\}$.    For $(\delta_1,\delta_2)=(0,0)$ in~\eqref{sym:lemproof:rel2a}, we have
\begin{equation}\label{sym:lemproof:rel3a}
\begin{split}
&\sum_{\gamma_1,\gamma_2=0}^1\mathtt F^{0,0}(u_1,u_2,\eps_1,\eps_2,\gamma_1,\gamma_2,\kappa_1,\kappa_2)=\mathrm{i}^{1+\eps_1+\eps_2}
 (-1)^{1+\eps_1+\eps_2+\eps_1 \eps_2+\kappa_2(1+\eps_1)+\kappa_1(1+\eps_2)}\\ &\times  s(-\mathrm{i}u_1)^{2(1-\kappa_1)} s(-\mathrm{i}u_2)^{2(1-\kappa_2)} u_1^{\eps_1} u_2^{\eps_2}
 (1+a^2)^2 \mathtt{y}_{1-\kappa_1,1-\kappa_2}^{\eps_1,\eps_2}(-\mathrm{i}u_1,-\mathrm{i}u_2),
\end{split}
\end{equation}  
where $\kappa_1,\kappa_2,\eps_1,\eps_2 \in \{0,1\}$.
For $(\delta_1,\delta_2)=(0,1)$ in~\eqref{sym:lemproof:rel2a}, we have
\begin{equation}\label{sym:lemproof:rel3b}
\begin{split}
&\sum_{\gamma_1,\gamma_2=0}^1\mathtt F^{0,1}(u_1,u_2,\eps_1,\eps_2,\gamma_1,\gamma_2,\kappa_1,\kappa_2)
=\mathrm{i}^{1+\eps_1+\eps_2}
 (-1)^{1+\eps_2+\eps_1 \eps_2+\kappa_2(1+\eps_1)+\kappa_1(1+\eps_2)}\\&\times s(-\mathrm{i}u_1)^{2(1-\kappa_1)}  u_1^{\eps_1} u_2^{\eps_2}
 (1+a^2)^2 \mathtt{y}_{1-\kappa_1,\kappa_2}^{\eps_1,\eps_2}(-\mathrm{i}u_1,-\mathrm{i}u_2),
\end{split}
\end{equation}  
where $\kappa_1,\kappa_2 ,\eps_1,\eps_2\in \{0,1\}$.
For $(\delta_1,\delta_2)=(1,0)$ in~\eqref{sym:lemproof:rel2a}, we have
\begin{equation}\label{sym:lemproof:rel3c}
\begin{split}
&\sum_{\gamma_1,\gamma_2=0}^1 \mathtt F^{1,0}(u_1,u_2,\eps_1,\eps_2,\gamma_1,\gamma_2,\kappa_1,\kappa_2)
=\mathrm{i}^{1+\eps_1+\eps_2}
 (-1)^{1+\eps_1+\eps_1 \eps_2+\kappa_2(1+\eps_1)+\kappa_1(1+\eps_2)}\\&\times s(-\mathrm{i}u_2)^{2(1-\kappa_2)}  u_1^{\eps_1} u_2^{\eps_2}
 (1+a^2)^2 \mathtt{y}_{\kappa_1,1-\kappa_2}^{\eps_1,\eps_2}(-\mathrm{i}u_1,-\mathrm{i}u_2),
\end{split}
\end{equation}  
where $\kappa_1,\kappa_2 ,\eps_1,\eps_2\in \{0,1\}$.
For $(\delta_1,\delta_2)=(1,1)$ in~\eqref{sym:lemproof:rel2a}, we have
\begin{equation}\label{sym:lemproof:rel3d}
\begin{split}
&\sum_{\gamma_1,\gamma_2=0}^1 F^{1,1}(u_1,u_2,\eps_1,\eps_2,\gamma_1,\gamma_2,\kappa_1,\kappa_2)=\mathrm{i}^{1+\eps_1+\eps_2}
 (-1)^{1+\eps_1 \eps_2+\kappa_2(1+\eps_1)+\kappa_1(1+\eps_2)}\\&\times u_1^{\eps_1} u_2^{\eps_2}
 (1+a^2)^2 \mathtt{y}_{\kappa_1,\kappa_2}^{\eps_1,\eps_2}(-\mathrm{i}u_1,-\mathrm{i}u_2),
\end{split}
\end{equation}  
where $\kappa_1,\kappa_2 ,\eps_1,\eps_2\in \{0,1\}$.
We sum up~\eqref{sym:lemproof:rel3a},~\eqref{sym:lemproof:rel3b},~\eqref{sym:lemproof:rel3c} and~\eqref{sym:lemproof:rel3d} to find~\eqref{sym:lemproof:rel2a} which we substitute back into~\eqref{sym:lemproof:rel2}. Then,~\eqref{sym:lemproof:rel2} becomes 
\begin{equation}\label{sym:lemproof:rel4}
\begin{split}
&\sum_{\delta_1,\delta_2=0}^1 \sum_{\kappa_1,\kappa_2=0}^1 \mathrm{i}^{1+\eps_1+\eps_2}
 (-1)^{1+(1-\delta_1)\eps_1+(1-\delta_2)\eps_2+\eps_1 \eps_2+\kappa_2(1+\eps_1)+\kappa_1(1+\eps_2)} u_1^{\eps_1} u_2^{\eps_2}(1+a^2)^2\\
& \times s(-\mathrm{i}u_1)^{2(1-\delta_1)(1-\kappa_1)+\kappa_1} s(-\mathrm{i}u_2)^{2(1-\delta_2)(1-\kappa_2)+\kappa_2}
\frac{\mathtt{y}^{\eps_1,\eps_2}_{1-h(\kappa_1,\delta_1),1-h(\kappa_2,\delta_2)} (-\mathrm{i} u_1, -\mathrm{i}u_2) }{ s(-\mathrm{i} u_1)^{1-\delta_1} s(-\mathrm{i} u_2)^{1-\delta_2}} \chi_{2\delta_1+\delta_2+1}.
\end{split}
\end{equation}
Applying the definition of $\chi_{2\delta_1+\delta_2+1}$ for $\delta_1,\delta_2 \in \{0,1\}$, which is given in  Lemma~\ref{bgf:lem:Gcompstatement}, into the above equation, 
we find that~\eqref{sym:lemproof:rel2}  becomes
\begin{equation}\label{sym:lemproof:rel5}
\begin{split}
&\sum_{\alpha_1,\alpha_2=0}^1\sum_{\delta_1,\delta_2=0}^1 \sum_{\kappa_1,\kappa_2=0}^1 \mathrm{i}^{1+\eps_1+\eps_2}
 (-1)^{1+(1-\delta_1)\eps_1+(1-\delta_2)\eps_2+\eps_1 \eps_2+\kappa_2(1+\eps_1)+\kappa_1(1+\eps_2)+\alpha_1(1+ \delta_1)+\alpha_2(1+\delta_2)} u_1^{\eps_1} u_2^{\eps_2}\\
&\times (1+a^2)^2 s(-\mathrm{i}u_1)^{(1-\delta_1)(1-2\kappa_1)+\kappa_1} s(-\mathrm{i}u_2)^{(1-\delta_2)(1-2\kappa_2)+\kappa_2}
{\mathtt{y}^{\eps_1,\eps_2}_{1-h(\kappa_1,\delta_1),1-h(\kappa_2,\delta_2)}  (-\mathrm{i} u_1, -\mathrm{i}u_2)} \lambda_{2\alpha_1+\alpha_2+1}^m.
\end{split}
\end{equation}
We proceed by expanding out the first sum~\eqref{sym:lemproof:rel5}  and consider the summand for $(\alpha_1,\alpha_2)=(0,0)$, $(0,1)$, $(1,0)$ and $(1,1)$. This is equivalent to extracting the coefficients of $ \lambda_{1}^m,\lambda_{2}^m,\lambda_{3}^m$ and $\lambda_{4}^m$.  We obtain the following four relations which are  verified using computer algebra 
\begin{equation}\label{sym:lemproof:rel5a}
\begin{split}
&\sum_{\delta_1,\delta_2=0}^1 \sum_{\kappa_1,\kappa_2=0}^1 \mathrm{i}^{1+\eps_1+\eps_2}
 (-1)^{1+(1-\delta_1)\eps_1+(1-\delta_2)\eps_2+\eps_1 \eps_2+\kappa_2(1+\eps_1)+\kappa_1(1+\eps_2)} u_1^{\eps_1} u_2^{\eps_2}(1+a^2)^2\\
&\times s(-\mathrm{i}u_1)^{(1-\delta_1)(1-2\kappa_1)+\kappa_1} s(-\mathrm{i}u_2)^{(1-\delta_2)(1-2\kappa_2)+\kappa_2}
{\mathtt{y}^{\eps_1,\eps_2}_{1-h(\kappa_1,\delta_1),1-h(\kappa_2,\delta_2)}  (-\mathrm{i} u_1, -\mathrm{i}u_2)} \lambda_{1}^m=0
\end{split}
\end{equation}
for all $\eps_1,\eps_2 \in \{0,1\}$,
\begin{equation}\label{sym:lemproof:rel5b}
\begin{split}
&\sum_{\delta_1,\delta_2=0}^1 \sum_{\kappa_1,\kappa_2=0}^1 \mathrm{i}^{1+\eps_1+\eps_2}
 (-1)^{1+(1-\delta_1)\eps_1+(1-\delta_2)\eps_2+\eps_1 \eps_2+\kappa_2(1+\eps_1)+\kappa_1(1+\eps_2)+1+\delta_2} u_1^{\eps_1} u_2^{\eps_2}(1+a^2)^2\\
&\times s(-\mathrm{i}u_1)^{(1-\delta_1)(1-2\kappa_1)+\kappa_1} s(-\mathrm{i}u_2)^{(1-\delta_2)(1-2\kappa_2)+\kappa_2}
{\mathtt{y}^{\eps_1,\eps_2}_{1-h(\kappa_1,\delta_1),1-h(\kappa_2,\delta_2)}  (-\mathrm{i} u_1, -\mathrm{i}u_2)} \lambda_{2}^m=0
\end{split}
\end{equation}
for all $\eps_1,\eps_2 \in \{0,1\}$,
\begin{equation}\label{sym:lemproof:rel5c}
\begin{split}
&\sum_{\delta_1,\delta_2=0}^1 \sum_{\kappa_1,\kappa_2=0}^1 \mathrm{i}^{1+\eps_1+\eps_2}
 (-1)^{1+(1-\delta_1)\eps_1+(1-\delta_2)\eps_2+\eps_1 \eps_2+\kappa_2(1+\eps_1)+\kappa_1(1+\eps_2)+1+ \delta_1} u_1^{\eps_1} u_2^{\eps_2}(1+a^2)^2\\
&\times s(-\mathrm{i}u_1)^{(1-\delta_1)(1-2\gamma_1)+\gamma_1} s(-\mathrm{i}u_2)^{(1-\delta_2)(1-2\gamma_2)+\gamma_2}
{\mathtt{y}^{\eps_1,\eps_2}_{1-h(\kappa_1,\delta_1),1-h(\kappa_2,\delta_2)}  (-\mathrm{i} u_1, -\mathrm{i}u_2)} \lambda_{3}^m=0
\end{split}
\end{equation}
for all $\eps_1,\eps_2 \in \{0,1\}$
and
\begin{equation}\label{sym:lemproof:rel5d}
\begin{split}
&\sum_{\delta_1,\delta_2=0}^1 \sum_{\kappa_1,\kappa_2=0}^1 \mathrm{i}^{1+\eps_1+\eps_2}
 (-1)^{1+(1-\delta_1)\eps_1+(1-\delta_2)\eps_2+\eps_1 \eps_2+\kappa_2(1+\eps_1)+\kappa_1(1+\eps_2)+1+ \delta_1+1+\delta_2} u_1^{\eps_1} u_2^{\eps_2}(1+a^2)^2\\
& \times s(-\mathrm{i}u_1)^{(1-\delta_1)(1-2\kappa_1)+\kappa_1} s(-\mathrm{i}u_2)^{(1-\delta_2)(1-2\kappa_2)+\kappa_2}
{\mathtt{y}^{\eps_1,\eps_2}_{1-h(\kappa_1,\delta_1),1-h(\kappa_2,\delta_2)}  (-\mathrm{i} u_1, -\mathrm{i}u_2)} \lambda_{4}^m\\
&=4\sum_{\kappa_1,\kappa_2=0}^1 (-\mathtt{i}^{\eps_1+\eps_2+1}) (-1)^{\kappa_1(1+\eps_2)+\kappa_2(1+\eps_1)+\eps_1 \eps_2}(1+a^2)^2u_1^{\eps_1} u_2^{\eps_2} 
\mathtt{y}_{\kappa_1,\kappa_2}^{\eps_1,\eps_2} (-\mathrm{i} u_1, -\mathrm{i}u_2) \\
&\times s(-\mathrm{i} u_1)^{\kappa_1} s(-\mathrm{i} u_2)^{\kappa_2}\lambda_4^m.
\end{split}
\end{equation}
Since summing up Eqs.~\eqref{sym:lemproof:rel5a},~\eqref{sym:lemproof:rel5b},~\eqref{sym:lemproof:rel5c} and~\eqref{sym:lemproof:rel5d} is equal to~\eqref{sym:lemproof:rel2}, we have verified Eq.~\eqref{sym:H:eq:rel2}.
\end{proof}

Using the computational lemmas given above, we now present the proof of Lemma~\ref{sym:lem:f3andf4}. 

\begin{proof}[Proof of Lemma~\ref{sym:lem:f3andf4}]
We first concentrate on verifying~\eqref{dc:lem:Htof3}. The change of variables $(w_1^2,w_2^2)\mapsto (u_1,u_2)$ and $(b_1^2,b_2^2) \mapsto(v_1,v_2)$ doubles each contour of integration but this is compensated by a factor of $1/2^4$ from the change of variables Jacobian.  Thus, the left hand-side of~\eqref{dc:lem:Htof3} is equal to 
\begin{equation} \label{dc:lem:Htof3b}
\begin{split}
&\frac{1}{(2\pi \mathrm{i})^4} \int_{\Gamma_{1-\epsilon}} \frac{du_1}{u_1}  \int_{\Gamma_{1-\epsilon}} \frac{du_2}{u_2} \int_{\Gamma_{1-\epsilon}} \frac{dv_1}{v_1}\int_{\Gamma_{1-\epsilon}} \frac{d v_2}{v_2} \frac{H(a,\sqrt{u_1},\sqrt{u_2},\sqrt{v_1},\sqrt{v_2})}{u_1^{x_1/2} u_2^{x_2/2} v_1^{y_1/2}v_2^{y_2/2}}.
\end{split}
\end{equation}
Using Lemma~\ref{sym:lem:split},~\eqref{dc:lem:Htof3b} is equal to 
\begin{equation}\label{dc:lem:Htof3c}
\begin{split}
&\frac{1}{(2\pi \mathrm{i})^4} \int_{\Gamma_{1-\epsilon}} \frac{du_1}{u_1}  \int_{\Gamma_{1-\epsilon}} \frac{du_2}{u_2} \int_{\Gamma_{1-\epsilon}} \frac{dv_1}{v_1}\int_{\Gamma_{1-\epsilon}} \frac{d v_2}{v_2}\\
&\times \sum_{i=1}^{2} (-1)^{i+1}\frac{ \sum_{\gamma_1,\gamma_2=0}^1B_{\gamma_1,h(\gamma_1,\eps_1)}^a(u_1,u_2) B_{\gamma_2,h(\gamma_2,\eps_2)}^a (v_2,v_1) S^i_{ \gamma_1,\gamma_2}(-u_1 \mathrm{i}, -v_2 \mathrm{i}) u_1^{\gamma_1+1/2} v_2^{\gamma_2+1/2} }{u_1^{x_1/2} u_2^{x_2/2} v_1^{y_1/2}v_2^{y_2/2} \tilde{c} (u_1,u_2) \tilde{c} (v_1,v_2)}\\
\end{split}
\end{equation}
for $(x_1,x_2) \in \mathtt{W}_{\eps_1}$ and $(y_1,y_2)\in\mathtt{B}_{\eps_2}$ and $\eps_1,\eps_2 \in \{0,1\}$.   To evaluate~\eqref{dc:lem:Htof3b}, we consider the following equation:   
 \begin{equation}\label{dc:lem:Htof3ci1}
\begin{split}
&\frac{1}{(2\pi \mathrm{i})^4} \int_{\Gamma_{1-\epsilon}} \frac{du_1}{u_1}  \int_{\Gamma_{1-\epsilon}} \frac{du_2}{u_2} \int_{\Gamma_{1-\epsilon}} \frac{dv_1}{v_1}\int_{\Gamma_{1-\epsilon}} \frac{d v_2}{v_2}\\
&\times\frac{ \sum_{\gamma_1,\gamma_2=0}^1B_{\gamma_1,h(\gamma_1,\eps_1)}^a(u_1,u_2) B_{\gamma_2,h(\gamma_2,\eps_2)}^a (v_2,v_1) S^i_{ \gamma_1,\gamma_2}(-u_1 \mathrm{i}, -v_2 \mathrm{i}) u_1^{\gamma_1+1/2} v_2^{\gamma_2+1/2} }{u_1^{x_1/2} u_2^{x_2/2} v_1^{y_1/2}v_2^{y_2/2} \tilde{c} (u_1,u_2) \tilde{c} (v_1,v_2)}.\\
\end{split}
\end{equation}
For the integral in~\eqref{dc:lem:Htof3ci1}, we make the change of variables $u_2\mapsto u_2^{-1}$ and $v_1\mapsto v_1^{-1}$ to obtain 
 \begin{equation}\label{dc:lem:Htof3ci1a}
\begin{split}
&\frac{1}{(2\pi \mathrm{i})^4} \int_{\Gamma_{1-\epsilon}} \frac{du_1}{u_1}  \int_{\Gamma_{1-\epsilon}} \frac{du_2}{u_2} \int_{\Gamma_{1-\epsilon}} \frac{dv_1}{v_1}\int_{\Gamma_{1-\epsilon}} \frac{d v_2}{v_2}\\
& \times\frac{ \sum_{\gamma_1,\gamma_2=0}^1B_{\gamma_1,h(\gamma_1,\eps_1)}^a(u_1,u_2^{-1}) B_{\gamma_2,h(\gamma_2,\eps_2)}^a (v_2,v_1^{-1}) S^i_{ \gamma_1,\gamma_2}(-u_1 \mathrm{i}, -v_2 \mathrm{i}) u_1^{\gamma_1+1/2} v_2^{\gamma_2+1/2}  u_2^{x_2/2} v_1^{y_1/2}}{u_1^{x_1/2}v_2^{y_2/2} \tilde{c} (u_1,u_2) \tilde{c} (v_1,v_2)}\\
&=\frac{1}{(2\pi \mathrm{i})^2} \int_{\Gamma_{1-\epsilon}} \frac{du_1}{u_1}  \int_{\Gamma_{1-\epsilon}} \frac{dv_2}{v_2}F_{x_2/2}(\nu(u_1)) F_{y_1/2}(\nu(v_2)) \sum_{\gamma_1,\gamma_2=0}^1S^i_{ \gamma_1,\gamma_2}(-u_1 \mathrm{i}, -v_2 \mathrm{i})\\
&\times\frac{B_{\gamma_1,h(\gamma_1,\eps_1)}^a\left(u_1,\left(-\frac{\mu(-\mathrm{i}u_1)}{c(u_1+1/u_1)} \right)^{-1}\right) B_{\gamma_2,h(\gamma_2,\eps_2)}^a \left(v_2,\left(-\frac{\mu(-\mathrm{i}v_2)}{c(v_2+1/v_2)} \right)^{-1}\right)   u_1^{\gamma_1+1/2} v_2^{\gamma_2+1/2}  }{4u_1^{x_1/2}v_2^{y_2/2} },\\
\end{split}
\end{equation}
where we have integrated with respect to $u_2$ and $v_1$ using Lemma~\ref{cuts:lem:F}. When $i=1$ in~\eqref{dc:lem:Htof3ci1a}, using the fact that~\eqref{sym:H:eq:rel1alt} is equal to the formula in Lemma~\ref{sym:H:lem:rel1}, we rewrite part of the integrand in the above equation, and obtain $\psi_{\eps_1,\eps_2}(a,x_1,x_2,y_1,y_2)$.   When $i=2$, using the fact that~\eqref{sym:H:eq:rel2alt} is equal to the formula in Lemma~\ref{sym:H:lem:rel2}, we rewrite part of the integrand in the above equation, and obtain~ $\mathcal{B}_{\eps_1,\eps_2}(a,x_1,x_2,y_1,y_2)$.  Hence, we have verified~\eqref{dc:lem:Htof3}.

\end{proof}

\subsubsection{Double Contour integrals from the contributions from $d_3$}

Here, we find the double contour integral for the coefficients of $d_3(a,w_1,w_2,b_1,b_2)$. For this subsection we set
\begin{equation}
\begin{split} \label{sym:eq:d3:y}
\mathtt{y}_{\eps_1,\eps_2} (u_1,u_2,v_1,v_2)= -\mathrm{i}^{\eps_2+1} B^a_{0,\eps_1}(u_1,u_2) (v_1^{1-\eps_2} +a v_1^{\eps_2}v_2)-\mathrm{i}^{2-\eps_2}u_1v_1 B^a_{1,\eps_1}(u_1,u_2) (a v_1^{\eps_2} +a v_1^{1-\eps_2}v_2).
\end{split}
\end{equation}
We have the following lemma
\begin{lemma}
For $\eps_1,\eps_2 \in \{0,1\}$ with $x \in \mathtt{W}_{\eps_1}$ and $y \in \mathtt{B}_{\eps_2}$,  we have
\begin{equation} \label{dc:lem:d3tof2}
\begin{split}
&\frac{1}{(2\pi \mathrm{i})^4} \int_{\Gamma_{1-\epsilon}} \frac{dw_1}{w_1}  \int_{\Gamma_{1-\epsilon}} \frac{dw_2}{w_2} \int_{\Gamma_{1-\epsilon}} \frac{db_1}{b_1}\int_{\Gamma_{1-\epsilon}} \frac{d b_2}{b_2} \frac{d_3(a,w_1,w_2,b_1,b_2)}{w_1^{x_1} w_2^{x_2} b_1^{y_1}b_2^{y_2}} \\
&= \psi_{\eps_1,\eps_2}(a,x_1,x_2,y_1,y_2). 
\end{split}
\end{equation}
\end{lemma}
\begin{proof}
For the proof, we first extract the terms of each vertex types in the expression $d_3(a,w_1,w_2,b_1,b_2)$.  Once we have the vertex types, we extract coefficients using a quadruple integral which is reduced to a double contour integral formula.  To this double contour integral formula, we apply a  change of variables, and after a computer algebra verification, we obtain $\psi_{\eps_1,\eps_2}(a,x_1,x_2,y_1,y_2)$.

To extract the terms of each vertex type from the expression $d_3(a,w_1,w_2,b_1,b_2)$, it is more convenient to take the change of variables $(w_1^2,w_2^2) = (u_1,u_2)$ and $(b_1^2,b_2^2)=(v_1,v_2)$.  We expand out $d_3$ with these change of variables and write in terms of its vertex types using the same convention used above.
This expansion gives
\begin{equation}
\begin{split}
\frac{1}{(1-u_1^2 v_1^2) \tilde{c}(u_1,u_2) \tilde{c}(v_1,v_2)} 
\left(
\begin{array}{c}
 -\mathrm{i} B_{0,0}^a(u_1,u_2) \frac{\sqrt{u_1}}{\sqrt{v_2}} (v_1+a v_2) + B_{1,1}^a(u_1,u_2) \frac{\sqrt{u_1}}{\sqrt{v_2}}u_1 v_1 (a+v_1 v_2) \\
 B_{0,0}^a(u_1,u_2) \frac{\sqrt{u_1}}{\sqrt{v_2}} (1+av_1 v_2) -\mathrm{i} B_{1,1}^a(u_1,u_2) \frac{\sqrt{u_1}}{\sqrt{v_2}}u_1 v_1 (av_1+ v_2) \\
 -\mathrm{i} B_{0,1}^a(u_1,u_2) \frac{\sqrt{u_1}}{\sqrt{v_2}} (v_1+a v_2) + B_{1,0}^a(u_1,u_2) \frac{\sqrt{u_1}}{\sqrt{v_2}}u_1 v_1 (a+v_1 v_2) \\
 B_{0,1}^a(u_1,u_2) \frac{\sqrt{u_1}}{\sqrt{v_2}} (1+av_1 v_2) -\mathrm{i} B_{1,0}^a(u_1,u_2) \frac{\sqrt{u_1}}{\sqrt{v_2}}u_1 v_1 (av_1+ v_2) 
\end{array}
\right).
\end{split}
\end{equation}

Since we have extracted the vertex types of $d_3(a,w_1,w_2,b_1,b_2)$, we now extract the coefficients using quadruple integrals.  
This means that the left hand-side of~\eqref{dc:lem:d3tof2} is equal to  
\begin{equation}
\begin{split}
&\frac{1}{(2\pi \mathrm{i})^4} \int_{\Gamma_{1-\epsilon}} \frac{du_1}{u_1}  \int_{\Gamma_{1-\epsilon}} \frac{du_2}{u_2} \int_{\Gamma_{1-\epsilon}} \frac{dv_1}{v_1}\int_{\Gamma_{1-\epsilon}} \frac{d v_2}{v_2}
 \frac{ \mathtt{y}_{\eps_1,\eps_2} (u_1,u_2,v_1,v_2)}{ (1-u_1^2 v_1^2) \tilde{c}(u_1,u_2) \tilde{c}(v_1,v_2)u_1^{(x_1-1)/2}u_2^{x_2/2} v_1^{y_1/2} v_2^{(y_2+1)/2}},
\end{split}
\end{equation}
where $\mathtt{y}_{\eps_1,\eps_2} (u_1,u_2,v_1,v_2)$ is given in~\eqref{sym:eq:d3:y}. We now apply the reduction to a double contour integral formula. We have that $(1-u_1^2 v_1^2)$ is not zero for $u_1, v_1 \in \Gamma_{1-\epsilon}$. We integrate with respect to $u_2$ and $v_1$ by applying the residue theorem and find that the above integral is equal to 
\begin{equation}
\frac{1}{(2\pi \mathrm{i})^2} \int_{\Gamma_{1-\epsilon}} \frac{du_1}{u_1}  \int_{\Gamma_{1-\epsilon}} \frac{dv_2}{v_2} \frac{ \mathtt{y}_{\eps_1,\eps_2} \left(u_1,\left(\frac{-\mu(-\mathrm{i}u_1 )}{c(u_1+1/u_1)} \right)^{-1},\left(\frac{-\mu(-\mathrm{i}v_2 )}{c(v_2+1/v_2)} \right)^{-1},v_2\right)}{4\left( 1-u_1^2 \left(\frac{-\mu(-\mathrm{i}v_2 )}{c(v_2+1/v_2)} \right)^{-2} \right)v_2 u_1^{\frac{x_1-1}{2}}u_2^{\frac{y_2-1}{2}}} F_{\frac{x_2}{2}} (\nu(u_1)) F_{\frac{y_1}{2}} (\nu(v_2)).
\end{equation} 
Using computer algebra, we evaluate the integrand in the above expression and find that
\begin{equation}
\frac{ \mathtt{y}_{\eps_1,\eps_2} \left(u_1,\left(\frac{-\mu(-\mathrm{i}u_1 )}{c(u_1+1/u_1)} \right)^{-1},\left(\frac{-\mu(-\mathrm{i}v_2 )}{c(v_2+1/v_2)} \right)^{-1},v_2\right)}{4\left( 1-u_1^2 \left(\frac{-\mu(-\mathrm{i}v_2 )}{c(v_2+1/v_2)} \right)^{-2} \right)v_2}
= -\mathrm{i}^{1+\eps_1+\eps_2}\sum_{\gamma_1,\gamma_2=0}^1 Y^{\eps_1,\eps_2}_{\gamma_1,\gamma_2}(u_1,u_2),
\end{equation}
where $Y^{\eps_1,\eps_2}_{\gamma_1,\gamma_2}(u_1,u_2)$ is given in~\eqref{asfo:eq:Y}.  Therefore, we have verified~\eqref{dc:lem:d3tof2}.

\end{proof}

\subsection{Completion of the proof of Theorem~\ref{thm:Kinverse}} \label{subsection:Completeproof}
We wrap up all our previous computations to prove Theorem~\ref{thm:Kinverse}.
\begin{proof}[Proof of Theorem~\ref{thm:Kinverse}]
From Lemma~\ref{sym:lemma:G01}, to find the formula for $K_{a,1}^{-1}(x,y)$ we can extract coefficients from $\mathcal{S} G^0(a,w_1,w_2,b_1,b_2)$. By noting the signs in the splitting of $G^0$ as given in~\eqref{sym:def:G}, we can use~\eqref{dc:lem:Htof3} and~\eqref{dc:lem:d3tof2}  to find
\begin{equation}
\begin{split}\label{sym:eq:finalB4}
&\frac{1}{(2\pi \mathrm{i})^4} \int_{\Gamma_{1-\epsilon}} \frac{dw_1}{w_1}  \int_{\Gamma_{1-\epsilon}} \frac{dw_2}{w_2} \int_{\Gamma_{1-\epsilon}} \frac{db_1}{b_1}\int_{\Gamma_{1-\epsilon}} \frac{d b_2}{b_2}
 \frac{(-d_3+H)(a,w_1,w_2,b_1,b_2)}{w_1^{x_1} w_2^{x_2} b_1^{y_1}b_2^{y_2}} \\
&=- \mathcal{B}_{\eps_1,\eps_2}(a,x_1,x_2,y_1,y_2).
\end{split}
\end{equation}
We extract coefficients of the remaining three expressions in $\mathcal{S}(-d_3+H)(a,w_1,w_2,b_1,b_2)$ which were not considered above.  This leads to three further double contour integral expressions.  
Since these computations are analogous to the long computation behind~\eqref{sym:eq:finalB4}, we will not list these computations explicitly.  These four double contour integral and Lemma~\ref{dc:lem:gaspart} combined give the formula for 
$K_{a,1}^{-1}((x_1,x_2),(y_1,y_2))$ as written in Theorem~\ref{thm:Kinverse}.

\end{proof}

\subsection{Proof of Corollary~\ref{corollary}} \label{subsection:corollary}
We  verify the formulas which are suitably nice for asymptotic computations that are given in  Corollary~\ref{corollary}. We start with the proof of Eq.~\eqref{B}.

\begin{proof}[Proof of Eq.~\eqref{B}]
From~\eqref{res:eq:B}, we take the change of variables $u_1 \mapsto \mathrm{i} u_1$ and $u_2 \mapsto \mathrm{i} u_2$.  This gives 
\begin{equation} \label{cuts:lemproof:change1}
\begin{split}
&\mathcal{B}_{\eps_1,\eps_2}(a,x_1,x_2,y_1,y_2) =
\frac{\mathrm{i}^{-(x_1+y_2)/2}}{(2 \pi \mathrm{i})^2} \int _{\Gamma_{1-\epsilon}} \frac{du_1}{u_1}\int_{\Gamma_{1-\epsilon}} \frac{du_2}{u_2} \frac{ F_{\frac{x_2}{2}}(\nu (u_1\mathrm{i})) F_{\frac{y_1}{2}} (\nu(u_2\mathrm{i}))}{u_1^{\frac{x_1-1}{2}} u_2^{\frac{y_2-1}{2}}}\lambda_4(u_1,  u_2)^m\\
&\times  (1+a^2)^2 \left( \sum_{\gamma_1,\gamma_2=0}^1 (-1)^{\eps_1+\eps_2+\eps_1 \eps_2+\gamma_1(1 +\eps_2)+\gamma_2(1+\eps_1)}  s(u_1)^{\gamma_1} s( u_2)^{\gamma_2} \mathtt{y}_{\gamma_1,\gamma_2}^{\eps_1,\eps_2} (u_1 , u_2)u_1^{\eps_1} u_2^{\eps_2}  \right),
\end{split}
\end{equation}
where we have deformed the contours of integration so that they are both within the unit circle.
To the above equation, we take the change of variables $u_i = G (\omega_i)$.    To do so, we use Lemma~\ref{cuts:lem:F} to expand out $F_{x_2/2}$ and $F_{y_1/2}$ and also use 
\begin{equation}
\frac{1}{u}	\frac{d u}{d \omega} = -\frac{1}{\sqrt{\omega^2+2c}}.
\end{equation}
Note that the contour of integration in the $\omega$ variables will be clockwise but can be made counter clockwise by compensating a factor of $(-1)$ and deforming to a contour $\Gamma_p$ for $\sqrt{2c}<p<1$.  We also have that
\begin{equation}
\lambda_4( G(\omega_1), G(\omega_2)) = \frac{ \omega_1^2 \omega_2^2 G \left(\omega_1 \right)^2G \left(\omega_2 \right)^2}{ G \left(\omega_1^{-1} \right)^2G \left(\omega_2^{-1} \right)^2}.
\end{equation} 
Putting all the parts together,~\eqref{cuts:lemproof:change1} under the change of variables $u_i = G(\omega_i)$ for $i \in \{1,2\}$ becomes  
\begin{equation}
\begin{split}
&\mathcal{B}_{\eps_1,\eps_2}(a,x_1,x_2,y_1,y_2) =
\frac{\mathrm{i}^{\frac{x_2-x_1+y_1-y_2}{2}}}{(2 \pi \mathrm{i})^2} \int _{\Gamma_{p}} \frac{d\omega_1}{\omega_1}\int_{\Gamma_{p}} \frac{d\omega_2}{\omega_2}  \frac{ \omega_1^{2m} \omega_2^{2m} G \left(\omega_1 \right)^{2m-\frac{x_1-1}{2}}G \left(\omega_2 \right)^{2m-\frac{y_2-1}{2}}}{ G \left(\omega_1^{-1} \right)^{2m-\frac{x_2}{2}}G \left(\omega_2^{-1} \right)^{2m-\frac{y_1}{2}}} \\
&\times \frac{1}{\prod_{j=1}^2 \sqrt{\omega_j^2+2c} \sqrt{\omega_j^{-2}+2c}} \sum_{\gamma_1,\gamma_2=0}^1 (-1)^{\eps_1+\eps_2 +\eps_1 \eps_2 + \gamma_1 (1+\eps_2)+\gamma_2(1+\eps_1) } \\
&\times s\left(G \left(\omega_1 \right) \right) ^{\gamma_1} s\left(G \left(\omega_2 \right) \right) ^{\gamma_2} \mathtt{y}_{\gamma_1,\gamma_2}^{\eps_1,\eps_2} \left( G \left(\omega_1 \right), G \left(\omega_2 \right) \right)G \left(\omega_1 \right)^{\eps_1}G \left(\omega_2 \right)^{\eps_2}.
\end{split}
\end{equation}
In the above equation, we remove a factor of $1/(1-\omega_1^2 \omega_2^2)$ in the expression $\mathtt{y}_{\gamma_1,\gamma_2}^{\eps_1,\eps_2} ( G (\omega_1 ), G (\omega_2 ) )$ and rewrite using~\eqref{cuts:def:x}.  Finally, taking the change of variable $\omega_2 \mapsto \omega_2^{-1}$ and using the definitions given in Eq.~\eqref{asfo:eq:Zstar} and Eq.~\eqref{asfo:lem:H}, gives the result.

\end{proof}
We now prove Eqs.~\eqref{asfo:lemproof:exponential:change3},~\eqref{asfo:lemproof:exponential:change4} and~\eqref{asfo:lemproof:exponential:change5} which will conclude the proof of Corollary~\ref{corollary}.

\begin{proof}[Proof of Eqs.~\eqref{asfo:lemproof:exponential:change3},~\eqref{asfo:lemproof:exponential:change4} and~\eqref{asfo:lemproof:exponential:change5}]

We will only find the formula for $\frac{1}{a}\mathcal{B}_{1-\eps_1,\eps_2}(a^{-1},2n-x_1,x_2,2n-y_1,y_2)$.  By symmetry, this will also lead to a good asymptotic formula for  $\frac{1}{a}\mathcal{B}_{\eps_1,1-\eps_2}(a^{-1},x_1,2n-x_2,y_1,2n-y_2)$ and $\mathcal{B}_{1-\eps_1,1-\eps_2}(a,2n-x_1,2n-x_2,2n-y_1,2n-y_2)$. 

We proceed by evaluating $\frac{1}{a}\mathcal{B}_{1-\eps_1,\eps_2}(a^{-1},2n-x_1,x_2,2n-y_1,y_2)$.  We find that
\begin{equation}\label{asfo:lemproof:exponential:change}
\begin{split}
&\frac{1}{a}\mathcal{B}_{1-\eps_1,\eps_2}(a^{-1},2n-x_1,x_2,2n-y_1,y_2) \\
&=-\frac{\mathrm{i}^{1-\eps_1+\eps_2+1}}{(2 \pi \mathrm{i})^2a} \int _{\Gamma_r} \frac{du_1}{u_1}\int_{\Gamma_r} \frac{du_2}{u_2} \frac{ F_{\frac{x_2}{2}}(\nu (u_1)) F_{\frac{2n-y_1}{2}} (\nu(u_2))}{u_1^{\frac{2n-x_1-1}{2}} v_2^{\frac{y_2-1}{2}}}\lambda_4(-\mathrm{i} u_1, -\mathrm{i} u_2)^m (1+a^2)^2\\
&\times   \sum_{\gamma_1,\gamma_2=0}^1 (-1)^{(1-\eps_1) \eps_2+\gamma_1(1 +\eps_2)+\gamma_2(2-\eps_1)}  s(-\mathrm{i}u_1)^{\gamma_1} s(-\mathrm{i} u_2)^{\gamma_2} \mathtt{y}_{\gamma_1,\gamma_2}^{1-\eps_1,\eps_2} (a^{-1},1,-\mathrm{i}u_1 ,- \mathrm{i} u_2)u_1^{1-\eps_1} u_2^{\eps_2},
\end{split}
\end{equation}
which follows from~\eqref{res:eq:B} and using the fact that $(1+a^2)F_s$ is invariant under $a \mapsto a^{-1}$. In the above equation, we make the change of variables $u_1 \mapsto \mathrm{i} u_1$ and $u_2 \mapsto \mathrm{i} u_2$ which gives
\begin{equation} \label{asfo:lemproof:exponential:change1}
\begin{split}
&-\frac{1}{(2 \pi \mathrm{i})^2a} \int _{\Gamma_r} \frac{du_1}{u_1}\int_{\Gamma_r} \frac{du_2}{u_2} \frac{ F_{\frac{x_2}{2}}(\nu (\mathrm{i}u_1)) F_{\frac{2n-y_1}{2}} (\nu(\mathrm{i}u_2))}{\mathrm{i}^{\frac{-x_1+y_2}{2}}u_1^{\frac{2n-x_1-1}{2}} v_2^{\frac{y_2-1}{2}}}\lambda_4( u_1, u_2)^m (1+a^2)^2\\
&\times   \sum_{\gamma_1,\gamma_2=0}^1 (-1)^{\eps_1+  \eps_1 \eps_2+\gamma_1(1 +\eps_2)+\gamma_2(2-\eps_1)}  s(u_1)^{\gamma_1} s( u_2)^{\gamma_2} \mathtt{y}_{\gamma_1,\gamma_2}^{1-\eps_1,\eps_2} (a^{-1},1,u_1 ,u_2)u_1^{1-\eps_1} u_2^{\eps_2} .
\end{split}
\end{equation} 
We deform the contour with respect to $u_1$ from $\Gamma_r$ to $\Gamma_{1/r}$ since the integrand is analytic. By noticing that 
\begin{equation}
\begin{split}
\frac{1}{a} \mathtt{y}_{\gamma_1,\gamma_2}^{1-\eps_1,\eps_2} (a^{-1},1,u_1 ,u_2) = \frac{1}{u_1^2} \mathtt{y}_{\gamma_1,\gamma_2}^{\eps_1,\eps_2} (a,1,u_1^{-1} ,u_2),
\end{split}
\end{equation}
which follows from Definition~\ref{def:coefficients} and that
\begin{equation}
u_1^4 \lambda_4(u_1^{-1},u_2) =  \lambda_4(u_1,u_2)
\end{equation}
we find that~\eqref{asfo:lemproof:exponential:change1} under the above contour deformation and the change of variables $u_1 \mapsto u_1^{-1}$ gives
\begin{equation}
\begin{split}~\label{asfo:lemproof:exponential:change2}
&-\frac{1}{(2 \pi \mathrm{i})^2} \int _{\Gamma_r} \frac{du_1}{u_1}\int_{\Gamma_r} \frac{du_2}{u_2} \frac{ F_{\frac{x_2}{2}}(-\nu (\mathrm{i}u_1)) F_{\frac{2n-y_1}{2}} (\nu(\mathrm{i}u_2))}{\mathrm{i}^{\frac{-x_1+y_2}{2}}u_1^{\frac{x_1+1}{2}} v_2^{\frac{y_2-1}{2}}}\lambda_4( u_1, u_2)^m (1+a^2)^2\\
&\times   \sum_{\gamma_1,\gamma_2=0}^1 (-1)^{\eps_1+  \eps_1 \eps_2+\gamma_1(1 +\eps_2)+\gamma_2(2-\eps_1)}  s(u_1)^{\gamma_1} s( u_2)^{\gamma_2} \mathtt{y}_{\gamma_1,\gamma_2}^{\eps_1,\eps_2} (u_1 ,u_2)u_1^{1+\eps_1} u_2^{\eps_2} .
\end{split}
\end{equation}
Using the integral representation of $F_s$ in Lemma~\ref{cuts:lem:F}
to rewrite $F_{\frac{x_2}{2}}(-\nu (\mathrm{i}u_1))$ as $(-1)^{\frac{x_2}{2}}F_{\frac{x_2}{2}}(\nu (\mathrm{i}u_1))$, we proceed in taking the change of variables $u_i=G(\omega_i)$ and follow the simplification procedures  detailed in the proof of Eq.~\eqref{B}.  After these computations we arrive at~\eqref{asfo:lemproof:exponential:change3}.
We follow similar procedures which give~\eqref{asfo:lemproof:exponential:change3} from~\eqref{asfo:lemproof:exponential:change} for $\frac{1}{a}\mathcal{B}_{\eps_1,1-\eps_2}(a^{-1},x_1,2n-x_2,y_1,2n-y_2)$ and  $ \mathcal{B}_{1-\eps_1,1-\eps_2}(a,2n-x_1,2n-x_2,2n-y_1,2n-y_2)$. We find~\eqref{asfo:lemproof:exponential:change4} and~\eqref{asfo:lemproof:exponential:change5}.

\end{proof}

\section{Combinatorial Boundary} \label{section:Combintorial}

In this section, we introduce a candidate for lattice paths which we believe describe the liquid-gas interface. We call these paths~\emph{tree paths}.  We remark that the red-blue particles~\cite{Joh:05} and the DR paths~\cite{LRS:01}(or Schr\"{o}der lattice paths) which have been previously used to study the solid-liquid boundary, are inadequate for giving a good description of the path separating the liquid and gas regions, even when adjusting for the two-periodicity.  These statements are motivated from numerous simulations.

\subsection{Height Function}

Thurston, in \cite{Thu:90}, introduced the notion of the height function for dimer coverings on a bipartite graph which is in correspondence with each dimer covering.  The height function for the square grid is defined as follows: the height is defined on the faces of the graph with the height difference between two faces equal to $\pm 3$ if a dimer covers the shared edge between faces and $\mp 1$ if there is no dimer covering the shared edge between the faces.  The definition of the height function is chosen so that the total change of the height around each vertex is equal to zero. 
We will use the convention that the height at the faces whose center are given by $(2i,0)$ and $(0,2i)$ (which border the Aztec diamond graph) are equal to $2i$ while the heights at the faces whose center are given by $(2n-2i,2n)$ and $(2n,2n-2i)$ are equal to $2i$ for $0 \leq i \leq n$.   With this convention, the height function can be computed on the faces of the Aztec diamond graph for each specific dimer covering.

\subsection{Tree Paths}

As the language of trees contains the words vertex and edge, a vertex on trees will be called a node and as these trees are directed, a tree edge will be called an arrow. We introduce a procedure that builds the trees paths on nodes in $\mathtt{W}_0$ with additional boundary vertices appended to the Aztec diamond graph.
\begin{itemize}
\item For an Aztec diamond of size $n=4m$, there are sources nodes at $\{(2i+1,0)\}_{i=0}^{n-1} \cap W_0$, and $\{(2i+1,2n)\}_{i=0}^{n-1} \cap W_0$ and sink nodes at $\{-1,4i+2 \}_{i=0}^{2m-1}$ and $\{2n+1,4i+4 \}_{i=0}^{2m-2}$.  These sink nodes are outside of the Aztec diamond graph.
\item Suppose that  $x \in \mathtt{W}_0$ is either a source node or incident to an incoming arrow.  Then, there is an outgoing arrow $(x, x \pm 2 e_i)$ if and only if there is a dimer covering the edge $(x,x\pm e_i)$ for $i \in \{1,2\}$. The path terminates when this outgoing arrow is incident to a sink vertex.
\end{itemize}

Similarly, there is a procedure which builds tree paths on nodes in $\mathtt{W}_1$ with additional boundary vertices appended to the Aztec diamond graph.
\begin{itemize}
\item For an Aztec diamond of size $n=4m$, there are sources nodes at $\{(2i+1,0)\}_{i=0}^{n-1} \cap W_1$, and $\{(2i+1,2n)\}_{i=0}^{n-1} \cap W_1$ and sink nodes at $\{-1,4i+4 \}_{i=0}^{2m-2}$ and $\{2n+1,4i+2 \}_{i=0}^{2m-1}$.  These sink nodes are outside of the Aztec diamond graph.
\item Suppose that $x \in \mathtt{W}_1$ is either a source vertex or incident to an incoming arrow.  Then, there is an outgoing arrow $(x, x \pm 2 e_i)$ if and only if there is a dimer covering the edge $(x,x\pm e_i)$ for $i \in \{1,2\}$. The path terminates when this outgoing arrow is incident to a sink vertex.

\end{itemize}
This construction is very much related to Temperley's correspondence~\cite{Tem:74}.  Fig.~\ref{fig:treepathexample} shows an example of one tree path on a relatively small Aztec diamond.
\begin{center}
\begin{figure}
\includegraphics[height=3.0in]{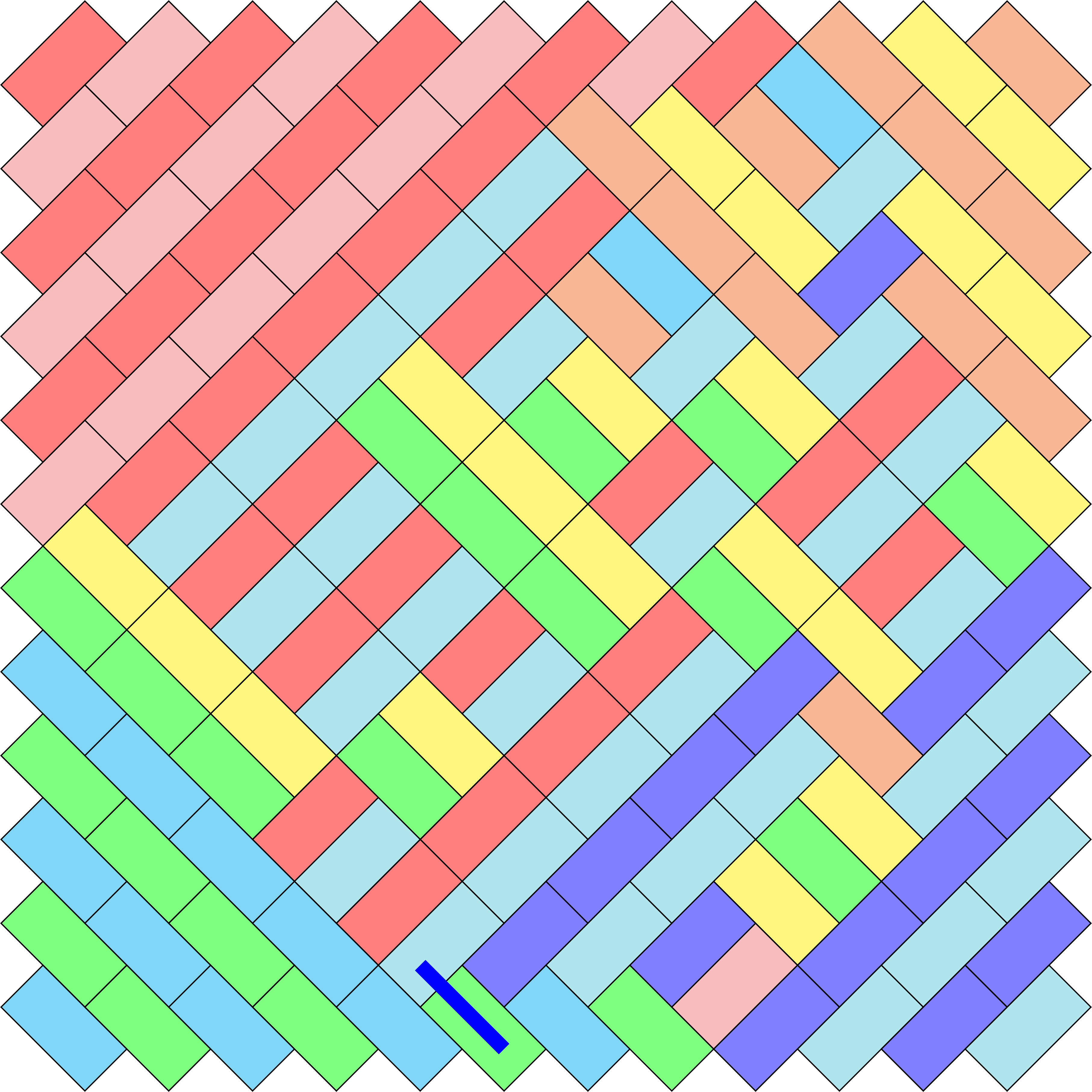}
\includegraphics[height=3.0in]{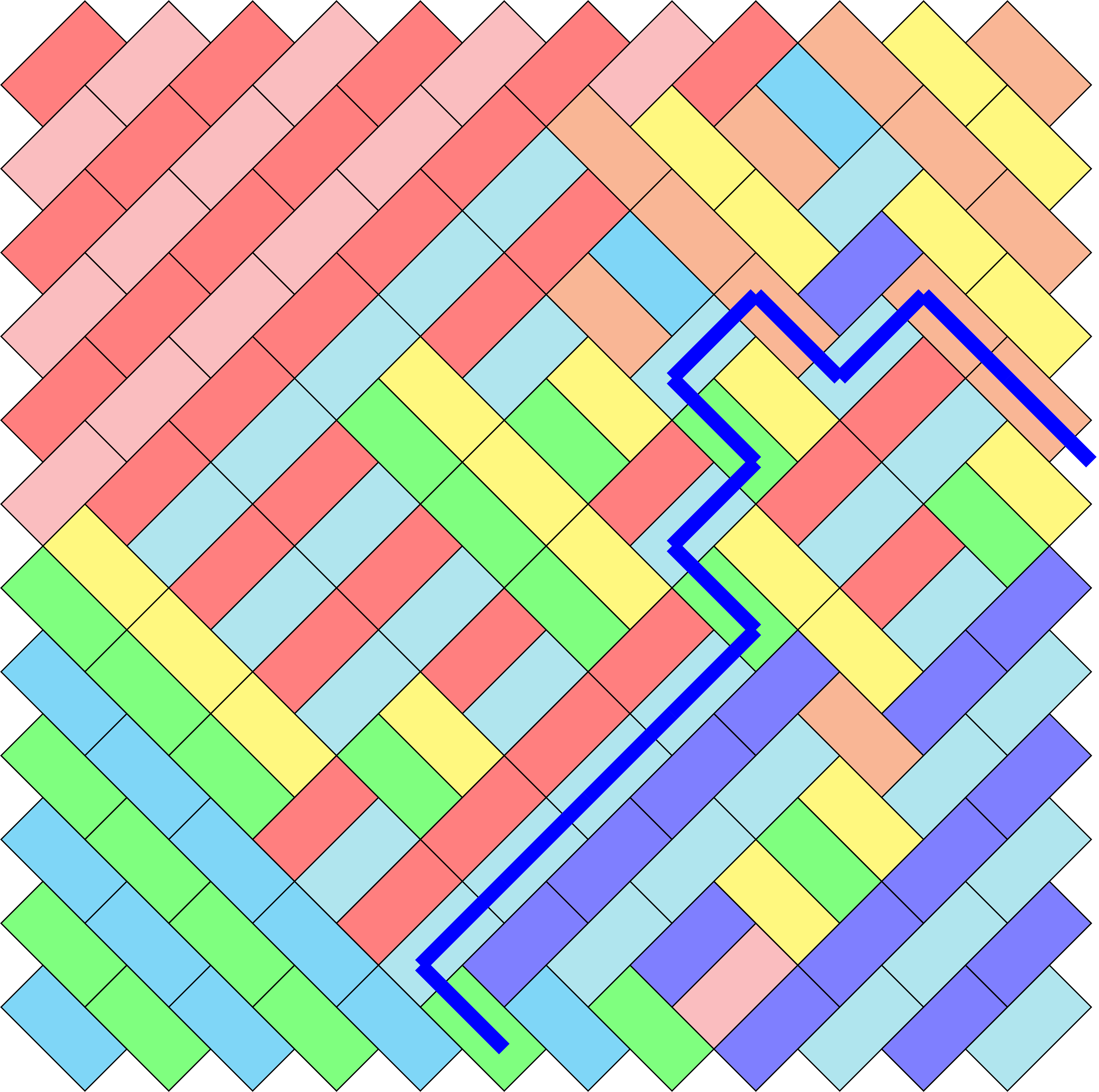}
\caption{An example of one tree path for a domino tiling of an Aztec diamond of size 12. This tree path starts at $(13,1)$ drawn in blue. The left figure shows the tree path after one step, that is a line from $(13,1)$ to $(11,3)$. The right figure shows the completed tree path.  }
\label{fig:treepathexample}
\end{figure}
\end{center}

\begin{center}
\begin{figure}
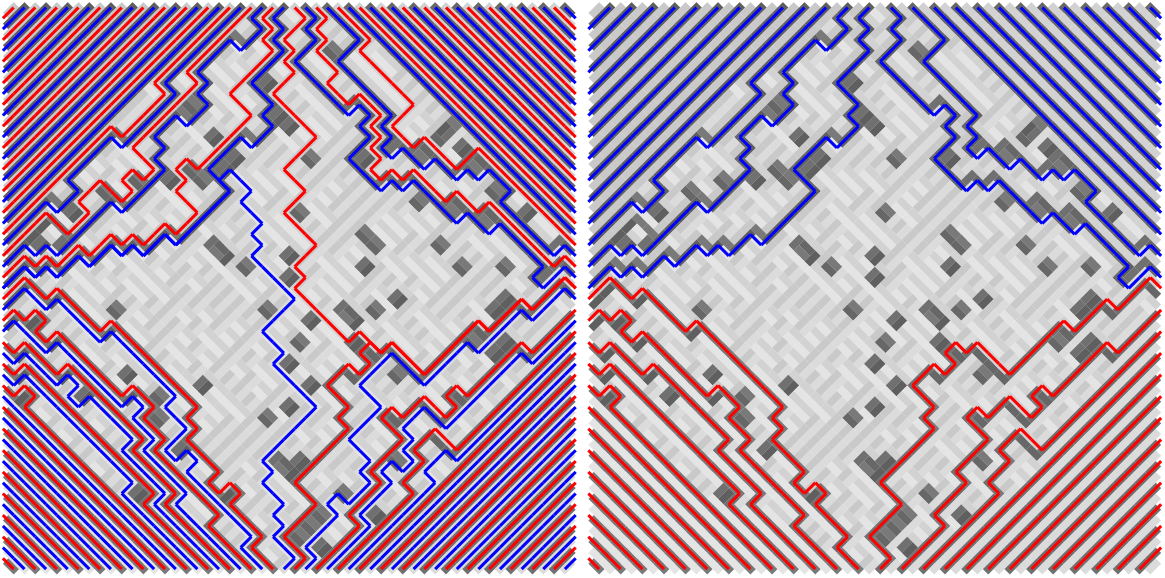

\includegraphics[height=3.0in]{ttpn52a05bwtrees2.pdf}
\includegraphics[height=3.0in]{ttpn52a05bwtrees.pdf}
\caption{ Two different drawings of an Aztec diamond of size 52 with $a=0.5$ with the two-periodic weighting.  In each figure, the red paths are the tree paths of type $W_0$ and the blue paths are the tree paths of type $W_1$. 
The left figure shows tree paths started from every white vertex on both the top and bottom boundaries while the right figure shows half of these tree paths. See the on-line version for colors.}
\label{fig:treepaths}
\end{figure}
\end{center}

We refer to each path built from each construction as a \emph{tree path}.  We shall also distinguish between tree paths on nodes of type $W_0$ and tree paths on nodes of type $W_1$.

\begin{lemma}
For the set of tree paths built from the above construction and assuming that  the edges $((1,0),(0,1))$, $((2n-1,0),(2n,1))$, $((0,2n-1),(1,2n))$ and $((2n-1,2n),(2n,2n-1))$ are covered by dimers, then
\begin{enumerate}
\item exactly one pair of tree paths of type $W_0$ coalesces and exactly one pair of tree paths of type $W_1$ coalesces, 
\item the tree paths which coalesce have source nodes on opposite boundaries, 
\item apart from these four tree paths, the rest of the tree paths are non-intersecting,  
\item the tree paths begin and end at the same height,
\item each tree path has zero winding when it terminates.
\end{enumerate}
\end{lemma}
The assumption in the lemma is an event which occurs with probability extremely close to 1, i.e. with probability $1-Ce^{-n^2}$ for a positive constant $C$. 

\begin{proof}

From the above construction, tree paths only terminate at the sink nodes. Since every white vertex in the Aztec diamond is covered by a dimer, then there is potentially one outgoing arrow  from a white vertex depending on whether the white vertex is incident to an incoming arrow.  A white vertex with an incoming arrow does not have an outgoing arrow if it is not covered by a dimer and so this white vertex must be a sink node. 

We next prove that two tree paths whose source vertices start on the same boundary are non-intersecting. It is immediate from the construction that two tree paths cannot change vertex type and a tree path can only coalesce with a tree path of the same type otherwise there is a violation of the dimer covering.  Suppose that two tree paths of type $W_0$ which start from the same boundary coalesce.  This immediately implies that the tree paths of type $W_1$ which start at nodes in between the tree paths of type $W_0$, must terminate before reaching either the left or right boundaries.  This a contradiction.   Therefore, the set of tree paths whose source nodes are on the bottom boundary of the Aztec diamond are non-intersecting.  A similar statement holds for the top boundary. 

By construction, each tree path only increases in height if it winds.  More precisely, each type of different arrow on the same tree path has the same height regardless of their locations, provided that the tree path does not wind.  This follows directly from the definition of the height function of the dimer model.  By the underlying dimer model, if two tree paths coalesce then they must have had the same height, otherwise, there is a violation with the height function rule and it is impossible for tree paths to coalesce with another tree path with a different winding number.  It is also impossible for any tree path to terminate if it has a non-zero winding.  Therefore, each tree path has zero overall winding which means that the tree paths have the same initial and terminating height. 

Since we assume that the edges $((1,0),(0,1))$, $((2n-1,0),(2n,1))$, $((0,2n-1),(1,2n))$ and $((2n-1,2n),(2n,2n-1))$ are covered by dimers,
 then there are exactly $2n$ source vertices and $2n-1$ sink vertices for tree paths of type $\mathtt{W}_0$ and similarly for tree paths of type $\mathtt{W}_1$.  This means that at least two pairs of tree paths must coalesce.  Each pair of tree paths which coalesce must start from opposite boundaries and must have the same initial height function since each path cannot spiral - there is no total winding on each path.  The heights of the source vertices on the bottom boundary increase from left to right whereas the heights of the sources vertices on the top boundary decrease from left to right. Since the paths starting from different heights cannot cross, it follows that the only pair of $W_0$ tree paths which coalesce have sources vertices $(4m+1,0)$ and $(4m-1,4m)$ and the only pair of $W_1$ tree paths which coalesce have source vertices $(4m+3,0)$ and $(4m+1,8m)$.  
\end{proof}
\begin{center}
\begin{figure}
\includegraphics[height=6.5in]{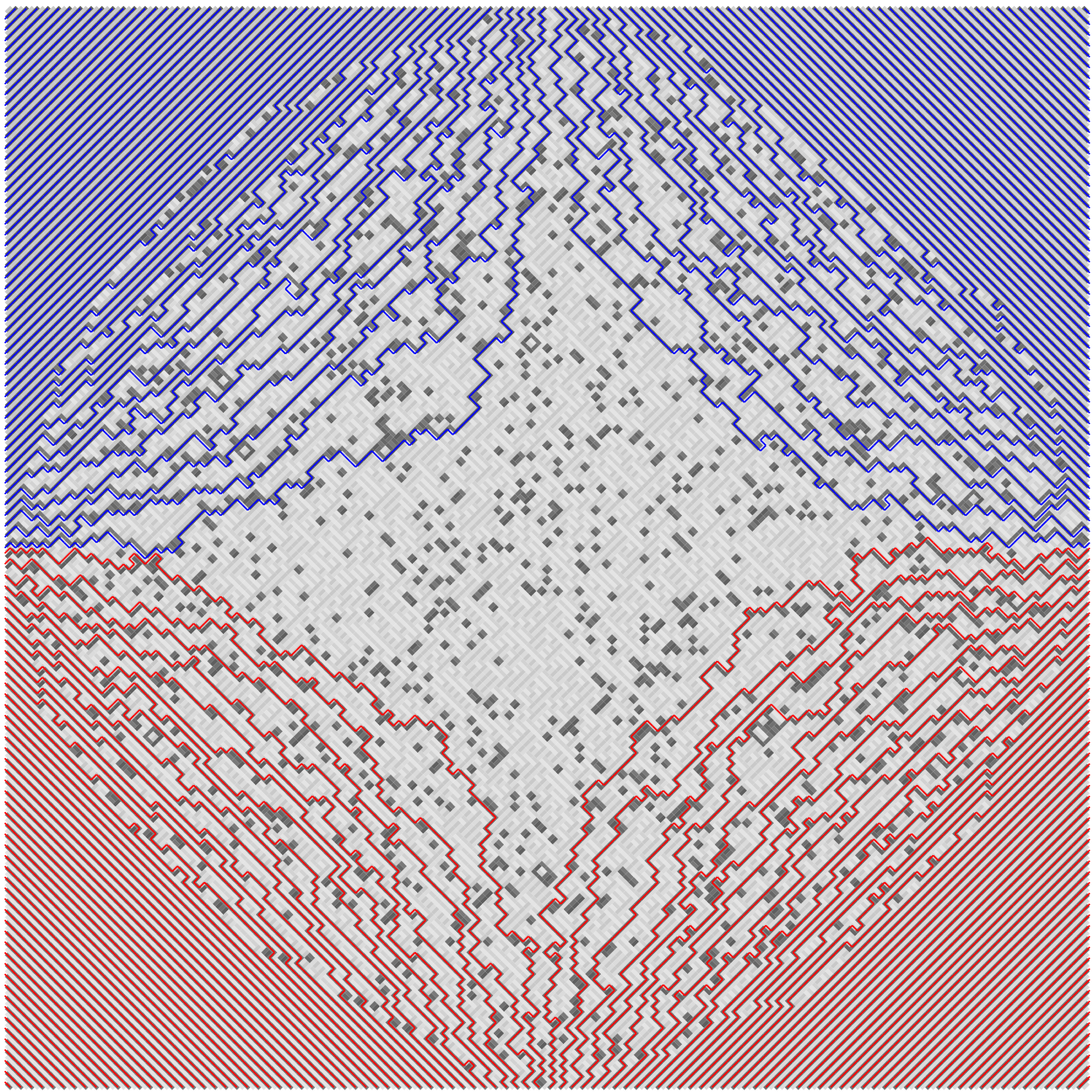}
\caption{The simulation considered in Fig.~\ref{fig:tpn200} with half of the tree paths overlaid, with the same convention as found in Fig.~\ref{fig:treepaths}. See the on-line version for colors. }
\label{fig:treepaths2}
\end{figure}
\end{center}

We now discuss heuristically  a possible candidate for the paths which separate the liquid-gas boundary. 
From the above lemma, exactly two pairs of tree paths coalesce. If we choose tree paths of one vertex type on the bottom boundary and the other vertex type on the top boundary, then there is no coalescence, that is, we remove half the tree paths; see Fig.~\ref{fig:treepaths}. 

From many simulations, it appears that for $0<a<1$ and for $n=4m$, we should choose the tree paths starting from $(4k+1,0)$ and $(4k+3,8m)$ for $0\leq k \leq 2m-2$.  We believe that the tree paths started from $(4m-3,0)$, $(4m+1,0)$, $(4m-1,8m)$ and $(4m+3,8m)$ separate the liquid and gas region.

Large simulations seem to show that this choice of tree paths give a good description of the paths that seem to be present in Fig.~\ref{fig:tpn200}.  Fig.~\ref{fig:treepaths2} shows this choice of tree paths, overlaid on that simulation.   Note that the other choice of tree paths seems to show a path crossing the gas region. 

We remark that the tree paths give good descriptions of the dark lines emerging in Fig.~\ref{fig:tpn200}.  There, due to the `clumping' dominoes, it is not clear exactly where the path is.  As remarked earlier in the paper, to compute the correlations between parts of these tree paths does not seem possible with the local statistical information we currently have available.

\begin{appendix}
\section{Limit shape}\label{limitshape} 
The  limit shape of the two-periodic Aztec diamond with corners $(-1,-1),(1,-1),(1,1)$ and $(-1,1)$ with $c=a/(1+a^2)$ is given by the equation
\begin{equation}
\begin{split} &64 c^6 \left(-1+\xi _1^2\right) \left(-1+\xi _2^2\right)-\left(\xi _1^4+\left(-1+\xi _2^2\right){}^2-2 \xi _1^2 \left(1+\xi _2^2\right)\right){}^2\\ &
-16 c^4 \left(3 \left(-1+\xi _2^2\right){}^2+\xi _1^2 \left(-6+27 \xi _2^2-20 \xi _2^4\right)+\xi _1^4 \left(3-20 \xi _2^2+16 \xi _2^4\right)\right) \\
&
-4 c^2 \left(3 \left(-1+\xi _2^2\right){}^3+\xi _1^6 \left(3+8 \xi _2^2\right)+\xi _1^4 \left(-9+13 \xi _2^2-16 \xi _2^4\right)+\xi _1^2 \left(9-30 \xi _2^2+13 \xi _2^4+8 \xi _2^6\right)\right) =0.
\end{split}
\end{equation}

\end{appendix}

\bibliographystyle{plain}
\bibliography{ref}

\begin{thebibliography}{10}

\bibitem{AR:05}
David Allison and Nicolai Reshetikhin.
\newblock Numerical study of the 6-vertex model with domain wall boundary
  conditions.
\newblock {\em Annales de l'Institut Fourier}, 55(6):1847--1869, 2005.

\bibitem{BKMM:07}
J.~Baik, T.~Kriecherbauer, K.~T.-R. McLaughlin, and P.~D. Miller.
\newblock {\em Discrete orthogonal polynomials}, volume 164 of {\em Annals of
  Mathematics Studies}.
\newblock Princeton University Press, Princeton, NJ, 2007.
\newblock Asymptotics and applications.

\bibitem{BC:11}
Alexei Borodin and Ivan Corwin.
\newblock Macdonald processes.
\newblock {\em Probab. Theory Related Fields}, 158(1-2):225--400, 2014.

\bibitem{BF:08}
Alexei Borodin and Patrik~L. Ferrari.
\newblock Anisotropic {G}rowth of {R}andom {S}urfaces in 2 + 1 {D}imensions.
\newblock {\em Comm. Math. Phys.}, 325(2):603--684, 2014.

\bibitem{BV:09}
Alexei Borodin and Vadim Gorin.
\newblock Shuffling algorithm for boxed plane partitions.
\newblock {\em Adv. Math.}, 220(6):1739--1770, 2009.

\bibitem{BGR:10}
Alexei Borodin, Vadim Gorin, and Eric~M. Rains.
\newblock {$q$}-distributions on boxed plane partitions.
\newblock {\em Selecta Math. (N.S.)}, 16(4):731--789, 2010.

\bibitem{BR:05}
Alexei Borodin and Eric~M. Rains.
\newblock Eynard-{M}ehta theorem, {S}chur process, and their {P}faffian
  analogs.
\newblock {\em J. Stat. Phys.}, 121(3-4):291--317, 2005.

\bibitem{BBCCR:15}
C\'{e}dric {Boutillier}, J\'{e}r\'{e}mie {Bouttier}, Guillaume {Chapuy}, Sylvie
  {Corteel}, and Sanjay {Ramassamy}.
\newblock Dimers on rail yard graphs.
\newblock arXiv:1504.05176, 2015.

\bibitem{BCC:14}
J\'{e}r\'{e}mie {Bouttier}, Guillaume {Chapuy}, and Sylvie {Corteel}.
\newblock {From {A}ztec diamonds to pyramids: steep tilings}.
\newblock 1407.0665, 2014.

\bibitem{CS:04}
Gabriel~D. Carroll and David Speyer.
\newblock The cube recurrence.
\newblock {\em Electron. J. Combin.}, 11(1):Research Paper 73, 31 pp.
  (electronic), 2004.

\bibitem{CJY:12}
Sunil Chhita, Kurt Johansson, and Benjamin Young.
\newblock Asymptotic domino statistics in the {A}ztec diamond.
\newblock {\em Ann. Appl. Probab.}, 25(3):1232--1278, 2015.

\bibitem{CY:13}
Sunil Chhita and Benjamin Young.
\newblock Coupling functions for domino tilings of {A}ztec diamonds.
\newblock {\em Adv. Math.}, 259:173--251, 2014.

\bibitem{CKP:01}
Henry Cohn, Richard Kenyon, and James Propp.
\newblock A variational principle for domino tilings.
\newblock {\em J. Amer. Math. Soc.}, 14(2):297--346 (electronic), 2001.

\bibitem{FSG:14}
Philippe Di~Francesco and Rodrigo Soto-Garrido.
\newblock Arctic curves of the octahedron equation.
\newblock arXiv:1402.4493, 2014.

\bibitem{EKLP:92}
N.~Elkies, G.~Kuperberg, M.~Larsen, and J.~Propp.
\newblock Alternating-sign matrices and domino tilings (part i).
\newblock {\em Journal of Algebraic Combinatorics}, 1(2):111--132, 1992.

\bibitem{Fab}
Carel Faber.
\newblock Personal communication.

\bibitem{Joh:05}
Kurt Johansson.
\newblock The arctic circle boundary and the {A}iry process.
\newblock {\em Ann. Probab.}, 33(1):1--30, 2005.

\bibitem{Kas:61}
P.W. Kasteleyn.
\newblock The statistics of dimers on a lattice: I. the number of dimer
  arrangements on a quadratic lattice.
\newblock {\em Physica}, 27:1209--1225, 1961.

\bibitem{Ken:97}
Richard Kenyon.
\newblock Local statistics of lattice dimers.
\newblock {\em Ann. Inst. H. Poincar\'e Probab. Statist.}, 33(5):591--618,
  1997.

\bibitem{KO:06}
Richard Kenyon and Andrei Okounkov.
\newblock Planar dimers and {H}arnack curves.
\newblock {\em Duke Math. J.}, 131(3):499--524, 2006.

\bibitem{KO:07}
Richard Kenyon and Andrei Okounkov.
\newblock Limit shapes and the complex {B}urgers equation.
\newblock {\em Acta Math.}, 199(2):263--302, 2007.

\bibitem{KOS:06}
Richard Kenyon, Andrei Okounkov, and Scott Sheffield.
\newblock Dimers and amoebae.
\newblock {\em Ann. of Math. (2)}, 163(3):1019--1056, 2006.

\bibitem{KP:13}
Richard Kenyon and Robin Pemantle.
\newblock Double-dimers, the {I}sing model and the hexahedron recurrence.
\newblock arXiv:1308.2998, 2013.

\bibitem{LRS:01}
Michael Luby, Dana Randall, and Alistair Sinclair.
\newblock Markov chain algorithms for planar lattice structures.
\newblock {\em SIAM J. Comput.}, 31(1):167--192 (electronic), 2001.

\bibitem{MPW:63}
E.~W. Montroll, R.~B. Potts, and J.~C. Ward.
\newblock Correlations and spontaneous magnetization of the two-dimensional
  {I}sing model.
\newblock {\em Journal of Mathematical Physics}, 4:308--322, 1963.

\bibitem{NHB:84}
B.~Nienhuis, H.~J. Hilhorst, and H.~W.~J. Bl{\"o}te.
\newblock Triangular {SOS} models and cubic-crystal shapes.
\newblock {\em J. Phys. A}, 17(18):3559--3581, 1984.

\bibitem{OR:03}
Andrei Okounkov and Nikolai Reshetikhin.
\newblock Correlation function of {S}chur process with application to local
  geometry of a random 3-dimensional {Y}oung diagram.
\newblock {\em J. Amer. Math. Soc.}, 16(3):581--603 (electronic), 2003.

\bibitem{PS:05}
T.~Kyle Petersen and David Speyer.
\newblock An arctic circle theorem for {G}roves.
\newblock {\em J. Combin. Theory Ser. A}, 111(1):137--164, 2005.

\bibitem{Pet:13}
Leonid Petrov.
\newblock Asymptotics of random lozenge tilings via {G}elfand–--{T}setlin
  schemes.
\newblock {\em Probability Theory and Related Fields}, pages 1--59, 2013.

\bibitem{PS:02}
Michael Pr{\"a}hofer and Herbert Spohn.
\newblock Scale invariance of the {PNG} droplet and the {A}iry process.
\newblock {\em J. Statist. Phys.}, 108(5-6):1071--1106, 2002.
\newblock Dedicated to David Ruelle and Yasha Sinai on the occasion of their
  65th birthdays.

\bibitem{Res:10}
N.~Reshetikhin.
\newblock Lectures on the integrability of the six-vertex model.
\newblock In {\em Exact methods in low-dimensional statistical physics and
  quantum computing}, pages 197--266. Oxford Univ. Press, Oxford, 2010.

\bibitem{She:07}
Scott Sheffield.
\newblock Gaussian free fields for mathematicians.
\newblock {\em Probab. Theory Related Fields}, 139(3-4):521--541, 2007.

\bibitem{Tem:74}
H.~N.~V. Temperley.
\newblock In {C}ombinatorics: {P}roceedings of the {B}ritish combinatorial
  conference.
\newblock {\em London Mathematical Society Lecture Notes Series}, pages
  202--204, 1974.

\bibitem{Thu:90}
William~P. Thurston.
\newblock Conway's tiling groups.
\newblock {\em Amer. Math. Monthly}, 97(8):757--773, 1990.

\end{thebibliography}

\end{document}